\documentclass[reqno,11pt]{amsart}

\usepackage[pagebackref=true]{hyperref}
\hypersetup{   pdfborder={0 0 0},
               colorlinks = false}

\renewcommand*{\backref}[1]{}
\renewcommand*{\backrefalt}[4]{%
    \ifcase #1 %
        \emph{(No citations.)}%
    \or
        \emph{(One citation on page #2.)}%
    \else
        \emph{(#1 citations on pages #2.)}%
    \fi
}

\usepackage[all]{xy}
\usepackage[usenames,dvipsnames]{xcolor}
\usepackage{tikz}
\usepackage{tikz-cd}
\usepackage{tikz-3dplot}
\usepackage{pgfplots}
\usetikzlibrary{3d}
\usetikzlibrary{shapes,decorations.pathreplacing}
\usetikzlibrary{calc}
\usetikzlibrary{backgrounds}
\usetikzlibrary{automata}
\usetikzlibrary{positioning}
\usetikzlibrary{decorations.markings}
\usetikzlibrary{arrows}
\usetikzlibrary{scopes}
\usetikzlibrary{intersections}
\tikzstyle myBG=[line width=3pt,opacity=1]

\newcommand\pgfmathsinandcos[3]{%
  \pgfmathsetmacro#1{sin(#3)}%
  \pgfmathsetmacro#2{cos(#3)}%
}
\newcommand\LongitudePlane[3][current plane]{%
  \pgfmathsinandcos\sinEl\cosEl{#2} 
  \pgfmathsinandcos\sint\cost{#3} 
  \tikzset{#1/.style={cm={\cost,\sint*\sinEl,0,\cosEl,(0,0)}}}
}
\newcommand\LatitudePlane[3][current plane]{%
  \pgfmathsinandcos\sinEl\cosEl{#2} 
  \pgfmathsinandcos\sint\cost{#3} 
  \pgfmathsetmacro\yshift{\cosEl*\sint}
  \tikzset{#1/.style={cm={\cost,0,0,\cost*\sinEl,(0,\yshift)}}} %
}
\newcommand\DrawLongitudeCircleFront[2][1]{
  \LongitudePlane{\elevationangle}{#2}
  \tikzset{current plane/.prefix style={scale=#1}}
  \pgfmathsetmacro\angVis{atan(sin(#2)*cos(\elevationangle)/sin(\elevationangle))} %
  \draw[current plane, very thick] (\angVis:1) arc (\angVis:\angVis+180:1);
}
\newcommand\DrawLongitudeCircleBack[2][1]{
  \LongitudePlane{\elevationangle}{#2}
  \tikzset{current plane/.prefix style={scale=#1}}
  \pgfmathsetmacro\angVis{atan(sin(#2)*cos(\elevationangle)/sin(\elevationangle))} %
  \draw[current plane, thick, dashed, color=black!50] (\angVis-180:1) arc (\angVis-180:\angVis:1);
}

\newcommand\DrawLatitudeCircle[2][1]{
  \LatitudePlane{\elevationangle}{#2}
  \tikzset{current plane/.prefix style={scale=#1}}
  \pgfmathsetmacro\sinVis{sin(#2)/cos(#2)*sin(\elevationangle)/cos(\elevationangle)}
  \pgfmathsetmacro\angVis{asin(min(1,max(\sinVis,-1)))}
  \draw[current plane, dotted, color=black!30] (\angVis:1) arc (\angVis:-\angVis-180:1);
  \draw[current plane, dotted, color=black!20] (180-\angVis:1) arc (180-\angVis:\angVis:1);
}

\definecolor{cof}{RGB}{219,144,71}
\definecolor{pur}{RGB}{186,146,162}
\definecolor{greeo}{RGB}{91,173,69}
\definecolor{greet}{RGB}{52,111,72}

\tdplotsetmaincoords{70}{165}

\usepackage{amssymb}
\usepackage{amsgen}
\usepackage{amsmath}
\usepackage{amsthm}
\usepackage{mathrsfs}
\usepackage{cite}
\usepackage{amsfonts}

\newcommand{\link}{\mathop{\mathrm{lk}}\nolimits}
\newcommand{\bd}{\partial}
\newcommand{\sgn}{\mathop{\mathrm{sgn}}}
\newcommand{\gldim}{\mathop{\mathrm{gl.~dim}}}
\newcommand{\Stab}{\mathop{\mathrm {Stab}}}
\newcommand{\Cliq}{\mathsf{Cliq}}

\renewcommand{\to}{\longrightarrow}

\newcommand{\rad}{\mathop{\mathrm{rad}}\nolimits}

\newcommand{\rk}{\mathop{\mathrm{rk}}\nolimits}
\newcommand{\R}{{\mathrel{\mathscr R}}} 

\newcommand{\res}{\mathop{\mathrm{res}}\nolimits}

\newcommand{\inv}{^{-1}}
\newcommand{\p}{\varphi}

\newcommand{\ov}[1]{\ensuremath{\overline {#1}}}
\newcommand{\til}[1]{\ensuremath{\widetilde {#1}}}

\newcommand{\wh}{\widehat}

\newcommand{\Hom}{\mathop{\mathrm{Hom}}\nolimits}
\newcommand{\Ext}{\mathop{\mathrm{Ext}}\nolimits}

\newcommand{\CW}{\Sigma}

\def\AAA{\mathcal A} 
\def\LLL{\mathcal L} 
\def\FFF{\mathcal F} 

\def\source{s}
\def\target{t}

\newcommand{\cdim}{\mathop{\mathrm{cd}}\nolimits}
\newtheorem{Thm}{Theorem}[section]
\newtheorem{Prop}[Thm]{Proposition}

\newtheorem{Lemma}[Thm]{Lemma}
{\theoremstyle{definition}
\newtheorem{Def}[Thm]{Definition}}
{\theoremstyle{remark}
\newtheorem{Rmk}[Thm]{Remark}}
\newtheorem{Cor}[Thm]{Corollary}
{\theoremstyle{remark}
\newtheorem{Example}[Thm]{Example}}

{\theoremstyle{remark}
}
{\theoremstyle{remark}
}
{\theoremstyle{remark}
}

\numberwithin{equation}{section}
\numberwithin{figure}{section}

\usepackage{graphicx}
\newcommand{\FeasibleSets}{\mathscr F}
\newcommand{\ext}{\operatorname{ext}} 
\newcommand{\SeparationSet}{S} 
\newcommand{\OIG}{\mathcal L} 
\newcommand{\conv}{\operatorname{conv}} 
\newcommand{\con}{\operatorname{con}}
\newcommand{\Flat}{\operatorname{flat}}

\title[Cell complexes and representation theory]{Cell complexes, poset topology and the representation theory of algebras arising in algebraic combinatorics and discrete geometry}

\author{Stuart Margolis}
\address[S.~Margolis]{%
    Department of Mathematics\\
    Bar Ilan University\\
    52900 Ramat Gan\\
    Israel}
\email[S.~Margolis]{margolis@math.biu.ac.il}

\author{Franco Saliola}
\address[F.~Saliola]{
D{\'e}partement de Math{\'e}matiques -- LaCIM\\
Universit{\'e} du Qu{\'e}bec {\`a} Montr{\'e}al\\
C.P. 8888, Succursale Centre-Ville\\
Montr{\'e}al, Qu{\'e}bec  H3C 3P8\\
Canada
}
\email[F.~Saliola]{saliola.franco@uqam.ca}

\author{Benjamin Steinberg}
\address[B.~Steinberg]{%
    Department of Mathematics\\
    City College of New York\\
    Convent Avenue at 138th Street\\
    New York, New York 10031\\
    USA}
\email[B.~Steinberg]{bsteinberg@ccny.cuny.edu}

\thanks{This work was partially supported by a grant from the Simons Foundation (\#245268 to Benjamin Steinberg) and the Binational Science Foundation of Israel and the US (\#2012080 to Stuart Margolis and Benjamin Steinberg).  Steinberg was also supported in part by a CUNY Collaborative Research Incentive Grant, by  NSA MSP \#H98230-16-1-0047, a PSC CUNY grant and a Fulbright Scholarship. Saliola was supported by an NSERC grant}
\date{Submitted, December 7, 2015, Revised September 26, 2018 }
\makeindex

\usepackage[refpage]{nomencl}
\makenomenclature
\usepackage{etoolbox}
\renewcommand\nomgroup[1]{%
  \bigskip
  \item[\bfseries
  \ifstrequal{#1}{A}{General}{%
  \ifstrequal{#1}{C}{CW complexes}{%
  \ifstrequal{#1}{H}{Hyperplane arrangements}{%
  \ifstrequal{#1}{L}{Left regular bands}{%
  \ifstrequal{#1}{M}{Matroids}{%
  \ifstrequal{#1}{P}{Posets}{%
  \ifstrequal{#1}{Q}{Quivers}{%
  \ifstrequal{#1}{R}{Rings and modules}{%
  \ifstrequal{#1}{S}{Simplicial complexes}{%
  }}}}}}}}}%
] \leavevmode \smallskip}

\begin{document}

\begin{abstract}
In recent years it has been noted that a number of combinatorial structures such as real and complex hyperplane arrangements, interval greedoids, matroids and oriented matroids have the structure of a finite monoid called a left regular band. Random walks on the monoid model a number of interesting Markov chains such as the Tsetlin library and riffle shuffle. The representation theory of left regular bands then comes into play and has had a major influence on both the combinatorics and the probability theory associated to such structures. In a recent paper, the authors established a close connection between algebraic and
combinatorial invariants of a left regular band by showing that certain
homological invariants of the algebra of a left regular band coincide with the
cohomology of order complexes of posets naturally associated to the left
regular band.

The purpose of the present monograph is to further develop and deepen the
connection between left regular bands and poset topology. This allows us to compute finite projective resolutions of all simple modules of unital left regular band algebras over fields and much more. In the process, we are led to define the class
of CW left regular bands as the class of left regular bands whose associated posets
are the face posets of regular CW complexes. Most of the examples that have arisen in the
literature belong to this class. A new and important class of examples is
a left regular band structure on the face poset of a CAT(0) cube complex.
Also, the recently introduced notion of a COM (complex of oriented matroids or conditional oriented matroid) fits nicely into our setting and includes CAT(0) cube complexes and certain more general CAT(0) zonotopal complexes. A fairly complete picture of the representation theory for CW left regular bands is obtained.
\end{abstract}

\maketitle
\tableofcontents

\section{Introduction}
We begin with a brief history of the background and motivation for the topics of this monograph before delving into the new results and the organization of the paper.  An announcement of the results herein appears in~\cite{MSSFPSAC}.

\subsubsection*{Application to Markov chains}
In a highly influential paper~\cite{BHR}, Bidigare, Hanlon and Rockmore showed
that a number of popular Markov chains, including the Tsetlin library and the
riffle shuffle, are random walks on the faces of a hyperplane arrangement (the
braid arrangement for these two examples).
More importantly, they showed that the representation theory of the monoid of
faces, where the monoid structure on the faces of a central hyperplane
arrangement is given by the Tits projections~\cite{Titsappendix}, could be used
to analyze these Markov chains and, in particular, to compute the spectrum of
their transition operators.

The face monoid of a hyperplane arrangement satisfies the semigroup identities $x^2 = x$ and $xyx = xy$. Semigroups satisfying these identities are known in the literature as left regular bands (although they were studied early on by Sch\"utzenberger~\cite{Schutzlrb} under the more descriptive name ``\textit{treillis gauches},'' translated by G.~Birkhoff in his Math Review as skew lattices, a term which nowadays has a different meaning). Brown developed~\cite{Brown1,Brown2} a theory of random walks on finite left regular bands. He gave numerous examples that do not come from hyperplane arrangements, as well as examples of hyperplane walks that could more easily be modeled on simpler left regular bands. For example, Brown considered random walks on bases of matroids. Brown used the representation theory of left regular bands to extend the spectral results of Bidigare, Hanlon and Rockmore~\cite{BHR} and gave an algebraic proof of the diagonalizability of random walks on left regular bands. Some of these results, including a simpler proof of diagonalizability, can be found in~\cite[Chapter~14]{repbook} with a textbook presentation.

Brown's theory has since been used and further developed by numerous authors.
Diaconis highlighted hyperplane face monoid and left regular band walks in his
1998 ICM lecture~\cite{DiaconisICM}.
Bj\"orner used it to develop the theory of random walks on complex hyperplane
arrangements and interval greedoids~\cite{bjorner1,bjorner2}.
Athanasiadis and Diaconis revisited random walks on hyperplane face monoids and left
regular bands in~\cite{DiaconisAthan}.
Chung and Graham considered further left regular band random walks associated
to graphs in~\cite{GrahamLRB}.
Saliola and Thomas proposed a definition of oriented interval greedoids by
generalizing the left regular bands associated to oriented matroids and
antimatroids~\cite{SaliolaThomas}.
A detailed study of symmetrized versions of random walks on hyperplane face
monoids was initiated in recent work of Reiner, Saliola and
Welker~\cite{RSW2014}; see also~\cite{SaliolaSym}.
Left regular bands have also appeared in Lawvere's work in topos theory~\cite{graphic1,moregraphic} under the name ``graphic monoids.''

The representation theory of hyperplane face monoids is also closely connected to Solomon's descent algebra~\cite{SolomonDescent}.  Bidigare showed in his thesis~\cite{Bidigarethesis} (see also~\cite{Brown2}) that if $W$ is a finite Coxeter group and $\AAA_W$ is the associated
reflection arrangement, then the descent algebra of $W$  is the (opposite of the) algebra of invariants for the action of $W$ on the algebra of the face monoid of $\AAA_W$. This, together with his study of the representation theory of hyperplane face monoids~\cite{Saliolahyperplane}, allowed the second author to compute the quiver of the descent algebra in types A and B~\cite{SaliolaDescent} (see also~\cite{schocker}),
and to compute the Loewy length of the descent algebra of type D~\cite{SaliolaTypeD}.

To further emphasize the scope and applicability of the theory of left regular
bands, we remark that the original version of Brown's book on
buildings~\cite{Brown:book1} makes no mention of the face monoid of
a hyperplane arrangement, whereas it plays a prominent role in the new
edition~\cite{Brown:book2}.
Hyperplane face monoids also have a salient position in the work of Aguiar and
Mahajan on combinatorial Hopf algebras and Hopf monoids in
species~\cite{Aguiar,aguiarspecies}. Their new work~\cite{THA}, which appeared
after the first version of this text was written, makes extensive use of
hyperplane face semigroups.
Among other things, they advance an intriguing relationship between
hyperplane face semigroups and the theory of
operads~\cite[Section~15.9]{THA}. We attempt to briefly explain the
relationship here without assuming familiarity with operads.

The starting point is the observation that several results concerning operads
can be connected to properties of the braid arrangement or its face semigroup.
For instance, one can associate certain algebras with an operad, and when this
is done with two prominent examples of operads---$\mathsf{Com}$ and
$\mathsf{Lie}$---one obtains the incidence algebra of the intersection lattice of the braid arrangement and
the monoid algebra of the hyperplane face monoid, respectively. Properties of these operads descend
to properties of these algebras: for instance, there is a notion of Koszul
duality for operads, and it turns out that $\mathsf{Com}$ and $\mathsf{Lie}$
are Koszul dual to each other. In other words, the fact that the (opposite)
incidence algebra of the intersection lattice of the braid arrangement is the Koszul dual to the
corresponding hyperplane face monoid can be seen as a consequence of
a relationship between $\mathsf{Com}$ and $\mathsf{Lie}$.
It turns out that this relationship persists when these notions are extended to
any arrangement $\mathcal A$: in this more general setting,  
$\mathsf{Com}[\mathcal A]$ and $\mathsf{Lie}[\mathcal A]$ are again Koszul dual
to each other and this fact results in the Koszul duality between the
(opposite) incidence algebra of the intersection lattice of $\mathcal A$ and the algebra of the hyperplane face monoid of
$\mathcal A$.  This Koszul duality is proved in Theorem~\ref{t:Koszul} for a much larger class
of left regular band algebras, leaving open the natural question of whether there is an operadic proof in this more general setting.

\subsubsection*{Techniques from the topology of hyperplane arrangements}

Using the topology of hyperplane arrangements, Brown and Diaconis~\cite{DiaconisBrown1} found resolutions of the simple modules for the face monoid that were later shown by the second author to be the minimal projective resolutions~\cite{Saliolahyperplane}. Brown and Diaconis used these resolutions to prove diagonalizability of the transition operator of the hyperplane chamber random walk. Bounds on rates of convergence to stationarity were obtained in~\cite{BHR,DiaconisBrown1}. They observed, moreover, that one can replace the faces of a hyperplane arrangement by the covectors of an oriented matroid~\cite{OrientedMatroids1999} and the theory carries through.

The resolutions of Brown and Diaconis coming from the topology of hyperplane
arrangements also played an important role in work of the second author.
They were used to compute the $\Ext$-spaces between simple modules for face
monoids of hyperplane arrangements~\cite{Saliolahyperplane}.
He also computed a quiver presentation of such algebras.
Consequently, he computed the global dimension of these algebras.
(Left regular band algebras have acyclic quivers and hence finite global dimension, which makes
computing their global dimension a natural question.)

In~\cite{Saliola} the second author computed the projective indecomposable
modules for arbitrary left regular band algebras and also the quiver.
The topological arguments to compute minimal projective resolutions and
$\Ext$-spaces did not immediately extend to this setting because one does
not have any obvious topology available.
Our previous paper~\cite{oldpaper} provided the appropriate
topological tools and will be expanded upon, and simplified, in this work.

The paper~\cite{Saliola} also contained an intriguing unpublished result of Ken Brown stating that the algebra of a free left regular band is hereditary. The proof is via a computation of the quiver. A theorem of Gabriel~\cite{assem} states that a split basic finite dimensional algebra is hereditary if and only if it is isomorphic to the path algebra of an acyclic quiver. Since left regular band algebras are split basic (over any field), Brown's proof amounts to proving that the dimension of the path algebra of the quiver of the free left regular band is the cardinality of this monoid.

The third author (unpublished) showed that Brown's counting argument could be modified to prove that the algebra of any left regular band whose poset of principal right ideals has a Hasse diagram which is a tree is hereditary. Brown's result is a direct consequence because the Hasse diagram of the poset of principal right ideals of a free left regular band is a tree. On the other hand, the poset of principal right ideals of a hyperplane monoid is dually isomorphic to the poset of faces of the arrangement. This offers a topological explanation for the computations of Brown and Diaconis and of Saliola cited in the previous paragraph.

These two examples strongly suggested that there was a deep connection between the representation theory and cohomology of the algebra of a left regular band $B$ and the topology of the poset of principal right ideals of $B$. There seem to be only a handful of results in the finite dimensional algebra literature that use topological techniques to compute homological invariants of algebras. The primary examples seem to be in the setting of incidence algebras, where the order complex of the poset plays a key role~\cite{Cibils,Gerstenhaber,Igusa}. A more general setting is considered in~\cite{Bustamente}.

In our previous paper~\cite{oldpaper}, we clarified this connection by using topological techniques to compute all the $\Ext$-spaces between simple modules of the algebra of a left regular band monoid. In particular, we used order complexes of posets and classifying spaces of small categories (in the sense of Segal~\cite{GSegal}) to achieve this. A fundamental role was played by the celebrated Quillen's Theorem A~\cite{Quillen}, which gives a sufficient condition for a functor between categories to induce a homotopy equivalence of classifying spaces. Somewhat surprisingly to us, a combinatorial invariant of simplicial complexes, the Leray number~\cite{Kalai1,Kalai2}, played an important part in~\cite{oldpaper}. The Leray number is tied to the Castelnuovo-Mumford regularity of Stanley-Reisner rings. In particular, that paper gives a new, non-commutative interpretation of the regularity of the Stanley-Reisner ring of a flag complex.

\subsubsection*{CW left regular bands}

Despite the use of this heavy machinery, we were not able to directly compute finite projective resolutions for the simple modules of the algebra of a left regular band in~\cite{oldpaper}. One of the main contributions of the current work is to do just that, provided the algebra is unital.
Our starting point is the well-known observation that the face poset of
a central hyperplane arrangement in $\mathbb{R}^n$ is dually isomorphic to the
face poset of a regular CW complex decomposition of a ball---in fact, the face poset of an important type of polytope known as a zonotope~\cite{Ziegler}---and that
the same is true for every contraction of the arrangement.
Moreover, the face poset of the zonotope is isomorphic to the
poset of principal right ideals of the face monoid.
These observations motivate the following important definition in this paper.

We define a left regular band $B$ to be a (connected) CW left regular band if every ``contraction" of the poset of principal right ideals of $B$ is isomorphic to the face poset of a connected regular CW complex. (Contractions are certain upper sets, corresponding to contractions in the case of the face poset of a central hyperplane arrangement or oriented matroid; a precise definition appears in Section~\ref{sssec:operations-on-lrbs}.) Most of the examples that have been studied in the literature have this property. We study actions of left regular bands $B$ on CW posets $P$, that is, face posets of regular CW complexes. We prove that under a technical restriction (that we term semi-freeness of the action; see Section~\ref{sssec:semi-free-actions}) and when the (order complex of the) poset $P$ is acyclic, then the augmented chain complex of the associated CW complex is a projective resolution of the trivial $\Bbbk B$ module for any commutative ring with unit $\Bbbk$.
\nomenclature[A]{$\Bbbk$}{a field or a commutative ring with unit}%

In particular, this allows us to show that if $B$ contains a two-sided identity
(that is, if $B$ is a left regular band \emph{monoid})
then the augmented chain complex of the order complex of $B$ is
a projective resolution of the trivial $\Bbbk B$-module.
More generally, this result holds for any connected left regular band (these are defined in terms of a topological condition that turns out to be equivalent to the semigroup having a unital algebra over any ground ring).
This allows us to compute
all $\Ext$-spaces between simple $\Bbbk B$-modules and thus the global dimension of
the algebra.

In the case that $B$ is a connected CW left regular band, we also compute
a quiver presentation of $\Bbbk B$, which is an exact generalization of the second
author's results in the case of a real central hyperplane arrangement monoid~\cite{Saliolahyperplane}.  However, even when restricted to the case of hyperplane arrangements, our approach here is fundamentally different from the original approach of~\cite{Saliolahyperplane} in that Brown and Diaconis~\cite{DiaconisBrown1} had to carefully choose orientations for the cells of the zonotope (or incidence numbers) in order to ensure that the augmented cellular chain complex of the zonotope was a chain complex of modules rather than just vector spaces.  This would have been virtually impossible to extend to the level of generality we are considering here.  Instead, we construct actions of the left regular bands on their associated cell complexes by cellular maps.  Functoriality of the cellular chain complex then comes into play to turn it into a chain complex of modules.  To the best of our knowledge it was not even previously known that hyperplane face monoids act on their associated zonotopes by cellular mappings.

Let us formulate here one of our main results.  Let $B$ be a connected CW left regular band, e.g., a real or complex hyperplane face semigroup, the monoid of covectors of an oriented matroid or the face semigroup of a CAT(0) cube complex.  The set $\Lambda(B)=\{Bb\mid b\in B\}$ of principal left ideals of $B$ is closed under intersection and hence is a meet semilattice. Moreover, the mapping $\sigma\colon B\to \Lambda(B)$ given by $\sigma(b)=Bb$ is a semigroup homomorphism.  Let $\Bbbk$ be a field.  Then the simple $\Bbbk B$-modules are one-dimensional and are indexed by $\Lambda(B)$.  We denote by $\Bbbk_X$ the simple module corresponding to $X\in \Lambda(B)$. The fundamental theorem of the representation theory of CW left regular bands is then the following, where we note that a left regular band that is a monoid is automatically connected.

\begin{Thm}
Let $B$ be a connected CW left regular band and $\Bbbk$ a field.
\begin{enumerate}
\item The regular CW complex with face poset $B$ is acyclic.
\item The poset $\Lambda(B)$ is graded.
\item Each open interval of $\Lambda(B)$ is a Cohen-Macaulay poset.
\item $\Bbbk B$ is a unital basic algebra with quiver the Hasse diagram $Q$ of $\Lambda(B)$.
\item $\Bbbk B\cong \Bbbk Q/I$ where $I$ is the ideal generated by the sum of all paths of length two in $Q$.
\item $\Bbbk B$ is a Koszul algebra with Koszul dual the incidence algebra of the opposite poset of $\Lambda(B)$.
\item The minimal projective resolution of the simple module $\Bbbk_X$ is given by the augmented cellular chain complex of the regular CW complex whose face poset is the contraction of $B$ to $X$.
\item If $X,Y\in \Lambda(B)$, then
\begin{equation*}
\Ext^q_{\Bbbk B}(\Bbbk_X,\Bbbk_Y) \cong \begin{cases} \Bbbk, & \text{if}\ X\leq Y\ \text{and}\ \rk[X,Y]=q\\ 0, & \text{else.}\end{cases}
\end{equation*}
\item The global dimension of $\Bbbk B$ is the dimension of the regular CW complex whose face poset is $B$.
\end{enumerate}
\end{Thm}

\subsubsection*{Outline}
We have tried to keep the text as self-contained as possible by providing background material in a number of fields because we are targeting it at a variety of audiences, in particular, toward specialists in algebraic combinatorics, representation theory, poset topology and semigroup theory.

In Section~\ref{s:lrbs} we recall the definitions and basic properties of left regular
bands. In particular, left regular bands are also posets and left multiplication induces an action on the corresponding poset and on the order complex of this poset. This observation plays a crucial role in the paper.

We define a number of new operations on left regular bands such as suspension and join that model the corresponding operations for posets and simplicial complexes. Also the fundamental notion of a contraction of a left regular band is introduced. We give examples of left regular bands that have arisen in the literature and that play an important part in this work. We pay particular attention to the left regular band structure on various combinatorial structures such as real and complex hyperplane arrangements~\cite{BHR,bjorner2}, matroids~\cite{Brown1,White}, oriented matroids~\cite{OrientedMatroids1999} and COMs~\cite{COMS}. We include the necessary background material for these structures as well.

Section~\ref{s:cwposets} begins with a survey of cell decompositions of various types. All are examples of regular CW complexes. We look at the corresponding notion of CW posets, which by definition are posets isomorphic to face posets of regular CW complexes. Novel to this text is that we promote the standard construction~\cite{BjornerCW} of a regular CW complex from a CW poset to a functor by establishing the functorial nature of the construction with respect to appropriate mappings of posets that we term ``cellular.'' This is crucial in order to transform an action of a semigroup on a CW poset into an action on the corresponding CW complex by regular cellular maps. This development could be of interest in its own right for workers in poset topology.

With this background out of the way, we are able to define the notion of CW left regular band as a left regular band all of whose contractions (as defined in Section~\ref{sssec:operations-on-lrbs}) are CW posets. Most of the left regular bands that are associated to combinatorial structures are CW left regular bands.  Oriented interval greedoids from~\cite{SaliolaThomas} are  presented as another naturally arising family of CW left regular bands.  The notion of an interval greedoid is discussed here as well.  Left regular bands associated to interval greedoids, generalizing Brown's construction for matroids~\cite{Brown1}, were considered in~\cite{bjorner2}.

An exciting new example of CW left regular bands is detailed in the penultimate subsection of Section~\ref{s:cwposets}. We recall the important notion of a cube complex and, in particular, that of a CAT(0) cube complex. There is a natural notion of hyperplane in such complexes and this leads us to define the structure of a left regular band associated to a CAT(0) cube complex by associating covectors to the faces of a CAT(0) cube complex. CAT(0) cube complexes have played an important role in geometric group theory and three-dimensional manifold theory in recent years, in particular, with respect to the work of Agol and Wise on Thurston's virtual Haken and virtual fibering conjectures~\cite{wisebook,agol}. In combinatorics, the Billera-Holmes-Vogtman phylogenetic tree space~\cite{treespace} has been a popular CAT(0) cube complex to study, particularly because  of its connections with tropical geometry, where it appears in the guise of the tropical Grassmannian~\cite{tropgrass}. We provide all the necessary background on cube complexes and CAT(0) spaces.  The left regular band structure for CAT(0) cube complexes was discovered independently in~\cite{COMS}, where it is shown that there is, in fact, a left regular band structure on the face poset of any CAT(0) zonotopal complex whose cells are Coxeter zonotopes.  We discuss this development briefly in the last subsection of Section~\ref{s:cwposets}.

In Section~\ref{s:algebraprelim} we review several elements of the representation theory of finite dimensional algebras that may not be familiar to all our target audience. We give the necessary background on the Gabriel quiver, quiver presentations and the connections with split basic algebras. We recall the definition of a Koszul algebra and its Koszul dual. We then review known results on the algebra of a left regular band monoid and its representation theory. We note that the algebra of any left regular band semigroup has a right identity element and show that the algebra has an identity element if and only if every contraction of the semigroup (in the sense noted above) is connected.  By extending the setting from monoids to semigroups with unital algebras, we are able to treat a number of important examples---such as affine hyperplane arrangements, affine oriented matroids, CAT(0) cube complexes and, more generally, COMs---whose associated left regular bands are not monoids but do have unital algebras.

Section~\ref{s:projrel} is devoted to one of the central results of this work. We introduce the notion of a semi-free action of a left regular band on a poset. We show that if a left regular band admits a semi-free action on an acyclic CW poset $P$, then the augmented cellular chain complex of the associated CW complex is a
projective resolution of the trivial $\Bbbk B$-module. In
particular, this applies to the poset of faces of the order complex of $B$ if $B$ is a monoid, or more generally if $B$ is connected (so that its algebra is unital), and so
the augmented simplicial chain complex of the order complex of $B$ is a projective
resolution of the trivial $\Bbbk B$-module. We use these projective resolutions to compute all $\Ext$-spaces between all simple modules for  monoid (and connected) left regular band algebras over a field, providing a more direct and enlightening proof  of the main results of our previous paper~\cite{oldpaper}.

We show that the global dimension of a CW left regular band $B$ is equal to the dimension of the CW complex associated to $B$.
For CW left regular bands, the chain complexes we discuss in this section turn out to provide the minimal
projective resolutions of every simple module over a field of coefficients.
In particular, this applies to the face monoid of a real or complex hyperplane
arrangement, the monoid of covectors of an oriented matroid or oriented
interval greedoid, the face semigroup of an affine hyperplane arrangement or
CAT(0) cube complex, or the semigroup of covectors associated to a COM.

We turn to quiver presentations in Section~\ref{s:quiverpres}. We abstract the argument used by the second author in the case of real hyperplane monoids~\cite{Saliolahyperplane} to prove that if $B$ is a connected CW left regular band (so, in particular, if $B$ is a monoid), then over a field of coefficients $\Bbbk$, the quiver of $\Bbbk B$ is isomorphic to the Hasse diagram of the maximal semilattice image (or support semilattice) $\Lambda(B)$ of $B$, which is a graded semilattice. Furthermore, $\Bbbk B$ is the quotient of the path algebra of its quiver by adding one relation for each interval of length $2$ in the Hasse diagram of $\Lambda(B)$, which is just the sum of all paths that begin at the bottom point of the interval and end at the top point. Thus $\Bbbk B$ is completely determined by $\Lambda(B)$.

It follows from the results of Section~\ref{s:quiverpres} that the algebra $\Bbbk B$ of a CW left regular band is a quadratic algebra---it has a basis of quiver relations of degree $2$.  We prove in Section~\ref{s:koszul} that the algebra of CW left regular band is, in fact, a Koszul algebra. We show that the Koszul dual of $\Bbbk B$ is the opposite algebra of the incidence algebra of the support semilattice $\Lambda(B)$. It follows that the $\Ext$-algebra of $\Bbbk B$ is isomorphic to the incidence algebra of $\Lambda(B)$. Results proved independently by Polo~\cite{Polo} and Woodcock~\cite{Woodcock} then yield that every open interval of the semilattice $\Lambda(B)$ is a Cohen-Macaulay poset.  This is particularly important because it has consequences for the M\"obius function of $\Lambda(B)$.  To the best of our knowledge, the results of Polo and Woodcock have never before been used to prove a poset is Cohen-Macaulay, but rather have been used to prove that certain incidence algebras are Koszul.

Section~\ref{s:injenv} studies injective envelopes.  In particular, we give an explicit geometric construction of the injective envelopes of the simple $\Bbbk B$-modules when $B$ is a CW left regular band whose cells are zonotopes (or, more generally, dual to face posets of oriented matroids).  In particular, the result applies to COMs.  The construction is based on the idea of performing a line shelling of the zonotope in order to construct a visual hemisphere whose faces in turn span a projective right ideal in the algebra.

There are many famous enumerative results that count cells in hyperplane arrangements and, more generally, in oriented matroids.  In particular, the theorem of Las Vergnas and Zaslavsky~\cite{OrientedMatroids1999} computes the number of chambers in a real hyperplane arrangement (or topes in an oriented matroid) in terms of the M{\"o}bius function of its intersection lattice (lattice of flats). From our perspective, the number of chambers is equal to the number of elements in the minimal ideal of the hyperplane face monoid. This leads to a natural direction for generalizing these results. The main results of Section~\ref{s:enum} do just that for connected CW left regular bands: we count the number of cells in each dimension in terms of the M\"obius function of the support semilattice.  When our result is restricted to hyperplane arrangements and oriented matroids we recover the Las Vergnas-Zaslavsky theorem discussed above.   In the case of complex hyperplane arrangements we recover a result of Bj\"orner~\cite{bjorner2} and in the  case of CAT(0) cube complexes we recover an enumerative result of Dress \textit{et. al}~\cite{dresscube}, originally expressed in the equivalent language of median graphs.  We use these results to generalize the second author's computation of the Cartan matrix of hyperplane face monoids to that of all connected CW left regular bands.

In the last section of the paper we compute the cohomological dimension of left regular band monoid algebras. We show that there is a surprising connection to Leray numbers. The cohomological dimension of the free partially commutative left regular band on a graph $\Gamma$ is the Leray number of the clique complex of $\Gamma$. We show how to use our result to construct easily a finite monoid (in fact, a regular band, that is, a submonoid  of a direct product of a left regular band and right regular band) that has left cohomological dimension $m$ and right cohomological dimension $n$ for any pair of natural numbers $m$ and $n$. (The reader is invited to compare with the construction in~\cite{GubaPride}.)

\section{Left regular bands, hyperplane arrangements, oriented matroids and generalizations}\label{s:lrbs}

In this section we review some basic facts about  left regular bands and then provide a miscellany of examples coming from combinatorics and topology.  In particular, we begin with some standard notions from semigroup theory that can be found, for example, in books such as~\cite[Appendix~A]{qtheor} and~\cite[Chapter~1]{repbook}.  A gentle introduction to the theory of $\mathscr R$-trivial monoids, including left regular bands, can be found in~\cite[Chapter~2]{repbook}.  After the introduction to left regular bands, we proceed with our first collection of important examples.  Our first examples are free left regular bands, and some close relatives like matroid left regular bands and free partially commutative left regular bands.  Then we turn to hyperplane face semigroups and oriented matroids, which are crucial examples.  The subsequent subsection considers the relatively new notion of a COM~\cite{COMS} (complex of oriented matroids), which generalizes oriented matroids.  The final subsection considers complex hyperplane face monoids after Bj\"orner~\cite{bjorner2}.

\subsection{Green's relations and the structure of left regular bands}

\subsubsection{Semigroups and ideals}
In this paper, all semigroups and monoids are assumed finite unless otherwise stated.
Standard references on semigroup theory are~\cite{CP,qtheor}.

A \emph{semigroup}\index{semigroup} $S$ is a set with an associative binary operation.  If $A,B$ are subsets of $S$, then $AB=\{ab\mid a\in A, b\in B\}$.  Multiplication of subsets is associative.

A \emph{left ideal}\index{left ideal} $L$ of $S$ is a non-empty subset with $SL\subseteq L$.  Right ideals are define dually.  A non-empty subset $I\subseteq S$ is a two-sided ideal, or simply an \emph{ideal}, if it is both a right and left ideal.  Each finite semigroup contains a unique minimal ideal.  This follows from the observation that if $I_1,I_2$ are ideals of $S$, then $I_1I_2\subseteq I_1\cap I_2$ and hence $I_1\cap I_2$ is an ideal.

\subsubsection{Green's relations}
We shall make some use of
\emph{Green's relations}\index{Green's relations}~\cite{Green}.
\nomenclature[L, 00]{$\mathscr L, \mathscr R, \mathscr J$}{Green's relations}%
 Two elements $s,t\in S$ of a semigroup $S$ are said to be \emph{$\mathscr J$-equivalent}\index{$\mathscr J$-equivalent}\index{equivalent!$\mathscr J$-} \index{Green's relations!$\mathscr J$}(respectively, \emph{$\mathscr R$-equivalent}\index{$\mathscr R$-equivalent}\index{equivalent!$\mathscr R$-}\index{Green's relations!$\mathscr R$}, respectively \emph{$\mathscr L$-equivalent}\index{$\mathscr L$-equivalent})\index{equivalent!$\mathscr L$-}\index{Green's relations!$\mathscr L$} if they generate the same principal two-sided (respectively, right, respectively, left) ideal. An element $e\in S$ is \emph{idempotent}\index{idempotent} if $e^2=e$. In this case, notice that the principal left ideal generated by $e$ is $Se$ and that $a\in Se$ if and only if $a=ae$.  Indeed, if $a\in Se$, then $a=xe$ with $x\in S$ and so $ae=xee=xe=a$.  Conversely, if $ae=a$, then trivially $a\in Se$. In particular, if $e,f\in S$ are idempotents, then they are $\mathscr L$-equivalent if and only if $ef=e$ and $fe=f$.  Dual statements hold for the principal right ideal generated by an idempotent.

\subsubsection{Left regular bands}
\nomenclature[L, 01]{$B$}{left regular band}%
A \emph{left regular band}\index{left regular band}\index{band!left regular} is a semigroup $B$ satisfying the identities $x^2=x$ and $xyx=xy$. More generally, a semigroup is called a \emph{band}\index{band} if all its elements are idempotent.  By a \emph{left regular band monoid}\index{left regular band!monoid}, we mean a left regular band with identity. Left regular band monoids have also been studied by Lawvere in the context of topos theory under the name \emph{graphic monoids}\index{graphic monoids}~\cite{graphic1,moregraphic,Lawvere}.
If $B$ is a left regular band, then $B^1$ denotes the monoid obtained by adjoining an external identity element to $B$ (even if it already had an identity). Notice that if $B$ is a left regular band, then so is $B^1$.

\subsubsection{Green's $\mathscr R$-order for left regular bands}
In any semigroup, there is a natural partial order on the set of idempotents defined by putting $e\leq f$ if $ef=fe=e$.  In a left regular band, $fe=e$ implies $ef=efe=e$ and so $e\leq f$ if and only if $fe=e$. Said differently, Green's relation $\mathscr R$ is trivial on a left regular band $B$ and the order on idempotents is Green's $\mathscr R$-order, that is, $e\leq f$ if and only if $eB\subseteq fB$.

It is crucial to what follows that the action of $B$ on itself by left multiplication is order preserving.

\begin{Lemma}\label{l:left.act.order}
The action of a left regular band $B$ on itself by left multiplication is order preserving.
\end{Lemma}
\begin{proof}
Indeed, if $e\leq f$, then $fe=e$ and so if $b\in B$, then $(bf)(be)=bfe=be$.  Thus $be\leq bf$.
\end{proof}

\subsubsection{Commutative left regular bands and meet semilattices}
If $P$ is a (meet) semilattice, that is, a partially ordered set with binary meets, then $P$ is a commutative left regular band with respect to the meet operation.  Moreover, $e\leq f$ if and only if $e=f\wedge e$ and so the natural partial order on $P$ is its original ordering.  Conversely, if $P$ is a commutative left regular band, then we claim that $P$ is a semilattice with respect to the natural partial order and the meet is the product in $P$. Indeed, if $e,f\in P$ then $e(ef)=ef$ and $f(ef)=(ef)f=ef$ and so $ef\leq e,f$.  If $a\leq e,f$, then $ea=a$ and $fa=a$, whence $efa=ea=a$.  Therefore, $a\leq ef$ and we conclude that $ef=e\wedge f$.   Let us record this observation as a proposition.

\begin{Prop}
Commutative left regular bands are precisely meet semilattices.
\end{Prop}

Thus left regular bands can be thought of as ``skew semilattices,'' and, for this reason, Sch\"utzenberger called left regular bands by the name \textit{treillis gauches}~\cite{Schutzlrb}.

\subsubsection{Support semilattice of a left regular band}
\nomenclature[L, 02]{$\Lambda(B)$}{support semilattice of a left regular band $B$}%
\nomenclature[L, 03]{$X, Y$}{elements of $\Lambda(B)$}%
The set $\Lambda(B)$ of principal left ideals of a left regular band is closed under intersection and hence is a semilattice called the \emph{support semilattice}\index{support semilattice} of $B$; see~\cite{Clifford,Brown1,Brown2}. More precisely we have the following proposition, which is a special case of a theorem of Clifford on semigroups all of whose elements belong to a subgroup~\cite{Clifford}.

\begin{Prop}\label{p:supp.lattice}
Let $B$ be a left regular band and $a,b\in B$.  Then $Aa\cap Ab=Aab=Aba$.  Thus $\Lambda(B)=\{Bb\mid b\in B\}$ is a semilattice.
\end{Prop}
\begin{proof}
First note that $aba=ab$ and $bab=ba$ implies that $Aab=Aba$.  Since $Aba\subseteq Aa$ and $Aab\subseteq Ab$, we conclude that $Aab\subseteq Aa\cap Ab$.  Conversely, if $x\in Aa\cap Ab$, then $xa=a$ and $xb=b$.  Thus $xab=xb=x$ and so $x\in Aab$.  We conclude that $Aa\cap Ab=Aab$.
\end{proof}

\nomenclature[L, 04]{$\sigma\colon B\xrightarrow{} \Lambda(B)$}{support map of the left regular band $B$}%
In the case that $B$ is a monoid, $\Lambda(B)$ will be a lattice with $B$ as the top.    Equipping $\Lambda(B)$ with the binary operation of intersection, the mapping \[\sigma\colon B\to \Lambda(B)\] given by $\sigma(a)=Ba$ becomes a surjective homomorphism (by Proposition~\ref{p:supp.lattice}) called the \emph{support map}\index{support map}.  Observe that $\sigma(a)=\sigma(b)$ if and only if $a,b$ are $\mathscr L$-equivalent, if and only if both $ab=a$ and $ba=b$.  The following proposition will be used many times throughout the text without explicit reference.

\begin{Prop}\label{p:support.lat.sub}
Let $B$ be a left regular band and let $B'\subseteq B$ be a subsemigroup.  Then $\Lambda(B')$ is isomorphic to the image of $B'$ under the support map $\sigma\colon B\to \Lambda(B)$.
\end{Prop}
\begin{proof}
It suffices to show that if $a,b\in B'$, then $Ba=Bb$ if and only if $B'a=B'b$.  But both these equalities hold if and only if $ab=a$ and $ba=b$.
\end{proof}

The support semilattice of $B$ is the abelianization of $B$ in the following categorical sense.

\begin{Prop}
Let $B$ be a left regular band and let $\Lambda$ be a meet semilattice.  Then any homomorphism $\tau\colon B\to \Lambda$ factors uniquely through $\Lambda(B)$, that is, there is a unique homomorphism $\tau'\colon \Lambda(B)\to \Lambda$ such that the diagram
\[\xymatrix{B\ar[rr]^{\sigma}\ar[rd]_{\tau} &&\Lambda(B)\ar[dl]^{\tau'}\\ & \Lambda &}\]
commutes.
\end{Prop}
\begin{proof}
If $\sigma(a)=\sigma(b)$, that is, $Ba=Bb$, then $ab=a$ and $ba=b$.  Therefore, as $\Lambda$ is commutative, we have that $\tau(a)=\tau(ab)=\tau(a)\tau(b)=\tau(b)\tau(a)=\tau(ba)=\tau(b)$.  It follows that $\tau$ factors uniquely through $\sigma$.
\end{proof}

\subsubsection{Characterisations of left regular bands}
It will be convenient to place into a single theorem several characterizations of left regular bands that we shall use without comment throughout the text.

\begin{Thm}
Let $S$ be a semigroup.  Then the following are equivalent.
\begin{enumerate}
\item  $S$ is a left regular band.
\item $S$ is an $\mathscr R$-trivial band, that is, a band in which $aS=bS$ implies $a=b$.
\item $S$ is a band in which each left ideal is two-sided.
\item $S$ is a semigroup with a homomorphism $\rho\colon S\to \Lambda$ to a semilattice $\Lambda$ such that each fiber $\rho^{-1}(\lambda)$ with $\lambda\in \Lambda$ (which is necessarily a semigroup) satisfies the identity $xy=x$.
\end{enumerate}
\end{Thm}
\begin{proof}
Suppose that $S$ is a left regular band and $aS=bS$.  Then $b=ab=aba=aa=a$ as $b\in aS$ and $a\in bS$ imply $ab=b$ and $ba=a$.  Thus $S$ is an $\mathscr R$-trivial band.  Conversely, if $S$ is an $\mathscr R$-trivial band and $a,b\in S$, then $abaS=abS$ because $abab=ab$.  Therefore, $aba=ab$ by $\mathscr R$-triviality and hence $S$ is a left regular band.  This proves the equivalence of (1) and (2).

Suppose that $S$ is a left regular band and $L$ is a left ideal of $S$.  Let $a\in S$ and $b\in L$.  Then from $ba=bab$ we conclude that $ba\in L$ and hence $L$ is a two-sided ideal.  Thus (1) implies (3).  Assume that (3) holds and let $a,b\in S$.  Then $Sa$ is a two-sided ideal and so $ab\in Sa$.  Therefore, $aba=ab$ and hence $S$ is a left regular band.  This establishes the equivalence of (1) and (3).

If $S$ is a left regular band, then the support map $\sigma\colon S\to \Lambda(S)$ is a homomorphism to a semilattice with the property that $\sigma(a)=\sigma(b)$ implies $ab=a$ and $ba=b$ (as $Ba=Bb$).  Therefore, (1) implies (4).  If (4) holds and $a\in S$, then $aa=a$ because the fiber over $\rho(a)$ satisfies $xy=x$.  Also, if $a,b\in S$, then $\rho(aba)=\rho(a)\rho(b)\rho(a)=\rho(a)^2\rho(b)=\rho(a)\rho(b)=\rho(ab)$ and so $ab,aba$ belong to the same fiber of $\rho$.  Thus $aba=(ab)^2a=(ab)(aba)=ab$ and so $S$ is a left regular band.  This completes the proof of the theorem.
\end{proof}

\subsubsection{Operations on left regular bands}
\label{sssec:operations-on-lrbs}

We define here some operations on left regular bands, inspired by topology,
that will come into play later in the theory.

\nomenclature[P, 01]{$P$}{poset}%
\nomenclature[P, 02]{$P_{\leq p}, P_{< p}$}{principal lower sets of the poset $P$ generated by $p \in P$}%
\nomenclature[P, 02]{$P_{\geq p}, P_{> p}$}{principal upper sets of the poset $P$ generated by $p \in P$}%
We first introduce notation for the subposet of a poset $P$
consisting of all elements greater than or equal to a specified element $p \in P$:
\[P_{\geq p} = \{q\in P\mid q\geq p\}.\]
The subposets $P_{\leq p}$, $P_{>p}$ and $P_{<p}$ are defined analogously.
For example, if $B$ is a left regular band and $a \in B$, then
$B_{\leq a}$ is the right ideal of $B$ generated by $a$:
\begin{equation*}
    B_{\leq a} = \{b \in B \mid b \leq a\} = a B.
\end{equation*}
\nomenclature[L, 07]{$B_{\leq a}$}{all elements of $B$ less than or equal to $a \in B$; also $aB$}%
\nomenclature[L, 08]{$B_{< a}$}{all elements of $B$ less than $a \in B$; also $\bd aB$}%
\nomenclature[L, 05]{$aB$}{right ideal generated by $a \in B$}%

\nomenclature[L, 09]{$B_{\geq X}$}{contraction of a left regular band $B$ to $X \in \Lambda(B)$}%
If $X\in \Lambda(B)$, then the \emph{contraction}\index{contraction} of $B$ to $X$ is the subsemigroup
\[B_{\geq X} = \sigma\inv(\Lambda(B)_{\geq X})=\{a\in B\mid \sigma(a)\geq X\}.\] Later we shall see that this corresponds to the operation of contraction in oriented matroid theory.  Observe that $\Lambda(B_{\geq X})\cong \Lambda(B)_{\geq X}$ by Proposition~\ref{p:support.lat.sub}.

\nomenclature[L, 06]{$\bd aB$}{right ideal $aB$ with $a$ removed; $aB \setminus \{a\}$}%
If $B$ is a left regular band and $a\in B$, we want to think of $aB = B_{\leq a}$ as the face poset of a closed cell of $B$ and hence we put \[\bd aB= aB\setminus \{a\}=B_{<a},\] which we think of as the boundary of $aB$ (and this, in fact,  will literally be the case for a number of important examples). In particular, if $B$ is a left regular band monoid, then $\partial B=B\setminus \{1\}$.   Note that $aB$ and $\bd aB$ are right ideals of $B$. The subsemigroup $aB$ is a left regular band monoid with identity $a$ as $aB=aBa$. Readers familiar with oriented matroid theory should think of $aB$ as being the result of applying a deletion to $B$.  Observe that $\Lambda(aB)\cong \Lambda(B)_{\leq Ba}$ by Proposition~\ref{p:support.lat.sub}.

\begin{Lemma}
If $Ba=Bb$, then $aB$ is isomorphic to $bB$ as a left regular band monoid via $x\mapsto bx$ for $x\in aB$ (with inverse $y\mapsto ay$ for $y\in bB$).
\end{Lemma}
\begin{proof}
Indeed, $b(xy) = bxby$ by the left regular band axiom and if $x\in aB$ and $y\in bB$, then $abx=ax=x$ and $bay=by=y$ because $ab=a$ and $ba=b$.
\end{proof}

\nomenclature[L, 16]{$B \ast B'$}{join of left regular bands $B$ and $B'$}%
If $B,B'$ are (disjoint) left regular bands, we define their \emph{join}\index{join} $B\ast B'$ to be $B\cup B'$ where $B$ and $B'$ are subsemigroups and $b'b=b=bb'$ for all $b\in B$ and $b'\in B'$. Notice that $B\ast B'$ is a monoid if and only if $B'$ is a monoid. Also observe that $\Lambda(B\ast B')=\Lambda(B)\ast \Lambda(B')$ (where the join is as left regular bands).

\nomenclature[L, 13]{$L = \{0, +, -\}$}{three-element left regular band}%
The variety (in the sense of universal algebra~\cite{Burris}) of left regular bands is well known to be generated by the three-element left regular band $L=\{0,+,-\}$ with the multiplication table in Table~\ref{multtableL}.  See~\cite[Proposition~7.3.2]{qtheor}.
\begin{figure}[tb]
\centering
\begin{gather*}
\begin{array}{|c|ccc|}
  \hline
    & 0 & + & - \\ \hline
  0 & 0 & + & - \\
  + & + & + & + \\
  - & - & - & - \\
  \hline
\end{array}
\end{gather*}
\caption{The multiplication table for $L$}\label{multtableL}
\end{figure}

One can view $L$ as the monoid of self-mappings of $[-1,1]$ (acting on the left) generated by the identity map $r_0$, and the constant maps $r_{\pm}$ with image $\pm 1$.  When we come to the relationship between hyperplane face semigroups and zonotopes, described later on, this way of thinking about $L$ will make sense because $L$ is the face monoid of the hyperplane arrangement in $\mathbb R$ with the origin as the hyperplane and $[-1,1]$ is the corresponding zonotope.
In oriented matroid theory, elements of $L^n$ are called \emph{covectors}\index{covectors}. Note that $\partial L=\{-,+\}$.

\nomenclature[L, 17]{$\mathsf S(B)$}{suspension of a left regular band $B$}%
Define the \emph{suspension}\index{suspension} of a left regular band $B$ to be $\mathsf S(B)=\partial L\ast B$; for instance, $L=S(\{0\})$. Note that $\mathsf S(B)$ is a monoid if and only if $B$ is a monoid. Also $\Lambda(\mathsf S(B)) = \Lambda(B)\cup \{-\infty\}$ where $-\infty$ is an adjoined minimum element.

We turn now to some particularly salient examples of left regular bands.

\subsection{Free left regular bands and matroids}\label{ss:matroid}
In this section we consider the free left regular band monoid $F(A)$ on a set $A$ and a generalization to matroids.

\subsubsection{Free left regular bands}
\nomenclature[L, 11]{$F(A)$}{free left regular band on the set $A$}%
The free monoid on the set $A$, denoted $A^*$, is the set of all words over the alphabet $A$ (including the empty word) with the binary operation of concatenation.  By standard universal algebra (or the adjoint functor theorem), there is a \emph{free left regular band monoid}\index{free left regular band} $F(A)$ on the set $A$.  Let $\rho\colon A^*\to F(A)$ be the canonical surjective homomorphism.  We claim that the repetition-free words form a cross-section to $\rho$, that is, they constitute a set of normal forms. Indeed, $\rho(u)=\rho(\overline{u})$ where $\overline{u}$ is the word obtained from $u$ by removing repetitions as you scan $u$ from left-to-right  (for example, $\overline{ababcbac}=abc$) because of the identities $x^2=x$ and $xyx=xy$ satisfied by $F(A)$.

The power set $P(A)$ is a commutative left regular band with respect to the operation of union and the mapping $c\colon A^*\to P(A)$ sending a word $w\in A^*$ to the set $c(w)$ of letters appearing in $w$ is a homomorphism.  Thus if $u,v$ are words with the same image under $\rho$, then $c(u)=c(v)$.  In order to show that the repetition-free words make up a set of normal forms, it suffices to show that if $u,v$ are distinct repetition-free words with $c(u)=c(v)$, then there is a homomorphism $\psi\colon A^*\to L$  with $\psi(u)\neq \psi(v)$. (This also shows that $L$ generates the variety of left regular bands.)   Let $w$ be the longest common prefix of $u$ and $v$.  Then $u=wax$ and $v=wby$ with $a\neq b$ in $A$ and $x,y\in A^*$.  Since $u,v$ are repetition-free, the letters $a$ and $b$ do not appear in $w$.  Hence if we define a homomorphism $\psi\colon A^*\to L$ on $A$ by $\psi(a)=+$, $\psi(b)=-$ and $\psi(c)=0$ for $c\in A\setminus \{a,b\}$, then we obtain that $\psi(u)=+$ and $\psi(v)=-$.

It follows that we can identify $F(A)$ with the set of repetition-free words over the alphabet $A$ with the binary operation $u\cdot v=\ov{uv}$ where $\ov{uv}$ is the result of removing repetitions as you scan $uv$ from left to right.  The free left regular band on $A$ is obtained from $F(A)$ by removing the empty word.  If $A$ has at least two elements, the free left regular band (semigroup) does not have a unital algebra over any base ring (as is easily verified).  For this reason we are only interested in free left regular band monoids.

It is not difficult to see that $u\leq v$ in $F(A)$ if and only if $v$ is a prefix of $u$ and that $F(A)u\subseteq F(A)v$ if and only if $c(v)\subseteq c(u)$.  Hence $\Lambda(F(A))$ can be identified with $P(A)$ ordered by reverse inclusion.  With this identification if $X\subseteq A$, then $F(A)_{\geq X}=F(X)$.    If $w\in F(A)$ with $c(w)=X$, then $wF(A)\cong F(A\setminus X)$.  Thus every contraction and ``deletion'' of a free left regular band monoid is again a free left regular band monoid.  See~\cite[Chapter~14]{repbook} for further details.

\subsubsection{Matroids}
Brown~\cite[Section~6]{Brown1} generalized the free left regular band monoid to matroids.  Matroids were introduced by Whitney in 1932 as an abstraction of linear independence in vector spaces, algebraic independence over fields and independence in graphs. Since then, matroid theory has expanded in many directions, with applications in geometry, topology, algebraic combinatorics and optimization theory. We give a brief introduction to matroids and the related notion of geometric lattice suited to the needs of the current work. We refer to~\cite{White} as a good reference on matroids.

\nomenclature[M, 01]{$\mathcal M = (E,\FeasibleSets)$}{matroid on ground set $E$ with independent sets $\FeasibleSets$}%
A \emph{matroid}\index{matroid} is a pair $\mathcal M = (E,\FeasibleSets)$ where $E$ is the \emph{ground set}\index{ground set} and $\FeasibleSets$ is a \emph{hereditary}\index{hereditary} collection of subsets of $E$ (that is, $X\in \FeasibleSets$ and $Y\subseteq X$ implies $Y\in \FeasibleSets$) satisfying the following Exchange Axiom.
\begin{itemize}
\item [(EA)] If $I_{1}$ and $I_{2}$ are elements of $\FeasibleSets$ with $|I_{1}| < |I_{2}|$, then there is an $e \in I_{2}- I_{1}$ such that $I_{1} \cup \{e\} \in \FeasibleSets$.
\end{itemize}
Elements of $\FeasibleSets$ are called \emph{independent subsets}\index{independent subsets}.

A \emph{loop}\index{loop} in $\mathcal M$ is an element $e \in E$ that belongs
to no independent subset or, equivalently, $\{e\} \notin \FeasibleSets$.
We will often restrict our attention to matroids that do not contain loops.
Two elements $e, f$ of $E$ are \emph{parallel}\index{parallel} if they are not
loops and $\{e, f\}$ is not an independent subset. This is a clear
generalization of the notion of parallel vectors in a vector space. A matroid
is \emph{simple}\index{simple} if it contains no loops and no distinct pair of
parallel elements.

The motivating examples are the  matroid of all linearly independent subsets of a vector space, the matroid of all forests in a graph (thought of as sets of edges of the graph) and the matroid of all algebraically independent subsets of a field. It follows easily from the exchange axiom that all maximal independent subsets (called \emph{bases}\index{bases}) in a matroid have the same size.   In the three motivating examples, we have bases in a vector space matroid are the usual bases of linear algebra, bases in a graphical matroid are spanning forests, and bases in the field example are transcendence bases.

\nomenclature[M, 03]{$r(A)$}{rank of a set $A$ in a matroid}%
\nomenclature[M, 04]{$r(E)$}{rank of a matroid}%
Let $\mathcal M = (E,\FeasibleSets)$ be a matroid and let $A$ be a subset of $E$. The \emph{rank}\index{rank} of $A$ is  defined by \[r(A)=\max\{|I|\mid  I\in \FeasibleSets, I \subseteq A\}.\] In the vector space matroid, the rank of a set of vectors is the dimension of the space that they span. The \emph{rank}\index{rank} of a matroid is the cardinality of a basis; this is also $r(E)$.

The rank function, $r\colon 2^{E} \to \mathbb{N}$ is a \emph{semimodular}\index{semimodular!rank function} rank function, meaning that it satisfies the following axioms.

\begin{itemize}
  \item[(R1)] $r(A) \leq |A|$
  \item[(R2)] A $\subseteq B \implies r(A) \leq r(B)$
  \item[(R3)] \emph{Upper Semimodularity}: $r(A \cup B) + r(A \cap B) \leq r(A) + r(B)$
\end{itemize}

One can prove that the independent sets $\FeasibleSets$ are precisely the subsets $I \subseteq E$ such that $r(I)=|I|$. Thus, the rank function gives another way to axiomatize matroids.

Flats are the matroid theoretic analogue of subspaces of a vector space. They are the maximal subsets of $E$ of a given rank. That is, a subset $A$ is a \emph{flat}\index{flat} if and only if $r(A) < r(A \cup \{e\})$ for all $e \in E\setminus A$.  It can be shown that the intersection of a collection of flats is a flat.

By the usual considerations, we can define a closure operator on $2^{E}$ by sending a subset $A$ of $E$ to the subset $\ov {A}$ which is the intersection of all the flats containing $A$. More concretely, \[\ov{A} = A \cup \{e \in E\mid r(A) = r(A \cup \{e\})\}.\] The assignment of $A$ to $\ov{A}$ is a closure operator on $2^E$, in that it is an order preserving, idempotent and non-decreasing function. In addition, it satisfies the following version of the exchange axiom, familiar in the case of vector spaces and linear independence.

\begin{itemize}
\item[(EX)] If $e,f \in E$ and $A\subseteq E$, then if $f \in \overline{A \cup \{e\}}\setminus \ov{A}$, then $e \in \ov{A \cup \{f\}}$.
\end{itemize}

It can be shown that, conversely, a closure operator $c\colon 2^{E} \to 2^{E}$ satisfying the Exchange Axiom (EX) defines a unique matroid where the flats are precisely those subsets $A$ of $E$ such that $c(A) = A$.

The collection  $\Lambda(\mathcal M)$ of all flats of a matroid $\mathcal M$ is a lattice where the meet operation is intersection and the join is defined by $F \vee G = \ov{F \cup G}$ for flats $F$ and $G$.  In fact, $\Lambda(M)$ is a geometric lattice. In order to define this notion, recall that an element $a$ in a lattice $\Lambda$ is an
\emph{atom}\index{atom} if  it covers the bottom element of $\Lambda$.  The lattice $\Lambda$ is \emph{ranked}\index{ranked!poset}\index{poset!ranked} or \emph{graded}\index{graded!poset}\index{poset!graded} if all maximal chains between $x$ and $y$ in $\Lambda$, with $x<y$, have the same length. The \emph{rank function}\index{rank function} of a graded lattice $\Lambda$ is the function $r\colon \Lambda \to \mathbb{N}$ where $r(x)$ is the length of a maximal chain from the bottom element of $\Lambda$ to $x$.
\nomenclature[P, 06]{$r(x)$}{rank of an element $x$ in a ranked poset}%

\begin{Def}[Geometric lattice]\label{d:geom.lattice}
A lattice $\Lambda$ is said to be a \emph{geometric lattice}\index{geometric lattice}\index{lattice!geometric} if it satisfies the following axioms.
\begin{itemize}
  \item[(G1)] $\Lambda$ is an atomic lattice, that is, every element in $\Lambda$ is a join of atoms.
  \item[(G2)] $\Lambda$ is ranked.
  \item[(G3)] $\Lambda$ is \emph{upper semimodular}\index{upper semimodular}: $r(A \vee B) + r(A \wedge B) \leq r(A) + r(B)$.
\end{itemize}
\end{Def}

A theorem of Birkhoff~\cite{White} shows that there is a one-to-one
correspondence between geometric lattices and simple matroids.
The correspondence associates with a geometric lattice $\Lambda$ the simple
matroid $\mathcal M(\Lambda)$ defined as follows. The ground set of this
matroid is the set $E$ of atoms of $\Lambda$. A collection $I$ of atoms is
independent if the rank of the join of $I$ is the cardinality of $I$. Then the
lattice of flats of $\mathcal M(\Lambda)$ is isomorphic to $\Lambda$.

\nomenclature[M, 02]{$\ov A$}{closure of a set $A$ in a matroid}%
 If $\mathcal M$ is a matroid with ground set $E$, then Brown considers the left regular band monoid whose elements are all linearly ordered independent subsets of $E$.  The product is given by concatenation and then removing elements that are dependent on the previous elements where we say $e\in E$ is \emph{dependent}\index{dependent} on a subset $A\subseteq E$ if $\ov{A\cup \{e\}}=\ov A$.  The free left regular band is obtained from the matroid in which all subsets of elements are independent. The Hasse diagram of the left regular band associated to a matroid is a rooted tree.  Bj\"orner extended this construction to interval greedoids, which form a natural extension of the notion of matroids.  See~\cite{bjorner2} for details.

\subsection{Free partially commutative left regular bands}
\nomenclature[L, 12]{$B(\Gamma)$}{free partially commutative left regular band on the graph $\Gamma$}%
If $\Gamma=(V,E)$ is a (finite simple) graph, then the \emph{free partially commutative left regular band}\index{free partially commutative left regular band} $B(\Gamma)$ is the left regular band \emph{monoid} with presentation \[\langle V\mid xy=yx, \text{if}\ \{x,y\}\in E\rangle\] (where this presentation is as a left regular band monoid).  They are left regular band analogues of free partially commutative monoids~\cite{CartierFoata} (or trace monoids~\cite{tracebook}) and right-angled Artin groups (cf.~\cite{BestvinaBrady,wisebook}).

The elements of $B(\Gamma)$ are in bijection with acyclic orientations of induced subgraphs of the complementary graph $\ov \Gamma$~\cite[Theorem~3.7]{oldpaper}.  Namely, if $w$ is a word in $V$ and $W$ is the set of letters appearing in $w$, then we associate to $w$ the induced subgraph of $\ov \Gamma$ on the vertex set $W$, acyclically oriented by orienting an edge $\{x,y\}$ from $x$ to $y$ if the first occurrence of $x$ in $w$ is \emph{before} the first occurrence of $y$ in $w$.  Conversely, given an acyclically oriented induced subgraph of $\ov \Gamma$ we can find a repetition-free word giving rise to that subgraph by performing a topological sorting of the vertices, that is, choosing an ordering of the vertices such that if there is a directed edge from $x$ to $y$, then $x$ occurs before $y$ in the ordering.

For example, if $E=\emptyset$, i.e., the graph $\Gamma$ is edgeless, then $B(\Gamma)$ is the free left regular band monoid on the vertices of $\Gamma$ and elements correspond to acyclic orientations of induced subgraphs of the complete graph on the vertex set $V$.  But these correspond exactly to repetition-free words over the alphabet $V$.  At the other extreme, if $\Gamma$ is complete, then $\ov \Gamma$ is edgeless.  Hence each induced subgraph has only one acyclic orientation. Thus the class of each word is determined by its support and so $B(\Gamma)\cong (P(V),\bigcup)$.

Like the case of free left regular band monoids, free partially commutative left regular bands behave well under contraction and ``deletion.''  First observe that the natural homomorphism $F(V)\to P(V)$ factors through $B(\Gamma)$ and hence $\Lambda(B(\Gamma))\cong (P(V),\bigcup)$ as a semilattice via the map taking an acyclically oriented induced subgraph of the complement of $\Gamma$ to its set of vertices.  If $X\subseteq V$, then one easily checks that $B(\Gamma)_{\geq X}\cong B(\Gamma[X])$ where $\Gamma[X]$ denotes the induced subgraph of $\Gamma$ with vertex set $X$.  On the other hand, if $b\in B(\Gamma)$ is the image of a word $w$ in $F(V)$ with $c(w)=X$, then $bB(\Gamma)\cong B(\Gamma[V\setminus X])$.  Details can be found in~\cite[Section~3.6]{oldpaper} or verified directly.

If $\Gamma_1,\Gamma_2$ are graphs, then it is not difficult to check that $B(\Gamma_1)\times B(\Gamma_2)\cong B(\Gamma_1\ast \Gamma_2)$, where $\Gamma_1\ast \Gamma_2$ is the join of the two graphs.  In other words, $\Gamma_1\ast \Gamma_2$ is the graph obtained by connecting each vertex of $\Gamma_1$ to each vertex of $\Gamma_2$ by an edge.

Free partially commutative left regular bands turn out to give a simple model of a family of random walks considered by Athanasiadis and Diaconis~\cite{DiaconisAthan} and their cohomology turns out to be connected to classical invariants of the clique (or flag) complex of $\Gamma$.
More details on free partially commutative left regular bands can be found in~\cite[Section~3.6]{oldpaper}.

\subsection{Hyperplane arrangements and oriented matroids}\label{ssec:hyperplane-arrangements-and-oriented-matroids}
We follow here Brown~\cite{Brown1,Brown2}, Brown and Diaconis~\cite{DiaconisBrown1} and Bj\"orner \textit{et al.}~\cite{OrientedMatroids1999}.  See also~\cite{Brown:book2} and the recent book~\cite{THA}.
\nomenclature[H, 01]{$\mathcal A$}{hyperplane arrangement}%
A \emph{hyperplane arrangement}\index{hyperplane arrangement} in a real vector space $V$ is a finite collection $\mathcal A=\{H_1,\ldots,H_n\}$ of (affine) hyperplanes in $V$.  The arrangement is called \emph{central}\index{central} $\bigcap_{i=1}^n H_i\neq \emptyset$ and otherwise is called \emph{affine}\index{affine}.  Without loss of generality, in the case of a central arrangement, we may (and do) assume that the hyperplanes are linear (that is, pass through the origin). A hyperplane arrangement in $V$ is said to be \emph{essential}\index{essential} if the unit normals to the hyperplanes span $V$.  A central arrangement is essential if and only if the intersection of all the hyperplanes is the origin.  We can  always assume a central arrangement is essential by considering $V/\bigcap_{i=1}^n H_i$. Similarly, we may restrict our attention to essential affine arrangements without loss of generality. See~\cite{Stanleyhyp} for details. The connected components of $V\setminus \bigcup_{i=1}^nH_i$ are called \emph{chambers}\index{chambers}.

\subsubsection{Central arrangements and oriented matroids}
\label{sssec:central-arrangements-and-oriented-matroids}
For the moment we consider the case that $\mathcal A$ is central and we assume that it is given by linear forms $f_i=0$, for $i=1,\ldots, n$, with $f_1,\ldots, f_n\in V^*$.   If $x\in \mathbb R$, put
\[\sgn (x)= \begin{cases} +, & \text{if}\ x>0 \\ -, & \text{if}\ x<0\\ 0, & \text{if}\ x=0.\end{cases}\] Then we can define a mapping $\theta\colon V\to L^n$ by
\begin{equation}\label{eq:define.theta}
\theta(x) = (\sgn(f_1(x)),\ldots,\sgn(f_n(x))).
 \end{equation}
  It turns out that the set of covectors $\theta(V)$ is a submonoid of $L^n$, introduced independently by Tits~\cite{Titsbuildings,Titsappendix} and Bland.  Trivially, we have $\theta(x)\theta(x)=\theta(x)$.  If $x\neq y$, then $\theta(x)\theta(y)=\theta(z)$ where $z$ is an element on the open line segment $(x,y)$ such that $f_i(z)=f_i(x)$ whenever $f_i(x)\neq 0$. (Such a $z$ exists because $f_1,\ldots, f_n$ are continuous so we can find $\varepsilon>0$ such that, for each $i$ with $f_i(x)\neq 0$, we have that $\sgn(f_i(B_{\varepsilon}(x)))=\sgn(f_i(x))$.)   If $z=(1-t)x+ty$ with $0<t<1$ and $f_i(x)=0$, then $f_i(z)=tf_i(y)$ and so $\sgn(f_i(z))=\sgn(f_i(y))$. It follows that $\theta(x)\theta(y)=\theta(z)$.   Note that $\theta(0)=0$ (the identity). The fibers of $\theta$ are relative interiors of polyhedral cones and the fibers over elements of $\theta(V)\cap (\bd L)^n$ are the chambers.

  \nomenclature[H, 05]{$\FFF(\mathcal A)$}{face monoid of the hyperplane arrangement $\mathcal A$}%
  Define $\FFF(\mathcal A)= \theta(V)$ to be the \emph{face monoid}\index{face monoid} of $\mathcal A$.  Note that up to isomorphism it does not depend on the choice of the $f_i$. Being a submonoid of $L^n$, the monoid $\FFF(\AAA)$ is a left regular band.  Figure~\ref{f:covectorsforhyp} gives an example of the covectors associated to a hyperplane arrangement in $\mathbb R^2$.
\begin{figure}[tbhp]
\begin{center}
\begin{tikzpicture}
    \draw[thick] (60:3) -- (240:3);
    \draw[thick] (120:3) -- (300:3);
    \draw[thick] (-3,0) -- (3,0);
    \draw (0,0)       node[fill=white,font=\tiny] {$(000)$};
    \draw (30:2.25)    node[font=\small]           {$(+++)$};
    \draw (60:1.5)       node[fill=white,font=\tiny] {$(0++)$};
    \draw (90:2.25)   node[font=\small]           {$(-++)$};
    \draw (120:1.5)  node[fill=white,font=\tiny] {$(-0+)$};
    \draw (150:2.25)    node[font=\small]           {$(--+)$};
    \draw (180:1.5) node[fill=white,font=\tiny] {$(--0)$};
    \draw (210:2.25)  node[font=\small]           {$(---)$};
    \draw (240:1.5)      node[fill=white,font=\tiny] {$(0--)$};
    \draw (270:2.25)   node[font=\small]           {$(+--)$};
    \draw (300:1.5)  node[fill=white,font=\tiny] {$(+0-)$};
    \draw (330:2.25)     node[font=\small]           {$(++-)$};
    \draw (0:1.5)   node[fill=white,font=\tiny] {$(++0)$};
\end{tikzpicture}
\end{center}
\caption{The covectors of the faces of the hyperplane arrangement in
$\mathbb R^2$ consisting of three distinct lines.\label{f:covectorsforhyp}}
\end{figure}

\nomenclature[H, 07]{$Z(\mathcal A)$}{zonotope associated with the arrangement $\mathcal A$}%
\nomenclature[H, 08]{$Z_\sigma$}{zonotope associated with a covector $\sigma \in L^n$}%
Associated to the arrangement $\mathcal A$ is the  Minkowski sum
\begin{align*}
Z(\mathcal A)&=[-f_1,f_1]+\cdots + [-f_n,f_n]\\
 &= \left\{\sum_{i=1}^n t_if_i\mid -1\leq t_i\leq 1,\ \text{for}\ i=1,\ldots, n\right\},
\end{align*}
 which  is a convex polytope in $V^*$ known as the \emph{zonotope polar to $\mathcal A$}\index{zonotope!polar to an arrangement}. We refer the reader to~\cite{Ziegler} for a good source on polytopes and, in particular, zonotopes.  Formally, a \emph{zonotope}\index{zonotope} is an image of a hypercube under an affine map or, equivalently, a Minkowski sum of line segments.
  If $\sigma=(\sigma_1,\ldots, \sigma_n)\in L^n$ is any covector, then we can associate to it the zonotope
\begin{equation}\label{facezonotope}
Z_{\sigma} = \sum_{\sigma_i=0} [-f_i,f_i]+\sum_{\sigma_i=-} \{-f_i\}+\sum_{\sigma_i=+}\{f_i\}.
\end{equation}
Clearly, $Z_{\sigma}\subseteq Z(\mathcal A)=Z_0$.  It turns out that the faces of $Z(\mathcal A)$ are exactly the zonotopes $Z_{\sigma}$ with $\sigma \in \FFF(\mathcal A)$ and that $\sigma\leq \tau$ if and only if $Z_{\sigma}\leq Z_{\tau}$. Consequently, $\FFF(\mathcal A)$ is isomorphic, as a poset, to the face poset of $Z(\mathcal A)$. Indeed, if $x\in (V^*)^*=V$, then the face of $Z(\mathcal A)$ on which $x$ (viewed as a functional) is maximized is $Z_{\theta(x)}$. See~\cite[Proposition~2.2.2]{OrientedMatroids1999} for details. Notice that the vertices of the zonotope are in bijection with the chambers of the arrangement.

\nomenclature[H, 06]{$\mathcal L(\mathcal A)$}{intersection lattice of the hyperplane arrangement $\mathcal A$}%
The \emph{intersection lattice}\index{intersection!lattice} $\mathcal L(\mathcal A)$ is the poset of subspaces which are intersections of hyperplanes from $\mathcal A$, ordered by reverse inclusion.  It is a geometric lattice~\cite[Proposition~3.8]{Stanleyhyp}.  One can verify directly that $\Lambda(\FFF(\mathcal A))\cong \mathcal L(\mathcal A)$ and the support map takes a covector $x\in \FFF(\mathcal A)$ to the linear span of the fiber $\theta\inv(x)$.  See~\cite{Brown1,Brown2} for details.

\nomenclature[H, 04]{$\mathcal B_n$}{Boolean arrangement}%
Some important examples are the following.  The \emph{boolean arrangement}\index{boolean arrangement} is the arrangement $\mathcal B_n$ of coordinate hyperplanes $x_i=0$ in $\mathbb R^n$.  The zonotope $Z(\mathcal A)$ is the $n$-cube $[-1,1]^n$.  One has $\FFF(\mathcal B_n)=L^n$. Figure~\ref{f:zonotope} depicts the zonotope polar to the hyperplane arrangement in Figure~\ref{f:covectorsforhyp}.
\begin{figure}[tbhp]
\begin{center}
\begin{tikzpicture}[vertices/.style={draw, fill=black, circle, inner sep=1pt}]
    \foreach \a in {30,90,150,210,270,330} {
      \node[vertices] at (\a:2) {};
    }
    \draw[fill=gray!20] (30:2)--(90:2)--(150:2)--(210:2)--(270:2)--(330:2)--(30:2)--cycle;
    \foreach \a in {30,90,150,210,270,330} {
            \draw[thick] (\a:2)--(\a+60:2);
    }
    \draw[dashed,semithick,blue] (60:3) -- (240:3);
    \draw[dashed,semithick,blue] (120:3) -- (300:3);
    \draw[dashed,semithick,blue] (-3,0) -- (3,0);
    \draw (0,0)       node[fill=white,font=\tiny] {$(000)$};
    \draw (30:2.75)    node[font=\small]           {$(+++)$};
    \draw (60:1.75)       node[fill=white,font=\tiny] {$(0++)$};
    \draw (90:2.75)   node[font=\small]           {$(-++)$};
    \draw (120:1.75)  node[fill=white,font=\tiny] {$(-0+)$};
    \draw (150:2.75)    node[font=\small]           {$(--+)$};
    \draw (180:1.75) node[fill=white,font=\tiny] {$(--0)$};
    \draw (210:2.75)  node[font=\small]           {$(---)$};
    \draw (240:1.75)      node[fill=white,font=\tiny] {$(0--)$};
    \draw (270:2.75)   node[font=\small]           {$(+--)$};
    \draw (300:1.75)  node[fill=white,font=\tiny] {$(+0-)$};
    \draw (330:2.75)     node[font=\small]           {$(++-)$};
    \draw (0:1.75)   node[fill=white,font=\tiny] {$(++0)$};
\end{tikzpicture}
\end{center}
\caption{The zonotope polar to the arrangement in Figure~\ref{f:covectorsforhyp}.\label{f:zonotope}}
\end{figure}
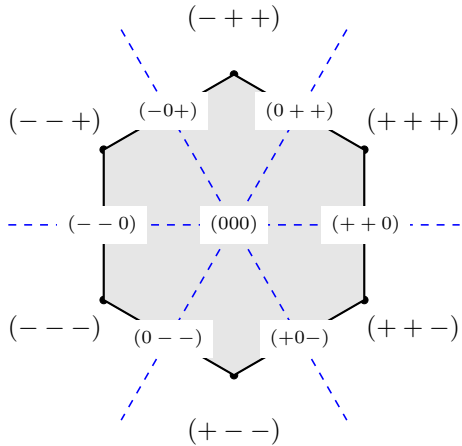

\nomenclature[H, 02]{$\mathcal A(\Gamma)$}{graphical arrangement associated with the graph $\Gamma$}%
\nomenclature[H, 03]{$\mathcal A_W$}{reflection arrangement of the finite Coxeter group $W$}%
\nomenclature[A]{$[n]$}{$\{1, 2, \ldots, n\}$}%
 If $\Gamma$ is a simple graph (say with vertex set $[n]$), then the \emph{graphical arrangement}\index{graphical arrangement} $\mathcal A(\Gamma)$ associated to $\Gamma$ is defined by the hyperplanes
 \[H_{ij} = \{x\in \mathbb R^n\mid x_i-x_j=0\}\]
  with $\{i,j\}$ an edge of $\Gamma$ and $j>i$ (note that this arrangement is not essential).  The special case $\mathcal A(K_n)$ of the complete graph is known as the \emph{braid arrangement}\index{braid arrangement}.  The zonotope corresponding to the braid arrangement is the \emph{permutohedron}\index{permutohedron}, which is the convex hull of the vectors \[\{(\sigma(1),\ldots,\sigma(n))\mid \sigma\in S_n\}.\] The braid arrangement is the Coxeter arrangement associated to $S_n$, viewed as a Coxeter group.  More generally, if $W$ is any finite Coxeter group, there is a corresponding reflection arrangement $\AAA_W$.  One can obtain the associated zonotope, called the \emph{$W$-permutohedron}\index{$W$-permutohedron}, $Z(\AAA_W)$ by choosing a point $x$ in some chamber and then taking the convex hull of the orbit $Wx$ of $x$.  The $1$-skeleton of $Z(\AAA_W)$ is the Cayley graph of $W$ with respect to its Coxeter generators.  Detailed information on Coxeter groups can be found in~\cite{Brown:book2}.

The sets of covectors of a central hyperplane arrangement are the primary examples of oriented matroids.  The reader is referred to the treatise~\cite{OrientedMatroids1999} for a definitive account of the theory of oriented matroids, including all facts that we state below. We need some notation to define this notion. There is a natural automorphism of $L$ that fixes $0$ and switches $+,-$, which we write $x\mapsto -x$.  Thus $L$ can be viewed as a unary monoid and hence $L^n$ can be viewed as a unary monoid via \[(x_1,\ldots, x_n)\longmapsto -(x_1,\ldots, x_n)=(-x_1,\ldots, -x_n).\]

For covectors $x,y\in L^n$, define their \emph{separation set}\index{separation set} to be \[S(x,y) = \{i\in [n]\mid x_i=-y_i\neq 0\}\] that is, it is the set of indices where $x,y$ have opposite signs.  The following properties are enjoyed by the set $\FFF(\mathcal A)$ of covectors of a central hyperplane arrangement and form the set of axioms for an oriented matroid.

\nomenclature[M, 05]{$\mathcal L \subseteq \{+, -, 0\}^E$}{covectors of an oriented matroid}%
\begin{Def}[Oriented matroid]\label{defn:oriented-matroids}
A set of $\mathcal L$ of covectors in $\{+,-,0\}^E$ is said to form an \emph{oriented matroid}\index{oriented matroid} (with (finite) \emph{ground set}\index{ground set} $E$) if it satisfies:
\begin{itemize}
\item [(OM0)] $0\in \mathcal L$;
\item [(OM1)] $x\in \mathcal L$ implies $-x\in \mathcal L$;
\item [(OM2)] $x,y\in \mathcal L$ implies $xy\in \mathcal L$;
\item [(OM3)] \emph{Strong Elimination}\index{strong elimination}: $x,y\in \mathcal L$ and $e\in S(x,y)$ implies there exists $z\in \mathcal L$ such that $z_e=0$ and $z_f=(xy)_f=(yx)_f$ for all $f\notin S(x,y)$.
\end{itemize}
Sometimes we say that the oriented matroid is the pair $(E,\mathcal L)$ to emphasize that the ground set is really part of the data.
\end{Def}

Note that (OM0)--(OM2) say that $\mathcal L$ is a unary submonoid of $L^E$.  Let us verify that the face monoid $\FFF(\mathcal A)$ of a central hyperplane arrangement $\mathcal A$ satisfies these four axioms. Here we will take $E$ to be the set of hyperplanes and view $\FFF(\mathcal A)\subseteq L^E$ (whereas before we viewed it as a submonoid of $L^n$ with $n=|E|$). It is a unary submonoid because $\theta(-x)=-\theta(x)$ (where $\theta$ is defined as per \eqref{eq:define.theta}).  To see that $\FFF(\mathcal A)$ satisfies (OM3), if $e\in S(x,y)$  and $x=\theta(x')$, $y=\theta(y')$, then the line segment $[x',y']$ will intersect the hyperplane corresponding to $e$ at some point $z'$. Putting $z=\theta(z')$, we have $z_e=0$, and if $f\notin S(x,y)$, then $z_f=(xy)_f=(yx)_f$. An oriented matroid is said to be \emph{realizable}\index{realizable!oriented matroid}\index{oriented matroid!realizable} if it comes from a hyperplane arrangement.  Not all oriented matroids are realizable (in fact, in some sense most are not).  However, if one replaces hyperplane arrangements by pseudosphere arrangements (roughly speaking arrangements in a sphere of subspaces homeomorphic to spheres of one dimension smaller), then one has a geometric realization of all oriented matroids; see~\cite[Chaper~5]{OrientedMatroids1999} for details.

Let us introduce some oriented matroid terminology and compare it with semigroup terminology.  Let $\mathcal L\subseteq L^E$ be an oriented matroid.  Let us warn the reader that  the oriented matroid literature customarily uses the reverse of the order that we have been considering on left regular bands.  However, we shall stick to using the  order $x\leq y$ if $yx=x$.  With our convention, maximal non-identity elements of $\mathcal L$ are called \emph{cocircuits}\index{cocircuit} and the elements of the minimal ideal (that is, the minimal elements) are called \emph{topes}\index{topes} (in hyperplane theory, topes are called \emph{chambers}\index{chambers}).  The set of all topes is denoted by $\mathcal T$.  It is known that the cocircuits generate $\mathcal L$ as a monoid and that  \[\mathcal L = \{x\in L^E\mid x\mathcal T\subseteq \mathcal T\};\] see the dual of~\cite[Proposition~3.7.2]{OrientedMatroids1999} and~\cite[Theorem~4.2.13]{OrientedMatroids1999}.  In particular, the set of topes determines the oriented matroid.

If $x\in \mathcal L$, the \emph{zero set}\index{zero set} of $x$ is \[Z(x) = \{e\in E\mid x_e=0\}\] The mapping $Z\colon \mathcal L\to (P(E),\cap)$ given by $x\mapsto Z(x)$ is a monoid homomorphism whose corestriction $Z\colon \mathcal L\to Z(\mathcal L)$ can be identified with the support map $\sigma\colon \mathcal L\to \Lambda(\mathcal L)$.  One has that $\Lambda(\mathcal L)$ is a geometric lattice and hence the lattice of flats of a unique simple matroid. The \emph{rank}\index{rank} of $\mathcal L$ is defined to be the rank of the corresponding matroid.  When $\mathcal L$ is the set of covectors associated to a central hyperplane arrangement $\mathcal A$, one has that the support lattice $\Lambda(\FFF(\mathcal A))$ is isomorphic to the intersection lattice of $\mathcal A$~\cite{Brown2,Stanleyhyp} and the rank is the codimension of the intersection of the hyperplanes in $\mathcal A$. We remark that there is a notion of isomorphism of oriented matroids and an unpublished result of the third author and Hugh Thomas shows that it coincides with isomorphism of the corresponding unary left regular band monoids~\cite{blog}.

\nomenclature[M, 06]{$\mathcal L/A$}{contraction of the oriented matroid $\mathcal L$ at $A$}%
If $A\subseteq E$, then the \emph{contraction}\index{contraction} of $(E,\mathcal L)$ at $A$ is the oriented matroid with ground set $E\setminus A$ and whose set $\mathcal L/A$ of covectors is $\{x|_{E\setminus A}\mid A\subseteq Z(x)\}$.  This monoid is isomorphic to the contraction  $\mathcal L_{\geq \ov A}$ where
\[\ov A=\bigcap \{B\in Z(\mathcal L)\mid B \supseteq A\}\]
(under the identification of $\Lambda(\mathcal L)$ with the zero sets of elements of $\mathcal L$). This is why we use the terminology contraction in this context for left regular bands in general.  In the case of an oriented matroid $(E,\mathcal L)$ coming from a hyperplane arrangement $\mathcal A$, one has that $\mathcal L/A$ is the face monoid of the hyperplane arrangement $\mathcal A/A$ whose underlying vector space $W$ is the intersection of the hyperplanes coming from $A$ and whose hyperplanes are the proper intersections of hyperplanes from $E\setminus A$ with $W$.

\nomenclature[M, 07]{$\mathcal L \setminus A$}{deletion of $A$ from the oriented matroid $\mathcal L$}%
If $(E,\mathcal L)$ is an oriented matroid and $A\subseteq E$, then the \emph{deletion}\index{deletion} of $A$ from $(E,\mathcal L)$ is the oriented matroid with ground set $E\setminus A$ and whose set $\mathcal L\setminus A$ of covectors is $\{x|_{E\setminus A}\mid x\in \mathcal L\}$.  If $x\in \mathcal L$ with $E\setminus Z(x)=A$, then one easily checks that $x\mathcal L\cong \mathcal L\setminus A$, cf.~the discussion after~\cite[Lemma~4.1.8]{OrientedMatroids1999}. For hyperplane arrangements, the operation of deletion corresponds to removing a subset of the hyperplanes.

An element $e\in E$ is called a \emph{loop}\index{loop} of $(E,\mathcal L)$ if $x_e=0$ for all $x\in \mathcal L$ or, equivalently, if $T_e=0$ for each tope $T$. If $e\in E$ is not a loop, then there is an associated partition of $\mathcal T$ into \emph{half-spaces}\index{half-spaces}.  Namely, we have
\[\mathcal T_e^+=\{T\in \mathcal T\mid T_e=+\}\quad \text{and}\quad \mathcal T_e^-=\{T\in \mathcal T\mid T_e=-\}.\]
For hyperplane arrangements these correspond to the sets of chambers in the positive and negative half-spaces determined by the hyperplane corresponding to $e$.

The reader should consult~\cite{OrientedMatroids1999} for more details and motivation for studying oriented matroids. In this paper, we mostly need results and notions from~\cite[Chapter~4]{OrientedMatroids1999}. However, in the next subsection we will explain how zonotopal tilings of a zonotope correspond to certain oriented matroids.  In particular, we will need the following notion. Let $(E,\mathcal L)$ be an oriented matroid.  Then an oriented matroid $(E\cup \{g\},\mathcal L')$ is a \emph{single-element lifting}\index{single-element lifting} of $(E,\mathcal L)$ if $g$ is not a loop of $\mathcal L'$ and $\mathcal L'/\{g\}=\mathcal L$.

\subsubsection{Affine arrangements, affine oriented matroids, $T$-convex sets of topes and zonotopal tilings}
Next we turn to affine hyperplane arrangements and affine oriented matroids.   Suppose that $\mathcal A=\{H_1,\ldots, H_n\}$ is an affine hyperplane arrangement in $V=\mathbb R^d$ given by equations $f_i(x)=c_i$ with $f_i\in V^*$ and $c_i\in \mathbb R$ for $i=1,\ldots, n$. See Figure~\ref{f:affinearrange}.
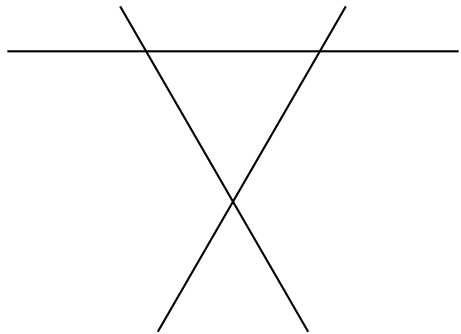
\begin{figure}[htbp]
\begin{center}
\begin{tikzpicture}
    \draw[thick] (60:3) -- (240:2);
    \draw[thick] (120:3) -- (300:2);
    \draw[thick] (-3,2) -- (3,2);
\end{tikzpicture}
\end{center}
\caption{An affine arrangement\label{f:affinearrange}}
\end{figure}
Then one can define a mapping $\theta\colon V\to L^n$ by \[\theta(x) = (\sgn(f_1(x)-c_1),\ldots,\sgn(f_n(x)-c_n))\] and obtain a subsemigroup of $L^n$.  It will not be a monoid for a non-central arrangement.  It turns out, however, to be more convenient to describe affine hyperplane face semigroups via central arrangements.

To do this, we embed $V$ in $\mathbb R^{d+1}$ as the affine hyperplane $x_{d+1}=1$. Now let $H_i'$ be the hyperplane in $\mathbb R^{d+1}$ defined by \[f'_i(x_1,\ldots, x_{d+1}) = f_i(x_1,\ldots, x_d)-x_{d+1}c_i\] and observe that $H_i'\cap \{x\in \mathbb R^{d+1}\mid x_{d+1}=1\}=\{(x,1)\mid x\in H_i\}$.  Let $H=\{x\in \mathbb R^{d+1}\mid x_{d+1}=0\}$ and put $\mathcal A'=\{H_1',\ldots, H_n',H\}$.  We now define the \emph{hyperplane face semigroup}\index{hyperplane face semigroup} of $\mathcal A$ to be $\FFF(\mathcal A)=\{y\in \FFF(\mathcal A')\mid y_{n+1}=+\}$. The reader should verify that this gives an isomorphic semigroup to the one constructed in the previous paragraph.
The support semilattice of $\FFF(\mathcal A)$ is isomorphic to the intersection semilattice of $\mathcal A$~\cite{Stanleyhyp} (that is the poset, of all non-empty intersections of hyperplanes of $\mathcal A$ ordered by reverse inclusion).  Notice that $\FFF(\mathcal A)$ is a right ideal in $\FFF(\mathcal A')$.

An \emph{affine oriented matroid}\index{affine oriented matroid} consists of a triple $(E,\mathcal L,g)$ where $\mathcal L$ is an oriented matroid on $E$ and $g\in E$ is not a loop.  The corresponding left regular band is \[\mathcal L^+(g)=\{x\in \mathcal L\mid x_g=+\}.\] 
For example, each affine hyperplane arrangement $\mathcal A$ determines an affine oriented matroid $([n+1],\FFF(\mathcal A'),n+1)$ in the manner described above. Affine oriented matroids play a role in linear programming over oriented matroids~\cite[Chapter 10]{OrientedMatroids1999}. See~\cite[Section~4.5]{OrientedMatroids1999} for more on affine oriented matroids.

Affine oriented matroids are closely related to zonotopal tilings via the Bohne-Dress theorem.  We follow here Richter-Gebert and Ziegler~\cite{GebertZiegler}.  Let $\mathcal A$ be a hyperplane arrangement with associated oriented matroid $\FFF(\mathcal A)$ and zonotope $Z(\mathcal A)$.  A (simple) \emph{zonotopal tiling}\index{zonotopal tiling} of $Z(\mathcal A)$ is a collection $\mathcal Z=\{Z_1,\ldots, Z_m\}$ of zonotopes such that:
\begin{itemize}
\item [(ZT1)] $Z(\mathcal A)=\bigcup_{i=1}^m Z_i$;
\item [(ZT2)] each face of an element of $\mathcal Z$ belongs to $\mathcal Z$;
\item [(ZT3)] if $Z_i,Z_j\in \mathcal Z$, then $Z_i\cap Z_j$ is a common face of $Z_i,Z_j$ (or empty);
\item [(ZT4)] every edge of $Z(\mathcal A)$ belongs to $\mathcal Z$.
\end{itemize}
Notice that (ZT1)--(ZT3) just assert that $\mathcal Z$ is the structure of a polyhedral complex on $Z(\mathcal A)$.
The Bohne-Dress theorem asserts that zonotopal tilings of $Z(\mathcal A)$ are in bijection with single-element liftings of $\FFF(\mathcal A)$.  If $(E\cup \{g\}, \mathcal L)$ is a single element lifting of $\FFF(\mathcal A)$, then the corresponding zonotopal tiling is $\mathcal Z=\{Z_{\sigma}\mid (\sigma,+)\in \mathcal L^+(g)\}$ where $Z_\sigma$ is defined as in \eqref{facezonotope}. Hence the faces of the polyhedral complex $\mathcal Z$ can be identified with the elements of the left regular band $\mathcal L^+(g)$ associated to the affine oriented matroid $(E\cup \{g\},\mathcal L,g)$.  The Bohne-Dress theorem is a good motivation for studying oriented matroids even when one is only interested in zonotopes. (It is known that not all zonotopal tilings come from realizable single-element liftings.)

The construction of affine oriented matroids admits the following generalization.  A set of topes of an oriented matroid $(E,\mathcal L)$ is called \emph{$T$-convex}\index{$T$-convex} if it is an intersection of half-spaces (in the sense defined earlier).  For example, the topes belonging to $\mathcal L^+(g)$ are precisely those in $\mathcal T_g^+$, and hence form a $T$-convex set.  In~\cite{Brown1}, $T$-convex sets of topes for a central hyperplane arrangement are called convex sets of chambers. Up to reorientation of the oriented matroid, we may assume without loss of generality that we are considering an intersection of positive half-spaces.  So let $A\subseteq E$ be a subset, none of whose elements are loops.  Put $\mathcal L^+(A) = \bigcap_{e\in A}\mathcal L^+(e)$.  Note that $\mathcal L^+(\emptyset) = \mathcal L$.  If non-empty, then we call $\mathcal L^+(A)$ the \emph{face semigroup}\index{face semigroup}
of the $T$-convex set of topes $\bigcap_{e\in A}\mathcal T_e^+$.  Random walks on these semigroups, in the case of hyperplane arrangements, were considered by Brown~\cite[Section~4]{Brown1} (actually, there is a slight difference in that Brown looks at the closure of the convex set of chambers, while we look at the interior).   Details on $T$-convex sets can be found in~\cite[Section~4.2]{OrientedMatroids1999}.

\subsection{Strong elimination systems, lopsided systems and COMs}
As we were preparing this manuscript for submission, the paper~\cite{COMS} appeared on ArXiv, which includes a large source of new examples of left regular bands and, in particular, includes all the examples of Section~\ref{ssec:hyperplane-arrangements-and-oriented-matroids}. We elaborate on these new classes of examples below.
The weakest notion considered in~\cite{COMS} is that of a strong elimination system.

\begin{Def}[Strong elimination system]
A \emph{strong elimination system}\index{strong elimination system} $(E,\mathcal L)$ with \emph{ground set}\index{ground set} $E$ is a collection $\mathcal L\subseteq L^E$ of covectors satisfying (OM2) and (OM3) from Definition~\ref{defn:oriented-matroids}.
\end{Def}

The notion of a strong elimination system abstracts the following setup.  A \emph{realizable strong elimination system}\index{realizable!strong elimination system}\index{strong elimination system!realizable} consists of a central hyperplane arrangement $\AAA$ in $\mathbb R^d$ and a convex subset $C\subseteq \mathbb R^d$.  It is convenient to assume that the strong elimination system is \emph{essential}\index{essential} meaning that $C$ spans $\mathbb R^d$ and that every hyperplane in $\AAA$ passes through $C$.  This assumption does not change the semigroups one can obtain.  Again, we choose forms defining the hyperplane arrangement to obtain a mapping $\theta\colon \mathbb R^d\to L^E$, as per~\eqref{eq:define.theta}, where $E$ is the set of hyperplanes.  Then $\theta(C)$ is a subsemigroup of the face monoid $\FFF(\AAA)$ because if $x,y\in C$, then the line segment $[x,y]$ from $x$ to $y$ is contained in $C$ and hence the element $z\in \mathbb R^n$ constructed in Section~\ref{sssec:central-arrangements-and-oriented-matroids} with $\theta(x)\theta(y)=\theta(z)$ belongs to $C$. Also, if $x,y\in C$ and $H$ is a hyperplane with $x,y$ on opposite sides of $H$, then the intersection $z$ of $[x,y]$ with $H$ belongs to $C$ and so $\theta(C)$ satisfies (OM3).    We remark that one could replace central arrangements with affine ones, but one obtains nothing new since our embedding of face semigroups of affine arrangements as right ideals in face monoids of central arrangements shows that the strong elimination system obtained from intersecting a convex set with the faces of an affine arrangement can also be obtained from intersecting a convex set with a central arrangement in a Euclidean space of one larger dimension.

Lawrence introduced the notion of a lopsided system in~\cite{lopsided}, which has applications in statistics, combinatorics, learning theory and computational geometry; see~\cite{BCDK1,BCDK2,Moran}.  A \emph{lopsided system}\index{lopsided system} $(E,\mathcal L)$ with ground set $E$ is a right ideal in $\mathcal L\subseteq L^E$ satisfying (OM3). It turns out that CAT(0) cube complexes are examples of lopsided systems, which we essentially proved before we became aware of the notion of a lopsided system, and which was proved explicitly and independently in~\cite{COMS} (see also the discussion following Theorem~\ref{t:cat(0).zono}).  Lopsided systems abstract the strong elimination systems obtained by considering the regions obtained by intersecting orthants (i.e., faces of the Boolean arrangement) with an open convex set in Euclidean space.

More generally, you can obtain a right ideal of a hyperplane face monoid satisfying (OM3) by considering those faces of a hyperplane arrangement $\AAA$ in $\mathbb R^d$  intersecting a convex open subset $C$ in $\mathbb R^d$. More precisely, following~\cite{COMS}, a \emph{realizable COM}\index{realizable!COM}\index{COM!realizable}\footnote{COM is used in~\cite{COMS} as an acronym for both ``complex of oriented matroids'' and ``conditional oriented matroid.''} is the set of covectors of a realizable strong elimination system corresponding to a central hyperplane arrangement $\AAA$ in $\mathbb R^d$ and a convex \emph{open}\index{open} subset $C$ of $\mathbb R^d$ (and again we may assume it is \emph{essential}\index{essential}, meaning that each hyperplane of $\AAA$ meets $C$).  The right ideal property follows because if $x\in C$ and $y\in \mathbb R^d$, then any point $z$ on $[x,y]$ sufficiently closed to $x$ belongs to $C$ and hence we may choose $z\in C$ with $\theta(x)\theta(y)=\theta(z)$ (by choosing $z$ in an $\varepsilon$-ball around $x$ small enough so that $z\in C$ and is in the same open half-space containing $x$ associated to each hyperplane of $\AAA$ not containing $x$). Note that any face semigroup of an affine hyperplane arrangement is a realizable COM.  Indeed, our embedding of the face semigroup of an affine arrangement in $\mathbb R^d$ into the face monoid of a central hyperplane arrangement in $\mathbb R^{d+1}$ consists of taking the covectors of elements in $\mathbb R^{d+1}$ belonging to the open convex set which is the positive half-space with $x_{d+1}>0$ in the homogenization of the affine arrangement. As was the case for strong elimination systems, there is no extra generality obtained by considering the faces of an affine arrangement intersecting a convex open subset.

Let us now give the formal definition of a COM from~\cite{COMS}.
\begin{Def}[COM]
A \emph{COM}\index{COM}\index{complex of oriented matroids}\index{conditional oriented matroid} with \emph{ground set}\index{ground set} $E$ is defined  to be a pair $(E,\mathcal L)$ with $\mathcal L\subseteq L^E$ such that $\mathcal L$ satisfies (OM3) of Definition~\ref{defn:oriented-matroids} and the \emph{face symmetry axiom}\index{face symmetry axiom} (FS):
\begin{itemize}
\item [(FS)] $x,y\in \mathcal L$ implies $x(-y)\in \mathcal L$.
\end{itemize}
\end{Def}

It is obvious that a right ideal  $\mathcal L$ in an oriented matroid satisfies (FS). In particular, every lopsided system is a COM.  It is apparently an open problem whether every COM is a right ideal in an oriented matroid.  It is observed in~\cite{COMS} that every COM is a strong elimination system and hence a left regular band.  We thus have the chain of inclusions depicted in Figure~\ref{fig:inclusions-LSs-COMs-SESs}.
\begin{figure}[h!]
    \centering
    \begin{equation*}
        \text{Lopsided systems}\subseteq \text{COMs}\subseteq \text{Strong elimination systems}
    \end{equation*}
    \caption{Relationship between the classes of lopsided systems, COMs and strong elimination systems}
    \label{fig:inclusions-LSs-COMs-SESs}
\end{figure}

\begin{Prop}\label{p:comis.an.lrb}
If $(E,\mathcal L)$ is a COM, then $\mathcal L$ is a subsemigroup of $L^E$.
\end{Prop}
\begin{proof}
If $x,y\in \mathcal L$, then $x(-y)\in \mathcal L$ and hence $x(-[x(-y)])=x((-x)y)=[x(-x)]y=xy$ belongs to $\mathcal L$.
\end{proof}

Affine oriented matroids and the face semigroups of $T$-convex sets of topes are examples of COMs.
\begin{Prop}
Every affine oriented matroid or face semigroup of a $T$-convex set of topes is a COM.
\end{Prop}
\begin{proof}
As any right ideal in an oriented matroid satisfies (FS), the only issue is to prove that (OM3) holds.  The key point is that if $x_g=+=y_g$ and $z$ is an element with $z_g=(xy)_g=(yx)_g$, then $z_g=+$.  Thus the set of covectors of an affine oriented matroid or the face semigroup of a $T$-convex set of topes will inherit (OM3) from the ambient oriented matroid.
\end{proof}

Elements of the minimal ideal of a COM are called \emph{topes}\index{topes}.  The following beautiful example of a COM associated to a poset is from~\cite{COMS}; see~\cite{reinerthesis} for the origins of this example and~\cite[Section~6.4]{THA} for more details on the connections between posets and the braid arrangement.

\begin{Example}[Ranking COM of a poset]
Let $(P,\preceq)$ be a poset.  Without loss of generality we may assume that the underlying set of $P$ is $\{1,\ldots, n\}$.  Consider then the braid arrangement $\AAA$ in $\mathbb R^n$ defined by the hyperplanes \[H_{ij}=\{x\in \mathbb R^n\mid x_i=x_j\}\] with $1\leq i<j\leq n$.  Let $C\subseteq \mathbb R^n$ be the convex open set defined by $x_i<x_j$ if $i\prec j$.  The corresponding realizable COM $\mathcal R(P)$ is called the \emph{ranking COM}\index{ranking COM}\index{COM!ranking} associated to $P$~\cite{COMS}.  The topes of $\mathcal R(P)$ are the linear extensions of $P$~\cite{COMS}.   The COM $\mathcal R(P)$ is closely connected to the order polytope~\cite{StanleyOrder} of $P$; see~\cite{COMS} for details and some worked out examples.

If we view $\FFF(\AAA)$ as consisting of all ordered partitions $(P_1,\ldots, P_r)$ of $\{1,\ldots,n\}$ (as per~\cite{BHR} or Example~\ref{ex:braid} below), then $\mathcal R(P)$ is the right ideal of $\FFF(\AAA)$ consisting of those ordered partitions satisfying $i\prec j$ implies the block of $i$ comes before the block of $j$.
\end{Example}

An important fact is that every contraction of a COM is again a COM and every ``deletion'' is an oriented matroid~\cite[Section~3]{COMS}.

\begin{Prop}\label{p:Com.minor}
Let $(E,\mathcal L)$ be a COM.  Then each contraction $\mathcal L_{\geq X}$ with $X\in \Lambda(\mathcal L)$ is a COM.  Moreover, if $x\in \mathcal L$, then $x\mathcal L$ is isomorphic to the monoid of covectors of an oriented matroid.
\end{Prop}

\subsection{Complex hyperplane arrangements}
In~\cite{complexstrat}, Bj\"orner and Ziegler associated covectors to complex hyperplane arrangements.
Bj\"orner later observed that these sets of covectors are left regular band monoids~\cite{bjorner2}.
Ziegler~\cite{ComplexMatroid} generalized the notion of oriented matroid to complex oriented matroids (which again are left regular band monoids) and our results should also apply in this context, but we did not work out the details.

Let $\til L=\{0,+,-,i,j\}$ with the multiplication table in Table~\ref{multtablecomplex}. The Hasse diagram of  $\til L$ is drawn in Figure~\ref{multtableRorder}.  Note that $\til L=\mathsf S(L)$, the suspension of $L$ as defined at the end of Subsection~\ref{sssec:operations-on-lrbs}.
\begin{table}[tbph]
\centering
\begin{gather*}
\begin{array}{c}
\begin{array}{c| c c c c c}
 & 0 & + & - & i & j\\ \hline
0& 0& + &- &i &j \\
+ &+&+ & + &i & j\\
- & - &- & -  & i & j\\
i & i & i &i & i & i\\
j & j & j & j & j & j
\end{array}
\end{array}
\end{gather*}
\caption{The multiplication table of $\til L$.
\label{multtablecomplex}}
\end{table}
\begin{figure}[tbhp]
\begin{center}
\begin{tikzpicture}
\node  	(A) 								{$0$};
\node  	(B) [below left =1cm of A] 		{$+$};
\node  	(C) [below right=1cm of A] 	{$-$};
\node 	(D) [below=1cm of B]			{$i$};
\node 	(E) [below=1cm of C]			{$j$};
\draw (A)--(B);
\draw (A)--(C);
\draw (B)--(D);
\draw (B)--(E);
\draw (C)--(D);
\draw (C)--(E);
\end{tikzpicture}
\end{center}
\caption{The Hasse diagram of $\til L$\label{multtableRorder}}
\end{figure}
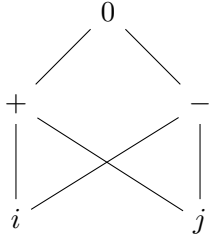

Elements of $\til L^n$ will again be called \emph{covectors}\index{covectors}. To define the covectors of a complex hyperplane arrangement, we must first associate elements of $\til L$ to complex numbers.
Define a function $\sgn\colon \mathbb C\to \til L$ by
\[\sgn(x+iy) = \begin{cases}i, & \text{if}\ y>0,\\ j, & \text{if}\ y <0,\\ +, & \text{if}\ y=0, x>0,\\ -, & \text{if}\ y=0, x<0, \\ 0, & \text{if}\ x=0=y \end{cases}\]
as depicted in Figure~\ref{f:complexplane}.
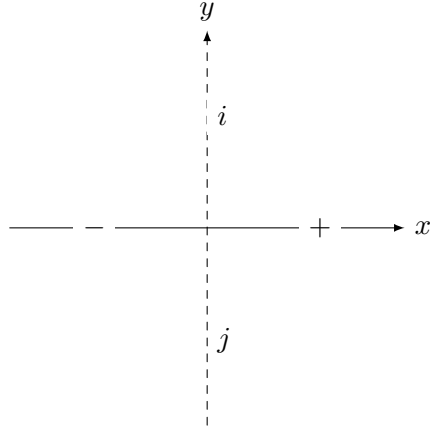
\begin{figure}[tbhp]
\begin{center}
     \begin{tikzpicture}[scale=1.5,cap=round,>=latex]
        \draw[->] (-1.75cm,0cm) -- (1.75cm,0cm) node[right,fill=white] {$x$};
        \draw[->,dashed] (0cm,-1.75cm) -- (0cm,1.75cm) node[above,fill=white] {$y$};
        \draw (-1.00cm,0cm) node[fill=white] {$-$}
              (1.00cm,0cm)  node[fill=white] {$+$}
              (0cm,-1.00cm) node[fill=white,right] {$j$}
              (0cm,1.00cm)  node[fill=white,right] {$i$};
 \end{tikzpicture}
\end{center}
\caption{Complex ``signs''\label{f:complexplane}}
\end{figure}

A (central) \emph{complex hyperplane arrangement}\index{complex hyperplane arrangement} in $\mathbb C^d$ is a collection  $\mathcal A=\{H_1,\ldots, H_n\}$ of (complex) codimension one subspaces, which we take to be the zero sets of complex linear forms $f_1,\ldots, f_n$ on $\mathbb C^d$.  We may always assume that the arrangement is \emph{essential}\index{essential}, that is, $H_1\cap\cdots \cap H_n=\{0\}$.  Analogously to the case of real hyperplane arrangements, we can define a map $\tau\colon \mathbb C^d\to \til L^n$ given by \[\tau(z_1,\ldots,z_d)= (\sgn(f_1(z_1)),\ldots, \sgn(f_d(z_d))).\] The image $\FFF(\mathcal A)=\tau(\mathbb C^d)$ is a submonoid of $\til L^n$ called the \emph{face monoid}\index{face monoid} of $\mathcal A$.
The minimal ideal of $\FFF(\mathcal A)$ consists of the elements of $\FFF(\AAA)\cap \{i,j\}^n$, cf.~\cite[Proposition~3.1]{bjorner2}.  It is shown in~\cite[Theorem~3.5]{complexstrat} that the right ideal $\FFF(\mathcal A)\cap (\bd \til L)^n$ of $\FFF(\AAA)$ is the face poset of a regular CW complex that is homotopy equivalent to the complement $\mathbb C^d\setminus (H_1\cup\cdots \cup H_n)$.

The \emph{augmented intersection lattice}\index{intersection!lattice!augmented}\index{augmented intersection lattice} $\LLL(\mathcal A)$ of $\mathcal A$ is the collection of all intersections of elements of
\[\mathcal A_{\mathrm{aug}} = \{H_1,\ldots, H_n, H_1^{\mathbb R},\ldots, H_n^{\mathbb R}\}\] ordered by reverse inclusion.  Here \[H_i^{\mathbb R} = \{z\in \mathbb C^d\mid \sgn(f_i(z))\in \{0,+,-\}\},\] which is a real hyperplane defined by $\Im (f_i(z))=0$.  One has that $\LLL(\mathcal A)$ is the support lattice of $\FFF(\mathcal A)$ and that the support map takes $F\in \FFF(\mathcal A)$ to the intersection of all elements of $\mathcal A_{\mathrm{aug}}$ containing $\tau\inv(F)$~\cite[Proposition~3.3]{bjorner2}.  The lattice $\LLL(\mathcal A)$ is a semimodular lattice of length $2d$~\cite[Proposition~3.2]{bjorner2}.   More generally, if $X<Y$ in $\LLL(\mathcal A)$, then the length of the longest chain from $X$ to $Y$ in $\LLL$ is $\dim_{\mathbb R} X-\dim_{\mathbb R} Y$.

It is not true in general that contractions of the left regular band $\FFF(\mathcal A)$ along elements of $\LLL(\mathcal A)$ are again face monoids of complex hyperplane arrangements, but their underlying posets are always face posets of regular (even PL) cell decompositions of a ball. See~\cite[Section~4.9]{oldpaper} for details.

\section{Regular CW complexes and CW posets}\label{s:cwposets}
In this section, we begin a systematic exploration of the topological aspects of left regular bands.  Left regular bands are posets and so they have an associated simplicial complex, called the order complex.  As left multiplications preserve order (Lemma~\ref{l:left.act.order}), the left regular band will act on its order complex by simplicial maps, which shall later be exploited to study resolutions of modules over the semigroup algebra of the left regular band.

Many of the left regular bands we have been considering, e.g., hyperplane face semigroups, covectors of COMS, etc.,  are in fact face posets of regular CW complexes. It is well known that regular CW complexes are determined up to isomorphism by their face posets~\cite{BjornerCW}.  Here we promote the construction taking a poset isomorphism to an isomorphism of regular CW complexes into a functor from the category of CW posets (face posets of regular CW complexes) with appropriate maps to the category of regular CW complexes with regular cellular mappings.  We shall see that the left action of the left regular band on itself is taken by the functor to a left action by cellular mappings on the corresponding CW complex.  This will form a crucial ingredient in constructing minimal projective resolutions of simple modules later in the text.   We believe that our functor from CW posets to regular CW complexes is of interest in its own right and is, to the best of our knowledge, novel to this text.

This section begins by reviewing the basics of simplicial complexes and regular CW complexes, as well as introducing our functor.  We then recall the notion of an oriented interval greedoid~\cite{SaliolaThomas}, which generalizes oriented matroids and certain complex hyperplane arrangement face monoids.    These also give left regular bands which are face posets of regular CW complexes.  This is followed by some results on the topology of the different left regular bands discussed in the previous section and this one.

The final two subsections concern left regular bands constructed from CAT(0) cube complexes and CAT(0) zonotopal complexes.  We quickly review some fundamentals of CAT(0) geometry and prove that the faces of a CAT(0) cube complex admit a left regular band structure; our proof in fact shows that they are lopsided systems.  This result was discovered independently in~\cite{COMS} in greater generality, which we also discuss.

The reader is referred to~\cite[Section~4.7]{OrientedMatroids1999} and~\cite[Appendix~A.2]{Brown:book2} for basics about regular CW complexes and CW posets.  See also~\cite{BjornerCW} and~\cite[Chapter~IX]{Massey}.
In this text all posets and CW complexes are finite unless otherwise stated.

\subsection{Simplicial complexes and order complexes of posets}

\subsubsection{Simplicial complexes and simplicial maps}
\nomenclature[S, 01]{$K$}{simplicial complex}%
A \emph{simplicial complex}\index{simplicial complex} $K$ is a pair $(V,\mathscr F)$ where $V$ is a (finite) set of \emph{vertices}\index{vertices} and $\mathscr F$ is a collection of non-empty subsets of $V$ such that:
\begin{itemize}
\item $\emptyset\neq \sigma\subseteq \tau\in \mathscr F\implies \sigma\in \mathscr F$;
\item $V=\bigcup \mathscr F$.
\end{itemize}
Elements of $\mathscr F$ are called \emph{simplices}\index{simplices} or \emph{faces}\index{faces}.    More precisely, if $|\sigma|=q+1$, then $\sigma$ is called a \emph{$q$-simplex}\index{$q$-simplex} and we write $\dim \sigma=q$. Maximal faces are called \emph{facets}\index{facets}.
A simplicial complex $K$ is said to be \emph{pure}\index{pure} if all its facets have the same dimension.  The dimension $\dim K$ of $K$ is the maximum dimension of a facet of $K$.

A \emph{simplicial map}\index{simplicial map} $f\colon K\to L$ of simplicial complexes is a map between their vertex sets such that the image of a simplex of $K$ under $f$ is a simplex of $L$. Simplicial complexes and simplicial maps form a category.

\subsubsection{Geometric realization}
\nomenclature[S, 02]{$\|K\|$}{geometric realization of a simplicial complex $K$}%
The \emph{geometric realization}\index{geometric realization} of a simplicial complex $K=(V,\mathscr F)$  is denoted $\|K\|$. By definition $\|K\|$ is the subspace of $\mathbb R^V$ consisting of all non-negative mappings $f\colon V\to \mathbb R$ such that $\sum_{v\in V}f(v)=1$ and whose support is a simplex of $K$.  Notice that the vertex of $\|K\|$ corresponding to a vertex $v\in V$ is the mapping $\delta_v\colon V\to \mathbb R$ with
\[\delta_v(x) = \begin{cases} 1, & \text{if}\ x=v\\ 0, & \text{else.}\end{cases}\]

\subsubsection{Subcomplex, $q$-skeleton, flag complex}
\nomenclature[S, 03]{$K[W]$}{induced subcomplex of the simplicial complex $K$ with vertex set $W$}%
A \emph{subcomplex}\index{subcomplex} of $K=(V,\mathscr F)$ is a simplicial complex $(V',\mathscr F')$ with $V'\subseteq V$ and $\mathscr F'\subseteq \mathscr F$.
If $W\subseteq V$, then the subcomplex of $K$ \emph{induced}\index{induced} by $W$ is $K[W] = (W, \mathscr F\cap P(W))$, that is, the vertex set of $K[W]$ is $W$ and a subset of $W$ is a simplex of $K[W]$ if and only if it is a simplex of $K$.  It is the largest subcomplex of $K$ whose vertex set is $W$. The \emph{$q$-skeleton}\index{$q$-skeleton} of $K$ is the subcomplex $K^q$ consisting of all simplices of $K$ of dimension at most $q$.
\nomenclature[S, 04]{$K^q$}{$q$-skeleton of the simplicial complex $K$}%

A \emph{flag complex}\index{flag complex} is a simplicial complex $K=(V,\mathscr F)$ such that whenever the $1$-skeleton of a simplex belongs to $K$, then so does the simplex.  Formally, this means that the minimal non-empty subsets of $V$ that do not belong to $\mathscr F$ have size two.

\nomenclature[S, 21]{$\Cliq(\Gamma)$}{clique complex of the graph $\Gamma$}%
An important class of examples arises by considering the set of cliques of a graph.
If $\Gamma=(V,E)$ is a (simple) graph, the \emph{clique complex}\index{clique complex} of $\Gamma$ is the simplicial complex  $\Cliq(\Gamma)$ whose vertex set is $V$ and whose simplices are the cliques (subsets of vertices inducing complete subgraphs) of $\Gamma$.  It is the unique flag complex whose $1$-skeleton is $\Gamma$.

\subsubsection{Operations on simplicial complexes}

\nomenclature[S, 07]{$\link_K(\sigma)$}{link of the simplex $\sigma$ in the simplicial complex $K$}%
If $\sigma$ is a simplex of $K=(V,\mathscr F)$, then the \emph{link}\index{link} of $\sigma$ is the simplicial complex $\link_K(\sigma)$ with vertex set $W=\{v\in V\setminus \sigma\mid \{v\}\cup \sigma\in \mathscr F\}$ and with simplices all non-empty subsets $\tau\subseteq W$ such that $\sigma\cup \tau\in \mathscr F$.  For example, if $v$ is a vertex, then $\link_K(v)$ has vertex set the neighbors of $v$.  A non-empty subset $\tau$ of neighbors of $v$ is a simplex in $\link_K(v)$ if and only if $\tau\cup \{v\}$ is a simplex of $K$.

\nomenclature[S, 07]{$K\ast v$}{cone on the simplicial complex $K$ with cone point $v$}%
If $K=(V,\mathscr F)$ is a simplicial complex and $v\notin V$, then the \emph{cone}\index{cone} on $K$ with \emph{cone point}\index{cone point} $v$ is the simplicial complex \[K\ast v=(V\cup \{v\},\{v\}\cup \{\sigma\subseteq V\cup \{v\}\mid \sigma\cap V\in \mathscr F\}).\] It is easy to see that $\|K\ast v\|$ is contractible via a straightline homotopy to the vertex $\delta_v$. In fact, $\|K\ast v\|$ is homeomorphic to the topological cone $(\|K\|\times I)/(\|K\|\times \{1\})$ on $\|K\|$.

\nomenclature[S, 08]{$K \ast K'$}{join of the simplicial complexes $K$ and $K'$}%
More generally, if $K=(V,\mathscr F), K'=(V',\mathscr F')$ with $V\cap V'=\emptyset$, then their \emph{join}\index{join} is the simplicial complex \[K\ast K'=(V\cup V', \mathscr F\cup \mathscr F'\cup \{\sigma\cup \tau\mid \sigma\in \mathscr F,\tau\in \mathscr F' \}).\]  One has that $\|K\ast K'\|\cong \|K\|\ast \|K'\|$ where the \emph{join}\index{join} $A\ast B$ of topological spaces $A,B$ is $(A\times B\times I)/R$ where $R$ is the equivalence relation identifying all elements of $A\times B\times \{0\}$ which project to the same element of $A$ and identifying all elements of $A\times B\times \{1\}$ that project to the same element of $B$. It is well known that the join is associative up to homeomorphism for locally compact spaces and that $S^n\ast S^m$ is homeomorphic to $S^{n+m+1}$.
\nomenclature[S, 09]{$\mathsf S(K)$}{suspension of the simplicial complex $K$}%
\nomenclature[C, 10]{$\mathsf S(X)$}{suspension of the space $X$}%
The join $\mathsf S(X)=S^0\ast X$ is called the \emph{suspension}\index{suspension} of $X$.  Notice that $\mathsf S(S^m)$ is homeomorphic to $S^{m+1}$. Similarly, we can define the suspension of a simplicial complex by taking the join with the simplicial complex consisting of two vertices and no simplices of dimension greater than zero.

\subsubsection{Order complexes of posets}
We specialize the above definitions and constructions to simplicial complexes associated with posets.

\nomenclature[P, 10]{$\Delta(P)$}{order complex of the poset $P$}%
The \emph{order complex}\index{order complex} $\Delta(P)$ of a poset $P$ is the simplicial complex whose vertex set is $P$ and whose simplices are the chains in $P$.

  If $P$ is a poset with a maximum $\wh 1$, then $\Delta(P)=\Delta(P\setminus \{\wh 1\})\ast \wh 1$ and a similar statement holds for a minimum.

\nomenclature[P, 11]{$P \ast Q$}{join or ordinal sum of two posets $P$ and $Q$}%
If $P,Q$ are disjoint posets, their \emph{join}\index{join} (or \emph{ordinal sum}\index{ordinal sum}) $P\ast Q$ is $P\cup Q$ where we define $p\leq q$ for all $p\in P$ and $q\in Q$ and where $P$ and $Q$ retain their respective orders.  Our notation follows that of Bj\"orner~\cite{bjornersurvey}.
It is easy to see that
\begin{equation*}
    \Delta(P \ast Q) = \Delta(P) \ast \Delta(Q).
\end{equation*}

\nomenclature[P, 12]{$\mathsf S(P)$}{suspension of the poset $P$}%
The \emph{suspension}\index{suspension} $\mathsf S(P)$ of a poset $P$ is the join of the two element antichain with $P$. The operations of taking joins or suspensions for left regular bands correspond to performing these operations on their underlying posets.

An important theorem in poset topology is Rota's cross-cut theorem.  A proof can be found in~\cite[Theorem~10.8]{bjornersurvey}.

\begin{Thm}[Rota]\label{t:crosscut}
Let $P$ be a finite poset such that any subset with a common lower bound has a meet.  Let $\mathscr M(P)$ be the set of maximal elements of $P$ and let $K$ be the simplicial complex with vertex set  $\mathscr M(P)$ and whose simplices are those subsets of $\mathscr M(P)$ with a common lower bound.Then $\|\Delta(P)\|$ is homotopy equivalent to $\|K\|$.
\end{Thm}

There is, of course, a dual version of the cross-cut theorem concerning upper bounds, joins and minimal elements.

\subsubsection{Cohen-Macaulay simplicial complexes}

Let $\Bbbk$ be a commutative ring with unit.  A simplicial complex $K$ is called $\Bbbk$-\emph{Cohen-Macaulay}\index{Cohen-Macaulay} if, for all simplices $\sigma$ of $K$, the reduced homology of $\link_K(\sigma)$ satisfies \[\til H_q(\link_K(\sigma);\Bbbk)=0\] for all $q<\dim \link_K(\sigma)$. When $\Bbbk$ is a field, this is equivalent to the Stanley-Reisner ring of $K$ over $\Bbbk$ being Cohen-Macaulay, whence the name.  See~\cite{Stanleycommut} for details. If $K$ is $\mathbb Z$-Cohen-Macaulay, we shall just say that $K$ is \emph{Cohen-Macaulay}\index{Cohen-Macaulay}.  By the universal coefficient theorem, $K$ is Cohen-Macaulay if and only if $K$ is $\Bbbk$-Cohen-Macaulay for all fields $\Bbbk$.

\subsubsection{Leray number of a simplicial complex}

\nomenclature[S, 10]{$L_\Bbbk(K)$}{$\Bbbk$-Leray number of the simplicial complex $K$}%
Let $\Bbbk$ be a commutative ring with unit.  Define the \emph{$\Bbbk$-Leray number}\index{$\Bbbk$-Leray number} $L_\Bbbk(K)$ of a simplicial complex $K=(V,\mathscr F)$ as follows:
\begin{align*}
    L_\Bbbk(K) =\min\{d\mid \til H^i(K[W];\Bbbk) =0, \forall i\geq d, \forall W\subseteq V\}
\end{align*}
where $\til H^i(L;\Bbbk)$ denotes the reduced cohomology of $L$ with coefficients in $\Bbbk$.  The Leray number is classically defined when $\Bbbk$ is a field, and usually the definition is formulated in terms of homology.  Of course, over a field homology and cohomology are dual vector spaces. It is known that $L_\Bbbk(K)=0$ if and only if $K$ is a simplex and $L_\Bbbk(K)\leq 1$ if and only if $K=\Cliq(\Gamma)$ for a chordal graph $\Gamma$ (recall that $\Gamma$ is \emph{chordal}\index{chordal} if it contains no induced cycle on four or more vertices).  See~\cite{Kalai1,Kalai2,Woodroofe,Wegner} for details, as well as for connections with Castelnuovo-Mumford regularity of Stanley-Reisner rings and with representability of simplicial complexes as nerves of families of compact convex subspaces of Euclidean space.

\subsection{Regular CW complexes and CW posets}
\nomenclature[C, 01]{$X$}{regular CW complex}%
We use the books of Fritsch and Piccinini~\cite{Fritsch} and Lundell and Weingram~\cite{Lundell} as our main references on CW complexes.
A \emph{regular CW complex}\index{regular CW complex} is a CW complex all of whose attaching maps are homeomorphisms to their images.  Simplicial complexes, convex polytopes and polyhedral complexes are examples.  In this text, we consider only finite CW complexes.  Let us give the formal definition.

\nomenclature[C, 03]{$\ov e$}{the closed cell corresponding to an open cell $e$}%
\nomenclature[C, 04]{$\bd e$}{the boundary of the closed cell $\ov e$}%
Let $E^n$ be the closed unit ball in $\mathbb R^n$ and $S^{n-1}$ be the unit sphere.  Let $X$ be a Hausdorff space  and  let $\Phi$ be a finite  collection of continuous injective maps $\varphi\colon E^n\to X$ (ranging over different $n$).  We call $\varphi(E^n)$ a \emph{closed $n$-cell}\index{closed $n$-cell} of $(X,\Phi)$ and $\varphi(E^n\setminus S^{n-1})$ an \emph{open $n$-cell}\index{open $n$-cell}.  We usually denote an open cell by $e$, the corresponding closed cell by $\ov e$ (which will be the topological closure of $e$) and we put $\bd e=\ov e-e$.  The pair $(X,\Phi)$ is called a \emph{regular CW complex}\index{regular CW complex} if:
\begin{enumerate}
\item the open cells partition $X$;
\item for each open $n$-cell $e$ with $n\geq 1$, we have that $\partial e$ is contained in the union of the open $k$-cells with $k<n$.
\end{enumerate}
 A subcomplex  of $(X,\Phi)$ is a CW complex $(Y,\Psi)$ with $Y$ a subspace of $X$ and $\Psi\subseteq \Phi$.  Equivalently, $Y$ is a subspace with the property that it contains the closure of every open cell that it intersects.  Note that each closed cell is a subcomplex in a regular CW complex; see~\cite[Theorem III.2.1]{Lundell} or~\cite[Theorem~1.4.10]{Fritsch}. A $0$-cell of a CW complex is called a \emph{vertex}\index{vertex}.

\nomenclature[C, 02]{$X^q$}{$q$-skeleton of the CW complex $X$}%
If $q\geq 0$, then the \emph{$q$-skeleton}\index{$q$-skeleton} of $X$ is the subcomplex $X^q$ consisting of all closed $k$-cells with $k\leq q$.  A continuous mapping $f\colon X\to Y$ of CW complexes is called \emph{cellular}\index{cellular} if $f(X^q)\subseteq Y^q$ for all $q\geq 0$. A cellular mapping $f\colon X\to Y$ is called \emph{regular}\index{regular} if it maps each open cell of $X$ onto an open cell of $Y$~\cite{Fritsch,Lundell}. One can then show that the image under a regular cellular mapping $f$ of each closed cell is a closed cell and that the image of each subcomplex is a subcomplex (cf.~\cite[Proposition~2.1.1]{Fritsch}); in fact, if $e$ is an open cell, then $f(\ov e)=\ov {f(e)}$, as is shown in the proof of~\cite[Proposition~2.1.1]{Fritsch}.  A bijective regular cellular map is called an \emph{isomorphism}\index{isomorphism!CW complexes}.  The inverse of a bijective regular cellular map is automatically a regular cellular map~\cite[Proposition~I.4.7]{Lundell}.

\nomenclature[P, 20]{$\mathcal P(X)$}{face poset of the regular CW complex $X$}%
\nomenclature[C, 07]{$\mathcal P(X)$}{face poset of the regular CW complex $X$}%
The face poset $\mathcal P(X)$ of a regular CW complex $X$ is the set of open cells of $X$ ordered by $\sigma\leq \tau$ if $\sigma\subseteq \ov \tau$.  We do not admit an empty face.

\nomenclature[P, 08]{$\dim(\sigma)$}{dimension of $\sigma$ in a graded (ranked) poset $P$}%
\nomenclature[P, 07]{$\rk[\sigma, \tau]$}{rank of a closed interval in a poset}%
A (finite) poset $P$ is said to be \emph{graded}\index{graded} (or ranked) if $\Delta(P_{\leq p})$ is a pure simplicial complex for each $p\in P$.  Define the \emph{dimension}\index{dimension} $\dim \sigma$ of $\sigma\in P$ by $\dim \sigma = \dim \Delta(P_{\leq \sigma})$. The \emph{rank}\index{rank} of a closed interval $[\sigma,\tau]$ in a graded poset is given by \[\rk[\sigma,\tau]=\dim \tau-\dim \sigma.\] It is the length of all maximal chains starting at $\sigma$ and ending at $\tau$.  The face poset of a regular CW complex is graded (cf.~\cite[Corollary~4.7.12]{OrientedMatroids1999}).

A graded poset $P$ is called a \emph{CW poset}\index{CW poset}\index{poset!CW} if $\|\Delta(P_{<\sigma})\|$ is homeomorphic to a $(\dim \sigma-1)$-sphere for each $\sigma\in P$ (where for this purpose we consider $S^{-1}=\emptyset$). In this case $\|\Delta(P_{\leq \sigma})\|$ is a cone on $\|\Delta(P_{<\sigma})\|$ and hence a topological $\dim \sigma$-ball.  Any (non-empty) lower set of a CW poset is a CW poset in its own right. A \emph{lower set}\index{lower set} of a poset is a  subset $L$ such that $x\in L$ and $y\leq x$ implies $y\in L$; an upper set\index{upper set} is defined dually.  An element $\sigma$ of a CW poset with $\dim \sigma=q$ will be called a \emph{$q$-cell}\index{$q$-cell}.

If $X$ is a regular CW complex, then $\|\Delta(\mathcal P(X))\|$ is homeomorphic to $X$ and $\Delta(\mathcal P(X))$ is the barycentric subdivision of $X$~\cite{OrientedMatroids1999,Massey,Fritsch,Lundell}.  In particular, if $e$ is an open $n$-cell of $X$, then $\bd e$ is a subcomplex and $\|\Delta(\mathcal P(X)_{<e})\|=\|\Delta(\mathcal P(\bd e))\|\cong S^{n-1}$. Thus if $X$ is a regular CW complex, then $\mathcal P(X)$ is a CW poset and lower sets correspond to subcomplexes of $X$.   If $f\colon X\to Y$ is a regular cellular map, then it induces an order preserving map \[\mathcal P(f)\colon \mathcal P(X)\to \mathcal P(Y)\] via $e\mapsto f(e)$.  An additional property of $\mathcal P(f)$ will be exploited later.

\nomenclature[C, 08]{$\CW(P)$}{regular CW complex associated with the CW poset $P$}%
\nomenclature[P, 21]{$\CW(P)$}{regular CW complex associated with the CW poset $P$}%
Conversely, there is a functor $\CW$ from CW posets and appropriate morphisms to regular CW complexes and regular cellular maps which can be described on objects as follows.  The geometric realization $\|\Delta(P)\|$ of the order complex of a CW poset $P$ admits a regular CW complex structure $\CW(P)$ whose closed cells are the subcomplexes of $\|\Delta(P)\|$ of the form $\ov e_a=\|\Delta(P_{\leq a})\|$ with $a\in P$. Note that $\ov e_a$ is indeed a $\dim a$-ball as described above.  The open cell $e_a$ corresponding to $a\in P$ is the union of the (relative interiors of the) geometric simplices of $\|\Delta(P)\|$ spanned by chains with maximal element $a$.  Hence each point of $\|\Delta(P)\|$ belongs to a unique open cell.  Indeed, a mapping $f\colon V\to \mathbb R$ from $\|\Delta(P)\|$ belongs to $e_a$ if and only if $a$ is the largest element of its support.  Also, $\bd e_a=\|\Delta(P_{<a})\|$ and is contained in the union of the open cells $e_b$ with $b<a$.  Thus $\|\Delta(P)\|$, together with the closed cells $\{\ov e_a\}_{a\in P}$, is a regular CW complex $\CW(P)$.   Notice that $P\cong \mathcal P(\CW(P))$ via the map $a\mapsto e_a$.
On the other hand, if $X$ is a regular CW complex, the homeomorphism $\|\Delta(\mathcal P(X))\|\to X$ of~\cite[Proposition~4.7.9]{OrientedMatroids1999}  is induced by an invertible regular cellular map $\CW(\mathcal P(X))\to X$ and so $X\cong \CW(\mathcal P(X))$.

Let us say that an order preserving map $f\colon P\to Q$ of posets is \emph{cellular}\index{cellular} if $\tau\leq f(\sigma)$ implies there exists $\sigma'\leq \sigma$ with $f(\sigma')=\tau$. Equivalently, $f$ is cellular if $f(P_{\leq \sigma}) = Q_{\leq f(\sigma)}$ for all $\sigma\in P$. Alternatively, $f$ is cellular if $f(L)$ is a lower set for each lower set $L$ of $P$.

\begin{Rmk}
 The reason for the terminology ``cellular'' is that we shall see that cellular maps correspond to regular cellular mappings between CW complexes.  However, there are other reasonable terminologies.

 We remark that one can place a topology, called the \emph{Alexandrov topology}\index{Alexandrov topology}, on any poset $P$ in which the lower sets of $P$ are the open sets.  Then order preserving maps are precisely the continuous mappings between posets equipped with their Alexandrov topologies.  Cellular mappings of posets correspond to the open mappings between posets with respect to their Alexandrov topologies.  However, the term open mapping would be confusing when we later realize these maps as maps between CW complexes.

 In~\cite{THA} a stronger notion of poset maps $f\colon P\to Q$ is considered, namely those which restrict to a bijection $P_{\leq \sigma}\to Q_{\leq f(\sigma)}$ for all $\sigma\in P$.  These are often called \emph{discrete fibrations}\index{discrete fibration} in category theory and correspond precisely to local homeomorphisms (or \'etale maps) of posets with respect to their Alexandrov topologies and are in bijection with presheaves on the poset.  Perhaps for this reason they are called ``coverings'' in~\cite{THA}.
\end{Rmk}

\begin{Prop}\label{faceposetasfunctor}
Let $f\colon X\to Y$ be a regular cellular map between regular cell complexes.  Then the induced map $\mathcal P(f)\colon \mathcal P(X)\to \mathcal P(Y)$ is cellular.
\end{Prop}
\begin{proof}
Clearly, $\mathcal P(f)$ is order preserving.  Let $e\in \mathcal P(X)$. Suppose that $e'\leq f(e)$, that is, $e'\subseteq \ov {f(e)}$.  Let $y\in e'$. Then as $f(\ov e)=\ov {f(e)}$ (by the proof of~\cite[Proposition~2.1.1]{Fritsch}), we can choose $x\in \ov e$ such that $f(x)=y$.  There is a unique open cell $e''$ contained in $\ov e$ such that $x$ belongs to $e''$.  Then $e''\leq e$ and $f(e'')$ is an open cell containing $y$.  Since each element of a CW complex belongs to a unique open cell, we conclude that $f(e'')=e'$.  This concludes the proof that $\mathcal P(f)$ is cellular.
\end{proof}

It follows that $\mathcal P$ is a functor from the category of regular CW complexes and regular cellular maps to the category of CW posets and cellular maps.  We now want to show that $\CW$ extends to a functor going the opposite way.  The functoriality of $\CW$ seems to be novel to this text.

\begin{Lemma}\label{l:chainlifting}
Let $f\colon P\to Q$ be a cellular map of posets.  Then given any chain $\tau_0<\cdots<\tau_q=\tau$ in $Q$ and any $\sigma\in P$ with $f(\sigma)=\tau$, there is a chain $\sigma_0<\cdots<\sigma_q=\sigma$ in $P$ with $f(\sigma_i)=\tau_i$ for $0\leq i\leq q$.
\end{Lemma}
\begin{proof}
Set $\sigma_q=\sigma$.  Let $0\leq i<q$ and assume that $\sigma_j$ has been defined for $i<j\leq q$ with $\sigma_{i+1}<\cdots<\sigma_q=\sigma$ and $f(\sigma_j) = \tau_j$.  By definition of a cellular map there exists $\sigma_i<\sigma_{i+1}$ such that $f(\sigma_i)=\tau_i$. This completes the proof.
\end{proof}

\begin{Lemma}
If $f\colon P\to Q$ is a cellular map of CW posets, then the induced simplicial map $\Delta(f)\colon \Delta(P)\to \Delta(Q)$ induces a regular cellular map $\CW(f)\colon \CW(P)\to \CW(Q)$.
\end{Lemma}
\begin{proof}
Lemma~\ref{l:chainlifting} easily implies that if $p\in P$, then $\Delta(f)$ takes the closed cell $\Delta(P_{\leq p})$ of $\CW(P)$ onto the closed cell  $\Delta(Q_{\leq f(p)})$ of $\CW(Q)$ and $\dim \Delta(Q_{\leq f(p)})\leq \dim \Delta(P_{\leq p})$.  Thus $\CW(f)$ is cellular.  The open cell of $\CW(P)$ corresponding to $p\in P$ is the union (of the relative interiors) of those simplices in $\|\Delta(P)\|$ spanned by chains with maximum element $p$.  Hence $\Delta(f)$ takes the open cell corresponding to $p$ into the open cell corresponding to $f(p)$, and in fact onto it by Lemma~\ref{l:chainlifting}. Thus $\Delta(f)$ is a regular cellular map.
\end{proof}

We conclude that $\CW$ is a functor from the category of CW posets and cellular maps to the category of regular CW complexes and regular cellular maps with the property that $\CW(\mathcal P(X))\cong X$ for any regular CW complex $X$.

\begin{Rmk}
It is easy to see that $\mathcal P\circ \CW$ is isomorphic to the identity functor.  But it is possible that $\mathcal P(f)=\mathcal P(g)$ for two distinct regular cellular maps $f,g\colon X\to Y$.  Nonetheless, it is not difficult to show that $\CW(\mathcal P(f))$ is always homotopy equivalent to $f$ for any regular cellular map (by an aspherical carrier type argument).  In particular, as far as homology is concerned there is not much of a difference between working with CW posets and cellular maps and with regular CW complexes and regular cellular maps.
\end{Rmk}

Notice that if $K,L$ are simplicial complexes and $f\colon K\to L$ is a simplicial map, then the induced map $\mathcal P(K)\to \mathcal P(L)$ is cellular.  Moreover, $\CW(f)\colon K=\CW(\mathcal P(K))\to \CW(\mathcal P(L))=L$ is the map of geometric realizations induced by the simplicial map $f$. So simplicial complexes with simplicial maps can be viewed either as a subcategory of CW posets and cellular maps, or a subcategory of regular CW complexes with regular cellular maps.

\begin{Rmk}
To summarize we have a diagram of functors
\begin{equation*}
    \xymatrix{%
        \bigg\{
            \substack{\text{\normalsize{regular CW complexes,}} \\
                      \text{\normalsize{regular cellular mappings}}}
        \bigg\}
        \ar@<1ex>[r]^{\mathcal P}
        & \ar@<1ex>[l]^{\CW}
        \bigg\{
            \text{\normalsize{CW posets, cellular maps}}
        \bigg\}
    }
\end{equation*}
where the functors induce bijections of isomorphism classes of objects but the left hand category has more morphisms.  The category of simplicial complexes embeds in both categories.
\end{Rmk}

We now state some further properties of cellular maps that will be used in our study of left regular bands. The first shows that the support homomorphism is cellular.

\begin{Prop}\label{p:supportiscellular}
If $B$ is a left regular band and $\sigma\colon B\to \Lambda(B)$ is the support map, then $\sigma$ is cellular.  Moreover, $a<b$ implies $\sigma(a)<\sigma(b)$.
\end{Prop}
\begin{proof}
The second statement follows because $a\leq b$ and $\sigma(a)=\sigma(b)$ implies $ba=a$ and $ba=b$, whence $a=b$.  To see that $\sigma$ is cellular, let $X\leq \sigma(b)$.  Choose $a$ with $Ba=X$.  Then $\sigma(ba)=\sigma(b)\wedge X=X$ and $ba\leq b$.
\end{proof}

\begin{Cor}\label{c:preservesgraded}
Let $f\colon P\to Q$ be a surjective cellular map of posets and suppose, furthermore, that $\sigma<\sigma'$ implies that $f(\sigma)<f(\sigma')$. Then $P$ is graded if and only if $Q$ is graded and, moreover, if this is the case then $f$ preserves ranks of intervals.
\end{Cor}
\begin{proof}
Let $p\in P$ and $q=f(p)$.
Because $f$ preserves strict inequalities, it is immediate that the simplicial map $f\colon \Delta(P_{\leq p})\to \Delta(Q_{\leq q})$ preserves dimensions of simplices.  Also Lemma~\ref{l:chainlifting} implies that each simplex of $\Delta(Q_{\leq q})$ is the image of a simplex of $\Delta(P_{\leq p})$. It follows that the facets of $\Delta(Q_{\leq q})$ are precisely the images under $f$ of the facets of $\Delta(P_{\leq p})$.  We conclude $\Delta(P_{\leq p})$ is pure if and only if $\Delta(Q_{\leq q})$ is pure.  The result now follows.
\end{proof}

Recall that a topological space is \emph{acyclic}\index{acyclic} if its reduced homology vanishes over $\mathbb Z$. Contractible spaces are of course acyclic.  If $X$ is acyclic, then the reduced homology of $X$ vanishes over any coefficient ring by the universal coefficient theorem.  If $P$ is a poset, we will say that $P$ has a topological property (e.g., contractibility, acyclicity, connectivity) if $\|\Delta(P)\|$ has that property.

\begin{Prop}\label{retract}
Let $f\colon P\to P$ be an idempotent cellular mapping on a poset.  Then $f(P)$ is a lower set and hence if $f(p)=p$, then $f(q)=q$ for all $q\leq p$.
If $P$ is a (connected/contractible/acyclic) CW poset then so is $f(P)$.
\end{Prop}
\begin{proof}
The image $f(P)$ of the lower set $P$ under a cellular map is a lower set.  Hence, if $f(p)=p$ and $q\leq p$, then $q\in f(P)$ and thus $f(q)=q$ because $f$ is idempotent. If $P$ is a CW poset, then the subcomplex $\CW(f(P))=\CW(f)(\CW(P))$ is a retract of $\CW(P)$ and the result follows.
\end{proof}

The following trivial observation will be useful later.

\begin{Prop}\label{p:maximumCWposet}
Let $P$ be a CW poset with a maximum.  Then $\Sigma(P)$ is a regular CW decomposition of a $\dim \Delta(P)$-ball.  Conversely, if $Q$ is the face poset of a regular CW decomposition of a $\dim \Delta(Q)$-sphere, then $P=Q\cup \{\wh 1\}$ (where $\wh 1$ is an adjoined maximum) is the face poset of a regular CW decomposition of a $\dim \Delta(P)$-ball.
\end{Prop}
\begin{proof}
Suppose first that $P$ is a CW poset with maximum $\wh 1$.  Then by definition $\|\Delta(P_{<\wh 1})\|$ is homeomorphic to a $(\dim \Delta(P)-1)$-sphere and so $\|\Delta(P)\|=\|\Delta(P_{<\wh 1})\ast \wh 1\|$ is homeomorphic to a $\dim \Delta(P)$-ball, being a cone on a sphere.  Suppose now that $Q$ is the face poset of a regular CW decomposition of a $\dim \Delta(Q)$-sphere and $P=Q\cup \{\wh 1\}$.  Then for $q\in Q$, we have $P_{<q}=Q_{<q}$ and $P_{<\wh 1}=Q$ and so $P$ is a CW poset.  By the above $P$ is the face poset of a $\dim \Delta(P)$-ball.
\end{proof}

The following propositions will also be helpful.

\begin{Prop}\label{p:joinsofcw}
Suppose that $P,Q$ are CW posets with \[\|\Delta(P)\|\cong S^{\dim(\Delta(P))}.\]  Then $P\ast Q$ is a CW poset.  In particular, the suspension $\mathsf S(Q)$ of $Q$ is a CW poset.
\end{Prop}
\begin{proof}
Clearly if $p\in P$, then $(P\ast Q)_{<p}=P_{<p}$ and hence has order complex a sphere of the appropriate dimension. If $q\in Q$, then $(P\ast Q)_{\leq q} = P\ast Q_{\leq q}$ has dimension $\dim \Delta(P)+\dim \Delta(Q_{\leq q})+1$.  On the other hand, $(P\ast Q)_{<q} = P\ast Q_{<q}$ and so $\|\Delta((P\ast Q)_{<q})\|$ is a sphere of the desired dimension from the formula $S^m\ast S^n=S^{m+n+1}$. The final statement is the special case when $\|\Delta(P)\|\cong S^0$.
\end{proof}

Recall that if $P,Q$ are posets, then their direct product is a poset by putting $(p,q)\leq (p',q')$ if $p\leq p'$ and $q\leq q'$.

\begin{Prop}\label{p:productofcw}
If $P,P'$ are  CW posets, then so is $P\times P'$.
\end{Prop}
\begin{proof}
Suppose that $(p,p')\in P\times P'$, $\dim \Delta(P_{\leq p})=r$ and $\dim \Delta(P'_{\leq p'})=s$. Clearly $\dim \Delta((P\times P')_{\leq (p,p')})=r+s$.    A general result in poset topology~\cite[Equation~(9.7)]{bjornersurvey} says that \[\|\Delta((P\times P')_{<(p,p')})\|\cong \|\Delta(P_{<p})\ast \Delta(P'_{<p'})\|\cong S^{r-1}\ast S^{s-1}\cong S^{r+s-1}.\]  Thus $P\times P'$ is a CW poset.
\end{proof}

\subsection{Oriented interval greedoids}\label{ss:OIGs}
In 2008, Billera, Hsiao, and Provan introduced a construction that associates
with every antimatroid (see below) a polytopal subdivision of a sphere~\cite{BilleraHsiaoProvan2008}.
This result is structurally similar to the sphericity theorem for oriented
matroids, which associates a regular cell decomposition of the sphere with
every oriented matroid (as already discussed in
Section~\ref{ssec:hyperplane-arrangements-and-oriented-matroids}).
The second author and Thomas developed a theory of ``oriented interval greedoids,'' akin
to the theory of ``oriented matroids,'' in order to unify these results~\cite{SaliolaThomas}. Their starting point was the observation that both
complexes can be endowed with the structure of a left regular band monoid.
We review the results here.

\subsubsection{Motivation from oriented matroids}

In order to motivate the definition of an oriented interval greedoid,
we begin by recasting the definition of a covector for an oriented matroid.

Let $\mathcal M=(E,\mathscr F)$ be a matroid, where we retain the notation of Section~\ref{ss:matroid}.
A note on terminology is in order. In the matroid
literature, the notions of ``closed set'' and ``flat'' coincide, but these
admit different generalizations for greedoids.
Each notion gives rise to a poset (with the former a lattice), which turn out
to be isomorphic to each other in the case of interval
greedoids~\cite[Theorem~3.6]{HomotopyPropertiesofGreedoids}.
In particular, the greedoid notion of flat does not conform with the matroid
notion of flat (cf.~Section~\ref{ss:matroid}).
To distinguish between these notions, we will use the term
``greedoid flat'' for flats in the sense of greedoid theory.

Introduce an equivalence relation on $\FeasibleSets$ by declaring
$X \sim Y$ if and only if $\ov X=\ov Y$.
The equivalence classes for this relation, denoted by $[X]$, are called the
\emph{greedoid flats}\index{greedoid flat} of the matroid.
The set $\Phi$ of all greedoid flats of the matroid admits a partial order induced by
inclusion: $[Y] \leq [X]$ if and only if there exists $Z \in E \setminus Y$ with $Y \cup
Z \in \FeasibleSets$ and $Y \cup Z \sim X$.
In particular, we have $[Y] \leq [X]$ whenever $Y \subseteq X$.

For $X \in \FeasibleSets$, let $\xi(X)$ be the union of all subsets in $[X]$:
\begin{equation*}
    \xi(X)= \bigcup_{Y \sim X} Y.
\end{equation*}
For matroids without loops (recall that a loop is an element that does not
belong to any set in $\FeasibleSets$), the sets $\xi(X)$ are the
\emph{flats}\index{flat} of the matroid (but will be called closed sets in the greedoid context). In this case,
the elements of $\xi(X)$ are those elements that can be added to $X$ without
increasing its rank.
Define $\Gamma(X)$ to be the set of  $e\in E$ such that $X\cup \{e\}\in \FeasibleSets$.

As briefly pointed out in
Section~\ref{ssec:hyperplane-arrangements-and-oriented-matroids},
there is a matroid underlying every oriented matroid. For convenience, assume that the oriented matroid contains no loops.
It is readily obtained by passing from the set of covectors
of the oriented matroid to their zero sets:
\begin{equation*}
    x \longmapsto Z(x) = \{e \in E \mid x_e = 0\}.
\end{equation*}
These are precisely the flats of a unique matroid (called the
\emph{underlying matroid}\index{underlying matroid} of the oriented matroid).
With this terminology and notation, it is straightforward to see that
for every covector $x$ of an oriented matroid, there exists a greedoid flat $[X]$
of the underlying matroid satisfying:
\begin{itemize}
    \item
        $x_e = 0$ for each $e \in \xi(X)$
        (i.e., $x_e = 0$ if and only if $e\in \ov X$);

    \item
        $x_e \in \{+, -\}$ for all $e \in \Gamma(X)$
        (i.e., $x_e = \pm$ if and only if $X \cup \{e\}$ is independent).
\end{itemize}
The notion of a covector for an interval greedoid is based on these properties.

\subsubsection{Interval greedoids}

Let $E$ denote a finite set and $\FeasibleSets$ a set of subsets of $E$.
An \emph{interval greedoid}\index{interval greedoid} is a pair $(E,\FeasibleSets)$ satisfying the
following properties for all $X,Y,Z\in\FeasibleSets$:
\begin{itemize}
    \item[(IG1)]
        If $X\neq\varnothing$, then there exists an $x \in X$ such that
        $X \setminus \{x\} \in \FeasibleSets$.
    \item[(IG2)]
        If $|X| > |Y|$, then there exists an $x \in X \setminus Y$ such that $Y \cup
        x \in \FeasibleSets$.
    \item[(IG3)]
        If $X \subseteq Y \subseteq Z$ and $e \in E \setminus Z$ with $X \cup e \in
        \FeasibleSets$ and $Z \cup e \in \FeasibleSets$, then $Y \cup e \in
        \FeasibleSets$.
\end{itemize}
A \emph{loop}\index{loop} in an interval greedoid $(E, \FeasibleSets)$ is an
element $e \in E$ that belongs to no subset in $\FeasibleSets$.
We will restrict our attention to interval greedoids that do not contain loops.

Every matroid is an interval greedoid. Further examples of interval greedoids
will be given below (in particular, antimatroids/convex geometries).

Let $(E, \FeasibleSets)$ be an interval greedoid. For $X \in \FeasibleSets$,
let $\FeasibleSets/X$ denote the collection of subsets that can be added to $X$
with the result still in $\FeasibleSets$:
\begin{gather*}
    \FeasibleSets/X = \left\{ Y \subseteq E \setminus X \mid Y \cup X \in \FeasibleSets \right\}.
\end{gather*}
Similarly, the set of \emph{continuations}\index{continuations} $\Gamma(X)$ of
$X \in \FeasibleSets$ is the set of elements in $E$ that can be added to $X$
with the result still in $\FeasibleSets$:
\begin{gather*}
    \Gamma(X) = \{ x \in E \setminus X \mid X \cup \{x\} \in \FeasibleSets\}.
\end{gather*}

Define an equivalence relation on $\FeasibleSets$ by declaring
$X \sim Y$ if and only if $\FeasibleSets/X = \FeasibleSets/Y$.
For interval greedoids, $\FeasibleSets/X = \FeasibleSets/Y$ if and only if
$\Gamma(X) = \Gamma(Y)$.
The equivalence classes of this relation are called \emph{greedoid flats}\index{greedoid flat}\index{flat!greedoid}
and we denote the equivalence class of $X$ by $[X]$.
The set of greedoid flats $\Phi$ is a poset with partial order induced by
inclusion, analogously to the matroid case. In fact, $\Phi$ is a \emph{semimodular lattice}\index{semimodular!lattice}\index{lattice!semimodular}~\cite[Theorem~5.3]{bjorner2}, that is, a lattice satisfying axioms (G2) and (G3) of Definition~\ref{d:geom.lattice}.
In particular, the minimal element of $\Phi$ is $[\varnothing]$, and $[Y] \leq
[X]$ whenever $Y \subseteq X$.
For a greedoid flat $[X]$, define
\begin{equation*}
    \xi(X) = \bigcup_{\substack{Y \in \FeasibleSets \\ \Gamma(Y) = \Gamma(X)}} Y.
\end{equation*}
The sets of the form $\xi(X)$ are the \emph{closed sets}\index{closed set} of the interval greedoid.

\subsubsection{Oriented interval greedoids}
A greedoid flat $[X]$ partitions the ground set $E$ into three subsets (possibly empty):
\begin{itemize}
    \item
        $\Gamma(X)$, the set of continuations of $X$;
    \item
        $\xi(X)$, the set of elements that lie on the greedoid flat; and
    \item
        the set of elements that belong to neither $\Gamma(X)$ nor $\xi(X)$.
\end{itemize}
A \emph{covector with underlying greedoid flat $[X]$}\index{covector!oriented interval greedoid} is any vector
$x \in \{0, +, -, 1\}^E$ satisfying:
\begin{itemize}
    \item
        $x_e \in \{+, -\}$ if $e \in \Gamma(X)$;
    \item
        $x_e = 0$ if $x \in \xi(X)$;
    \item
        $x_e = 1$ if $x \notin \Gamma(X) \cup \xi(X)$.
\end{itemize}
Let $x\mapsto \Flat(x)$ denote the map that sends a covector to its
underlying greedoid flat.
For two covectors $x$ and $y$ of $(E,\FeasibleSets)$, define their
\emph{separation set}\index{separation set} by
\begin{gather*}
\SeparationSet(x,y) = \left\{e \in E \mid x_e = -y_e \in \{+,-\}\right\}.
\end{gather*}

To define the \emph{composition of covectors}\index{composition!of covectors}, we first need an auxiliary
operation.
For $x, y \in \{0, +, -, 1\}^E$, define
\begin{align*}
    (x \star y)_e =
    \begin{cases}
        y_e, & \text{if } y_e < x_e, \\
        x_e, & \text{otherwise},
    \end{cases}
\end{align*}
where the symbols $0,+,-,1$ are partially ordered so that $1$ and $0$ are the
(unique) minimal and maximal elements, respectively, and $+$ and $-$ are
incomparable. See Figure~\ref{f:Hasse.OIG} for the Hasse diagram of $\{0,+,-,1\}$.  Note that we use the opposite ordering convention of~\cite{SaliolaThomas}.
\begin{figure}[tbhp]
\begin{center}
\begin{tikzpicture}
\node  	(A) 								{$0$};
\node  	(B) [below left =1cm of A] 		{$+$};
\node  	(C) [below right=1cm of A] 	{$-$};
\node 	(D) [below=2cm of A]			{$1$};
\draw (A)--(B);
\draw (A)--(C);
\draw (B)--(D);
\draw (C)--(D);
\end{tikzpicture}
\end{center}
\caption{The Hasse diagram of $\{0,+,-,1\}$\label{f:Hasse.OIG}}
\end{figure}
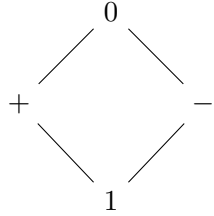
Then the composition $x \circ y$ of two covectors $x$ and $y$ with underlying
greedoid flats $[X]$ and $[Y]$, respectively, is defined as
\begin{align*}
    (x \circ y)_e =
    \begin{cases}
        (x \star y)_e, & \text{if } e \in \Gamma(W) \cup \xi(W), \\
        \hfill 1, & \text{otherwise}.
    \end{cases}
\end{align*}
where $W$ is maximal among the sets in $\FeasibleSets$ that are contained in
$\xi(X) \cap \xi(Y)$.
(Note that, unlike for oriented matroids, this operation is not necessarily
computed componentwise. See Figure~\ref{f:ProductOfCovectorsInAConvexGeometry}
and the surrounding discussion.)

Since we are restricting our attention to interval greedoids without loops, the
all zero covector is the identity element for this operation.
(If the interval greedoid admits loops, then the identity element is the
covector $w$ with $w_e = 1$ if $e$ is a loop and $w_e = 0$, otherwise.)

An \emph{oriented interval greedoid}\index{oriented interval greedoid} is a triple $(E, \FeasibleSets, \OIG)$,
where $(E,\FeasibleSets)$ is an interval greedoid and $\OIG$ is a set of
covectors of $(E,\FeasibleSets)$ satisfying the following axioms.
\begin{itemize}
    \item[(OG1)]
        The map $\Flat\colon \OIG \to \Phi$ is surjective,
        where $\Phi$ is the lattice of greedoid flats of $(E, \FeasibleSets)$.

    \item[(OG2)]
        If $x \in \OIG$, then $-x \in \OIG$ (where $-1=1$).

    \item[(OG3)]
        If $x, y \in \OIG$, then $x \circ y \in \OIG$.

    \item[(OG4)]
        If $x,y\in\OIG$, $e\in\SeparationSet(x,y)$ and
        $(x\circ y)_e \neq1$, then there exists $z\in\OIG$ such
        that $z_e = 0$ and for all $f\notin\SeparationSet(x,y)$,
        if $(x \circ y)_f \neq 1$,
        then $z_f = (x \circ y)_f = (y \circ x)_f$.
\end{itemize}

These axioms are modeled on the covector axioms for oriented matroids; compare
with Definition~\ref{defn:oriented-matroids}.
Moreover, if $(E, \FeasibleSets)$ is a matroid without loops, then this notion
coincides with that of an oriented matroid.
More precisely, $\mathcal L$ is the collection of covectors of an oriented
matroid with underlying matroid $(E,\FeasibleSets)$ if and only if $(E, \FeasibleSets,
\mathcal L)$ is an oriented interval greedoid~\cite[Theorem~3.12]{SaliolaThomas}.
Similarly, if the underlying interval greedoid is an antimatroid (see below),
then this recovers the Billera--Hsiao--Provan complex~\cite[Theorem~3.14]{SaliolaThomas}.
Other examples of oriented interval greedoids include complexified hyperplane
arrangements~\cite[Theorem~3.20]{SaliolaThomas}.

The covectors of an oriented interval greedoid form a left regular band
monoid.

\begin{Prop}
    Let $\OIG$ be the set of covectors for an oriented interval greedoid on an
    interval greedoid without loops $(E, \FeasibleSets)$.
\begin{enumerate}
\item  $\OIG$ equipped with the composition of covectors forms a left regular
    band monoid with identity element equal to the all zero covector.
\item The support lattice $\Lambda(\OIG)$ is isomorphic to the lattice of greedoid flats
    $\Phi$ of $(E, \FeasibleSets)$ and the support map corresponds to
    $\Flat\colon \OIG \to \Phi$.
\end{enumerate}
\end{Prop}

Since $\OIG$ is a left regular band monoid, it is a poset
with respect to the $\R$-order, as we have been considering previously. The order is characterized by
$x \geq y$ if and only if $x \circ y = y$, or equivalently
$x \geq y$ if and only if $x_e \geq y_e$ for all $e \in E$~\cite[Lemma~3.7]{SaliolaThomas}.
The poset $\OIG^{op}\setminus \{0\}$ obtained from $\OIG$ by considering the \emph{opposite}
of the $\R$-order and removing the identity is the face poset of a regular CW complex.   The following result is~\cite[Theorem~6.1]{SaliolaThomas}, where the reader is referred to~\cite[Chapter~4]{OrientedMatroids1999} for the notions of shellable and piecewise linear regular CW decompositions of a sphere.

\begin{Thm}\label{t:oig.sphere}
    Let $\OIG$ denote the set of covectors of an oriented interval
    greedoid. Then $\OIG^{op}\setminus \{0\}$ is isomorphic to the face poset of
    a shellable (hence piecewise linear) regular CW decomposition of a sphere.
\end{Thm}

Here is an example extracted from~\cite[Example~6.2]{SaliolaThomas}, which
we illustrate in Figures~\ref{fig:oig-lattice-of-greedoid flats},~\ref{fig:oig-sphere}
and~\ref{fig:oig-sphere-poset}.
Consider the interval greedoid on the ground set $E=\{a,b,c,e\}$
defined by
\begin{equation}
    \label{eqn:feasible-sets-of-oig-example}
    \FeasibleSets=\{\emptyset\} \cup F \cup
    \{X\cup \{e\} \mid X\in F\},
\end{equation}
where $F=\{\{a\},\{b\},\{c\},\{a,b\},\{a,c\},\{b,c\}\}$.
The lattice of greedoid flats $\Phi$ of this interval greedoid is depicted in
Figure~\ref{fig:oig-lattice-of-greedoid flats}.
\begin{figure}[htpb]
    \centering
    \begin{tikzpicture}[>=latex,line join=bevel]
      \node (node_8) at (107.5bp,114.5bp) [draw,draw=none] { $\Big\{ \{a,b\}, \{a,c\}, \{b, c\} \Big\}$};
      \node (node_7) at (19.5bp,114.5bp) [draw,draw=none]  { $\Big\{ \{c, e\} \Big\}$};
      \node (node_6) at (83.5bp,61.5bp) [draw,draw=none]   { $\Big\{ \{c\} \Big\}$};
      \node (node_5) at (195.5bp,114.5bp) [draw,draw=none] { $\Big\{ \{a, e\} \Big\}$};
      \node (node_4) at (179.5bp,61.5bp) [draw,draw=none]  { $\Big\{ \{b\} \Big\}$};
      \node (node_3) at (131.5bp,61.5bp) [draw,draw=none]  { $\Big\{ \{a\} \Big\}$};
      \node (node_2) at (252.5bp,114.5bp) [draw,draw=none] { $\Big\{ \{b, e\} \Big\}$};
      \node (node_1) at (131.5bp,8.5bp) [draw,draw=none]   { $\Big\{ \varnothing \Big\}$};
      \node (node_0) at (151.5bp,167.5bp) [draw,draw=none] { $\Big\{ \{a, b, e\}, \{a, c, e\}, \{b, c, e\} \Big\}$};
      \draw [black] (node_1) ..controls (145.6bp,24.483bp) and (156.55bp,36.114bp)  .. (node_4);
      \draw [black] (node_4) ..controls (157.71bp,77.935bp) and (139.91bp,90.539bp)  .. (node_8);
      \draw [black] (node_8) ..controls (120.36bp,130.41bp) and (130.26bp,141.88bp)  .. (node_0);
      \draw [black] (node_6) ..controls (64.32bp,77.784bp) and (48.914bp,90.061bp)  .. (node_7);
      \draw [black] (node_2) ..controls (221.26bp,131.27bp) and (194.81bp,144.63bp)  .. (node_0);
      \draw [black] (node_5) ..controls (182.64bp,130.41bp) and (172.74bp,141.88bp)  .. (node_0);
      \draw [black] (node_1) ..controls (131.5bp,23.805bp) and (131.5bp,34.034bp)  .. (node_3);
      \draw [black] (node_3) ..controls (150.68bp,77.784bp) and (166.09bp,90.061bp)  .. (node_5);
      \draw [black] (node_1) ..controls (117.4bp,24.483bp) and (106.45bp,36.114bp)  .. (node_6);
      \draw [black] (node_4) ..controls (201.59bp,77.935bp) and (219.63bp,90.539bp)  .. (node_2);
      \draw [black] (node_3) ..controls (124.66bp,77.031bp) and (119.63bp,87.72bp)  .. (node_8);
      \draw [black] (node_7) ..controls (61.009bp,131.54bp) and (97.114bp,145.49bp)  .. (node_0);
      \draw [black] (node_6) ..controls (90.338bp,77.031bp) and (95.368bp,87.72bp)  .. (node_8);
    \end{tikzpicture}
    \caption{The lattice of greedoid flats $\Phi$ of the interval greedoid defined
        on the ground set $E = \{a, b, c, e\}$ by
        \eqref{eqn:feasible-sets-of-oig-example}.}
    \label{fig:oig-lattice-of-greedoid flats}
\end{figure}
Next, consider a two-dimensional sphere: draw three great circles on it, each
passing through the north and south poles, and mark the points where each great
circle crosses the equator (see Figure~\ref{fig:oig-sphere}).
\begin{figure}[htpb]
    \centering

    \tikzset{%
      >=latex, 
      inner sep=1pt,%
      outer sep=3pt,%
      mark coordinate/.style={inner sep=0pt, outer sep=0pt, minimum size=5pt, fill=black, circle, draw=gray},%
      coordinate label front/.style={draw=gray, fill=white, fill opacity=0.8, draw opacity=0.8, text opacity=1, text=black, rounded corners=1pt},%
      coordinate label back/.style={draw=gray, fill=white, fill opacity=0.3, draw opacity=0.3, text opacity=0.5, text=black, rounded corners=1pt},%
    }

    \resizebox{.75\linewidth}{!}{%
    \begin{tikzpicture}


    \def\sphereradius{3}
    \def\elevationangle{20}


    \makeatletter
    \pgfdeclareradialshading[tikz@ball]{ball}{\pgfqpoint{-10bp}{10bp}}{%
     color(0bp)=(tikz@ball!30!white);
     color(9bp)=(tikz@ball!75!white);
     color(18bp)=(tikz@ball!90!black);
     color(25bp)=(tikz@ball!70!black);
     color(50bp)=(black)}
    \makeatother

    \fill[ball color=white] (0,0) circle (\sphereradius);

    \pgfmathsetmacro\H{\sphereradius*cos(\elevationangle)} 


    \DrawLatitudeCircle[\sphereradius]{0}


    \DrawLongitudeCircleBack[\sphereradius]{0}
    \DrawLongitudeCircleBack[\sphereradius]{60}
    \DrawLongitudeCircleBack[\sphereradius]{120}

    \DrawLongitudeCircleFront[\sphereradius]{0}
    \DrawLongitudeCircleFront[\sphereradius]{60}
    \DrawLongitudeCircleFront[\sphereradius]{120}


    \coordinate[mark coordinate] (N) at (0, \H);
    \coordinate[mark coordinate, fill=black!50] (S) at (0,-\H);

    \coordinate[mark coordinate] (P1) at (-3, -0);
    \coordinate[mark coordinate] (P2) at (-1.50, -0.89);
    \coordinate[mark coordinate] (P3) at ( 1.50, -0.89);

    \coordinate[mark coordinate, fill=black   ] (P1') at ( 3   , 0   );
    \coordinate[mark coordinate, fill=black!50] (P2') at ( 1.50, 0.89);
    \coordinate[mark coordinate, fill=black!50] (P3') at (-1.50, 0.89);


    \def\pointlabelsize{\scriptscriptstyle}

    \node[coordinate label front, above, draw=ForestGreen, text=ForestGreen] at (N)   {$\pointlabelsize 000+$};
    \node[coordinate label front, left , draw=ForestGreen, text=ForestGreen] at (P1)  {$\pointlabelsize 0-+0$};
    \node[coordinate label front, left , draw=ForestGreen, text=ForestGreen] at (P2)  {$\pointlabelsize +0+0$};
    \node[coordinate label front, right, draw=ForestGreen, text=ForestGreen] at (P3)  {$\pointlabelsize ++00$};
    \node[coordinate label front, right, draw=ForestGreen, text=ForestGreen] at (P1') {$\pointlabelsize 0+-0$};
    \node[coordinate label front, right, draw=ForestGreen, text=ForestGreen] at (P2') {$\pointlabelsize -0-0$};
    \node[coordinate label front, left , draw=ForestGreen, text=ForestGreen] at (P3') {$\pointlabelsize --00$};
    \node[coordinate label front, below, draw=ForestGreen, text=ForestGreen] at (S)   {$\pointlabelsize 000-$};

    \def\regionlabelsize{\scriptstyle}

    \def\xcd{0}
    \def\ycd{0.8}
    \node[coordinate label front, text=BrickRed, draw=BrickRed] at (\xcd,  -\ycd) {$\regionlabelsize +++1$};
    \node[coordinate label back , text=BrickRed, draw=BrickRed] at (\xcd,   \ycd) {$\regionlabelsize ---1$};

    \def\xcd{2.2}
    \def\ycd{0.35}
    \node[coordinate label front, text=BrickRed, draw=BrickRed] at (-\xcd, -\ycd) {$\regionlabelsize +-+1$};
    \node[coordinate label front, text=BrickRed, draw=BrickRed] at ( \xcd, -\ycd) {$\regionlabelsize ++-1$};
    \node[coordinate label back , text=BrickRed, draw=BrickRed] at ( \xcd,  \ycd) {$\regionlabelsize -+-1$};
    \node[coordinate label back , text=BrickRed, draw=BrickRed] at (-\xcd,  \ycd) {$\regionlabelsize --+1$};

    \def\edgelabelsize{\scriptscriptstyle}

    \def\xcd{1.0}
    \def\ycd{-1.5}
    \node[coordinate label front, text=RoyalBlue, draw=RoyalBlue] at ( \xcd, -\ycd) {$\edgelabelsize ++0+$};
    \node[coordinate label back , text=RoyalBlue, draw=RoyalBlue] at (-\xcd,  \ycd) {$\edgelabelsize --0-$};
    \node[coordinate label front, text=RoyalBlue, draw=RoyalBlue] at (-\xcd, -\ycd) {$\edgelabelsize +0++$};
    \node[coordinate label back , text=RoyalBlue, draw=RoyalBlue] at ( \xcd,  \ycd) {$\edgelabelsize -0--$};

    \def\xcd{2.1}
    \def\ycd{2.0}
    \node[coordinate label front, text=RoyalBlue, draw=RoyalBlue] at (-\xcd,  \ycd) {$\edgelabelsize 0-++$};
    \node[coordinate label back , text=RoyalBlue, draw=RoyalBlue] at ( \xcd, -\ycd) {$\edgelabelsize 0+--$};

    \node[coordinate label front, text=RoyalBlue, draw=RoyalBlue] at ( \xcd,  \ycd) {$\edgelabelsize 0+-+$};
    \node[coordinate label back , text=RoyalBlue, draw=RoyalBlue] at (-\xcd, -\ycd) {$\edgelabelsize 0-+-$};

    \def\xcd{1.25}
    \def\ycd{2.25}
    \node[coordinate label front, text=RoyalBlue, draw=RoyalBlue] at (-\xcd, -\ycd) {$\edgelabelsize +0+-$};
    \node[coordinate label back , text=RoyalBlue, draw=RoyalBlue] at ( \xcd,  \ycd) {$\edgelabelsize -0-+$};

    \node[coordinate label front, text=RoyalBlue, draw=RoyalBlue] at ( \xcd, -\ycd) {$\edgelabelsize ++0-$};
    \node[coordinate label back , text=RoyalBlue, draw=RoyalBlue] at (-\xcd,  \ycd) {$\edgelabelsize --0+$};

    \end{tikzpicture}%
    }

    \caption{A regular CW decomposition of a sphere corresponding to an oriented interval greedoid
        on the interval greedoid defined in \eqref{eqn:feasible-sets-of-oig-example}. The cells are labelled by the
        covectors of the oriented interval greedoid. See
        Figure~\ref{fig:oig-sphere-poset} for its (opposite) face poset.}
    \label{fig:oig-sphere}
\end{figure}
The result is a regular CW decomposition of a sphere with $6$ two-cells, $12$ one-cells, and
$8$ zero-cells, whose cells correspond to a collection of covectors for an
oriented interval greedoid structure for $(E,\FeasibleSets)$
(see Figure~\ref{fig:oig-sphere-poset}).
\begin{figure}[htpb]
    \centering
    \resizebox{.8\linewidth}{!}{%
    \begin{tikzpicture}[>=latex,line join=bevel,]
      \node (node_26) at (267.0bp,157.0bp) [draw,text=black]{$0000$};

      \node (node_24) at (290.0bp,107.0bp) [draw,color=ForestGreen,fill=none]{$000+$};
      \node (node_25) at (244.0bp,107.0bp) [draw,color=ForestGreen,fill=none]{$000-$};
      \node (node_23) at (60.0bp,107.0bp) [draw,color=ForestGreen,fill=none]{$0-+0$};
      \node (node_20) at (474.0bp,107.0bp) [draw,color=ForestGreen,fill=none]{$0+-0$};
      \node (node_8) at (152.0bp,107.0bp) [draw,color=ForestGreen,fill=none]{$+0+0$};
      \node (node_4) at (336.0bp,107.0bp) [draw,color=ForestGreen,fill=none]{$++00$};
      \node (node_17) at (382.0bp,107.0bp) [draw,color=ForestGreen,fill=none]{$-0-0$};
      \node (node_14) at (198.0bp,107.0bp) [draw,color=ForestGreen,fill=none]{$--00$};

      \node (node_22) at (14.0bp,57.0bp) [draw,color=RoyalBlue,fill=none]{$0-+-$};
      \node (node_21) at (60.0bp,57.0bp) [draw,color=RoyalBlue,fill=none]{$0-++$};
      \node (node_7) at (106.0bp,57.0bp) [draw,color=RoyalBlue,fill=none]{$+0+-$};
      \node (node_6) at (152.0bp,57.0bp) [draw,color=RoyalBlue,fill=none]{$+0++$};
      \node (node_3) at (290.0bp,57.0bp) [draw,color=RoyalBlue,fill=none]{$++0-$};
      \node (node_2) at (336.0bp,57.0bp) [draw,color=RoyalBlue,fill=none]{$++0+$};
      \node (node_19) at (474.0bp,57.0bp) [draw,color=RoyalBlue,fill=none]{$0+--$};
      \node (node_18) at (520.0bp,57.0bp) [draw,color=RoyalBlue,fill=none]{$0+-+$};
      \node (node_16) at (382.0bp,57.0bp) [draw,color=RoyalBlue,fill=none]{$-0--$};
      \node (node_15) at (428.0bp,57.0bp) [draw,color=RoyalBlue,fill=none]{$-0-+$};
      \node (node_13) at (198.0bp,57.0bp) [draw,color=RoyalBlue,fill=none]{$--0-$};
      \node (node_12) at (244.0bp,57.0bp) [draw,color=RoyalBlue,fill=none]{$--0+$};

      \node (node_9) at (451.0bp,7.0bp) [draw,color=BrickRed,fill=none]{$-+-1$};
      \node (node_5) at (60.0bp,7.0bp) [draw,color=BrickRed,fill=none]{$+-+1$};
      \node (node_1) at (370.0bp,7.0bp) [draw,color=BrickRed,fill=none]{$++-1$};
      \node (node_0) at (195.0bp,7.0bp) [draw,color=BrickRed,fill=none]{$+++1$};
      \node (node_11) at (284.0bp,7.0bp) [draw,color=BrickRed,fill=none]{$---1$};
      \node (node_10) at (106.0bp,7.0bp) [draw,color=BrickRed,fill=none]{$--+1$};

      \draw [black,-] (node_19) ..controls (456.95bp,62.385bp) and (453.35bp,63.251bp)  .. (450.0bp,64.0bp) .. controls (372.1bp,81.388bp) and (349.47bp,81.543bp)  .. (node_25);
      \draw [black,-] (node_11) ..controls (322.1bp,20.699bp) and (374.21bp,38.07bp)  .. (node_15);
      \draw [black,-] (node_13) ..controls (198.0bp,69.947bp) and (198.0bp,80.897bp)  .. (node_14);
      \draw [black,-] (node_3) ..controls (302.07bp,70.591bp) and (314.09bp,83.133bp)  .. (node_4);
      \draw [black,-] (node_16) ..controls (382.0bp,69.947bp) and (382.0bp,80.897bp)  .. (node_17);
      \draw [black,-] (node_9) ..controls (445.17bp,20.162bp) and (439.68bp,31.632bp)  .. (node_15);
      \draw [black,-] (node_1) ..controls (361.23bp,20.377bp) and (352.73bp,32.377bp)  .. (node_2);
      \draw [black,-] (node_15) ..controls (390.7bp,70.972bp) and (342.34bp,87.794bp)  .. (node_24);
      \draw [black,-] (node_0) ..controls (232.57bp,20.789bp) and (283.03bp,37.968bp)  .. (node_2);
      \draw [black,-] (node_0) ..controls (170.32bp,21.313bp) and (143.56bp,35.744bp)  .. (node_7);
      \draw [black,-] (node_18) ..controls (502.95bp,62.385bp) and (499.35bp,63.251bp)  .. (496.0bp,64.0bp) .. controls (418.1bp,81.388bp) and (395.47bp,81.543bp)  .. (node_24);
      \draw [black,-] (node_23) ..controls (108.24bp,119.19bp) and (200.29bp,140.53bp)  .. (node_26);
      \draw [black,-] (node_19) ..controls (474.0bp,69.947bp) and (474.0bp,80.897bp)  .. (node_20);
      \draw [black,-] (node_5) ..controls (60.0bp,19.947bp) and (60.0bp,30.897bp)  .. (node_21);
      \draw [black,-] (node_9) ..controls (469.72bp,21.021bp) and (489.34bp,34.673bp)  .. (node_18);
      \draw [black,-] (node_17) ..controls (349.5bp,121.56bp) and (313.28bp,136.68bp)  .. (node_26);
      \draw [black,-] (node_10) ..controls (93.933bp,20.591bp) and (81.915bp,33.133bp)  .. (node_21);
      \draw [black,-] (node_0) ..controls (221.49bp,21.385bp) and (250.44bp,36.01bp)  .. (node_3);
      \draw [black,-] (node_15) ..controls (415.93bp,70.591bp) and (403.91bp,83.133bp)  .. (node_17);
      \draw [black,-] (node_25) ..controls (249.83bp,120.16bp) and (255.32bp,131.63bp)  .. (node_26);
      \draw [black,-] (node_24) ..controls (284.17bp,120.16bp) and (278.68bp,131.63bp)  .. (node_26);
      \draw [black,-] (node_22) ..controls (26.067bp,70.591bp) and (38.085bp,83.133bp)  .. (node_23);
      \draw [black,-] (node_20) ..controls (425.76bp,119.19bp) and (333.71bp,140.53bp)  .. (node_26);
      \draw [black,-] (node_3) ..controls (277.93bp,70.591bp) and (265.91bp,83.133bp)  .. (node_25);
      \draw [black,-] (node_1) ..controls (399.16bp,21.457bp) and (431.27bp,36.278bp)  .. (node_19);
      \draw [black,-] (node_11) ..controls (260.21bp,21.277bp) and (234.53bp,35.611bp)  .. (node_13);
      \draw [black,-] (node_21) ..controls (60.0bp,69.947bp) and (60.0bp,80.897bp)  .. (node_23);
      \draw [black,-] (node_2) ..controls (323.93bp,70.591bp) and (311.91bp,83.133bp)  .. (node_24);
      \draw [black,-] (node_5) ..controls (47.933bp,20.591bp) and (35.915bp,33.133bp)  .. (node_22);
      \draw [black,-] (node_18) ..controls (507.93bp,70.591bp) and (495.91bp,83.133bp)  .. (node_20);
      \draw [black,-] (node_5) ..controls (85.585bp,21.349bp) and (113.43bp,35.877bp)  .. (node_6);
      \draw [black,-] (node_6) ..controls (152.0bp,69.947bp) and (152.0bp,80.897bp)  .. (node_8);
      \draw [black,-] (node_4) ..controls (317.28bp,121.02bp) and (297.66bp,134.67bp)  .. (node_26);
      \draw [black,-] (node_10) ..controls (131.59bp,21.349bp) and (159.43bp,35.877bp)  .. (node_13);
      \draw [black,-] (node_10) ..controls (143.3bp,20.972bp) and (191.66bp,37.794bp)  .. (node_12);
      \draw [black,-] (node_5) ..controls (72.067bp,20.591bp) and (84.085bp,33.133bp)  .. (node_7);
      \draw [black,-] (node_1) ..controls (409.1bp,20.513bp) and (464.54bp,38.254bp)  .. (node_18);
      \draw [black,-] (node_13) ..controls (210.07bp,70.591bp) and (222.09bp,83.133bp)  .. (node_25);
      \draw [black,-] (node_7) ..controls (143.3bp,70.972bp) and (191.66bp,87.794bp)  .. (node_25);
      \draw [black,-] (node_12) ..controls (256.07bp,70.591bp) and (268.09bp,83.133bp)  .. (node_24);
      \draw [black,-] (node_11) ..controls (273.57bp,20.52bp) and (263.27bp,32.88bp)  .. (node_12);
      \draw [black,-] (node_7) ..controls (118.07bp,70.591bp) and (130.09bp,83.133bp)  .. (node_8);
      \draw [black,-] (node_6) ..controls (189.3bp,70.972bp) and (237.66bp,87.794bp)  .. (node_24);
      \draw [black,-] (node_12) ..controls (231.93bp,70.591bp) and (219.91bp,83.133bp)  .. (node_14);
      \draw [black,-] (node_11) ..controls (311.4bp,21.421bp) and (341.46bp,36.144bp)  .. (node_16);
      \draw [black,-] (node_16) ..controls (344.7bp,70.972bp) and (296.34bp,87.794bp)  .. (node_25);
      \draw [black,-] (node_14) ..controls (216.72bp,121.02bp) and (236.34bp,134.67bp)  .. (node_26);
      \draw [black,-] (node_9) ..controls (432.28bp,21.021bp) and (412.66bp,34.673bp)  .. (node_16);
      \draw [black,-] (node_9) ..controls (456.83bp,20.162bp) and (462.32bp,31.632bp)  .. (node_19);
      \draw [black,-] (node_10) ..controls (80.415bp,21.349bp) and (52.569bp,35.877bp)  .. (node_22);
      \draw [black,-] (node_1) ..controls (348.06bp,21.164bp) and (324.68bp,35.195bp)  .. (node_3);
      \draw [black,-] (node_22) ..controls (30.673bp,62.488bp) and (33.943bp,63.301bp)  .. (37.0bp,64.0bp) .. controls (117.81bp,82.47bp) and (139.1bp,81.942bp)  .. (220.0bp,100.0bp) .. controls (220.1bp,100.02bp) and (220.21bp,100.05bp)  .. (node_25);
      \draw [black,-] (node_0) ..controls (183.72bp,20.591bp) and (172.49bp,33.133bp)  .. (node_6);
      \draw [black,-] (node_8) ..controls (184.5bp,121.56bp) and (220.72bp,136.68bp)  .. (node_26);
      \draw [black,-] (node_2) ..controls (336.0bp,69.947bp) and (336.0bp,80.897bp)  .. (node_4);
      \draw [black,-] (node_21) ..controls (76.673bp,62.488bp) and (79.943bp,63.301bp)  .. (83.0bp,64.0bp) .. controls (163.81bp,82.47bp) and (185.1bp,81.942bp)  .. (266.0bp,100.0bp) .. controls (266.1bp,100.02bp) and (266.21bp,100.05bp)  .. (node_24);
    \end{tikzpicture}%
    }
    \caption{The underlying poset of the left regular band of covectors from
        Figure~\ref{fig:oig-sphere}. It is isomorphic to the opposite poset of
        the face poset (including an empty face) of the regular CW decomposition of the sphere in
        Figure~\ref{fig:oig-sphere}.}
    \label{fig:oig-sphere-poset}
\end{figure}
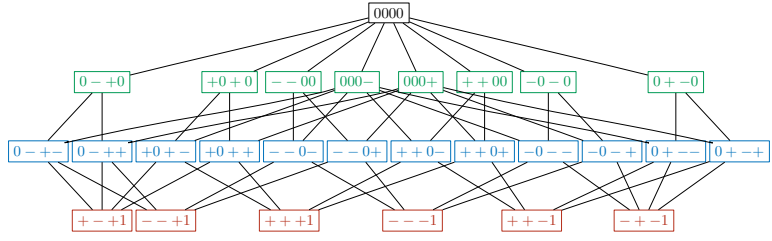

\subsubsection{Contraction and deletion of oriented interval greedoids}

Let $(E, \FeasibleSets)$ be an interval greedoid with lattice of greedoid flats $\Phi$.
The \emph{contraction}\index{contraction} of $(E, \FeasibleSets)$ along $X \in \FeasibleSets$ is
the interval greedoid $(\bigcup_{Y \in \FeasibleSets/X} Y , \FeasibleSets/X)$,
where
\begin{gather*}
    \FeasibleSets/X = \left\{ Y \subseteq E \setminus X \mid Y \cup X \in \FeasibleSets \right\}.
\end{gather*}
Let $\Phi/X$, $\Gamma/X$ and $\xi/X$ denote the corresponding notions in the
contraction.

Fix $X \in \FeasibleSets$. Let $y$ be a covector of $(E, \FeasibleSets)$ such
that $\Flat(y) \geq [X]$ in $\Phi$. Then there exists $Y \in \FeasibleSets/X$
such that $\Flat(y) = [X \cup Y]$. Define the \emph{contraction}\index{contraction} of $y$ along $X$
as
\begin{align}
    \label{e:CovectorOfContraction}
    \con_X(y)_e =
    \begin{cases}
        y_e, & \text{if } e \in (\Gamma/X)(Y), \\
        0, & \text{if } e \in (\xi/X)(Y), \\
        1, & \text{otherwise.}
    \end{cases}
\end{align}

If $\OIG$ is the set of covectors of an oriented interval greedoid on $(E,
\FeasibleSets)$, then
\begin{gather*}
    \OIG/X = \left\{ \con_X(y) \mid y \in \OIG \text{~and~} \Flat(y) \geq [X] \right\}
\end{gather*}
defines an oriented interval greedoid over the contraction of $(E,
\FeasibleSets)$ by $X$~\cite[Proposition~4.3]{SaliolaThomas}.
By~\cite[Proposition~4.4]{SaliolaThomas},
$\OIG/X$ is isomorphic to the contraction of $\OIG$ to $[X]$ (in the
sense of left regular bands) via the monoid homomorphism
\begin{equation*}
    \con_X\colon \OIG_{\geq [X]}\to \OIG/X.
\end{equation*}

We now describe the deletion operation on oriented interval greedoids.
Let $W$ be an arbitrary subset of $E$.
The \emph{deletion}\index{deletion} of $(E, \FeasibleSets)$ to $W$
is the interval greedoid $(W, \FeasibleSets|_W)$
defined by
\begin{gather*}
    \FeasibleSets|_W = \left\{ X \in \FeasibleSets \mid X \subseteq W \right\}.
\end{gather*}
Let $\Phi|_W$, $\Gamma|_W$ and $\xi|_W$ denote the corresponding notions for
the deletions.

Fix $W \subseteq E$.
Let $y$ be a covector of $(E, \FeasibleSets)$ with $\Flat(y) = [Y]$.
Define the \emph{deletion}\index{deletion} of $y$ to $W$ as
\begin{gather*}
    \res_W(y)_w =
    \begin{cases}
        y_w, & \text{if } w \in \Gamma|_W(Y), \\
        0, & \text{if } w \in \xi|_W(Y), \\
        1, & \text{otherwise}.
    \end{cases}
\end{gather*}
If $\OIG$ is the set of covectors of an oriented interval greedoid on $(E,
\FeasibleSets)$ and $W \subseteq E$, then define the \emph{deletion}\index{deletion} of
$\OIG$ to $W$ to be
\begin{gather*}
    \OIG|_W = \left\{ \res_W(y) \mid y \in \OIG \right\}.
\end{gather*}
It is not true in general that deletions of oriented interval
greedoids to arbitrary subsets $W \subseteq E$ are again
oriented interval greedoids.
However, it is true for deletion to $\Gamma(\varnothing)$,
in which case the deletion is an oriented matroid~\cite[Proposition~4.12]{SaliolaThomas};
and for deletion to $\xi(X)$ for any $X \in \FeasibleSets$~\cite[Theorem~4.15]{SaliolaThomas},
in which case
\begin{equation*}
    \res_{\xi(X)}\colon x \OIG\to \OIG|_{\xi(X)}
\end{equation*}
is a monoid isomorphism
for any covector $x$ with $\Flat(x) = [X]$,
with inverse $\res_{\xi(X)}(y) \mapsto x \circ y$
for all $y \in \OIG$~\cite[Proposition~4.16]{SaliolaThomas}.

\subsubsection{Oriented interval greedoids from convex geometries}

Convex geometries give rise to a special class of interval greedoids called
antimatroids. In this setting the covectors and the composition of covectors
admit a very nice geometric description. We begin by defining convex
geometries.

Just as matroids can be viewed as an abstraction of linear independence of
vectors in $\mathbb R^d$, convex geometries can be viewed as an abstraction of
convexity of vectors in $\mathbb R^d$.
In the following, think of $E$ as a finite subset of $\mathbb R^d$;
of $\tau$ as the convex hull operator $\tau(A) = \conv(A) \cap E$ for $A\subseteq E$;
and of $\ext(A)$ as the set of extreme points of the convex hull of $A$.

A \emph{convex geometry}\index{convex geometry} is a pair $(E, \tau)$ with $E$ a finite
set and $\tau\colon 2^E \to 2^E$ a non-decreasing, order preserving and
idempotent function, satisfying the following \emph{anti-exchange}\index{anti-exchange}
property:
\begin{itemize}
    \item[(AE)]
        If $x,y \not\in \tau(X), x \neq y$, and $y \in \tau(X \cup x)$, then $x
        \not\in \tau(X \cup y)$.
\end{itemize}
The subsets $A \subseteq E$ satisfying $\tau(A) = A$ are called
\emph{closed sets}\index{closed set}.
The \emph{extreme points}\index{extreme points} $\ext(A)$ of $A \subseteq E$ are the points $x \in A$
satisfying $x \not\in \tau(A \setminus x)$.

If $(E,\tau)$ is a convex geometry, then $(E, \FeasibleSets)$ is an interval
greedoid, where $\FeasibleSets$ consists of the complements of the closed sets.
These interval greedoids are called \emph{antimatroids}\index{antimatroids} or \emph{upper interval
greedoids}.
If $X \in \FeasibleSets$, then $E \setminus X$ is a closed set of the convex
geometry and we define $\Gamma(X) = \ext(E \setminus X)$.

A covector for an antimatroid admits a nice geometric description in terms of
the corresponding convex geometry.
Let $X \in \FeasibleSets$. Then $X$ corresponds to the closed
set $C = E \setminus X$ in the convex geometry.
A covector with underlying greedoid flat $[X]$
assigns $+$ or $-$ to the extreme points of $C$,
assigns $0$ to the points of the exterior of $C$,
and assigns $1$ to the non-extreme points contained in $C$.
See Figure~\ref{f:TwoCovectorsInAConvexGeometry}.

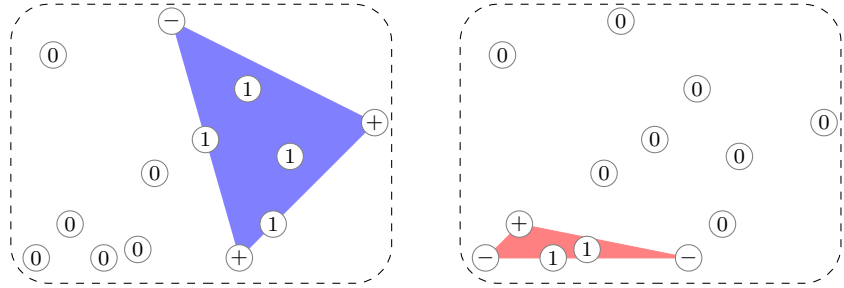
\begin{figure}[!ht]
\centering
\begin{tikzpicture}[scale=0.45]
    \tikzset{%
      >=latex,%
      line join=bevel,%
      inner sep=1pt,%
      outer sep=3pt,%
      mark point/.style={inner sep=0pt, outer sep=0pt, minimum size=10pt, fill=white, circle, draw=gray, font=\scriptsize},%
    }
    \draw[color=black, fill=none, rounded corners=5mm, dashed] (5,-0.75) -- (10.5,-0.75) -- (10.5,7.5) -- (-0.75,7.5) -- (-0.75,-0.75) -- (5,-0.75);
    \filldraw[color=blue!50] (6,0) -- (10,4) -- (4,7);
    \draw (6,0) node[mark point] {$+$};
    \draw (7,1) node[mark point] {$1$};
    \draw (10,4) node[mark point] {$+$};
    \draw (4,7) node[mark point] {$-$};
    \draw (5,3.5) node[mark point] {$1$};
    \draw (6.25,5) node[mark point] {$1$};
    \draw (7.5,3) node[mark point] {$1$};
    \draw (0,0) node[mark point] {$0$};
    \draw (2,0) node[mark point] {$0$};
    \draw (3,0.25) node[mark point] {$0$};
    \draw (1,1) node[mark point] {$0$};
    \draw (3.5,2.5) node[mark point] {$0$};
    \draw (0.5,6) node[mark point] {$0$};
\end{tikzpicture}
\qquad
\begin{tikzpicture}[scale=0.45]
    \tikzset{%
      >=latex,%
      line join=bevel,%
      inner sep=1pt,%
      outer sep=3pt,%
      mark point/.style={inner sep=0pt, outer sep=0pt, minimum size=10pt, fill=white, circle, draw=gray, font=\scriptsize},%
    }
    \draw[color=black, fill=none, rounded corners=5mm, dashed] (5,-0.75) -- (10.5,-0.75) -- (10.5,7.5) -- (-0.75,7.5) -- (-0.75,-0.75) -- (5,-0.75);
    \filldraw[color=red!50] (0,0) -- (6,0) -- (1,1);
    \draw (0,0) node[mark point]     { $-$};
    \draw (0.5,6) node[mark point]   { $0$};
    \draw (1,1) node[mark point]     { $+$};
    \draw (10,4) node[mark point]    { $0$};
    \draw (2,0) node[mark point]     { $1$};
    \draw (3,0.25) node[mark point]  { $1$};
    \draw (3.5,2.5) node[mark point] { $0$};
    \draw (4,7) node[mark point]     { $0$};
    \draw (5,3.5) node[mark point]   { $0$};
    \draw (6,0) node[mark point]     { $-$};
    \draw (6.25,5) node[mark point]  { $0$};
    \draw (7,1) node[mark point]     { $0$};
    \draw (7.5,3) node[mark point]   { $0$};
\end{tikzpicture}
\caption{Covectors of an antimatroid: extreme points of a closed set are
    labelled by $+$ or $-$; non-extreme interior points are labelled $1$;
    and exterior points are labelled $0$.}
\label{f:TwoCovectorsInAConvexGeometry}
\end{figure}

The composition $x \circ y$ of two covectors $x$ and $y$ is computed as
follows. (Figure~\ref{f:ProductOfCovectorsInAConvexGeometry} depicts the
product of the two covectors in Figure~\ref{f:TwoCovectorsInAConvexGeometry}.)
Consider the convex hull of the union of the corresponding closed sets.
The exterior points are labelled $0$.
The non-extreme interior points are labelled $1$.
If $e$ is an extreme point, then it is an extreme point for $x$ or for $y$
(possibly both). If it is an extreme point for $x$, then $e$ takes the sign
$x_e$; otherwise, it takes the sign $y_e$.

This example exemplifies that the product on covectors is not always computed
componentwise. It is possible that $x_e = y_e = 0$ yet $(x \circ y)_e \neq 0$.
This happens if $e$ is exterior to both closed sets, yet it is interior to the
convex hull of the union of the closed sets.
Similarly, extreme points of the closed sets can become interior to the convex
hull of their union.

\begin{figure}[!ht]
\centering
\begin{tikzpicture}[scale=0.45]
    \tikzset{%
      >=latex,%
      line join=bevel,%
      inner sep=1pt,%
      outer sep=3pt,%
      mark point/.style={inner sep=0pt, outer sep=0pt, minimum size=10pt, fill=white, circle, draw=gray, font=\scriptsize},%
    }
    \draw[color=black, fill=none, rounded corners=5mm, dashed] (5,-0.75) -- (10.5,-0.75) -- (10.5,7.5) -- (-0.75,7.5) -- (-0.75,-0.75) -- (5,-0.75);
    \filldraw[color=blue!50!red!50] (6,0) -- (10,4) -- (4,7) -- (0,0);
    \draw (0,0) node[mark point]     { $-$};
    \draw (0.5,6) node[mark point]   { $0$};
    \draw (1,1) node[mark point]     { $1$};
    \draw (10,4) node[mark point]    { $+$};
    \draw (2,0) node[mark point]     { $1$};
    \draw (3,0.25) node[mark point]  { $1$};
    \draw (3.5,2.5) node[mark point] { $1$};
    \draw (4,7) node[mark point]     { $-$};
    \draw (5,3.5) node[mark point]   { $1$};
    \draw (6,0) node[mark point]     { $+$};
    \draw (6.25,5) node[mark point]  { $1$};
    \draw (7,1) node[mark point]     { $1$};
    \draw (7.5,3) node[mark point]   { $1$};
\end{tikzpicture}
\qquad
\begin{tikzpicture}[scale=0.45]
    \tikzset{%
      >=latex,%
      line join=bevel,%
      inner sep=1pt,%
      outer sep=3pt,%
      mark point/.style={inner sep=0pt, outer sep=0pt, minimum size=10pt, fill=white, circle, draw=gray, font=\scriptsize},%
    }
    \draw[color=black, fill=none, rounded corners=5mm, dashed] (5,-0.75) -- (10.5,-0.75) -- (10.5,7.5) -- (-0.75,7.5) -- (-0.75,-0.75) -- (5,-0.75);
    \filldraw[color=red!50!blue!50] (6,0) -- (10,4) -- (4,7) -- (0,0);
    \draw (0,0) node[mark point]     { $-$};
    \draw (0.5,6) node[mark point]   { $0$};
    \draw (1,1) node[mark point]     { $1$};
    \draw (10,4) node[mark point]    { $+$};
    \draw (2,0) node[mark point]     { $1$};
    \draw (3,0.25) node[mark point]  { $1$};
    \draw (3.5,2.5) node[mark point] { $1$};
    \draw (4,7) node[mark point]     { $-$};
    \draw (5,3.5) node[mark point]   { $1$};
    \draw (6,0) node[mark point]     { $-$};
    \draw (6.25,5) node[mark point]  { $1$};
    \draw (7,1) node[mark point]     { $1$};
    \draw (7.5,3) node[mark point]   { $1$};
\end{tikzpicture}
\caption{The products $x \circ y$ (left) and $y \circ x$ (right) of the
    covectors $x$ and $y$ depicted in Figure~\ref{f:TwoCovectorsInAConvexGeometry}.}
\label{f:ProductOfCovectorsInAConvexGeometry}
\end{figure}
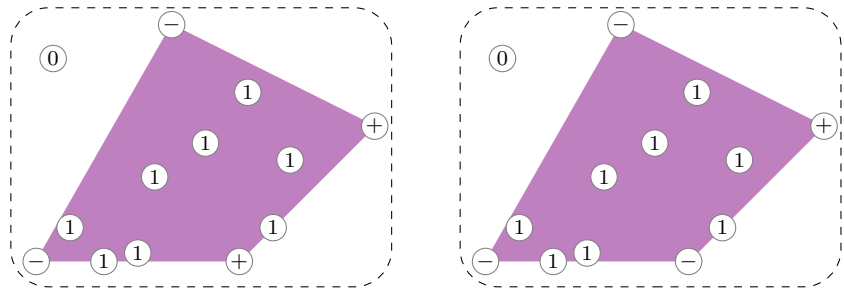

The set of all covectors of an antimatroid forms an oriented interval greedoid~\cite[Theorem~3.14]{SaliolaThomas}.
In fact, this is the only oriented interval greedoid structure on an
antimatroid~\cite[Proposition~3.15]{SaliolaThomas}.
This collection of covectors is the central object of study in the work of
Billera, Hsiao and Provan~\cite{BilleraHsiaoProvan2008}.
To state their results, we recall some definitions from~\cite{BilleraHsiaoProvan2008}.

Define a poset $Q_L$ as follows. Let $L$ denote the lattice of closed
sets of the convex geometry ordered by inclusion. The elements of $Q_L$ are
equivalence classes of pairs $(A, \varepsilon)$ with $A \in L$ and
$\varepsilon\colon E \to \{+, -\}$, where $(A, \alpha) \sim (B, \beta)$ if $A = B$
and $\alpha|_{\ext(A)} = \beta|_{\ext(B)}$.
The partial order on $Q_L$ is defined as $(A, \alpha) \leq (B, \beta)$
if and only if $A \subseteq B$ and $\alpha$ and $\beta$ agree on $\ext(A) \cap \ext(B)$.
Note that $Q_L$ admits a minimal element corresponding to the pair
$(\varnothing, \varepsilon)$, where $\varepsilon$ is the empty map.
$Q_L$ is also has a maximal element $\wh 1$ adjoined.

Elements of $Q_L$ are in bijection with covectors of the oriented interval
greedoid on the antimatroid. To see this, note that elements of $Q_L$ determine
covectors, and vice-versa.
Indeed, the equivalence class of $(A, \alpha)$ is determined by $A$ and
$\alpha|_{\ext(A)}$, so one obtains a covector $\wh\alpha$ by declaring
$\wh\alpha_e = \alpha(e)$ if $e \in \ext(A)$, $\wh\alpha_e = 1$ if $e \in
A \setminus \ext(A)$ and $\wh\alpha_e = 0$ if $e \notin A$.
Conversely, starting with a covector $x$, consider the closed set
$A = \{e \in E\mid x_e \neq 0\}$ and any map $\alpha\colon E \to \{+,-\}$
satisfying $\alpha(e) = x_e$ on $\ext(A)$.

The partial order on $Q_L$ is opposite to the partial order on $\OIG$.
Indeed, if $A \subseteq B$ and $\alpha$ and $\beta$ agree on $\ext(A) \cap
\ext(B)$, then the corresponding covectors $\wh\alpha$ and $\wh\beta$
satisfy $\wh\alpha_e \geq \wh\beta_e$ for all $e \in E$:
\begin{itemize}
    \item
        if $e \notin A$, then $\wh\alpha_e = 0 \geq \wh\beta_e$;
    \item
        if $e \in \ext(A) \cap \ext(B)$, then
        $\wh\alpha_e = \wh\beta_e$ by assumption;
    \item
        if $e \in \ext(A) \setminus \ext(B)$, then
        $\wh\alpha_e \in \{+, -\}$ and $\wh\beta_e = 1$,
        so $\wh\alpha_e \geq \wh\beta_e$;
    \item
        if $e \in A \setminus \ext(A)$, then $e \notin B \setminus \ext(B)$
        and $\wh\alpha_e = 1 = \wh\beta_e$.
\end{itemize}

Billera, Hsiao and Provan prove that $Q_L \setminus \{\wh 0, \wh 1\}$ is the
face poset of a regular CW decomposition of a sphere. It follows that $\OIG \setminus \{0\}$
is the opposite poset of a face poset of a regular CW decomposition of a sphere.

\subsection{The topology of left regular bands}
In this subsection, we describe the topology of the various families of left regular bands that we have been considering (where by the topology of a left regular band, we mean its topology as a poset).  Note that if $B$ is a left regular band monoid and also a CW poset, then it is, in fact, the face poset of a regular CW decomposition of a ball by Proposition~\ref{p:maximumCWposet} (since the identity is a maximum).  This need not be the case for left regular bands without an identity element.

\begin{Prop}\label{p:topology}
The following hold.
\begin{enumerate}
\item If $\mathcal L$ is an oriented matroid, then $\mathcal L_{\geq X}$ is the face poset of a regular cell decomposition of a ball for each $X\in \Lambda(\mathcal L)$.
\item If $(E,\mathcal L,g)$ is an affine oriented matroid and $X\in\Lambda(\mathcal L^+(g))$, then $\mathcal L^+(g)_{\geq X}$ is the face poset of a regular CW decomposition of a ball.
\item Let $(E,\mathcal L)$ be a COM and $X\in \Lambda(\mathcal L)$.  Then $\mathcal L_{\geq X}$ is the face poset of a contractible regular CW complex.
\item If $\mathcal A$ is a complex hyperplane arrangement, then $\FFF(\mathcal A)_{\geq X}$ is the face poset of a regular CW decomposition of a ball for each $X\in \Lambda(\FFF(\mathcal A))$.
\item If $(E,\FeasibleSets,\OIG)$ is an oriented interval greedoid and $X\in \Lambda(\OIG)$, then $\OIG_{\geq X}$ is the face poset of a regular CW decomposition of a ball.
\item If $\Gamma=(V,E)$ is a graph, then $\Lambda(B(\Gamma))\cong (P(V),\cup)$.  If $Y\subseteq X\subseteq V$ and $e_Y\in B(\Gamma)$ has support $Y$, then $\|\Delta(\bd e_YB(\Gamma)_{\geq X})\|$ is homotopy equivalent to $\|\Cliq(\Gamma[X\setminus Y])\|$.
\end{enumerate}
\end{Prop}
\begin{proof}
For (1), we observe that $\mathcal L_{\geq X}$ is the set of covectors of a contraction of $\mathcal L$ and so it suffices to prove that $\mathcal L$ is a CW poset for any oriented matroid.  From~\cite[Corollary~4.3.4]{OrientedMatroids1999}, we have that $\mathcal L\setminus \{0\}$ is the face poset of a regular CW decomposition of a sphere.  Thus $\mathcal L$ is the face poset of a regular CW decomposition of a ball by Proposition~\ref{p:maximumCWposet}.

Similarly to (1), each contraction of an affine oriented matroid is again an affine oriented matroid and so to prove (2), we just need that $\mathcal L^+(g)$ is the face poset of a regular CW decomposition of a ball.  But this is the content of~\cite[Corollary~4.5.8]{OrientedMatroids1999}.

To prove (3), again we can use contractions to reduce to proving that $\mathcal L$ is a contractible CW poset by Proposition~\ref{p:Com.minor}.  But this is proved in~\cite[Section~11]{COMS}.  To see that $\mathcal L$ is the face poset of a regular cell complex, one just uses that each $x\mathcal L=\mathcal L_{\leq x}$ with $x\in \mathcal L$ is isomorphic to the monoid of covectors of an oriented matroid by Proposition~\ref{p:Com.minor} and hence $\mathcal L_{<x}$ is the face poset of a regular CW decomposition of a sphere by~\cite[Corollary~4.3.4]{OrientedMatroids1999}.

Item (4) is the most complicated case because contractions of complex hyperplane face monoids need not again be complex hyperplane face monoids.  It is shown in the proof of~\cite[Proposition~4.21]{oldpaper} that if $X\in \Lambda(\FFF(\mathcal A))$, then $\FFF(\mathcal A)_{\geq X}^{op}\setminus \{0\}$ is the face poset of a PL (piecewise linear) regular CW decomposition of a sphere (note that in~\cite{oldpaper} we followed the convention that regular CW complexes have empty faces).  Thus $\FFF(\mathcal A)_{\geq X}\setminus \{0\}$ is the face poset of a regular CW decomposition of a sphere by~\cite[Proposition~4.7.26]{OrientedMatroids1999} and so $\FFF(\mathcal A)_{\geq X}$ is the face poset of a regular CW decomposition of a ball by Proposition~\ref{p:maximumCWposet}.

To prove (5), we use closure of oriented interval greedoids under contraction to reduce to proving that $\OIG$ is the face poset of a regular CW decomposition of a ball.  By Theorem~\ref{t:oig.sphere}, $\OIG^{op}\setminus \{0\}$ is the face poset of a PL regular CW decomposition of a sphere and hence $\OIG\setminus \{0\}$ is the face poset of a regular CW decomposition of a sphere by~\cite[Proposition~4.7.26]{OrientedMatroids1999}.  Therefore, $\OIG$ is  the face poset of a regular CW decomposition of a ball by Proposition~\ref{p:maximumCWposet}.

\nomenclature[L, 09]{$\Delta(B)$}{order complex of the poset structure on $B$}%
The final item is proved in the course of the proof of~\cite[Theorem~4.16]{oldpaper}.  As $e_YB(\Gamma_{\geq X})\cong B(\Gamma[X\setminus Y])$, the key point is to show that $\Delta(\bd B(\Gamma))$ is homotopy equivalent to $\Cliq(\Gamma)$.  We sketch the argument from~\cite[Theorem~4.16]{oldpaper}, where the reader can find the details.  We first observe that any set of elements of $B(\Gamma)$ with a common lower bound has a meet. This is because the set of upper bounds of an element $b\in B(\Gamma)$, which is the left stabilizer of $b$, is a commutative submonoid and hence a meet semilattice.  Indeed, if $b$ corresponds to an acyclically oriented subgraph $\Phi$ of the complement of $\Gamma$, then the left stabilizer of $b$ is generated by all vertices that can appear first in a topological sorting of the partial order induced by $\Phi$.  These vertices must then form a clique in $\Gamma$ and hence commute.   The maximal elements of $\bd B(\Gamma)$ are the elements of $V$.  A collection of elements of $V$ have a common lower bound if and only if they mutually commute, which occurs if and only if they form a clique in $\Gamma$.  Rota's cross-cut theorem (Theorem~\ref{t:crosscut}) then provides the desired homotopy equivalence.
\end{proof}

A central notion in this paper is that of a CW left regular band.  We say that a left regular band $B$ is a
\emph{CW left regular band}\index{CW left regular band} if $B_{\geq X}$ is a CW poset for each $X\in \Lambda(B)$.  We say that $B$ is a connected CW left regular band if each $B_{\geq X}$ is the face poset of a connected regular CW complex.  Important examples include face semigroups of hyperplane arrangements and the set of covectors of an (affine) oriented matroid, a COM or an oriented interval greedoid.  Face monoids of complex hyperplane arrangements form another example.  On the other hand, free left regular band monoids, and more generally, free partially commutative left regular bands are not CW left regular bands.

\nomenclature[L, 19]{$\mathfrak P$}{a class of posets closed under isomorphism; used in defining a $\mathfrak P$-left regular band}%
More generally, let $\mathfrak P$ be a class of posets closed under isomorphism, e.g., CW posets, connected posets or contractible posets.  We say that a left regular band $B$ is a \emph{$\mathfrak P$-left regular band}\index{$\mathfrak P$-left regular band} if $B_{\geq X}\in \mathfrak P$ for all $X\in \Lambda(B)$. For example, the set of covectors of an (affine) oriented matroid is a contractible CW left regular band. A \emph{spherical left regular band}\index{spherical left regular band} will then mean a left regular band  $B$ such that $B_{\geq X}$ is the face poset of a regular CW decomposition of a $\dim \Delta(B_{\geq X})$-sphere for all $X\in \Lambda(B)$.  For example, if $\mathcal A$ is a real or complex hyperplane arrangement, then $\FFF(\mathcal A)\setminus \{0\}$ is a spherical left regular band, as is $\mathcal L\setminus \{0\}$ for any oriented matroid $(E,\mathcal L)$.

The next two propositions give us easy ways to create new CW left regular bands from old ones.
\begin{Prop}
If $B,B'$ are (connected/contractible) CW left regular bands, then so is $B\times B'$.
\end{Prop}
\begin{proof}
First note that $\Lambda(B\times B')\cong \Lambda(B)\times \Lambda(B')$ and that if $(X,Y)\in \Lambda(B\times B')$, then $(B\times B')_{\geq (X,Y)}=B_{\geq X}\times B_{\geq Y}$.  Thus it suffices to show that the poset $B\times B'$ is a (connected/contractible) CW poset.  It is a CW poset by Proposition~\ref{p:productofcw}.  As products of connected/contractible spaces are connected/contractible and $\|\Delta(B\times B')\|\cong \|\Delta(B)\|\times \|\Delta(B')\|$ (cf.~\cite[Equation~(9.6)]{bjornersurvey}), the result follows.
\end{proof}

\begin{Prop}\label{p:joinsuspendCWLRB}
If $B'$ is a spherical left regular band and $B$ is a connected CW left regular band (monoid), then $B'\ast B$ is a connected CW left regular band (monoid) with support lattice $\Lambda(B')\ast \Lambda(B)$.  In particular, $\mathsf S(B)$ is a connected CW left regular band (monoid) with support lattice $\Lambda(B)\cup \{-\infty\}$ where $-\infty$ is an external minimum.
\end{Prop}
\begin{proof}
It is clear that $\Lambda(B'\ast B)\cong \Lambda(B')\ast \Lambda(B)$ and we identify these two semilattices.  If $X\in \Lambda(B)$, then $(B'\ast B)_{\geq X}=B_{\geq X}$ and hence is a connected CW poset.  If $X\in \Lambda(B')$, then $(B'\ast B)_{\geq X} = B'_{\geq X}\ast B$ and hence is a CW poset by Proposition~\ref{p:joinsofcw}.  Moreover, the join of a sphere with a connected space is always connected.  The monoid statements are clear.
\end{proof}

In particular, if $\mathcal A,\mathcal A'$ are real or complex hyperplane arrangements, then $(\FFF(\mathcal A)\setminus \{0\})\ast \FFF(\mathcal A')$ is a connected CW left regular band.
We now give another example of a family of connected CW left regular band monoids, which is essentially from~\cite{complexstrat}.

\begin{Example}[Ladders]\label{ladders}
\nomenclature[L, 14]{$L_n$}{ladders; the $n$-fold suspension of the left regular band $L$}%
Let $L_n=\mathsf S^{n}(\{0\})$ be the $n$-fold suspension of the trivial left regular band.  So $L_1=L$ and $L_2=\til L$.  It is easy to see that
$L_n\cong\{0,\pm 1,\ldots, \pm n\}$ with the product given by
\[xy = \begin{cases} y, &\text{if}\ |x|<|y|\\ x, & \text{if}\ |x|\geq |y|.\end{cases}\]   Observe that $\Lambda(L_n)\cong \{0,\ldots, n\}$ ordered by $\geq$ and the support map is given by $\sigma(x)=|x|$.  We call the $L_n$ \emph{ladders}\index{ladders} because the Hasse diagram of $L_n\setminus \{0\}$ is a ladder with $n$ rungs.  The partial order on $L_n$ is given by $x<y$ if and only if $|x|>|y|$. It follows from Proposition~\ref{p:joinsuspendCWLRB} that $L_n$ is a CW poset of dimension $n$ and hence $\|\Delta(L_n)\|$ is homeomorphic to a closed $n$-ball.  See Figure~\ref{f:ladder} for the Hasse diagram and the corresponding cell decomposition of the ball for $n=3$.
\begin{figure}[tbhp]
\begin{center}
\begin{tikzpicture}
\begin{scope}[xshift=-2cm,shorten >=1pt,%
auto,semithick,
inner sep=1pt,bend angle=45].
\tikzstyle{every node}=[font=\footnotesize]
\node (G) at (0,0) {$0$};
\node (E) [below left of=G] {$1$};
\node (F) [below right of=G] {$-1$};
\node (C) [below of=E] {$2$};
\node (D) [below of=F] {$-2$};
\node (A) [below of=C] {$3$};
\node (B) [below of=D] {$-3$};
\draw [red,  thick] (A) -- (C);
\draw [blue, thick] (A) -- (D);
\draw [red,  thick] (B) -- (C);
\draw [blue, thick] (B) -- (D);
\draw [red,  thick] (C) -- (E);
\draw [blue, thick] (C) -- (F);
\draw [red,  thick] (D) -- (E);
\draw [blue, thick] (D) -- (F);
\draw [red,  thick] (E) -- (G);
\draw [blue, thick] (F) -- (G);
\end{scope}
\begin{scope}[xshift=2cm,yshift=-1.5cm,inner sep=1pt,vertices/.style={draw, fill=black, circle, inner sep=1pt}]
\node[vertices] (A) at (-1,0) {};
\node[vertices] (B) at (1,0) {};
    \draw (-1,0) arc (180:360:1cm and 0.5cm);
    \draw[dashed] (-1,0) arc (180:0:1cm and 0.5cm);
    \draw (0,0) circle (1cm);
    \shade[ball color=blue!10!white,opacity=0.20] (0,0) circle (1cm);
\end{scope}
\end{tikzpicture}
\end{center}
\caption{The Hasse diagram of $L_3$ and corresponding cell decomposition of the $3$-ball\label{f:ladder}}
\end{figure}
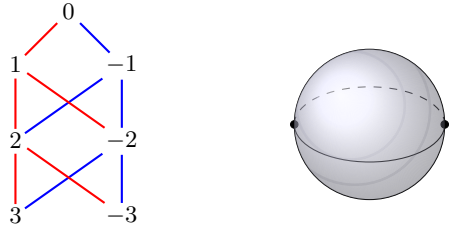
\end{Example}

The following lemma provides an alternative characterization of connectedness for left regular bands.

\begin{Lemma}\label{l:Franco.graph}
Let $B$ be a left regular band and $X\in \Lambda(B)$.  Define a graph $\Gamma(X)$ as follows.  The vertex set of $\Gamma(X)$ is  $L_X=\{b\in B\mid Bb=X\}$.  Two vertices $x,y\in L_X$ are adjacent if they have a common upper bound, that is, $bx=x$ and $by=y$ for some $b\in B$ (necessarily belonging to $B_{\geq X}$).  Then $\Delta(B_{\geq X})$ is connected if and only if $\Gamma(X)$ is connected.
\end{Lemma}
\begin{proof}
Suppose first that $\Gamma(X)$ is connected.  Fix $e_X\in L_X$.  First note that If $a\in B_{\geq X}\setminus L_X$, then $\{a,ae_X\}$ is an edge from $a$ to $ae_X$ in $\Delta(B_{\geq X})$.  It thus suffices to show that any two elements of $L_X$ are connected in $\Delta(B_{\geq X})$.   If $x,y\in L_X$ are adjacent and $bx=x$, $by=y$ with $b\in B$, then $\{x,b\}$, $\{b,y\}$ is an edge path from $x$ to $y$ in $\Delta(B_{\geq X})$.  We now deduce the connectivity of $\Delta(B_{\geq X})$ from the connectivity of $\Gamma(X)$.

Assume that $\Delta(B_{\geq X})$ is connected and that $x,y\in L_X$. Then there is an edge path from $x$ to $y$ in $\Delta(B_{\geq X})$.  Consider a shortest edge path.  If the path contains consecutive edges $\{a,b\}$ and $\{b,c\}$ where $a<b<c$ or $a>b>c$, then these two edges can be replaced by the single edge $\{a,c\}$, resulting in a shorter path.  As $x,y$ are minimal elements of $B_{\geq X}$, it follows that any minimal length edge path from $x$ to $y$ in $\Delta(B_{\geq X})$ is of the form
\[\{x_0,x_1\}\{x_1,x_2\}\cdots \{x_{n-2},x_{n-1}\}\{x_{n-1},x_n\}\] with $x_0=x$, $x_n=y$ and
\[x_0<x_1>x_2<x_3>\cdots >x_{n-2}<x_{n-1}>x_n.\]  As $a\geq ay\in L_X$ for all $a\in B_{\geq X}$, we deduce that
\[x_0<x_1>x_2y<x_3>\cdots >x_{n-2}y<x_{n-1}>x_n\] and so $x\sim x_2y\sim x_4y\sim\cdots \sim x_{n-2}y\sim y$ where $\sim$ denotes the adjacency relation of $\Gamma(X)$.  This establishes that $\Gamma(X)$ is connected.
\end{proof}

The graph constructed in Lemma~\ref{l:Franco.graph} is implicit in~\cite{Saliola} and is closely related to tope graphs of oriented matroids~\cite{OrientedMatroids1999} and COMs~\cite{COMS}.  We use the lemma to prove that strong elimination systems give rise to connected left regular bands.

\begin{Prop}\label{p:se.conn}
Let $(E,\mathcal L)$ be a strong elimination system.  Then $\mathcal L$ is a connected left regular band.
\end{Prop}
\begin{proof}
Let $X\in \Lambda(\mathcal L)$.  We prove that the graph $\Gamma(X)$ is connected.  Let $x,y\in L_X$.  We prove that $x,y$ are in the same component of $\Gamma(X)$ by induction on $|S(x,y)|$.  As $Z(x)=Z(y)$, if $S(x,y)=\emptyset$, then $x=y$ and there is nothing to prove.  So assume that if $v,w\in L_X$ with $|S(v,w)|<|S(x,y)|$, then $v,w$ are in the same component of $\Gamma(X)$.  Let $e\in S(x,y)$.  By (OM3), we can find $z\in \mathcal L$ such that $z_e=0$ and $z_f=(xy)_f=(yx)_f$ for $f\in E\setminus S(x,y)$.  Notice that since $Z(x)=Z(y)$, this implies that $z\in \mathcal L_{\geq X}$.  Then we have that $zx,zy\in L_X$, $(zx)_f=x_f=y_f=(zy)_f$ for $f\in E\setminus S(x,y)$ and $(zx)_e=x_e$, $(zy)_e=y_e$.  Therefore, $|S(x,zx)|<|S(x,y)|$ and $|S(y,zy)|<|S(x,y)|$, whence we may assume by induction that $x,zx$ are in the same component of $\Gamma(X)$ and that $y,zy$ are in the same component of $\Gamma(X)$.  But $z$ is a common upper bound of $zx,zy$ and so $zx$ and $zy$ are adjacent in $\Gamma(X)$.  We conclude that $x,y$ are in the same component of $\Gamma(X)$, completing the proof.
\end{proof}

\subsection{CAT(0) cube complexes}
Our next goal is to show that the face poset of every finite CAT(0) cube complex has the structure of a connected CW left regular band. We refer the reader to~\cite{wisebook} for a survey of important developments in the theory of groups acting on CAT(0) cube complexes.  In particular, they have played an important rule in Agol's solution to Thurston's virtual fibering and virtual Haken conjectures~\cite{agol}, using heavily earlier work of Wise~\cite{wisebook}.
We follow here Abramenko and Brown~\cite[Appendix~A.2]{Brown:book2} for cube complexes and the book of Bridson and Haefliger~\cite{Bridson} for CAT(0) metric spaces.  The unpublished monograph of Roller~\cite{roller} is a good reference for combinatorially minded people on CAT(0) cube complexes. See also~\cite{chepoibandelt} for a more graph theoretic point of view.  Independently, it is shown in~\cite{COMS} that CAT(0) cube complexes are lopsided systems (and hence left regular bands) and that certain, more general, CAT(0) zonotopal complexes are COMs.  We have decided to leave here our original presentation, which was written and spoken about at various conferences in 2014--2015, and we shall discuss these further developments in the next subsection.

A (finite) \emph{cube complex}\index{cube complex} is a regular CW complex $K$ such that:
\begin{enumerate}
\item for each open $n$-cell $e$ of $K$, the subcomplex carried by $\ov e$ is isomorphic to the standard $n$-cube (with its usual cell structure with face poset $L^n$);
\item the intersection of two closed cells, if non-empty, is again a closed cell.
\end{enumerate}

Notice the second condition implies that if $C$ is a cube all of whose vertices belong to a cube $C'$, then $C\subseteq C'$.

For example, if $K$ is a simplicial complex, put $\wh {\mathcal P(K)}=\mathcal P(K)\cup \{\emptyset\}$.  Let \[\mathcal C(K)=\|\Delta(\wh{\mathcal P(K)})\|\cong \|K\ast v\|\] where $v\notin K$.  Each closed interval $[a,b]$ of $\wh{\mathcal P}(K)$ is a boolean lattice and hence $\|\Delta([a,b])\|$ is homeomorphic to a cube (with the Hasse diagram as the $1$-skeleton).  We take as our set of cubes for $C(K)$ the cubes $\|\Delta([a,b])\|$ to give $C(K)$ the structure of a cube complex.

If $K$ is a cube complex and $v$ is a vertex of $K$, then the \emph{link}\index{link} $\link_K(v)$ of $v$ is the simplicial complex with vertices all edges of $K$ incident on $v$ and where a set of edges incident on $v$ forms a simplex if the edges in question belong to a common cube.  A \emph{CAT(0) cube complex}\index{CAT(0)!cube complex}\index{cube complex!CAT(0)} is a simply connected cube complex such that the link of each vertex is a flag complex.  By a theorem of Gromov~\cite[Theorem~II.5.20]{Bridson} this is equivalent to $K$ carrying a CAT(0) metric in which each $n$-cube $C$ has a characteristic map which is simultaneously an isomorphism of CW complexes and an isometry from the $n$-cube $I^n=[0,1]^n$ equipped with its Euclidean metric to $C$ (and so in particular each edge length has  $1$).  Let us briefly recall the definition and basic properties of a CAT(0) metric space.  Details can be found~\cite[Section~II.1]{Bridson}.

If $x,y$ are points of a metric space $(X,d)$, then a \emph{geodesic}\index{geodesic} from $x$ to $y$ is a continuous map $\gamma\colon [0,\ell]\to X$ with $\gamma(0)=x$, $\gamma(\ell)=y$ and \[d(\gamma(t_1),\gamma(t_2))=|t_1-t_2|\] for all $t_1,t_2\in [0,\ell]$ (and so, in particular, $\ell=d(x,y)$).  A \emph{geodesic metric space}\index{geodesic!metric space}\index{metric space!geodesic} is a metric space $(X,d)$ in which any two points are connected by a geodesic.

A \emph{geodesic triangle}\index{geodesic!triangle} $\Delta$ in $(X,d)$ consists of three points $x,y,z\in X$ and geodesics $\gamma_{xy}$, $\gamma_{xz}$ and $\gamma_{yz}$, where $\gamma_{ab}$ is a geodesic from $a$ to $b$.  A \emph{comparison triangle}\index{comparison triangle} for $x,y,z$ is a triangle $\Delta'$ in the Euclidean plane $(\mathbb R^2,\|\cdot \|_2)$ with vertices $x',y',z'$ with $d(a',b')=d(a,b)$ for $a,b\in \{x,y,z\}$ (such a triangle always exists and is unique up to isometry of labelled triangles). One says that a geodesic metric space $X$ is a \emph{CAT(0) metric space}\index{CAT(0)!metric space}\index{metric space!CAT(0)} if $\Delta$ is thinner than $\Delta'$ for all geodesic triangles $\Delta$ in $X$.  Formally speaking, this means that if $a,b,c\in \{x,y,z\}$, if $p=\gamma_{ab}(t)$, $q=\gamma_{ac}(t')$ and if $p',q'$ are the points of the edges $[a',b']$ and $[a',c']$, respectively, with $d(a,p)=\|a'-p'\|_2$ and $d(a,q)=\|a'-q'\|_2$, then $d(p,q)\leq \|p'-q'\|_2$; see Figure~\ref{f:CAT0ineq}.
\begin{figure}[tbhp]
\begin{center}
\begin{tikzpicture}[vertices/.style={draw, fill=black, circle, inner sep=1pt}, bend angle =30]
\begin{scope}[xshift=-3cm]
\node[vertices,label=left:{$x$}]        (A) at (0,0)						{};
\node[vertices,label=above:{$y$}]  	    (B) at (2,2)         		        {};
\node[vertices,label=below:{$z$}]       (C) at (3,-1) 	                    {};
\node[vertices,inner sep=.5pt,label=right:{$p$}]       (D) at (1,.44)                        {};
\node[vertices,inner sep=.5pt,label=right:{$q$}]       (E) at (1.5,0)                        {};
\path (A) edge[bend right] node [above left] {$\gamma_{xy}$}   (B)
      (A) edge[bend left]  node [below left] {$\gamma_{xz}$}   (C)
      (B) edge[bend right]  node [right] {$\gamma_{yz}$}   (C)
      (D) edge[dashed]                                     (E);
\end{scope}
\begin{scope}[xshift=3cm]
\node[vertices,label=left:{$x'$}]        (A) at (0,0)						{};
\node[vertices,label=above:{$y'$}]  	    (B) at (2,2)         		        {};
\node[vertices,label=below:{$z'$}]       (C) at (3,-1) 	                    {};
\node[vertices,inner sep=.5pt,label=right:{$p'$}]       (D) at (1,1)                        {};
\node[vertices,inner sep=.5pt,label=right:{$q'$}]       (E) at (1.5,-.5)    {};
\path (A) edge    (B)
      (A) edge    (C)
      (B) edge    (C)
      (D) edge[dashed]                                     (E);
\end{scope}
\end{tikzpicture}
\end{center}
\caption{Comparison triangle for a CAT(0) space\label{f:CAT0ineq}}
\end{figure}
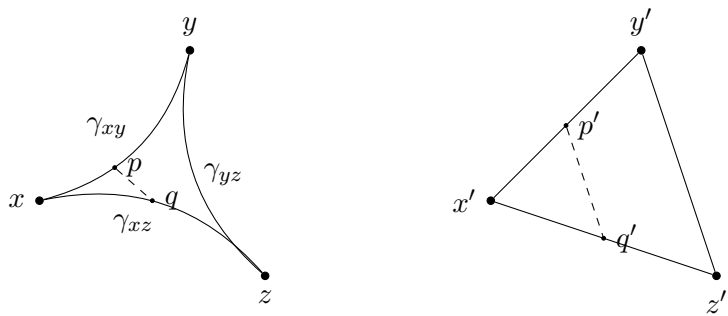
CAT(0) metric spaces were popularized by Gromov and  have played an important role in geometric group theory over the past 20 years.

The following theorem encompasses part of~\cite[Proposition~II.1.4]{Bridson} and~\cite[Corollary~II.1.5]{Bridson}.

\begin{Thm}
Let $(X,d)$ be a CAT(0) metric space.
\begin{enumerate}
\item $X$ is contractible.
\item For any points $x,y\in X$, there is a unique geodesic from $x$ to $y$.
\end{enumerate}
\end{Thm}

A non-empty subspace $C$ of a geodesic metric space $(X,d)$ is \emph{convex}\index{convex} if it contains all geodesics between points of $C$. In particular, if $(X,d)$ is CAT(0), then $C$ is convex if and only if it contains the unique geodesic between any two of its points.  A convex subspace of a CAT(0) metric space is CAT(0) in the induced metric, cf.~\cite[Example~II.1.15]{Bridson}.  Intersections of convex subspaces, if non-empty, are again convex.

\begin{Example}[CAT(0) cone]
Let $K$ be a simplicial complex. The reader should check that the cube complex $\mathcal C(K)$ constructed above is CAT(0)  if and only if $K$ is a flag complex.  In this case, $\mathcal C(K)$ is called the \emph{CAT(0) cone}\index{CAT(0)!cone}\index{cone!CAT(0)} on $K$.
\end{Example}

A graph is a CAT(0) cube complex (consisting of $0$-cubes and $1$-cubes) if and only if it is a tree. A simply connected square complex (pure $2$-dimensional cube complex)  is CAT(0) if and only if each internal vertex is surrounded by at least $4$ squares because a graph is a flag complex if and only if it is triangle-free and the link of a vertex in a square complex is a graph. See Figure~\ref{f:cat0squarecomplex} for a CAT(0) square complex whose unique internal vertex has a link which is a $5$-cycle.
\begin{figure}[htbp]
\begin{center}
\begin{tikzpicture}[vertices/.style={draw, fill=black, circle, inner sep=1pt}]
\draw[fill=gray!20] (-1,1)--(0,1)--(0,0)--(-1,0)--(-1,1)--cycle;
\draw[fill=gray!20] (0,1)--(1,1)--(1,0)--(0,0)--(0,1)--cycle;
\draw[fill=gray!20] (-1,0)--(0,0)--(-.5,-1)--(-1.5,-1)--(-1,0)--cycle;
\draw[fill=gray!20] (0,0)--(1,0)--(1.5,-1)--(.5,-1)--(0,0)--cycle;
\draw[fill=gray!20] (0,0)--(-.5,-1)--(0,-2)--(.5,-1)--(0,0)--cycle;
\node[vertices] at (-1,1)   (A){};
\node[vertices] at (0,1)    (B){};
\node[vertices] at (1,1)    (C){};
\node[vertices] at (-1,0)   (D){};
\node[vertices] at (0,0)    (E){};
\node[vertices] at (1,0)    (F){};
\node[vertices] at (-1.5,-1)   (G){};
\node[vertices] at (-.5,-1)    (H){};
\node[vertices] at (0,-2)    (I){};
\node[vertices] at (.5,-1)    (J){};
\node[vertices] at (1.5,-1)    (K){};
\draw (A)--(B)--(E)--(D)--(A)--cycle;
\draw (B)--(C)--(F)--(E)--(B)--cycle;
\draw (D)--(E)--(H)--(G)--(D)--cycle;
\draw (E)--(F)--(K)--(J)--(E)--cycle;
\draw (E)--(H)--(I)--(J)--(E)--cycle;
\draw [dashed,semithick,color=red] (0,.5)--(-.5,0)--(-0.25,-.5)--(.25,-.5)--(.5,0)--(0,.5)--cycle;
\end{tikzpicture}
\end{center}
\caption{A CAT(0) square complex with a $5$-cycle link\label{f:cat0squarecomplex}}
\end{figure}
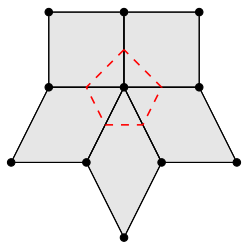

\begin{Example}[Phylogenetic Tree Space]
In~\cite{treespace}, Billera, Holmes and Vogtman introduced a CAT(0) cube complex structure on the space of phylogenetic trees called \emph{phylogenetic tree space}\index{phylogenetic tree space}.  Speyer and Sturmfels~\cite{tropgrass} later showed that phylogenetic tree space is the tropical Grassmannian.  This seems to be the most commonly studied CAT(0) cube complex outside of geometric group theory.  Strictly speaking, phylogenetic tree space is not a finite cube complex, and so we shall work with the truncation of this complex in which the edge lengths of interior edges of phylogenetic trees are at most $1$.

Following the formalism of~\cite{treespace}, the \emph{phylogenetic tree space}\index{phylogenetic tree space} $\mathcal T_n$ will be defined as follows.  We consider labelled $n$-trees with $n\geq 2$.  These are trees with $n+1$ leaves labelled by $\{0,\ldots, n\}$ and each non-leaf has valence at least $3$.  We think of the leaf labelled by $0$ as the root.  See Figure~\ref{f:3trees} for the case $n=3$.
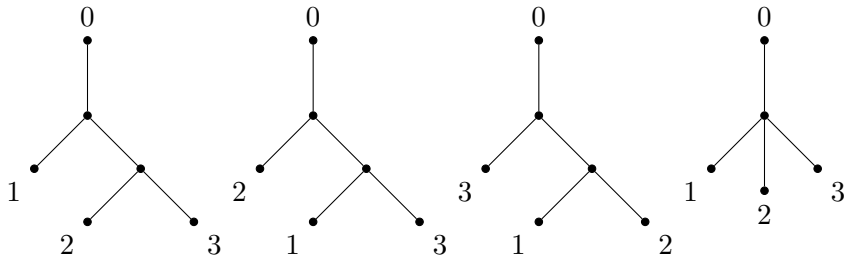
\begin{figure}[htbp]
\begin{center}
\begin{tikzpicture}[vertices/.style={draw, fill=black, circle, inner sep=1pt}]
\begin{scope}[xshift=-4.5cm]
\node[vertices,label=above:{$0$}]                   (A) {};
\node[vertices] [below of=A]                        (B) {};
\node[vertices,label=below left:{$1$}] [below left of=B] (C) {};
\node[vertices]                        [below right of=B] (D) {};
\node[vertices,label=below left:{$2$}] [below left of=D] (E) {};
\node[vertices,label=below right:{$3$}][below right of=D] (F) {};
\draw (A)--(B);
\draw (B)--(C);
\draw (B)--(D);
\draw (D)--(E);
\draw (D)--(F);
\end{scope}
\begin{scope}[xshift=-1.5cm]
\node[vertices,label=above:{$0$}]                   (A) {};
\node[vertices] [below of=A]                        (B) {};
\node[vertices,label=below left:{$2$}] [below left of=B] (C) {};
\node[vertices]                        [below right of=B] (D) {};
\node[vertices,label=below left:{$1$}] [below left of=D] (E) {};
\node[vertices,label=below right:{$3$}][below right of=D] (F) {};
\draw (A)--(B);
\draw (B)--(C);
\draw (B)--(D);
\draw (D)--(E);
\draw (D)--(F);
\end{scope}
\begin{scope}[xshift=1.5cm]
\node[vertices,label=above:{$0$}]                   (A) {};
\node[vertices] [below of=A]                        (B) {};
\node[vertices,label=below left:{$3$}] [below left of=B] (C) {};
\node[vertices]                        [below right of=B] (D) {};
\node[vertices,label=below left:{$1$}] [below left of=D] (E) {};
\node[vertices,label=below right:{$2$}][below right of=D] (F) {};
\draw (A)--(B);
\draw (B)--(C);
\draw (B)--(D);
\draw (D)--(E);
\draw (D)--(F);
\end{scope}
\begin{scope}[xshift=4.5cm]
\node[vertices,label=above:{$0$}]                   (A) {};
\node[vertices] [below of=A]                        (B) {};
\node[vertices,label=below left:{$1$}] [below left of=B] (C) {};
\node[vertices,label=below:{$2$}]      [below of=B] (D) {};
\node[vertices,label=below right:{$3$}] [below right of=B] (E) {};
\draw (A)--(B);
\draw (B)--(C);
\draw (B)--(D);
\draw (B)--(E);
\end{scope}
\end{tikzpicture}
\end{center}
\caption{The four labelled $3$-trees\label{f:3trees}}
\end{figure}

By an \emph{interior edge}\index{interior edge} of a labelled $n$-tree, we mean an edge not incident on a leaf.   A \emph{metric}\index{metric} labelled $n$-tree is one whose interior edges are assigned positive lengths in the interval $(0,1]$.  (Note that in~\cite{treespace} arbitrary positive lengths are allowed.) See Figure~\ref{f:metrictree}.
\begin{figure}[htbp]
\begin{center}
\begin{tikzpicture}[vertices/.style={draw, fill=black, circle, inner sep=1pt}]
\node[vertices,label=above:{$0$}]                   (A) {};
\node[vertices] [below of=A]                        (B) {};
\node[vertices,label=below left:{$1$}] [below left of=B] (C) {};
\node[vertices]                        [below right of=B] (D) {};
\node[vertices,label=below left:{$2$}] [below left of=D] (E) {};
\node[vertices,label=below right:{$3$}][below right of=D] (F) {};
\path (A) edge (B)
      (B) edge (C)
      (B) edge node [above right] {$\frac{2}{3}$} (D);
\draw (D)--(E);
\draw (D)--(F);
\end{tikzpicture}
\end{center}
\caption{A metric tree\label{f:metrictree}}
\end{figure}
There is an obvious notion of isomorphism of $n$-labelled metric trees as graph isomorphisms preserving the labellings and the edge lengths.  We shall identify isomorphic metric trees.

It will be convenient to allow $n$-labelled trees to have interior edges of length $0$ and to identify such a tree as being isomorphic to the metric tree obtained by contracting to a point each edge of length $0$.   In this way, each metric tree can be represented by a trivalent tree (all inner nodes have degree $3$) with $n-2$ interior edges with lengths in the interval $[0,1]$. See Figure~\ref{f:degeneratetree}.
\begin{figure}[htbp]
\begin{center}
\begin{tikzpicture}[vertices/.style={draw, fill=black, circle, inner sep=1pt}]
\begin{scope}[xshift=-3cm]
\node[vertices,label=above:{$0$}]                   (A) {};
\node[vertices] [below of=A]                        (B) {};
\node[vertices,label=below left:{$1$}] [below left of=B] (C) {};
\node[vertices]                        [below right of=B] (D) {};
\node[vertices,label=below left:{$2$}] [below left of=D] (E) {};
\node[vertices,label=below right:{$3$}][below right of=D] (F) {};
\node                                  [right=3cm of B] {$\cong$};
\draw (A)--(B);
\draw (B)--(C);
\path (B) edge node [above right] {$0$} (D);
\draw (D)--(E);
\draw (D)--(F);
\end{scope}
\begin{scope}[xshift=3cm]
\node[vertices,label=above:{$0$}]                   (A) {};
\node[vertices] [below of=A]                        (B) {};
\node[vertices,label=below left:{$1$}] [below left of=B] (C) {};
\node[vertices,label=below:{$2$}]      [below of=B] (D) {};
\node[vertices,label=below right:{$3$}] [below right of=B] (E) {};
\draw (A)--(B);
\draw (B)--(C);
\draw (B)--(D);
\draw (B)--(E);
\end{scope}
\end{tikzpicture}
\end{center}
\caption{A degenerate metric tree\label{f:degeneratetree}}
\end{figure}
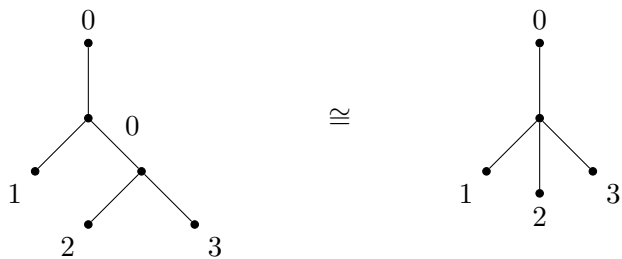
Let us say an $n$-labelled tree $T'$ is a \emph{degeneration}\index{degeneration} of an $n$-labelled tree $T$ if it is obtained from $T$ by contracting some number of interior edges to a point.  For example, the $n$-labelled tree with no interior edges is a degeneration of every $n$-labelled tree.

Fix a trivalent $n$-labelled tree $T$.  Then we can identify the set $C_T$ of isomorphism classes of metric trees obtained from assigning lengths from $[0,1]$ to the $n-2$ interior edges with the cube $[0,1]^{n-2}$. The proper faces of $C_T$ come from allowing some interior edges to have length $0$ and correspond bijectively to $n$-labelled trees which are degenerations of $T$.  If $T,T'$ are two trivalent trees with a common degeneration, then we isometrically glue together the faces of $C_T$ and $C_{T'}$ corresponding to their common degenerations. The resulting cube complex $\mathcal T_n$ is CAT(0). The metric $n$-labelled tree with no interior edges is a cone point for $\mathcal T_n$. See Figure~\ref{f:T3} for $\mathcal T_3$.
\begin{figure}[htbp]
\begin{center}
\begin{tikzpicture}[vertices/.style={draw, fill=black, circle, inner sep=1pt}]
\node[vertices]                   (A) {};
\node[vertices] [left of=A]       (B) {};
\node[vertices] [above right of=A](C) {};
\node[vertices] [below right of=A](D) {};
\draw (A)--(B);
\draw (A)--(C);
\draw (A)--(D);
\end{tikzpicture}
\end{center}
\caption{Phylogenetic tree space $\mathcal T_3$\label{f:T3}}
\end{figure}
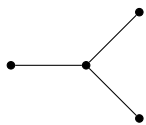

In fact, $\mathcal T_n$ is the CAT(0) cone on the following flag complex.  Call a partition of $\{0,\ldots, n\}$ into two subsets \emph{thick}\index{thick} if both parts have at least two elements.  If $T$ is an $n$-labelled tree and $e$ is an interior edge of $T$, then $T-\{e\}$ has two connected components which induce a thick partition on $\{0,\ldots, n\}$. Define a graph with vertices the thick partitions. If $\{A,B\}$ and $\{A',B'\}$ are thick partitions, then we connect them by an edge if \[A\subseteq A'\ \text{or}\ A\subseteq B'\ \text{or}\ B\subseteq A'\ \text{or}\ B\subseteq B'.\]  Let $L_n$ be the clique complex of this graph.  Then the simplices of $L_n$ are in bijection with $n$-labelled trees.  Given an $n$-labelled tree $T$, one obtains a simplex by considering all thick partitions obtained by removing some interior edge from $T$.  The facets of $L_n$ are in bijection with the trivalent trees.  It is then easy to check that $\mathcal T_n$ is the CAT(0) cone on $L_n$ and that $L_n$ is the link of the origin. See~\cite{treespace,tropgrass} for details.
\end{Example}

Fundamental to the theory of CAT(0) cube complexes is the notion of a hyperplane.

A \emph{midcube}\index{midcube} of $I^n$ is the intersection of $I^n$ with one of the affine hyperplanes \[H_i=\left\{x\in \mathbb R^n\mid x_i=\frac{1}{2}\right\}.\]   If $e$ is an edge of $I^n$, then the midcube dual to $e$ is the intersection of $I^n$ with the perpendicular bisector of $e$.  See Figure~\ref{f:midcube}.
\begin{figure}[htbp]
\begin{center}
\begin{tikzpicture}[x  = {(0.5cm,0.5cm)},
                    y  = {(0.95cm,-0.25cm)},
                    z  = {(0cm,0.9cm)}]
\begin{scope}[canvas is yz plane at x=-1]
  \shade[left color=blue!50,right color=blue!20] (-1,-1) rectangle (1,1);
\end{scope}
\begin{scope}[canvas is xz plane at y=1]
  \shade[right color=blue!70,left color=blue!20] (-1,-1) rectangle (1,1);
\end{scope}
\begin{scope}[canvas is yx plane at z=1]
  \shade[top color=blue!80,bottom color=blue!20] (-1,-1) rectangle (1,1);
\end{scope}
\begin{scope}[canvas is xz plane at y=0]
  \shade[bottom color=red!50,top color=red!20] (-1,-1) rectangle (1,1);
\end{scope}
\end{tikzpicture}
\end{center}
\caption{A midcube\label{f:midcube}}
\end{figure}

Midcubes of an $n$-cube of a CAT(0) cube complex can be defined via the characteristic map identifying it isometrically with $I^n$.  Let $X$ be a CAT(0) cube complex and let $\sim$ be the smallest equivalence relation on the edges of $X$ such that opposite edges of any square of $X$ are equivalent. If $e$ is an edge of a CAT(0) cube complex, then the \emph{hyperplane}\index{hyperplane} $H_e$ of $X$ dual to $e$ is the union of all midcubes dual to edges in the $\sim$-equivalence class of $e$. Equivalently, a hyperplane is a maximal connected subspace which is a union of midcubes. See~\cite{wisebook,roller} for details.  Figure~\ref{f:hyperplane} depicts a hyperplane in a CAT(0) cube complex.
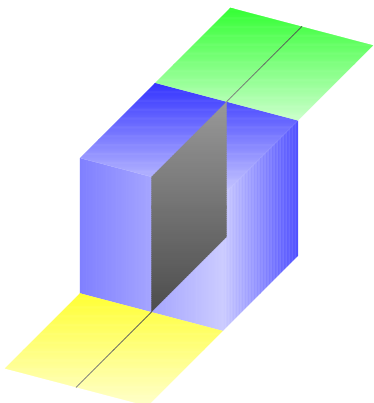
\begin{figure}[htbp]
\begin{center}
\begin{tikzpicture}[x  = {(0.5cm,0.5cm)},
                    y  = {(0.95cm,-0.25cm)},
                    z  = {(0cm,0.9cm)}]
\begin{scope}[canvas is yz plane at x=-1]
  \shade[left color=blue!50,right color=blue!20] (-1,-1) rectangle (1,1);
\end{scope}
\begin{scope}[canvas is xz plane at y=1]
  \shade[right color=blue!70,left color=blue!20] (-1,-1) rectangle (1,1);
\end{scope}
\begin{scope}[canvas is yx plane at z=1]
  \shade[top color=blue!80,bottom color=blue!20] (-1,-1) rectangle (1,1);
\end{scope}
\begin{scope}[canvas is xz plane at y=0]
  \shade[bottom color=black!70,top color=black!40] (-1,-1) rectangle (1,1);
\end{scope}
\begin{scope}[canvas is yx plane at z=1]
  \shade[top color=green!80,bottom color=green!20] (-1,1) rectangle (1,3);
\end{scope}
\begin{scope}[canvas is yx plane at z=-1]
  \shade[top color=yellow!80,bottom color=yellow!20] (-1,-3) rectangle (1,-1);
\end{scope}
\begin{scope}[canvas is yx plane at z=-1]
  \shade[top color=black!70,bottom color=black!50] (-.006,-3) rectangle (.006,-1);
\end{scope}
\begin{scope}[canvas is yx plane at z=1]
  \shade[bottom color=black!50,top color=black!70] (-.006,1) rectangle (.006,3);
\end{scope}
\end{tikzpicture}
\end{center}
\caption{A hyperplane in a CAT(0) cube complex\label{f:hyperplane}}
\end{figure}
Figure~\ref{f:more.hyper} displays all the hyperplanes for the CAT(0) cube complex in Figure~\ref{f:cat0squarecomplex}.
\begin{figure}[htbp]
\begin{center}
\begin{tikzpicture}[vertices/.style={draw, fill=black, circle, inner sep=1pt}]
\draw[fill=gray!20] (-1,1)--(0,1)--(0,0)--(-1,0)--(-1,1)--cycle;
\draw[fill=gray!20] (0,1)--(1,1)--(1,0)--(0,0)--(0,1)--cycle;
\draw[fill=gray!20] (-1,0)--(0,0)--(-.5,-1)--(-1.5,-1)--(-1,0)--cycle;
\draw[fill=gray!20] (0,0)--(1,0)--(1.5,-1)--(.5,-1)--(0,0)--cycle;
\draw[fill=gray!20] (0,0)--(-.5,-1)--(0,-2)--(.5,-1)--(0,0)--cycle;
\node[vertices] at (-1,1)   (A){};
\node[vertices] at (0,1)    (B){};
\node[vertices] at (1,1)    (C){};
\node[vertices] at (-1,0)   (D){};
\node[vertices] at (0,0)    (E){};
\node[vertices] at (1,0)    (F){};
\node[vertices] at (-1.5,-1)   (G){};
\node[vertices] at (-.5,-1)    (H){};
\node[vertices] at (0,-2)    (I){};
\node[vertices] at (.5,-1)    (J){};
\node[vertices] at (1.5,-1)    (K){};
\draw (A)--(B)--(E)--(D)--(A)--cycle;
\draw (B)--(C)--(F)--(E)--(B)--cycle;
\draw (D)--(E)--(H)--(G)--(D)--cycle;
\draw (E)--(F)--(K)--(J)--(E)--cycle;
\draw (E)--(H)--(I)--(J)--(E)--cycle;
\draw [dashed,semithick,color=red] (-.5,1)--(-.5,0)--(-1.05,-1);
\draw [dashed,semithick,color=red] (.5,1)--(.5,0)--(1.05,-1);
\draw [dashed,semithick,color=blue] (-1,.5)--(1,.5);
\draw [dashed,semithick,color=blue] (-1.25,-.5)--(-.25,-.5)--(.25,-1.525);
\draw [dashed,semithick,color=orange] (1.25,-.5)--(.25,-.5)--(-.25,-1.525);
\end{tikzpicture}
\end{center}
\caption{The hyperplanes from the CAT(0) cube complex in Figure~\ref{f:cat0squarecomplex}\label{f:more.hyper}}
\end{figure}
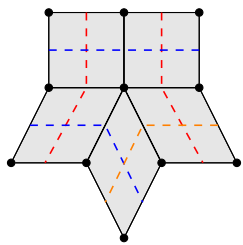

We need the following fundamental facts about CAT(0) cube complexes. See~\cite[Theorem~2.13]{wisebook}.

\begin{Thm}\label{t:cat0geometry}
Let $K$ be a CAT(0) cube complex.
\begin{enumerate}
\item Each hyperplane of $K$ is convex and hence a CAT(0) cube complex (whose cubes are the faces of the midcubes that comprise it).
\item If $C$ is an $n$-cube of $K$, then each of the $n$-midcubes of $C$ lie on different hyperplanes of $K$.
\item If $H$ is a hyperplane, then $K\setminus H$ has two connected components called \emph{half-spaces}\index{half-spaces}.
\end{enumerate}
\end{Thm}

Although we do not use it, we mention the following important theorem obtained, independently by Chepoi~\cite{chepoi} and Roller~\cite{roller}. We recall that a \emph{median graph}\index{median graph} is a connected graph $\Gamma$ such that if $x,y,z$ are vertices of $\Gamma$ and $\gamma_{ab}$ are geodesic edge paths from $a$ to $b$ for $a,b\in \{x,y,z\}$, then $|\gamma_{xy}\cap \gamma_{xz}\cap \gamma_{yz}|=1$.  For example, the hypercube graph is a median graph, as is the Hasse diagram of any distributive lattice.

\begin{Thm}\label{catcube1skel}
A CAT(0) cube complex $K$ is uniquely determined by its $1$-skeleton, which is a median graph. Namely, the cubes of $K$ are obtained by filling in with an $n$-cube each induced subgraph of $K^1$ which is isomorphic to the $1$-skeleton of an $n$-cube.

Conversely, if $\Gamma$ is a median graph, then the cube complex obtained by filling in $1$-skeleta of induced hypercube subgraphs is a CAT(0) cube complex.
\end{Thm}

Let $K$ be a CAT(0) cube complex with hyperplanes $H_1,\ldots, H_n$.  We define the \emph{intersection semilattice}\index{intersection!semilattice} $\mathcal L(K)$ to consist of all non-empty intersections $H_{i_1}\cap\cdots\cap H_{i_k}$ of hyperplanes, where if $k=0$, then we interpret the intersection to be $K$.  We order it by reverse inclusion.  Note that the semilattice operation takes two elements $X,Y$ of $\mathcal L(K)$ and brings it to the intersection of all hyperplanes containing both $X,Y$ (with the empty intersection being $K$). Hyperplanes in CAT(0) cube complexes are convex in the CAT(0) metric by Theorem~\ref{t:cat0geometry} and hence each element  $X\in \mathcal L(K)$ is convex and therefore a CAT(0) cube complex in its own right.  The cubes of $X$ are exactly the non-empty intersections of cubes of $K$ with the subspace $X$.  One can recover the set $A=\{H_{i_1},\ldots, H_{i_k}\}$ of hyperplanes from the vertices of $X$, namely each vertex is the midpoint of a $k$-cube whose midcubes are precisely its intersections with the elements of $A$.  In other words, each element $X$ of $\mathcal L(K)$ can be uniquely expressed as an intersection of hyperplanes. In general, if $C$ is an $n$-cube of $K$ with $C\cap X\neq \emptyset$, then $C\cap X$ is an $(n-k)$-cube of $X$.

The set of hyperplanes in a CAT(0) cube complex satisfies the \emph{Helly property}\index{Helly property}: any collection of hyperplanes with pairwise non-empty intersections has non-empty intersection; see~\cite[Page~143]{chepoi} for a proof.  Recall that if $\mathcal F=\{X_1,\ldots, X_n\}$ is a collection of subsets of a set $X$, then the \emph{nerve}\index{nerve} of $\mathcal F$ is the simplicial complex $\mathcal N(\mathcal F)$ with vertex set $\mathcal F$ and where $\{X_{i_0},\ldots, X_{i_q}\}$ is a $q$-simplex if $X_{i_0}\cap \cdots\cap X_{i_q}\neq \emptyset$.  Clearly, $\mathcal N(\mathcal F)$ is a flag complex if and only if $\mathcal F$ has the Helly property.

\begin{Prop}\label{p:intersectcat0faceposet}
Let $K$ be a finite CAT(0) cube complex with set $\mathcal H$ of hyperplanes. Then the intersection semilattice $\mathcal L(K)$ can be identified with the face poset of the nerve $\mathcal N(\mathcal H)$ of $\mathcal H$, including an empty face.   Moreover, $\mathcal N(\mathcal H)$ is a flag complex and every finite flag complex comes about this way.
\end{Prop}
\begin{proof}
We already saw above that each element $X$ of $\mathcal L(K)$ can be uniquely expressed as an intersection of a subset $A_X$ of hyperplanes.  The mapping $X\mapsto A_X$ gives an isomorphism of $\mathcal L(K)$ with $\mathcal P(\mathcal N(\mathcal H))\cup \{\emptyset\}$ (sending $K$ to the empty face).  Since $\mathcal H$ has the Helly property, $\mathcal N(\mathcal H)$ is a flag complex.  Conversely, the reader should verify that if $X$ if a flag complex and $K$ is the CAT(0) cone on $X$, then $\mathcal N(\mathcal H)$ is isomorphic to $X$ (where $\mathcal H$ is the set of hyperplanes of $K$).
\end{proof}

One can consider the vertex set $K^0$ of a CAT(0) cube complex as a metric space with the path metric $d_p$.  The distance $d_p(v,w)$ between two vertices $v,w$ is the length of a shortest edge path between them in the $1$-skeleton $K^1$.  Let $S(v,w)$ be the set of hyperplanes $H$ separating $v$ and $w$, meaning that $v$ and $w$ are in different components of $K\setminus H$.  Then it is known that \[d_p(v,w)=|S(v,w)|.\]  In particular, two vertices are connected by an edge if and only if they are separated by exactly one hyperplane, which is necessarily the hyperplane dual to that edge. See~\cite[Remark~2.14]{wisebook} or~\cite{roller}.

We shall need the following lemma stating that a cube can be separated from each vertex outside of it by a hyperplane.

\begin{Lemma}\label{separatefromcube}
Let $K$ be a CAT(0) cube complex, let $v$ be a vertex of $K$ and let $C$ be a cube of $K$ not containing $v$. Then there is a hyperplane $H$ separating $v$ from $C$, that is, $v$ and $C$ are contained in different connected components of $K\setminus H$.
\end{Lemma}
\begin{proof}
Let $w$ be a vertex of $C$ at a minimal distance from $v$ with respect to the path metric $d_p$.  Then $|S(v,w)|=d_p(v,w)\geq 1$.  Let $H\in S(v,w)$.  Suppose first that $H\cap C\neq \emptyset$.  Then there is a unique edge $e$ of $C$ dual to $H$ and incident on $w$. Let $w'$ be the other endpoint of $e$.  Then $w'$ is in the same component of $K\setminus H$ as $v$ and so $H\notin S(v,w')$. If $H'\in S(v,w')$, then because $S(w,w')=\{H\}$, we must have that $w,w'$ are in the same component of $K\setminus H'$ and hence $H'\in S(v,w)$. We conclude that $S(v,w')\subseteq S(v,w)\setminus \{H\}$ and so $d_p(v,w')<d_p(v,w)$, a contradiction. Thus $H\cap C=\emptyset$.  Since $C$ is connected and $w$ is in a different component of $K\setminus H$ than $v$, we conclude that $H$ separates $v$ and $C$.
\end{proof}

Another important fact is that every finite CAT(0) cube complex embeds as a subcomplex of a hypercube; this was first proved by Bandelt in the language of median graphs~\cite{Bandelt}.  This is the third item of the following theorem.  We provide a proof here using covectors for both completeness, and to make the semigroup structure transparent.
The final two items of the next theorem are new; the last item was obtained jointly with Daniel Wise.

\begin{Thm}\label{embedinhypercube}
Let $K$ be a finite CAT(0) cube complex with hyperplanes $\{H_1,\ldots, H_d\}$.  Fix, for each $H_i$, positive and negative half-spaces $H_i^+$ and $H_i^-$ of $K\setminus H_i$.   To each cube $C$ of $K$ assign the covector $\tau(C)\in L^d$ given by
\[\tau(C)_i = \begin{cases} +, & \text{if}\ C\subseteq H_i^+\\
                            -, & \text{if}\ C\subseteq H_i^-\\
                            0, & \text{if}\ C\cap H_i\neq \emptyset.\end{cases}\]
Let $\FFF(K)$ be the set of all $\tau(C)$ such that $C$ is a cube of $K$.
Then the following hold.
\begin{enumerate}
\item $\tau\colon \mathcal P(K)\to L^d$ is an order embedding.
\item $\FFF(K)$ is a right ideal in $L^d$.
\item $K$ is isomorphic to a subcomplex of $[-1,1]^d$.
\item $\Lambda(\FFF(K))\cong \mathcal L(K)$.
\item If $X\in \mathcal L(K)$, then $\FFF(K)_{\geq X}\cong \FFF(X)$.
\end{enumerate}
In particular, $\FFF(K)$ is a connected CW left regular band.
\end{Thm}
\begin{proof}
Clearly if $C$ is a face of $C'$, then $\tau(C)\leq \tau(C')$.  Suppose that $\tau(C)\leq \tau(C')$.
Let $v$ be a vertex of $C$.  If $v\notin C'$, then there is a hyperplane $H_i$ separating $v$ from $C'$ by Lemma~\ref{separatefromcube}.  But then $\tau(C')_i=-\tau(v)_i\neq 0$ and $\tau(v)_i\leq \tau(C)_i\leq \tau(C')_i$, a contradiction.  Thus each vertex of $C$ is a vertex of $C'$, whence $C\subseteq C'$. Therefore, $\tau$ is an order embedding.

Suppose that $C$ is an $n$-cube of $K$.  Then exactly $n$ hyperplanes intersect $C$ and they intersect it in its midcubes by Theorem~\ref{t:cat0geometry}. Thus $\tau(C)$ has exactly $n$ zeroes and therefore $3^n$ faces.  But $C$ also has $3^n$ faces and $\tau$ is an order embedding.  Thus every face of $\tau(C)$ is in the image of $\tau$ and so $\FFF(K)$ is a lower set, that is, a right ideal of $L^d$.

Since $\mathcal P(K)$ is order isomorphic to a lower set of $L^d=\mathcal P([-1,1]^d)$, it follows that $K$ is isomorphic to a subcomplex of $[-1,1]^d$ because regular cell complexes are determined up to isomorphism by their face posets.

Because $\FFF(K)$ is a right ideal in $L^d$, it follows that $\sigma\colon \FFF(K)\to \Lambda(\FFF(K))$ can be identified with the zero set map $Z\colon \FFF(K)\to Z(\FFF(K))$.  If $A\subseteq \{1,\ldots, n\}$, then $A=Z(\tau(C))$ if and only if $A=\{i\mid H_i\cap C\neq \emptyset\}$ and hence $\bigcap_{i\in A}H_i\in \mathcal L(K)$.  Conversely, if $X=\bigcap_{i\in A}H_i\in \mathcal L(K)\neq \emptyset$ choose a vertex $v$ of $X$. By definition of the cube complex structure on $X$, there is a $|A|$-cube $C$ of $K$ with $C\cap X=\{v\}$.  Then $Z(\tau(C))=A$.  Thus $\Lambda(\FFF(K))$ can be identified with $\mathcal L(K)$ as a meet semilattice.

If $A\in \Lambda(\FFF(K))$ and $X=\bigcap_{i\in A}H_i\in \mathcal L(K)\neq \emptyset$ are as above, then the cubes $C$ with $A\subseteq Z(\tau(C))$ are exactly the cubes $C$ of $K$ meeting $X$.  But the cubes of $X$ are precisely the non-empty intersections of cubes of $K$ with $X$.  Thus the map $\tau(C)\mapsto C\cap X$ gives an order isomorphism between $\FFF(K)_{\geq A}$ and the face poset $\mathcal P(X)$ of $X$. In particular, $\FFF(K)_{\geq A}$ is a connected CW poset.  But more is true.  The hyperplanes of $X$ are precisely the $H'_i=H_i\cap X$ with $i\notin A$ and so if we choose the positive (respectively, negative) half-space of $H'_i$ to be $H_i^+\cap X$ (respectively, $H_i^-\cap X$) and let $\tau'\colon \mathcal P(X)\to L^{[n]\setminus A}$ be the corresponding embedding, then the map $\tau(C)\mapsto \tau'(C\cap X)$ is a semigroup isomorphism between $\FFF(K)_{\geq A}$ and $\FFF(X)$.
\end{proof}

We call $\FFF(K)$ the \emph{face semigroup}\index{face semigroup} of  the CAT(0) cube complex $K$.  We shall identify $\FFF(K)$ with the set of faces of $K$.  The reader should verify that the semigroup structure does not depend on the choice of orientation for the hyperplanes.  In fact, there is the following geometric interpretation of the product.  If $C,C'$ are cubes, then $C\cdot C'$ is the face of $C$ on which the mapping $x\mapsto d_p(x,C')$ is minimized.

\subsection{CAT(0) zonotopal complexes}
By a \emph{Coxeter zonotope}\index{Coxeter zonotope} we mean a zonotope (called an \emph{even polyhedron}\index{even polyhedron} in~\cite{HaglundPaulin}) that is symmetric around the perpendicular bisector, or \emph{mid-hyperplane}\index{mid-hyperplane}, of each edge.  A zonotope is combinatorially isomorphic to a Coxeter zonotope if and only if it is a $W$-permutohedron for a finite Coxeter group $W$, that is, the zonotope dual to a Coxeter hyperplane arrangement $\mathcal A_W$.  Of course, the $n$-cube is the Coxeter zonotope corresponding to the finite right-angled Coxeter group $(\mathbb Z/2\mathbb Z)^n$. See~\cite{COMS,HaglundPaulin} for details.

By a \emph{zonotopal complex}\index{zonotopal complex}, we mean a polyhedral complex whose cells are zonotopes.  So it is a collection of zonotopes closed under taking faces and such that the intersection of any two zonotopes in the collection is a common face of them both (if non-empty). A \emph{CAT(0) zonotopal complex}\index{CAT(0)!zonotopal complex}\index{zonotopal complex!CAT(0)} is a zonotopal complex equipped with a CAT(0) metric. If $K$ is a CAT(0) zonotopal complex, each of whose zonotopes is isometric to a Coxeter zonotope, then one can define a hyperplane in $K$ similarly to the CAT(0) cube complex case.  Namely, consider the equivalence relation on the set of edges generated  by pairs of edges that are on opposite sides of a two-dimensional face (two-dimensional zonotopes always have an even number of sides).  The union of all mid-hyperplanes from an equivalence class of edges gives a hyperplane in $K$ and one has that hyperplanes do not self-intersect and have many of the same properties of hyperplanes  in CAT(0) cube complexes such as having exactly two connected components in their complement and being convex.  They are CAT(0) zonotopal complexes in their own right with Coxeter zonotopes as cells.  See~\cite[Section~11]{COMS} for details.

If $K$ is a CAT(0) zonotopal complex with Coxeter zonotopes as cells, then one can define a collection of covectors $\mathcal L\subseteq L^E$, where $E$ is the set of hyperplanes (which we assume have an orientation selected for their two half-spaces), via the same recipe we used for CAT(0) cube complexes. If $Z$ is a face of $K$, we assign it the covector $\tau(Z)$ with
\[\tau(Z)_e = \begin{cases} 0, & \text{if}\ e\cap Z\neq \emptyset\\ +, & \text{if}\ Z\subseteq e^+\\ -, &\text{if}\ Z\subseteq e^-\end{cases}\] where $e^+$ and $e^-$ denote the positive and negative half-spaces associated with the hyperplane $e$.  The following beautiful theorem was proved in~\cite[Section~11]{COMS}.

\begin{Thm}\label{t:cat(0).zono}
If $K$ is a CAT(0) zonotopal complex whose cells are Coxeter zonotopes, then the collection of covectors defined above provides the face poset of $K$ with the structure of a COM. Moreover, each contraction of $K$ is isomorphic to the semigroup of covectors of an intersection of hyperplanes, which is again a CAT(0) zonotopal complex whose cells are Coxeter zonotopes.
\end{Thm}

We remark that the support semilattice for a CAT(0) zonotopal complex with Coxeter zonotope cells is more complicated than in the cube complex case because the intersection of a zonotope with a subset of hyperplanes does not determine the subset.  For example, the regular hexagon is the Coxeter zonotope associated to the symmetric group $S_3$.  There are $3$ hyperplanes, but any two of them have the same intersection point: the center of the hexagon.

Note that in the case of a CAT(0) cube complex $K$, the covectors we obtain form a right ideal in $L^n$, where $n$ is the number of hyperplanes, and so the COM we obtain is in fact a lopsided system.

\section{Algebras}\label{s:algebraprelim}
We introduce the necessary ring theory in order to talk about the representation theory of left regular bands.  The representation theory of a left regular band is that of its semigroup ring.  Much of the modern theory of finite dimensional algebras, e.g. quivers and quiver presentations, etc., is predicated upon the algebra being unital.  Many of the examples of left regular bands that we have been considering are not monoids.  This means we have to introduce semigroup algebras that are not, \textit{a priori}, unital.  Thus, we begin by discussing the notion of the Jacobson radical for rings that do not necessarily have a unit and developing some of the basic theory.  We then recall some standard notions and results about finite dimensional algebras including quivers, basic algebras, graded algebras and Koszul algebras.  Afterward, we begin the detailed study of a left regular band algebra and, in particular, characterize the property of having a unital algebra as the topological property of being a connected left regular band.  For connected left regular bands, we describe the simple modules and projective indecomposable modules and compute a complete set of primitive idempotents, as well as the Cartan matrix.  This generalizes the second author's results in the case of left regular band monoids~\cite{Saliola}.

\subsection{Rings and radicals}
\nomenclature[R, 01]{$\rad(R)$}{Jacobson radical of the ring $R$}%
Let $R$ be a ring, not necessarily unital.  We recall the definition of the Jacobson radical of $R$~\cite{rowen}.  A left ideal $L$ of $R$ is said to be \emph{modular}\index{modular} if there is an element $a\in R$ such that $r-ra\in L$ for all $r\in R$.  Note that if $R$ has a right identity, in particular, if $R$ is unital, then every left ideal is modular (take $a$ to be the right identity).  By a \emph{maximal}\index{maximal} modular left ideal is meant a maximal proper modular left ideal. One can define (maximal) modular right ideals analogously. The \emph{(Jacobson) radical}\index{(Jacobson) radical} $\rad(R)$ is the intersection of all maximal modular left ideals. It is a two-sided ideal and it is also the intersection of all maximal modular right ideals. If $R$ has a right identity (e.g.\ if $R$ is unital), then $\rad(R)$ is the intersection of all maximal left ideals of $R$ or, equivalently is the intersection of all annihilators of simple left $R$-modules.  Every nilpotent ideal of $R$ is contained in the radical.   A ring $R$ is called \emph{semiprimitive}\index{semiprimitive} if $\rad(R)=0$.  If $R$ is semiprimitive and finite dimensional over a field, then $R$ is automatically unital and semisimple~\cite{rowen}.  A commutative ring with unit is semiprimitive if and only if it is a subdirect product of fields.

In this text, module unmodified shall always means left module. If we want to discuss right modules we shall explicitly say so. Sometimes, we shall write ``left module'' for emphasis.  A module over a unital ring is called \emph{unitary}\index{unitary} if $1M=M$.

Let us remark that if $R$ is a ring with a right identity $e$, then $R$ is a projective left module, as is any direct summand in $R$. Indeed, $\Hom_R(R,M)\cong eM$ for any $R$-module $M$ and the functor $M\mapsto eM$ is obviously exact.

\subsection{Finite dimensional algebras}\label{ss:finite.dim}
The reader is referred to~\cite{assem,benson,AuslanderReiten,curtis} for background on the modern theory of finite dimensional algebras. A reference for homological algebra is~\cite{CartanEilenberg}. We summarize here what we need for this paper. For this subsection, we fix a field $\Bbbk$ and all $\Bbbk$-algebras are assumed unital  and finite dimensional (except path algebras of non-acyclic quivers).  All modules will be assumed unitary. A $\Bbbk$-algebra $A$ is said to be \emph{split}\index{split} if its semisimple quotient $A/\rad(A)$ is isomorphic to a direct product of matrix algebras over $\Bbbk$ or, equivalently, each simple $A$-module is absolutely irreducible. In particular, if $\Bbbk$ is algebraically closed, then all finite dimensional $\Bbbk$-algebras are split.

A $\Bbbk$-algebra is \emph{basic}\index{basic} if $A/\rad(A)$ is isomorphic to a direct product of division rings.  Consequently, $A$ is split basic if and only if $A/\rad(A)\cong \Bbbk^n$ for some $n\geq 1$ or, in other words, $A$ is split basic if and only if every simple $A$-module is one-dimensional.  Every split finite dimensional algebra over $\Bbbk$ is Morita equivalent to a unique (up to isomorphism) split basic algebra. Recall that two algebras are \emph{Morita equivalent}\index{Morita equivalent} if their module categories are equivalent.  Therefore, the representation theory of finite dimensional algebras over an algebraically closed field can be reduced to the case of (split) basic algebras.

An $A$-module $M$ is \emph{indecomposable}\index{indecomposable} if it cannot be expressed as a direct sum of proper submodules. The Krull-Schmidt theorem guarantees that every finite dimensional $A$-module can be expressed as a direct sum of indecomposables and that the number of indecomposable summands isomorphic to a given indecomposable module is independent of the decomposition.  In particular, the regular $A$-module can be expressed as
\begin{equation}\label{eq:dec.proj}
A=P_1\oplus\cdots\oplus P_r
\end{equation}
 where the $P_i$ are projective indecomposable modules and, moreover, each projective indecomposable $A$-module is isomorphic to some $P_i$.  Furthermore, we can find a complete set of orthogonal primitive idempotents $e_1,\ldots, e_r$ such that $P_i=Ae_i$ for $i=1,\ldots, r$ and all complete sets of orthogonal primitive idempotents arise in this way.  Moreover, any two complete sets of orthogonal primitive idempotents are conjugate by an element of the group of units of the algebra.  Recall that a non-zero idempotent $e\in A$ is \emph{primitive}\index{primitive} if $e=e_1+e_2$ with $e_1,e_2$ orthogonal idempotents implies $e_1=0$ or $e_2=0$.

\nomenclature[R, 02]{$\rad(M)$}{radical of the module $M$}%
The \emph{radical}\index{radical} $\rad(M)$ of a finite dimensional  $A$-module $M$ is the intersection of all maximal submodules of $M$.  One has the equality \[\rad(M)=\rad(A)M.\]  The quotient $M/\rad(M)$ is semisimple and is called the \emph{top}\index{top} of $M$.  Recall that a module is \emph{semisimple}\index{semisimple} if it is a direct sum of simple modules.  An important fact about finite dimensional algebras is that if $P$ is a projective indecomposable module, then $P/\rad(P)$ is simple.  Moreover, if $P$ and $Q$ are projective indecomposable modules, then $P\cong Q$ if and only if $P/\rad(P)\cong Q/\rad(Q)$.  In addition, if \eqref{eq:dec.proj} is a decomposition of $A$ into projective indecomposables, then \[A/\rad(A)\cong P_1/\rad(P_1)\oplus\cdots \oplus P_r/\rad(P_r)\] and every simple $A$-module is isomorphic to one of the form $P_i/\rad(P_i)$.  Consequently, there is a bijection between isomorphism classes of projective indecomposable modules and simple modules.

 We recall that a surjective homomorphism $\p\colon P\to M$ from a finite dimensional projective $A$-module $P$ to a module $M$ is a \emph{projective cover}\index{projective cover} if $\ker \p\subseteq \rad(P)$.  An injective homomorphism $\p\colon M\to I$ where $I$ is a finite dimensional injective module is called an \emph{injective envelope}\index{injective!envelope} (or \emph{injective hull}\index{injective!hull}) if every non-zero submodule of $I$ intersects $\p(M)$ non-trivially. Each projective indecomposable module is the projective cover of its simple top and each injective indecomposable module is the injective envelope of its simple socle.  Recall that the \emph{socle}\index{socle} of a module is its largest semisimple submodule.

Let $e_1,\ldots, e_r$ be a complete set of orthogonal primitive idempotents.  Then each $A$-module $I_j=\Hom_{\Bbbk}(e_jA,\Bbbk)$ is an injective indecomposable module with simple socle isomorphic to $(A/\rad{A})e_j$.  Moreover, $I_j\cong I_k$ if and only if $(A/\rad{A})e_j\cong (A/\rad{A})e_k$ and every injective indecomposable $A$-module is isomorphic to one of the form $I_j$ with $1\leq j\leq r$.  Thus isomorphism classes of injective indecomposable modules and simple modules are also in bijection.

\nomenclature[R, 03]{$\gldim A$}{global dimension of the algebra $A$}%
A $\Bbbk$-algebra $A$ is said to be \emph{hereditary}\index{hereditary} if each of its left ideals is a projective module or, equivalently, submodules of projective modules are projective.  Alternatively, hereditary algebras are algebras of global dimension at most one. Recall that the \emph{projective dimension}\index{projective dimension} of a module $M$ is the shortest length of a projection resolution \[P_\bullet\to M = \cdots\to P_q\to P_{q-1}\to\cdots\to P_1\to P_0\to M.\] The \emph{global dimension}\index{global dimension} $\gldim A$ of $A$ is the supremum of the projective dimensions of all $A$-modules.  When $A$ is finite dimensional, $\gldim A$ is the maximum projective dimension of a simple $A$-module~\cite[Theorem~A.4.8]{assem}.  Alternatively, $\gldim A$ is the largest $n$ such that $\Ext_A^n(S,S')\neq 0$ for some simple $A$-modules $S,S'$ (when $\dim A<\infty$), cf.~\cite[Corollary~16.2]{repbook}.

\subsection{Quivers and basic algebras}
\nomenclature[Q, 01]{$Q$}{quiver}%
\nomenclature[Q, 02]{$Q^{op}$}{opposite quiver of the quiver $Q$}%
A \emph{quiver}\index{quiver} $Q$ is a directed graph (possibly with loops and multiple edges), always assumed finite here. The vertex set is denoted $Q_0$ and the edge set $Q_1$. More generally, $Q_n$ will denote the set of (directed) paths of length $n$ with $n\geq 0$ (where the distinction between empty paths and vertices is blurred). Denote by $Q^{op}$ the quiver obtained from $Q$ by reversing the orientation of each arrow.
A quiver is \emph{acyclic}\index{acyclic}\index{quiver!acyclic} if it has no directed cycles.

\nomenclature[Q, 03]{$\Bbbk Q$}{path algebra of the quiver $Q$ with coefficients in $\Bbbk$}%
\nomenclature[Q, 04]{$\varepsilon_v$}{trivial path at the vertex $v$ of a quiver}%
The \emph{path algebra}\index{path algebra} $\Bbbk Q$ of $Q$ has a basis consisting of all paths in $Q$ (including an empty path $\varepsilon_v$ at each vertex) with the product induced by concatenation (where undefined concatenations are set equal to zero).  Here we compose paths from right to left like category theorists: so if $p\colon v\to w$ and $q\colon w\to z$ are paths, then there composition is denoted $qp$.  Note that $\Bbbk Q=\bigoplus_{n\geq 0} \Bbbk Q_n$. The identity of $\Bbbk Q$ is the sum of all empty paths.

For example, if $Q$ consists of a single vertex with $n$ loops, then $\Bbbk Q$ is the free algebra on $n$ generators.  If $Q$ is a directed path with $n$ nodes, then $\Bbbk Q$ is isomorphic to the algebra of $n\times n$ upper triangular matrices over $\Bbbk$.

\nomenclature[R, 04]{$Z(A)$}{center of the algebra $A$}%
If $A$ is a finite dimensional algebra, then the primitive idempotents of the center $Z(A)$ are called \emph{central primitive idempotents}\index{central primitive idempotents}.
The central primitive idempotents of $\Bbbk Q$ are precisely the sums of empty paths ranging over a connected component of $Q$.  If $f$ is a central idempotent, then $Af=fAf$ is a unital $\Bbbk$-algebra with identity $f$. There is a unique complete set of orthogonal central primitive idempotents $f_1,\ldots, f_n$ and, moreover, \[A\cong Af_1\times \cdots \times Af_n\] as a $\Bbbk$-algebra.  The algebras $Af_j$ are called the \emph{blocks}\index{blocks} of $A$.

\nomenclature[Q, 05]{$(Q, I)$}{bound quiver, where $I$ is an admissible ideal of $\Bbbk Q$}%
An important theorem of Gabriel is that a split basic algebra is hereditary if and only if it is isomorphic to $\Bbbk Q$ for some acyclic quiver $Q$. More generally, Gabriel showed that every split basic algebra has a quiver presentation.  A \emph{bound quiver}\index{bound quiver} $(Q,I)$ consists of a quiver $Q$ and an admissible ideal $I$ of $\Bbbk Q$.  To define an admissible ideal, let $J=\bigoplus_{n\geq 1}\Bbbk Q_n$ be the \emph{arrow ideal}\index{arrow ideal} of $\Bbbk Q$, that is, the ideal generated by all paths of length one.  Then an ideal $I$ is \emph{admissible}\index{admissible} if $J^n\subseteq I\subseteq J^2$ for some $n\geq 2$. The algebra $A=\Bbbk Q/I$ will be a finite dimensional split basic $\Bbbk$-algebra.   Note that $\rad(A) = J/I$ and the cosets $\varepsilon_v+I$ with $v\in Q_0$ form a complete set of orthogonal primitive idempotents. In particular, if $Q$ is acyclic, then $\{0\}$ is admissible,  $\rad(\Bbbk Q)=J$ and the empty paths $\varepsilon_v$ with $v\in Q_0$ form a complete set of orthogonal primitive idempotents of $\Bbbk Q$.  Indeed, if $Q$ is acyclic with $n$ vertices, then no path in $Q$ has length greater than $n-1$ and so $J^n=0$.

Conversely, every finite dimensional split basic algebra $A$ is the algebra of a bound quiver for a unique quiver $Q(A)$ (the admissible ideal $I$ is not unique).  The vertices of $Q(A)$ are the (isoclasses of) simple $A$-modules. If $S,S'$ are simple $A$-modules, then the number of directed edges $S\to S'$ in $Q(A)$ is $\dim_\Bbbk \Ext^1_A(S,S')$.

To explain how the admissible ideal $I$ is obtained, fix a complete set $E$ of orthogonal primitive idempotents for $A$ and let $e_S\in E$ be the idempotent with $S\cong (A/\rad(A))e_S$. One has that \[\Ext^1_A(S,S')\cong e_{S'}[\rad(A)/\rad^2(A)]e_S\] as $\Bbbk$-vector spaces.  See~\cite[Proposition~2.4.3]{benson}. A quick sketch is as follows.  Since the projection $Ae_S\to S$ induces an isomorphism $\Hom_A(S,S')\to \Hom_A(Ae_S,S')$ (because $S'$ is simple and $Ae_S$ has simple top $S$), the exact sequence
\[0\to \rad(A)e_S\to Ae_S\to S\to 0\] implies, via the long exact sequence for $\Ext$, that
\begin{align*}
\Ext^1_A(S,S') &\cong \Hom_A(\rad(A)e_S,S')\\ &\cong \Hom_{A/\rad(A)}([\rad(A)/\rad^2(A)]e_S,S')\\ &\cong \Hom_{A/\rad(A)}(S',[\rad(A)/\rad^2(A)]e_S)\\ &\cong \Hom_A(Ae_{S'},[\rad(A)/\rad^2(A)]e_S)\\ &\cong  e_{S'}[\rad(A)/\rad^2(A)]e_S.
\end{align*}

One can define a homomorphism $\Bbbk Q(A)\to A$  by sending the empty path at $S$ to $e_S$ and by choosing a bijection of the set of  arrows $S\to S'$ with a subset of $e_{S'}\rad(A)e_S$ mapping to a basis of $e_{S'}[\rad(A)/\rad^2(A)]e_S$. The map is extended to paths of length greater than one in the obvious way.  Such a homomorphism is automatically surjective and the kernel will be an admissible ideal.  See~\cite[Proposition~4.1.7]{benson} for details (where the result is stated for algebraically closed fields but just uses that the algebra is split).  The main idea is that because $A$ is split basic, the homomorphism we have just defined is surjective modulo $\rad^2(A)$ and hence surjective by~\cite[Proposition~1.2.8]{benson}; the kernel is admissible by construction.

We remark that if there is no path of length greater than one from $S$ to $S'$ in $Q(A)$, then $e_{S'}\rad^2(A)e_S=0$ because $\rad(A)=J/I$. Similarly, if there is no path at all from $S$ to $S'$ through $S''$, then \[e_{S'}\rad(A)e_{S''}\rad(A)e_S=e_{S'}(J/I)e_{S''}(J/I)e_S=0.\]  Finally, we remark that if $S\neq S'$, then $e_{S'}Ae_S=e_{S'}\rad(A)e_S$, cf.~\cite[Lemma~2.12]{assem} or~\cite[Proposition~2.4.3]{benson}.

\nomenclature[Q, 06]{$(Q, R)$}{quiver presentation, where $R$ is a system of relations}%
A \emph{system of relations}\index{system of relations} for an admissible ideal $I$ is a finite subset \[R\subseteq \bigcup_{v,w\in Q_0} \varepsilon_wI\varepsilon_v\] generating $I$ as an ideal. Every admissible ideal admits a system of relations $R$ and the pair $(Q,R)$ is called a \emph{quiver presentation}\index{quiver!presentation} of $A$.  A system of relations $R$ is called \emph{minimal}\index{minimal} if no proper  subset of $R$ generates the ideal $I$.  The following result, giving a homological interpretation of the number of elements of a minimal generating set, encapsulates~\cite[1.1 and 1.2]{Bongartz}.  We include a proof for completeness both because it will play an essential role in this text and to make things accessible for the diverse audience we are targeting (the exposition in the original reference is aimed at specialists in finite dimensional algebras).

\begin{Thm}[Bongartz]\label{t:bongartz}
Let $(Q,I)$ be a bound quiver with $Q$ acyclic and let $A=\Bbbk Q/I$.  For $v\in Q_0$, let $\Bbbk_v$ be the simple $A$-module $(A/\rad(A))\varepsilon_v$.  Then, for any minimal system of relations $R$, one has that the cardinality of $R\cap \varepsilon_wI\varepsilon_v$ is $\dim_\Bbbk\Ext^2_A(\Bbbk_v,\Bbbk_w)$.
\end{Thm}
\begin{proof}
For $v,w\in Q_0$, let $R(v,w)=R\cap \varepsilon_wI\varepsilon_v$.  Let $J$ be the arrow ideal of $\Bbbk Q$. We claim that
\begin{equation}\label{eq:relation.dim}
|R(v,w)|=\dim_{\Bbbk} \varepsilon_w\left[I/(JI+IJ)\right]\varepsilon_v.
\end{equation}
First we prove that the cosets of the elements of $R(v,w)$ are linearly independent in $\varepsilon_w\left[I/(JI+IJ)\right]\varepsilon_v$.  Suppose that this is not the case. Then there exists $\alpha_0\in R(v,w)$ such
\begin{equation}\label{eq:fun.equation}
\alpha_0\in \sum_{\alpha\in R(v,w)\setminus \{\alpha_0\}}\Bbbk \alpha+JI+IJ.
\end{equation}
 Let $K$ be the ideal generated by $R\setminus \{\alpha_0\}$ and let $I_0$ be the ideal generated by $\alpha_0$.  Then we deduce from \eqref{eq:fun.equation} that $\alpha_0\in K+JI_0+I_0J$ and so $I_0\subseteq K+JI_0+I_0J$.  We claim that, for all $n\geq 1$,
 \[I_0\subseteq K+\sum_{i=0}^nJ^{n-i}I_0J^{i}\]
  where $J^0$ is interpreted as $\Bbbk Q$.  The base case $n=1$ is already handled.  Assume it is true for $n$.  Then we have
 \begin{align*}
 I_0&\subseteq K+JI_0+I_0J\subseteq K+JK+J\sum_{i=0}^nJ^{n-i}I_0J^{i}+KJ+\sum_{i=0}^nJ^{n-i}I_0J^{i}J\\ &\subseteq K+\sum_{i=0}^{n+1}J^{n+1-i}I_0J^{i}.
 \end{align*}
As $J$ is nilpotent (since $Q$ is a acyclic) and $J^{n-i}I_0J^{i}\subseteq J^n$, we conclude that $I_0\subseteq K$.  This contradicts the minimality of $R$.  Therefore, the cosets of the elements of $R(v,w)$ are linearly independent in $\varepsilon_w\left[I/(JI+IJ)\right]\varepsilon_v$.   Next we show that they span this $\Bbbk$-vector space.

Let $r\in \varepsilon_wI\varepsilon_v$.  Then \[r=\sum_{\alpha\in R}\sum_{i=1}^{n_{\alpha}} c_{\alpha,i}u_{\alpha,i}\alpha v_{\alpha,i}\] with the $c_{\alpha,i}\in \Bbbk$ and the $u_{\alpha,i}$ and $v_{\alpha,i}$ paths.  If $u_{\alpha,i}$ or $v_{\alpha,i}$ is non-empty, then $u_{\alpha,i}\alpha v_{\alpha,i}\in \varepsilon_w(JI+IJ)\varepsilon_v$.  It follows that \[r+\varepsilon_w(JI+IJ)\varepsilon_v=\sum_{\alpha\in R(v,w)}c_{\alpha}\alpha+\varepsilon_w(JI+IJ)\varepsilon_v\] for some $c_{\alpha}\in \Bbbk$.  This completes the proof of \eqref{eq:relation.dim}.

It remains to prove that \[\dim_\Bbbk\Ext^2_A(\Bbbk_v,\Bbbk_w)=\dim_\Bbbk \varepsilon_w\left[I/(JI+IJ)\right]\varepsilon_v.\]  Consider the exact sequence of $\Bbbk Q/I$-modules
\[\xymatrix{0\ar[r] & (I/IJ)\varepsilon_v\ar[r]^h & (J/IJ)\varepsilon_v\ar[r]^j & (\Bbbk Q/I)\varepsilon_v\ar[r]^k & \Bbbk_v\ar[r] & 0}\] where $h(x+IJ)=x+IJ$, $j(x+IJ)=x+I$ and $k(x+I)=x+J/I$.
We write $\source(e)$ for the source of an edge $e\in Q_1$ and $\target(e)$ for the target.  Then there is an isomorphism of $\Bbbk Q$-modules \[\bigoplus_{\{e\in Q_1\mid \source(e)=v\}}\Bbbk Q\varepsilon_{\target(e)}\to J\varepsilon_v\] that takes $p\in \Bbbk Q\varepsilon_{\target(e)}$ to $pe$.  Thus \[(J/IJ)\varepsilon_v\cong \Bbbk Q/I\otimes_{\Bbbk Q} J\varepsilon_v\cong \bigoplus_{\{e\in Q_1\mid \source(e)=v\}}(\Bbbk Q/I)\varepsilon_{\target(e)}\] is a projective $\Bbbk Q/I$-module.

Let $M\cong (J/I)\varepsilon_v$ be the cokernel of $h$. Then we have exact sequences
\begin{gather*}
\xymatrix{0\ar[r] & (I/IJ)\varepsilon_v\ar[r]^h & (J/IJ)\varepsilon_v\ar[r] &M\ar[r] & 0}\\
\xymatrix{0\ar[r] & M\ar[r] & (\Bbbk Q/I)\varepsilon_v\ar[r] & \Bbbk_v\ar[r] & 0}
\end{gather*}
of $\Bbbk Q/I$-modules with the middle terms projective modules.    The long exact sequence for the $\Ext$-functor then yields that
\[\Ext^2_{\Bbbk Q/I}(\Bbbk_v,\Bbbk_w)\cong \Ext^1_{\Bbbk Q/I}(M,\Bbbk_w)\cong \mathop{\mathrm{coker}} h^*\] where $h^*\colon \Hom_{\Bbbk Q/I}((J/IJ)\varepsilon_v,\Bbbk_w)\to \Hom_{\Bbbk Q/I}((I/IJ)\varepsilon_v,\Bbbk_w)$ is the homomorphism induced by $h$.  But if $f\colon (J/IJ)\varepsilon_v\to \Bbbk_w$ is a homomorphism, then because $I$ is admissible, and hence $I\subseteq J^2$, we conclude that $fh((I/IJ)\varepsilon_v)\subseteq f((J^2/IJ)\varepsilon_v)\subseteq (J/I)f((J/IJ)\varepsilon_v)\subseteq (J/I)\Bbbk_w=0$ as $J/I=\rad(\Bbbk Q/I)$.   We conclude that $h^*=0$ and so \[\mathop{\mathrm{coker}} h^*\cong \Hom_{\Bbbk Q/I}((I/IJ)\varepsilon_v,\Bbbk_w)\cong \Hom_{\Bbbk Q/J}([I/(JI+IJ)]\varepsilon_v,\Bbbk_w).\]  As $\Bbbk Q/J$ is split semisimple, it follows that
$\Hom_{\Bbbk Q/J}([I/(JI+IJ)]\varepsilon_v,\Bbbk_w)\cong \Hom_{\Bbbk Q/J}(\Bbbk_w,[I/(JI+IJ)]\varepsilon_v)\cong \varepsilon_w\left[I/(JI+IJ)\right]\varepsilon_v$.  We conclude that $\Ext^2_{A}(\Bbbk_v,\Bbbk_w)\cong  \varepsilon_w\left[I/(JI+IJ)\right]\varepsilon_v$, completing the proof.
\end{proof}

Obtaining a quiver presentation of a finite dimensional algebra is a crucial step in order to be able to apply the apparatus developed by modern representation theorists.

In~\cite[Theorem~3.10]{oldpaper}, the authors showed that for any acyclic quiver $Q$, there is a left regular band monoid $B_Q$ with $\Bbbk B_Q\cong \Bbbk Q$ for any field $\Bbbk$; in other words, the representation theory of acyclic quivers (i.e., split (basic) hereditary algebras) is a special case of the representation theory of left regular bands!

\subsection{Gradings, quadratic algebras and Koszul algebras}
In this subsection we recall the notions of quadratic and Koszul algebras.  Motivation for these notions can be found in~\cite{beilinson,MartinezVillaIntroKoszul,GreenVilla}.
We continue to hold fixed a field $\Bbbk$ and assume that all $\Bbbk$-algebras are unital.
A $\Bbbk$-algebra $A$ is said to be \emph{graded}\index{graded!algebra}\index{algebra!graded} if it has a vector space decomposition $A=\bigoplus_{n\geq 0} A_n$ such that $A_iA_j\subseteq A_{i+j}$. An ideal $I$ of $A$ is \emph{homogeneous}\index{homogeneous ideal}\index{ideal!homogeneous} if it is generated as an ideal by the elements of $I\cap A_i$, for $i\geq 0$.  In this case, $A/I=\bigoplus_{n\geq 0} (A_n+I)/I$ is a grading.
Note that $I=\bigoplus_{n\geq 1} A_i$ is a homogeneous ideal and $A/I\cong A_0$ with the grading trivial (i.e., concentrated in degree $0$).  In particular, we can view $A_0$ as an $A$-module. A \emph{graded $A$-module}\index{graded!module}\index{module!graded} is an $A$-module $M=\bigoplus_{n\geq 0} M_n$ such that $A_iM_j\subseteq M_{i+j}$ for all $i,j$.  We say that $M$ is generated in degree $i$ if $M_j=0$ for $j<i$ and $M_j=A_{j-i}M_i$ for $j\geq i$. An $A$-module homomorphism $\p\colon M\to N$ between graded $A$-modules has \emph{degree}\index{degree} $d$ if $\p(M_i)\subseteq N_{i+d}$ for all $i\geq 0$.

The path algebra $\Bbbk Q=\bigoplus_{n\geq 0}\Bbbk Q_n$ of a quiver is naturally graded by path length.  That is, the homogeneous component of degree $n\geq 0$ is the $\Bbbk$-span of the set $Q_n$ of paths of length $n$.  The arrow ideal \[J=\bigoplus_{n\geq 1}\Bbbk Q_n\] is graded and $\Bbbk^{|Q_0|}\cong \Bbbk Q_0\cong \Bbbk Q/J$.
If $I$ is a homogeneous admissible ideal, then $A=\Bbbk Q/I$ is a graded finite dimensional $\Bbbk$-algebra.

A graded $\Bbbk$-algebra $A$ is \emph{quadratic}\index{quadratic!algebra}\index{algebra!quadratic} if $A_0$ is semisimple and $A$ is generated by $A_1$ over $A_0$ with relations of degree $2$. Rather than define this formally at this level of generality, we stick to the case of path algebras factored by homogeneous ideals.  If $I$ is a homogeneous admissible ideal of $\Bbbk Q$ that is generated as an ideal by elements in degree $2$ (i.e., by elements which are linear combinations of paths of length $2$), then $A=\Bbbk Q/I$ is called a \emph{quadratic algebra}\index{quadratic!algebra}\index{algebra!quadratic}.  We now wish to define, restricted to this setting, the quadratic dual.  Our presentation follows that of~\cite{MartinezVillaIntroKoszul} and~\cite{GreenVilla}. Define a non-degenerate bilinear form \[\langle \cdot,\cdot\rangle \colon \Bbbk Q_2\times \Bbbk Q^{op}_2\to \Bbbk\] by putting, for paths $p,q$ in $Q_2$ and $Q_2^{op}$, respectively, \[\langle p,q\rangle = \begin{cases} 1, & \text{if}\ q=p^{\circ}\\ 0, & \text{else}\end{cases}\] where $p^{\circ}$ is the path in $Q^{op}$ obtained from $p$ by reversing all the edges and their order. Let $I_2=I\cap \Bbbk Q_2$ and let $I_2^\perp\subseteq \Bbbk Q^{op}_2$ be the orthogonal complement of $I_2$.  Then $A^!=\Bbbk Q^{op}/\langle I_2^{\perp}\rangle$ is called the \emph{quadratic dual}\index{quadratic!dual}\index{dual!quadratic} of $A$.

Suppose that $A$ is a graded $\Bbbk$-algebra.  Then a graded $A$-module $M$ has a \emph{linear resolution}\index{linear resolution}\index{resolution!linear} if it is generated in degree $0$ and there is a projective resolution $P_\bullet\to M$ such that each $P_i$ is a graded $A$-module generated in degree $i$ for all $i\geq 0$ and the maps in the resolution all have degree $0$. One says that $A$ is a \emph{Koszul algebra}\index{Koszul!algebra}\index{algebra!Koszul} if $A_0$ is a semisimple $\Bbbk$-algebra and $A_0$, considered as a graded $A$-module concentrated in degree $0$, has a linear resolution.  In the case $A=\Bbbk Q/I$ for a homogeneous admissible ideal, one has that $A_0\cong \Bbbk^{|Q_0|}$ is a semisimple algebra.  Also observe that in this case $A_0$ admits a linear resolution if and only if each simple $A$-module (viewed as a graded module concentrated in degree $0$) has a linear resolution.  A $\Bbbk$-algebra $A$ admits at most one grading (up to isomorphism of graded algebras) making it Koszul.

\nomenclature[R, 06]{$\Ext(A)$}{$\Ext$-algebra of the algebra $A$}%
Koszul algebras are always quadratic~\cite{beilinson} but the converse is false. The quadratic dual of a Koszul algebra is also called its \emph{Koszul dual}\index{Koszul!dual}\index{dual!Koszul}.  If $A$ is a graded $\Bbbk$-algebra, the \emph{$\Ext$-algebra}\index{$\Ext$-algebra} of $A$ is the graded $\Bbbk$-algebra \[\Ext(A)=\bigoplus_{n\geq 0}\Ext_A^n(A_0,A_0)\] with multiplication given by the Yoneda composition.  If $A$ is a Koszul algebra, then $A^!\cong \Ext(A)^{op}$ and $\Ext(\Ext(A))\cong A$. See~\cite{beilinson} for details.

\subsection{The algebra of a left regular band}\label{s:lrbalgebra}
A general reference for algebras of finite monoids is~\cite{repbook}.  Here we shall have to consider semigroup algebras.  See~\cite[Chapter~5]{CP} or~\cite{oknisemigroupalgebra} for semigroup algebras.

Fix a (finite) left regular band $B$ and a commutative ring with unit $\Bbbk$ for this subsection. The semigroup algebra $\Bbbk B$ is the free $\Bbbk$-module with basis $B$ and multiplication extending that of $B$.  Formally, the product is given by
\[\sum_{b\in B}c_bb\cdot \sum_{b\in B}d_bb= \sum_{b,b'\in B}c_bd_{b'}bb'.\]  Note that $\Bbbk B$ is not in general unital.  The most elementary example comes from \emph{left zero semigroups}\index{left zero semigroup}. If $X$ is a finite set with at least two elements and we make $X$ into a left regular band via the \emph{left zero multiplication}\index{left zero multiplication}, that is, by setting $xy=x$ for all $x\in X$, then $\Bbbk X$ is not unital. Indeed, for each $x\in X$, one has that  \[\sum_{y\in X}c_yy\cdot x= \sum_{y\in X}c_yy\] and hence each element of $X$ is a right identity for $\Bbbk X$.  A unital ring cannot have multiple right identities.  It turns out, more generally, that $\Bbbk B$ does have a right identity (not necessarily unique) for each left regular band $B$.  However, we shall see that $\Bbbk B$ can be unital even when $B$ is not a monoid. In particular, we shall characterize when $\Bbbk B$ is unital in terms of the poset structure on $B$.

The \emph{trivial module}\index{trivial module} for $\Bbbk B$ is the module $\Bbbk$ with trivial action $bk=k$ for all $b\in B$ and $k\in \Bbbk$.  Trivial modules play an important role in monoid cohomology.

We first recall some basic results on the structure of the semigroup algebra $\Bbbk B$.  These results can be found in~\cite{Saliola} for the case when $\Bbbk$ is a field and in~\cite{oldpaper} for general commutative rings under the assumption that $B$ is a monoid.  For completeness we reprove the results here for arbitrary finite left regular bands (not necessarily monoids).

\subsubsection{Orthogonal idempotents}
\label{sssec:orthogonal-idempotents}
Let us recall the construction of a set of orthogonal idempotents from~\cite{Saliola}. They from a complete set of orthogonal idempotents when $B$ is a left regular band monoid; in general, they just sum up to a right identity.   When $\Bbbk$ is a field and $\Bbbk B$ is unital, they are primitive idempotents.
Fix, for each $X\in \Lambda(B)$, an element $e_X$ with $X=Be_X$.  Denoting the minimum element of $\Lambda(B)$ by $\wh 0$,  define $\eta_X$
recursively by $\eta_{\wh 0} = e_{\wh 0}$ and, for $X>\wh 0$,
\begin{equation}\label{defidempotents}
\eta_X =e_X-\sum_{Y<X}e_X\eta_Y.
\end{equation}
Notice that, by induction, one can write
\begin{align*}
    \eta_X=\sum_{b\in B} c_bb
\end{align*}
with the $c_b$ integers such that $e_X\geq b$ for all $b$ with $c_b\neq 0$ and
the coefficient of $e_X$ in $\eta_X$ is $1$.  Also note that by construction
\begin{equation}\label{eq:eXformula}
e_X=\sum_{Y\leq X}e_X\eta_Y.
\end{equation}

The following results are proved in Lemma~4.1 and Theorem~4.2 of~\cite{Saliola}
when $\Bbbk$ is a field and $B$ is a monoid.

\begin{Thm}\label{primidempotentprops}
The elements $\{\eta_X\}_{X\in \Lambda(B)}$ enjoy the following properties.
\begin{enumerate}
\item If $a\in B$ and $X\in \Lambda(B)$ are such that $a\notin B_{\geq X}$, then
    $a\eta_X=0$.
\item One has that $\{\eta_X\mid X\in \Lambda(B)\}$ is a set of orthogonal idempotents and that
    \begin{align*}
        \eta=\sum_{X\in \Lambda(B)} \eta_X
    \end{align*}
is a right identity for $\Bbbk B$. In particular, if $\Bbbk B$ is unital (for instance, if $B$ is a monoid), then $\eta=1$ and $\{\eta_X\}_{X\in \Lambda(B)}$ is a complete set of orthogonal idempotents.
\item If $\Bbbk$ is a field and $\Bbbk B$ is unital, then each $\eta_X$, with $X\in \Lambda(B)$, is a primitive idempotent.
\end{enumerate}
\end{Thm}
\begin{proof}
We prove the first item by induction on $X$.  If $X=\wh 0$, then there is nothing to prove.  So assume that it is true for $Y<X$ and we prove it for $X$.  Let $a\in B\setminus B_{\geq X}$ and put $Z=X\wedge Ba=Be_X\cap Ba=Bae_X$.  Then we have that
\begin{align*}
a\eta_X &= ae_X-\sum_{Y<X}ae_X\eta_Y\\ &= ae_X-\sum_{Y\leq Z}ae_X\eta_Y-\sum_{Y<X, Y\nleq Ba}ae_Xa\eta_Y\\
&= ae_X-ae_X\sum_{Y\leq Z}\eta_Y
\end{align*}
where the last equality uses the induction hypothesis.  As $Bae_X=Z=Be_Z$, we conclude that $ae_X=ae_Xe_Z$ and so
\[a\eta_X=ae_X-ae_X\sum_{Y\leq Z}\eta_Y = ae_Xe_Z-ae_Xe_Z\sum_{Y\leq Z}\eta_Y = ae_Xe_Z-ae_Xe_Z=0\] because, by \eqref{eq:eXformula},
\[e_Z\sum_{Y\leq Z}\eta_Y=e_Z.\]  This completes the proof of the first item.

For the second item, we first verify that $\eta_X$ is idempotent.  Indeed,
\[\eta_X^2=e_X\eta_X-\sum_{Y<X}e_X\eta_Y\eta_X=e_X\eta_X=\eta_X\] by the first item because each element $a\in B$ with non-zero coefficient in $e_X\eta_Y$ for $Y<X$ satisfies $Ba\nsupseteq X$.

Next we check that $\eta_X\eta_Z=0$ if $X\neq Z$.  First suppose that $Z\nleq X$.  As every element of $B$ with a non-zero coefficient in $\eta_X$ belongs to $e_XB\subseteq Be_X$, we conclude that $\eta_X\eta_Z=0$ by the first item.  Thus we are reduced to the case that $Z<X$.  We proceed by induction.  Assume that $\eta_Y\eta_Z=0$ for all $Y$ with $Z<Y<X$.  Then we have that
\[\eta_X\eta_Z= e_X\eta_Z-\sum_{Y<X}e_X\eta_Y\eta_Z=e_X\eta_Z-e_X\eta_Z=0\] because $\eta_Z$ is idempotent, $\eta_Y\eta_Z=0$ if $Z\nleq Y$ by the first case and $\eta_Y\eta_Z=0$ if $Z<Y<X$ by induction.  This completes the proof that $\{\eta_X\mid X\in \Lambda(B)\}$ is an orthogonal set of idempotents.  It remains to verify that \[\eta=\sum_{X\in \Lambda(B)}\eta_X\] is a right identity for $\Bbbk B$.
Indeed, by the first item, if $b\in B$ and $X=Bb$, then $be_X=b$ and
\[b\eta=\sum_{Y\leq X}b\eta_Y=be_X\sum_{Y\leq X}\eta_Y=be_X=b\] by \eqref{eq:eXformula}, as required.

The final item will be proved in Theorem~\ref{Schutz} below.
\end{proof}

\begin{Cor}\label{c:regularisproj}
If $B$ is a left regular band, then $\Bbbk B$ is a projective $\Bbbk B$-module as are its direct summands.
\end{Cor}
\begin{proof}
We have $\Bbbk B=\Bbbk B\eta$, with $\eta$ as in Theorem~\ref{primidempotentprops}, and so $\Hom_{\Bbbk B}(\Bbbk B,M)\cong \eta M$ for any left $\Bbbk B$-module $M$. As the functor $M\mapsto \eta M$ is obviously exact, this completes the proof.
\end{proof}

The following is~\cite[Corollary~4.4]{Saliola} for fields and is in~\cite{oldpaper} for the general case, but with $B$ a monoid.

\begin{Cor}\label{basisofidempotents}
The set $\{b\eta_{\sigma(b)}\mid b\in B\}$ is a basis of idempotents for $\Bbbk B$.
\end{Cor}
\begin{proof}
Clearly, $b\eta_{\sigma(b)}b\eta_{\sigma(b)}=b\eta_{\sigma(b)}^2=b\eta_{\sigma(b)}$ by the left regular band axiom and idempotence of $\eta_{\sigma(b)}$.  As $be_{\sigma(b)}=b$, it follows that \[b\eta_{\sigma(b)}=b+\sum_{a\in \bd{bB}}c_aa\] and so the mapping $\Bbbk B\to \Bbbk B$ induced by $b\mapsto b\eta_{\sigma(b)}$ has matrix which is unipotent upper triangular with respect to any linear extension of the natural partial order on $B$.  The result follows.
\end{proof}

\subsubsection{Sch\"utzenberger representations}
\label{ss:Schutz}
\nomenclature[L, 15]{$L_X$}{left Sch\"utzenberger representation associated with an element $X \in \Lambda(B)$}%
We recall the classical (left) \emph{Sch\"utz\-en\-ber\-ger representation}\index{Sch\"utzenberger representation}\index{representation!Sch\"utzenberger} associated
to an element $X\in \Lambda(B)$.  We retain the notation of the previous subsection.  Let $L_X=\sigma\inv (X)$.  It is the
$\mathscr L$-class of $e_X$ in the sense of semigroup theory~\cite{CP,repbook} (see~\cite[Appendix~A]{qtheor}), that is, it consists of all elements that generate the same principal left ideal of $B$ as $e_X$.  Define a
$\Bbbk B$-module structure on $\Bbbk L_X$ by putting, for $a\in B$ and $b\in L_X$,
\begin{align*}
    a\cdot b =
    \begin{cases}
        ab & \text{if}\ \sigma(a)\geq X\\
        0 & \text{else.}
    \end{cases}
\end{align*}
Let us compute the endomorphism ring of $\Bbbk L_X$.

\begin{Prop}\label{p:end.schutz}
Let $\Bbbk$ be a commutative ring with unit and $B$ a left regular band. Let $X\in \Lambda(B)$.  Then $\mathrm{End}_{\Bbbk B}(\Bbbk L_X)\cong \Bbbk$.  Consequently, if $\Bbbk$ is a field, then $\Bbbk L_X$ is indecomposable.
\end{Prop}
\begin{proof}
Fix $e_X$ with $Be_X=X$, as usual.  Let $\lambda\colon \Bbbk L_X\to \Bbbk$ be the $\Bbbk$-module homomorphism defined by $\lambda(a)=1$ for $a\in L_X$ and
define \[\Psi\colon \mathrm{End}_{\Bbbk B}(\Bbbk L_X)\to \Bbbk\] by $\Psi(\p) = \lambda(\p(e_X))$.  We claim that $\p(a)=\Psi(\p)a$ for all $a\in L_X$.  The proposition will follow from this and the observation that $\Psi(c1_{\Bbbk L_X})=c$ for $c\in \Bbbk$.
Indeed, suppose that $\p(e_X) = \sum_{b\in L_X}c_bb$. If $a\in L_X$, then we have that $ab=a$ for all $b\in L_X$.  Therefore, we have that
\[\p(a)=\p(ae_X)=a\p(e_X)=\sum_{b\in L_X}c_bab=\left(\sum_{b\in L_X}c_b\right)a=\Psi(\p)a.\]  This establishes that $\mathrm{End}_{\Bbbk B}(\Bbbk L_X)\cong \Bbbk$.

The second statement follows from the first because a field has no non-trivial idempotents.
\end{proof}

The next theorem was proved in~\cite{Saliola} for $\Bbbk$ a field and over a general base commutative ring in~\cite{oldpaper} under the assumption that $B$ has an identity.

\begin{Thm}\label{Schutz}
Let $B$ be a left regular band and $\Bbbk$ a commutative ring with unit.  Then the
Sch\"utzenberger representation $\Bbbk L_X$ is isomorphic to $\Bbbk B\eta_X$, for $X\in \Lambda(X)$ via the mapping $b\mapsto b\eta_X$ and, consequently, is a projective module.  Moreover,
\begin{equation}\label{eq:direct.sum.schutz}
    \Bbbk B\cong \bigoplus_{X\in \Lambda(B)} \Bbbk L_X.
\end{equation}
If $\Bbbk$ is a field and $\Bbbk B$ is unital, then this is the decomposition of $\Bbbk B$ into projective
indecomposables and hence each $\eta_X$ is a primitive idempotent.
\end{Thm}
\begin{proof}
Fix $X\in \Lambda(B)$.  Define a $\Bbbk$-linear map $\p\colon \Bbbk L_X\to \Bbbk B\eta_X$ by $\p(b)=b\eta_X$ for $b\in L_X$.  The mapping $\p$ is injective as a consequence of Corollary~\ref{basisofidempotents}.  It is surjective because if $a\in B$, then $a\eta_X=0$ unless $Ba\supseteq X$ by Theorem~\ref{primidempotentprops}.  But if $X=Be_X$ and $Ba\supseteq X$, then $a\eta_X=ae_X\eta_X$ and $ae_X\in L_X$.  Thus $a\eta_X=\p(ae_X)$ and so $\p$ is surjective.  We check that $\p$ is $\Bbbk B$-linear. If $a\in L_X$ and $b\in B$ with $Bb\supseteq X$, then \[\p(b\cdot a) = (ba)\eta_X=b(a\eta_X)=b\p(a).\]  If $Bb\nsupseteq X$, then $ba\eta_X=0$ by Theorem~\ref{primidempotentprops} and so \[\p(b\cdot a)=\p(0)=0=b(a\eta_X)=b\p(a).\] This establishes that $\Bbbk B\eta_X\cong \Bbbk L_X$.

To prove \eqref{eq:direct.sum.schutz}, we recall that \[\eta=\sum_{X\in \Lambda(B)}\eta_X\] is a right identity and the $\eta_X$ with $X\in \Lambda(B)$ form a set of orthogonal idempotents.  Thus \eqref{eq:direct.sum.schutz} follows from
\[\Bbbk B=\Bbbk B\eta=\bigoplus_{X\in \Lambda(B)}\Bbbk B\eta_X\]
and the first part of the theorem.  The final statement is a consequence of Proposition~\ref{p:end.schutz} because $\Bbbk B\eta_X\cong \Bbbk L_X$ is indecomposable.
\end{proof}

If $X\in \Lambda(B)$, then there is a $\Bbbk$-algebra homomorphism $\rho_X\colon \Bbbk B\to
\Bbbk B_{\geq X}$ given by
\begin{align*}
    \rho_X(b) =
    \begin{cases}
        b & \text{if}\ \sigma(b)\geq X\\
        0 & \text{else}
    \end{cases}
\end{align*}
for $b\in B$.
This homomorphism allows us to consider any $\Bbbk B_{\geq X}$-module $M$ as a
$\Bbbk B$-module via the action $b \cdot m = \rho_X(b) m$ for all $b \in B$ and $m
\in M$.  Denote by $\Bbbk_{X}$ the trivial $\Bbbk B_{\geq X}$-module, viewed as a $\Bbbk B$-module.  It is known that if $\Bbbk$ is a field, then the $\Bbbk_X$ with $X\in \Lambda(B)$ form a complete set of non-isomorphic simple $\Bbbk B$-modules and the augmentation map $\lambda_X\colon \Bbbk L_X\to \Bbbk_X$ is the projective cover~\cite{Saliola}; in fact, this is an immediate consequence of Theorem~\ref{Schutz} because the $\Bbbk L_X$ are a complete set of projective indecomposable modules and hence have simple tops.  Let us state this as a corollary.

\begin{Cor}\label{c:simples}
Let $\Bbbk$ be a field and $B$ a left regular band such that $\Bbbk B$ is unital.  Then the simple $\Bbbk B$-modules $\Bbbk_X$ with $X\in \Lambda(B)$ form a complete set of representatives of the isomorphism classes of simple $\Bbbk B$-modules and the mapping $\p_X\colon \Bbbk L_X\to \Bbbk_X$ given by $\p_X(b)=1$ for $b\in L_X$ is the projective cover.
\end{Cor}

The fact that the $\Bbbk_X$ are the simple $\Bbbk B$-modules over a field can be deduced also from the classical representation theory of finite semigroups, cf.~\cite{myirreps}, \cite[Chapter~5]{CP} or~\cite[Chapter~5]{repbook}.

 If $Y\geq X$, then $\Bbbk L_Y$ is a
$\Bbbk B_{\geq X}$-module and is the projective indecomposable module for $\Bbbk B_{\geq
X}$ corresponding to $Y\in  \Lambda(B)_{\geq X}\cong \Lambda(B_{\geq X})$.  We, therefore, obtain the following corollary of Theorem~\ref{Schutz}.

\begin{Cor}\label{niceprojectives}
If $X\in \Lambda(B)$, then $\Bbbk B_{\geq X}$ is a projective $\Bbbk B$-module and the
decomposition
\begin{align*}
    \Bbbk B\cong \Bbbk B_{\geq X}\oplus \bigoplus_{Y\ngeq X}\Bbbk L_Y
\end{align*}
holds. Consequently, any projective $\Bbbk B_{\geq X}$-module is a projective
$\Bbbk B$-module (via $\rho_X$).
\end{Cor}

As a corollary, it follows that we can compute projective resolutions for $\Bbbk B_{\geq X}$-modules, and compute $\Ext$ between them, over either $\Bbbk B$ or
$\Bbbk B_{\geq X}$.

\begin{Cor}\label{computeExt}
Let $X\in \Lambda(B)$ and let $M,N$ be $\Bbbk B_{\geq X}$-modules.  Any projective resolution $P_\bullet\to M$ of $M$ as a $\Bbbk B_{\geq X}$-module is also a projective resolution as a $\Bbbk B$-module. Consequently,
\begin{align*}
    \Ext^n_{\Bbbk B}(M,N)\cong \Ext^n_{\Bbbk B_{\geq X}}(M,N)
\end{align*}
for all $n\geq 0$.
\end{Cor}

\subsubsection{The Jacobson radical}\label{ss:jacobson}
In this subsection we compute the Jacobson radical of the algebra of a left regular band over a semiprimitive commutative ring with unit $\Bbbk$. Munn characterized more generally the radical of a band over any base ring~\cite{MunnJacobson}.

\nomenclature[P, 30]{$I(P; \Bbbk)$}{incidence algebra of the finite poset $P$ with coefficients in $\Bbbk$}%
First we need Solomon's theorem describing the algebra of a finite semilattice~\cite{Burnsidealgebra}.  We include a proof for completeness.  The \emph{incidence algebra}\index{incidence algebra} $I(P;\Bbbk)$ of a finite poset $P$ is the $\Bbbk$-algebra with $\Bbbk$-basis consisting all pairs $(\sigma,\tau)$ with $\sigma\leq \tau$ and whose product on basis vectors is given by \[(\sigma,\tau)(\alpha,\beta)= \begin{cases}(\alpha,\tau), &\text{if}\ \beta=\sigma\\ 0, &\text{else.} \end{cases}\]  If $\Bbbk$ is field, then it is easy to see that $I(P;\Bbbk) \cong \Bbbk Q/\langle R\rangle$ where $Q$ is the Hasse diagram of $P$ and $R$ is the set of all differences $p-q$ of parallel paths in $Q$.

We can identify elements of $I(P;\Bbbk)$ with mappings $f\colon P\times P\to \Bbbk$ with $f(p,q)=0$ whenever $p\nleq q$.  The mapping $f$ corresponds to the element
\[\sum_{p\leq q} f(p,q)\cdot (p,q)\] of $I(P;\Bbbk)$.  The identity element $\sum_{p\in P}(p,p)$ of $I(P;\Bbbk)$  then corresponds to the Kronecker  function $\delta$. The product is given by \[(f\cdot g)(p,q) = \sum_{p\leq x\leq q}f(x,q)g(p,x)\] for $p\leq q$.   An element $f\in I(P;\Bbbk)$ is invertible if and only if $f(p,p)$ is a unit of $\Bbbk$ for all $p\in P$.  The inverse of $f$ is given recursively by
\begin{equation}\label{eq:inv.inc.alg}
\begin{split}
f^{-1}(p,p) &= f(p,p)^{-1}\\
f^{-1}(p,q) &= -f(q,q)^{-1}\sum_{p\leq r<q}f(r,q)f^{-1}(p,r),\quad \text{for}\ p<q.
\end{split}
\end{equation}
The \emph{zeta function}\index{zeta function} of $P$ is the element $\zeta\in I(P;\Bbbk)$ given by
\[\zeta(p,q) = \begin{cases} 1, & \text{if}\ p\leq q\\ 0, & \text{else.}\end{cases}.\] It is invertible over any base ring $\Bbbk$ and has inverse the \emph{M\"obius function}\index{M\"obius function} $\mu$ of $P$.  Normally, one considers the M\"obius function to be the inverse of $\zeta$ in $I(P;\mathbb Z)$ as the general case is then a specialization.
\nomenclature[P, 24]{$\mu$}{M\"obius function of a poset}%
Details on incidence algebras, zeta functions and M\"obius functions can be found in~\cite[Chapter~3]{Stanley}.

\begin{Thm}\label{t:solomon}
Let $\Bbbk$ be a commutative ring with unit and $\Lambda$ a finite (meet) semilattice.  Then $\Bbbk \Lambda\cong \Bbbk^{\Lambda}$ and, in particular, is unital.

Moreover, if \[\varepsilon_x=\sum_{y\leq x}\mu(y,x)y\] for $x\in \Lambda$, then the $\varepsilon_x$ form a complete set of orthogonal idempotents of $\Bbbk \Lambda$ and $\Bbbk\Lambda\varepsilon_x=\Bbbk \varepsilon_x$.  Furthermore, if $y\in \Lambda$ and $c\in \Bbbk$, then \[yc\varepsilon_x = \begin{cases}c\varepsilon_x, & \text{if}\ y\geq x\\ 0, & \text{else}\end{cases}\] holds.  In particular, if $\Bbbk$ is a field, then the $\varepsilon_x$ form a complete set of orthogonal primitive idempotents.
\end{Thm}
\begin{proof}
For $y\in \Lambda$, put \[\delta_y(z) = \begin{cases} 1, & \text{if}\ z=y \\ 0, & \text{else.}\end{cases}\] and note that the $\delta_y$ with $y\in \Lambda$ form a $\Bbbk$-basis for $\Bbbk^{\Lambda}$.

Define $\p\colon \Bbbk \Lambda\to \Bbbk^{\Lambda}$ and $\psi\colon \Bbbk^{\Lambda}\to \Bbbk \Lambda$ by
\begin{align*}
\p(x) &= \sum_{z\leq x}\delta_z\\
\psi(f) &= \sum_{x\in \Lambda}\sum_{z\leq x}f(x)\mu(z,x)z
\end{align*}
for $x\in \Lambda$.  In particular, \[\psi(\delta_x) = \sum_{z\leq x}\mu(z,x)z.\]

First we verify that $\p$ is a homomorphism.  It suffices to check that $\p(x)\p(y)=\p(xy)$ for $x,y\in \Lambda$.  Indeed, we have
\[\p(x)\p(y) = \sum_{z\leq x}\delta_z\cdot \sum_{u\leq y}\delta_u = \sum_{z\leq x,y}\delta_z=\sum_{z\leq xy}\delta_z=\p(xy).\]  Next we establish that $\p$ and $\psi$ are inverse mappings. For $x\in \Lambda$, we have
\begin{align*}
\psi(\p(x)) &= \sum_{y\leq x}\sum_{z\leq y}\mu(z,y)z=\sum_{z\leq x}\left(\sum_{z\leq y\leq x}\zeta(y,x)\mu(z,y)\right)z=x\\
\p(\psi(\delta_x))&= \sum_{y\leq x}\sum_{z\leq y}\mu(y,x)\delta_z=\sum_{z\leq x}\left(\sum_{z\leq y\leq x}\mu(y,x)\zeta(z,y)\right)\delta_z=\delta_x
\end{align*}
as required.  This completes the proof that $\p$ is an isomorphism of $\Bbbk$-algebras with inverse $\psi$.

As $\varepsilon_x=\psi(\delta_x)$, it remains to consider the action of $y\in\Lambda$ on $\Bbbk\Lambda\varepsilon_x$ or, equivalently, the action of $\p(y)$ on $\Bbbk^{\Lambda}\delta_x=\Bbbk\delta_x$.  But we have \[\p(y)\delta_x = \sum_{z\leq y}\delta_z\delta_x= \begin{cases} \delta_x, & \text{if}\ y\geq x\\ 0, & \text{else.}\end{cases}\]  This completes the proof.
\end{proof}

Recall that if $B$ is a left regular band, then $\Lambda(B)=\{Bb\mid b\in B\}$ is a meet semilattice with $Ba\cap Bb=Bab=Bba$ and we have the support homomorphism $\sigma\colon B\to \Lambda(B)$ given by $\sigma(b)=Bb$.  As before, if $X\in \Lambda(B)$, then $L_X=\sigma\inv(X)$ is an $\mathscr L$-class of $B$.   Abusing notation, we shall also write $\sigma$ for the unique extension $\sigma\colon \Bbbk B\to \Bbbk\Lambda(B)$.  The following result is a special case of the results of~\cite{MunnJacobson}; see also~\cite{Brown1}.

\begin{Thm}\label{t:nilkernel}
Let $B$ be a finite left regular band and $\Bbbk$ a commutative ring with unit.  Let $J$ be the kernel of the homomorphism $\sigma\colon \Bbbk B\to \Bbbk \Lambda(B)$ induced by $b\mapsto Bb$.  Let $n$ be the largest integer such that there is a chain $X_1>X_2>\cdots>X_n$ in $\Lambda(B)$.  Then $J^{n+1}=0$ and hence $J$ is nilpotent.
\end{Thm}
\begin{proof}
If $r=\sum_{b\in B}c_bb$ with the $c_b\in \Bbbk$ and $X\in \Lambda(B)$, then put $r_X=\sum_{b\in L_X}c_bb$.  Notice that $r=\sum_{X\in \Lambda(X)}r_X$ and that \[\sigma(r)=\sum_{b\in B}c_b\sigma(b)=\sum_{X\in \Lambda}\left(\sum_{b\in L_X}c_b\right)X.\] Therefore, $r\in J$ if and only if, for each $X\in \Lambda(B)$, one has that $\sum_{b\in L_X}c_b=0$.  It follows from this that $r\in J$ if and only if $r_X\in J$ for all $X\in \Lambda(B)$.

Let us put $J_X=\{r\in J\mid r=r_X\}$.  Then we have that \[J=\bigoplus_{X\in \Lambda(B)} J_X\] from the discussion in the previous paragraph.  Thus to prove our theorem it suffices to prove that if $Y_1,\ldots, Y_{n+1}\in \Lambda(B)$, then $J_{Y_1}\cdots J_{Y_{n+1}}=0$. Put $X_i=Y_1\wedge Y_2\wedge \cdots \wedge Y_i$ for $1\leq i\leq n+1$.

We claim that if $Y\leq Z$, then $J_YJ_Z=0$.  Indeed, if $r=\sum_{a\in L_Y}c_aa$ and $b\in B$ with $\sigma(b)=Z\geq Y$, then \[rb=\sum_{a\in L_Y}c_aab=\sum_{a\in L_Y}c_aa=r\] because $ab=a$ if $Ba\subseteq Bb$.  Therefore, if $r\in J_Y$ and $s=\sum_{b\in L_Z}d_bb\in J_Z$, then \[rs=\sum_{b\in L_Z}d_brb=\left(\sum_{b\in L_Z}d_b\right)r=0\] as $s\in J_Z$ implies $\sum_{b\in L_Z}d_b=0$.  We conclude that $J_YJ_Z=0$ whenever $Y\leq Z$.

Next observe that $J_{Y_1}\cdots J_{Y_i}\subseteq J_{X_i}$ for $1\leq i\leq n+1$ because if $\sigma(y_k)= Y_k$ for $1\leq k\leq i$, then $\sigma(y_1\cdots y_i)= \sigma(y_1)\cdots \sigma(y_i)=X_i$.  As \[X_1\geq X_2\geq \cdots \geq X_{n+1}\] we deduce by choice of $n$ that $X_i=X_{i+1}$ for some $1\leq i<n+1$.  But then $X_i=X_{i+1}=X_i\wedge Y_{i+1}\leq Y_{i+1}$.  Therefore, we have
\[J_{Y_1}\cdots J_{Y_i}J_{Y_{i+1}}\subseteq J_{X_i}J_{Y_{i+1}}=0\] where the last equality follows from the claim because $X_i\leq Y_{i+1}$.  We conclude that $J_{Y_1}\cdots J_{Y_{n+1}}=0$, completing the proof.
\end{proof}

As a consequence, we can deduce that $J$ is the Jacobson radical of $\Bbbk B$ if the ground ring $\Bbbk$ is sufficiently nice.

\begin{Cor}\label{c:jac.radical}
Let $B$ be a finite left regular band and $\Bbbk$ a semiprimitive commutative ring with unit (e.g.\ a field).  Then the Jacobson radical of $\Bbbk B$ is the kernel of the homomorphism $\sigma\colon \Bbbk B\to \Bbbk \Lambda(B)$ induced by $b\mapsto Bb$ for $b\in B$.
\end{Cor}
\begin{proof}
A nilpotent ideal is contained in the Jacobson radical of any ring.  Thus $\ker\sigma\subseteq \rad(\Bbbk B)$ by Theorem~\ref{t:nilkernel}. The isomorphism $\Bbbk\Lambda(B)\cong \Bbbk^{\Lambda}$ of Theorem~\ref{t:solomon} implies that $\Bbbk \Lambda(B)$ is semiprimitive. We deduce that $\rad(\Bbbk B)\subseteq \ker \sigma$ and the theorem follows.
\end{proof}

It follows from Corollary~\ref{c:jac.radical} that $\Bbbk B$ is a split basic algebra over any field $\Bbbk$ whenever it is unital.

Note that as a $\Bbbk B$-module we have that
\begin{equation}\label{eq:decomposelambdaB}
\Bbbk \Lambda(B)\cong \bigoplus_{X\in \Lambda(B)} \Bbbk_X
\end{equation}
by Theorem~\ref{t:solomon}.  This provides another proof of Corollary~\ref{c:simples}, which states that when $\Bbbk$ is a field and $\Bbbk B$ is unital, then the $\Bbbk_X$ with $X\in \Lambda(B)$ are the simple $\Bbbk B$-modules.

It will be convenient to have  a basis for the radical of a left regular band algebra.

\begin{Prop}\label{rad.basis}
Let $B$ be a left regular band and $\Bbbk$ a commutative ring with unit.  Let $J$ be the kernel of the homomorphism $\sigma\colon \Bbbk B\to \Bbbk \Lambda(B)$ induced by $\sigma(b)=Bb$ for $b\in B$.  Fix, for each $X\in \Lambda(B)$, an element $e_X\in B$ with $Be_X=X$.  Then the non-zero elements of the form $b-e_{Bb}$ form a basis for $J$ over $\Bbbk$.
\end{Prop}
\begin{proof}
The linear independence of the non-zero elements of the form $b-e_{Bb}$ is immediate from the linear independence of $B$.  Trivially, $\sigma(b-e_{Bb})=Bb-Bb=0$.  So it remains to show that these elements span $J$ as a $\Bbbk$-module. So suppose that
\[x=\sum_{b\in B}c_bb=\sum_{X\in \Lambda(B)}\sum_{Bb=X}c_bb\]
belongs to $J$.    Then we have \[0=\sum_{X\in \Lambda(B)}\sum_{Bb=X}c_bX\] and hence for each $X\in \Lambda(B)$, one has that \[\sum_{Bb=X}c_b=0.\]  It follows that \[x=\sum_{X\in \Lambda(B)}\sum_{Bb=X}c_b(b-e_{Bb})\] as required.
\end{proof}

A consequence of Proposition~\ref{rad.basis} that we shall use in the next section is the following.

\begin{Cor}\label{c:add.one}
Let $B$ be a left regular band and $\Bbbk$ a field.  Let $J$ be the kernel of the natural homomorphism $\sigma\colon \Bbbk B\to \Bbbk \Lambda(B)$ induced by $b\mapsto Bb$ for $b\in B$.  Then $J$ is the Jacobson radical of $\Bbbk B^1$.
\end{Cor}
\begin{proof}
Let $R$ be the kernel of the natural homomorphism $\rho\colon \Bbbk B^1\to \Bbbk \Lambda(B^1)$ induced by $b\mapsto B^1b$ for $b\in B^1$; its the radical of $\Bbbk B^1$ by Corollary~\ref{c:jac.radical}.  As $1$ is the unique generator of the left ideal $B^1$ and $Ba=Bb$ if and only if $B^1a=B^1b$ for $a,b\in B$, it follows from Proposition~\ref{rad.basis} that $J$ and $R$ have the same basis and hence coincide.
\end{proof}

\subsection{Existence of identity elements in left regular band algebras}\label{ss:existid}
Let $B$ be a finite left regular band.  It will be important to characterize when $\Bbbk B$ is a unital ring.  The answer turns out to be independent of the ground ring $\Bbbk$ and depends, instead, on the topology of $B$ and its contractions as  posets.   Following our earlier convention, we say that $B$ is \emph{connected}\index{connected} if $\Delta(B_{\geq X})$ (or, equivalently, the Hasse diagram of $B_{\geq X}$) is connected for each $X\in \Lambda(B)$.  Of course, monoids are connected since all contractions of a monoid are monoids and hence have contractible order complexes.  It turns out that $\Bbbk B$ is unital if and only if $B$ is connected.

The proof is surprisingly indirect in that some representation theory is used.  It would be nice to have a completely elementary proof of the existence of an identity.   Later we shall see that if $B$ is a connected left regular band, then, in fact, $\Delta(B_{\geq X})$ is an acyclic topological space for all $X\in \Lambda(B)$ using the existence of an identity.

\begin{Thm}\label{t:unital}
The following conditions are equivalent for a finite left regular band $B$.
\begin{enumerate}
\item $B$ is connected.
\item $\mathbb ZB$ is unital.
\item $\Bbbk B$ is unital for each commutative ring with unit $\Bbbk$.
\end{enumerate}
\end{Thm}
\begin{proof}
Clearly the second item implies the third as $\Bbbk B\cong \Bbbk\otimes_{\mathbb Z}\mathbb ZB$ has identity $1\otimes 1$. Suppose that $\Delta(B_{\geq X})$ is not connected.  Denote by $\pi_0(Y)$ the set of connected components of a topological space $Y$.  There is a surjective semigroup homomorphism $B_{\geq X}\to \pi_0(\Delta(B_{\geq X}))$, sending $b\in B$ to its component, where $\pi_0(\Delta(B_{\geq X}))$ is given the left zero multiplication (in fact, this is the maximal left zero image of $B_{\geq X}$; see~\cite{embedpaper}). Indeed, $a\geq ab$, for $a,b\in B_{\geq X}$, and hence $a$ and $ab$ belong to the same connected component. Thus we have surjective homomorphisms \[\Bbbk B\to \Bbbk B_{\geq X}\to \Bbbk \pi_0(\Delta(B_{\geq X}))\] and hence $\Bbbk B$ is not unital, as $\Bbbk \pi_0(\Delta(B_{\geq X}))$ is not by the discussion at the beginning of Section~\ref{s:lrbalgebra}.  This proves that the third item implies the first.

Suppose now that $B$ is connected. Fix, for each $X\in \Lambda(B)$, an element $e_X\in B$ with $X=Be_X$. If $\mathbb ZB$ is unital, then its identity must be $\eta=\sum_{X\in \Lambda(B)}\eta_X$, where the $\eta_X$ are as defined as per \eqref{defidempotents}, since $\eta$ is a right identity by Theorem~\ref{primidempotentprops}.  It clearly suffices to show that  $\eta$ is, in fact, a left identity for $\mathbb QB$.

Consider $\mathbb QB^1$ and let $\wh 1=B^1\in \Lambda(B^1)$.  Then $\eta_{\wh 1}=1-\eta$, together with the $\eta_X$ with $X\in \Lambda(B)$, form a complete set of orthogonal primitive idempotents for $\mathbb QB^1$ by Theorem~\ref{primidempotentprops}.   Note that $\mathbb QB$ is an ideal in the unital $\mathbb Q$-algebra $\mathbb QB^1$. Our goal is to show that $(1-\eta)\mathbb QB=0$, as this implies that $x=\eta x$ for all $x\in \mathbb QB$, as required.  Let $J$ be the kernel of the surjective homomorphism $\sigma\colon \mathbb QB\to \mathbb Q\Lambda(B)$ induced by $\sigma(b)=Bb$ for $b\in B$.    The algebra $\mathbb Q\Lambda(B)$ is unital by Theorem~\ref{t:solomon} and hence $\sigma(\eta)$ is the identity of $\mathbb Q\Lambda(B)$ because surjective homomorphisms map right identities to right identities.  Therefore, $\sigma(x-\eta x)=0$ for all $x\in \mathbb QB$ and so $(1-\eta)\mathbb QB\subseteq J$, whence $(1-\eta)\mathbb QB=(1-\eta)J$.  So it suffices to prove that $(1-\eta)J=0$.

As the ideal $J$ is nilpotent by Theorem~\ref{t:nilkernel}, it suffices to prove that $(1-\eta)J=(1-\eta)J^2$ or, equivalently, $(1-\eta)J\subseteq J^2$.  Thus, to complete the proof of the theorem, we are reduced to proving that $(1-\eta)[J/J^2]=0$.  Recall that $J$ is the Jacobson radical of $\mathbb QB^1$ by Corollary~\ref{c:add.one} and $1-\eta$ is the primitive idempotent corresponding to the simple module $\mathbb Q_{\wh 1}$, which is annihilated by $B$ and fixed by $1$.  Therefore, we have
\begin{align*}
(1-\eta)[J/J^2]& \cong \Hom_{\mathbb QB^1/J}([\mathbb QB^1/J](1-\eta),J/J^2)\\ &=\Hom_{\mathbb QB^1/J}(\mathbb Q_{\wh 1},J/J^2)\\ &\cong \Hom_{\mathbb QB^1/J}(J/J^2,\mathbb Q_{\wh 1})\\& \cong \Hom_{\mathbb QB^1}(J,\mathbb Q_{\wh 1})
\end{align*}
where the penultimate isomorphism uses that $\mathbb QB^1/J$ is a split semisimple $\mathbb Q$-algebra and the last isomorphism uses that $J^2=\rad(J)$ and that $\mathbb Q_{\wh 1}$ is simple.  In summary, to prove that $\eta$ is an identity for $\mathbb QB$, it remains to show that there are no non-zero $\mathbb QB^1$-module homomorphisms from $J$ to $\mathbb Q_{\wh 1}$.

So let $\p\colon J\to \mathbb Q_{\wh 1}$ be a $\mathbb QB^1$-module homomorphism.  By Proposition~\ref{rad.basis}, $J$ has  $\mathbb Q$-vector space basis consisting of elements of the form $b-e_X$ where $Bb=X$ and $b\neq e_X$.  Thus to prove that $\p$ is the zero mapping, it suffices to show that if $x,y\in B$ with $Bx=By=X$, then $\p(x-y)=0$.

 Let $\Gamma(X)$ be the connected graph from Lemma~\ref{l:Franco.graph}.  Assume first that $x,y\in L_X$ are adjacent vertices.  Then there exists $b\in B$ with $bx=x$ and $by=y$.  Therefore, we  have that \[\p(x-y)=\p(b(x-y))=b\p(x-y)=0.\]  If $x,y$ are not adjacent, then there is a sequence $x=x_0,x_1,\ldots, x_n=y$ with $x_i,x_{i+1}$ adjacent vertices of $\Gamma(X)$ for $i=0,\ldots, n-1$. But then we have the telescoping sum
\[x-y=\sum_{i=0}^{n-1}(x_i-x_{i+1})\] and so
\[\p(x-y) = \sum_{i=0}^{n-1}\p(x_i-x_{i+1})=0\] where the last equality holds by the previous case.  We conclude that $\p=0$, as required.  This concludes the proof that $\mathbb ZB$ has an identity.
\end{proof}

\begin{Rmk}
The unit of $\Bbbk B$, when $B$ is a connected left regular band, is given by $\eta=\sum_{X\in \Lambda(B)}\eta_X$ and hence can be explicitly written down.  But the formula is not very enlightening.
Using methods similar to ours, a slightly more concrete formula in the case of an affine arrangement is given in \cite[Section~14.3.3]{THA}.
\end{Rmk}

Most of our examples of left regular bands that are not monoids are strong elimination systems and hence are connected left regular bands.

\begin{Cor}
Let $(E,\mathcal L)$ be a strong elimination system and $\Bbbk$ a field. Then $\Bbbk \mathcal L$ is a unital $\Bbbk$-algebra.
\end{Cor}
\begin{proof}
This follows from Theorem~\ref{t:unital} because $\mathcal L$ is a connected left regular band by Proposition~\ref{p:se.conn}.
\end{proof}

\subsection{Cartan invariants}
Let $A$ be a unital finite dimensional $\Bbbk$-algebra over a field $\Bbbk$.  Let $P_1,\ldots, P_r$  be a complete set of representatives of the isomorphism classes of projective indecomposable $A$-modules and $S_1,\ldots, S_r$ be a complete set of representatives of the isomorphism classes of simple $A$-modules such that $S_i\cong P_i/\rad(P_i)$. Also, let $I_1,\ldots, I_r$ be a complete set of representatives of the isomorphism classes of injective indecomposable $A$-modules such that $I_i$ has socle isomorphic to $S_i$. The \emph{Cartan matrix}\index{Cartan matrix} of $A$ is the $r\times r$-matrix $C$ with $C_{ij}$ the number of composition factors of $P_j$ isomorphic to $S_i$, or equivalently the number of composition factors of $I_i$ isomorphic to $S_j$.  If $A$ is a split $\Bbbk$-algebra, then
\[C_{ij}=\dim_{\Bbbk}\Hom_A(P_i,P_j)=\dim_{\Bbbk} e_iAe_j\] where $e_i, e_j$ are primitive idempotents with $P_i\cong Ae_i$ and $P_j\cong Ae_j$.

Saliola computed the Cartan matrix of the algebra of a left regular band monoid in terms of the M\"obius function of its support lattice.  We provide here the straightforward generalization for the algebra of a connected left regular band.  If $B$ is a connected left regular band, then both the projective indecomposable $\Bbbk B$-modules and the simple $\Bbbk B$-modules are indexed by $\Lambda(B)$ and so it is natural to index the Cartan matrix by this set.  More precisely, if $X,Y\in \Lambda(B)$, then
$C_{X,Y}=\dim_{\Bbbk}\Hom(\Bbbk L_X,\Bbbk L_Y)$.

\begin{Thm}\label{t:Cartan}
Let $B$ be a connected left regular band with support semilattice $\Lambda(B)$ and let $\Bbbk$ be a field. Fix $e_X\in B$ with $Be_X=X$ for all $X\in \Lambda(B)$.  Then the Cartan matrix of $\Bbbk B$ is given by
\begin{equation}\label{eq:cartan}
C_{X,Y}=\sum_{Z\leq X}|e_ZB\cap L_Y|\cdot \mu(Z,X)
\end{equation}
for $X,Y\in \Lambda(B)$ where $\mu$ is the M\"obius function of $\Lambda(B)$. In particular, $C_{X,Y}$ is unipotent lower triangular with respect to any linear extension of the partial order on $\Lambda(B)$.
\end{Thm}
\begin{proof}
Let $\rho\colon B\to M_n(\Bbbk)$ be the matrix representation afforded by $\Bbbk L_Y$ with respect to the basis $L_Y$ and let $\chi$ be the character of $\rho$, that is, $\chi(b)$ is the trace of $\rho(b)$.   Note that $\chi(b) = |bB\cap L_Y|$ because $b$ is idempotent.   Recall from Corollary~\ref{c:simples} that if $Z\in \Lambda(B)$, then the irreducible representation $\chi_Z$ corresponding to the simple module $\Bbbk_Z$ is given by
\[\chi_Z(a) = \begin{cases} 1, & \text{if}\ Ba\geq Z\\ 0, & \text{else}\end{cases}\]
for $a\in B$.

Choosing a basis for $\Bbbk L_Y$ adapted to a composition series, we conclude that
\[\chi=\sum_{Z\in \Lambda(B)}C_{Z,Y}\cdot \chi_Z\] and hence
\[|e_XB\cap L_Y|=\chi(e_X)=\sum_{Z\in \Lambda(B)}C_{Z,Y}\cdot \chi_Z(e_X) =\sum_{Z\leq X}C_{Z,Y}.\]  An application of Rota's M\"obius inversion theorem yields \eqref{eq:cartan}.

For the final statement, observe that $e_ZB\cap L_Y=\emptyset$ if $Z\ngeq Y$. Thus $C_{X,Y}=0$ unless $X\geq Y$ and $C_{X,X}=|e_XB\cap L_X|\cdot \mu(X,X)=1$. This completes the proof.
\end{proof}

\section{Projective resolutions and global dimension}\label{s:projrel}
In this section we use various simplicial and CW complexes associated to connected left regular bands to compute projective resolutions of their simple modules.  As a consequence, we obtain a more direct proof of the results of our previous paper on the global dimension of left regular bands~\cite{oldpaper}, which went through classifying spaces of small categories. Also by allowing left regular bands that are not monoids, the theory can include a number of new and interesting examples.  The strongest results are obtained for connected CW left regular bands, for which we are able to construct the minimal projective resolutions of the simple modules using topology.

\subsection{Actions of left regular bands on CW posets}
Although at the end of the day we are interested in actions of left regular bands on regular CW complexes, it is more convenient to work with actions on CW posets as these are more combinatorial in nature. Thanks to the functor from CW posets to regular CW complexes, we can then obtain actions on the associated CW complexes and hence on their cellular chain complexes.

\subsubsection{}
\nomenclature[L, 18]{$\Stab(x)$}{stabilizer of $x$}%
 If $B$ is a left regular band acting on a set $X$ and $x\in X$, then \[\Stab(x) = \{b\in B\mid bx=x\}\] denotes the \emph{stabilizer}\index{stabilizer} of $x$.  We say that the action of $B$ on $X$ is \emph{unitary}\index{unitary} if $BX=X$.  Note that this is equivalent to $\Stab(x)\neq \emptyset$ for all $x\in X$. Indeed, if $BX=X$ and $x=by$ with $y\in X$, then $x=bx$ and hence $\Stab(x)\neq \emptyset$.  Conversely, if $b\in \Stab(x)$, then $x=bx\in BX$.  Notice that if $B$ is a monoid, then the action of $B$ on $X$ is unitary if and only if the identity acts via the identity mapping.  Clearly, $BX=X$ if the identity of $B$ acts by the identity mapping. If $\Stab(x)\neq\emptyset$, say $bx=x$, then $1x=1(bx)=(1b)x=bx=x$.  Therefore, the identity acts via the identity mapping.
Notice that the action of $B$ on the left of itself is unitary because $b\in \Stab(b)$ for all $b\in B$.  The importance of unitary actions stems from the following proposition.

\begin{Prop}\label{p:unitary}
Let $B$ be a connected left regular band and $\Bbbk$ a commutative ring with unit.  Suppose that $B$ has a unitary action on a set $X$.  Then $\Bbbk X$ is a unitary $\Bbbk B$-module.
\end{Prop}
\begin{proof}
Let $\eta$ be the identity of $\Bbbk B$ and let $x\in X$.  Suppose that $b\in B$ with $bx=x$.  Then $\eta x=\eta (bx)=(\eta b)x=bx=x$ and so $\eta$ fixes the basis $X$ of $\Bbbk X$.  It follows that $\Bbbk X$ is a unitary $\Bbbk B$-module.
\end{proof}

The following proposition will be useful several times.

\begin{Prop}\label{stabilizersonsimplicesfromstabs}
Let $B$ be a connected left regular band acting by cellular maps on a poset $P$.  Let $\sigma_0<\sigma_1<\cdots<\sigma_q$ form a simplex $\sigma$ of $\Delta(P)$.  Then $\Stab(\sigma)=\Stab(\sigma_q)$.  In particular, the action of $B$ on $P$ is unitary if and only if all of the simplices of $\Delta(P)$ have non-empty stabilizers under the action of $B$ on $\Delta(P)$ by simplicial maps (i.e., the action of $B$ on the set of simplices of $\Delta(P)$ is unitary).
\end{Prop}
\begin{proof}
If $b\in B$ and $b\sigma_q=\sigma_q$, then Proposition~\ref{retract} implies that $b$ stabilizes $\sigma$.  Conversely, if $b$ stabilizes $\sigma$, then from $b\sigma_0\leq b\sigma_1\leq\cdots\leq b\sigma_q$ and $\{\sigma_0,\ldots, \sigma_q\}=\{b\sigma_0,\ldots, b\sigma_q\}$, we deduce that $b\sigma_q=\sigma_q$.
\end{proof}

It is crucial that $B$ acts on itself by cellular mappings.

\begin{Prop}\label{cellactonB}
Let $B$ be a left regular band.  Then the action of $B$ on itself by left multiplication is unitary and by cellular maps.
\end{Prop}
\begin{proof}
By Lemma~\ref{l:left.act.order} the action is order preserving and we have already observed that it is unitary.  Let $a,b\in B$ and suppose that $b'\leq ab$.  Then $bb'\leq b$ and $a(bb')=(ab)b'=b'$.  Thus left multiplication by $a$ is a cellular mapping.
\end{proof}

\subsubsection{Simplicial homology}
Let us remind the reader of some basics of simplicial homology.  See~\cite[Chapter~1.5]{Munkres} for details.

Fix for this discussion a commutative ring with unit $\Bbbk$.
Let $K$ be a (finite) simplicial complex.  If $\sigma=\{v_0,\ldots, v_q\}$ is a $q$-simplex, then an \emph{orientation}\index{orientation} of $\sigma$ is an equivalence class of total orderings of the vertices of $\sigma$, where two orderings are equivalent if they differ by an even permutation.  Each $q$-simplex admits exactly two orientations if $q>0$ (and one if $q=0$).   We write $[v_{i_0},\ldots, v_{i_q}]$ for the class of the ordering $v_{i_0},\ldots, v_{i_q}$. An \emph{oriented $q$-simplex}\index{oriented $q$-simplex} consists of a $q$-simplex together with an orientation.

\nomenclature[R, 07]{$C_{\bullet}(K; \Bbbk)$}{simplicial chain complex of $K$ with coefficients in $\Bbbk$}%
To construct the \emph{simplicial chain complex}\index{simplicial chain complex} $C_{\bullet}(K;\Bbbk)$ of $K$ with coefficients in $\Bbbk$, we recall that the
$q^{th}$-chain module $C_q(K;\Bbbk)$  is  the quotient of the free $\Bbbk$-module on the set of oriented simplices by the relations identifying an oriented simplex $\sigma$ with $-\sigma'$ where $\sigma'$ has the same underlying simplex as $\sigma$ but with the opposite orientation (for $q>0$). Of course, by fixing an orientation of each simplex, it is easy to see that $C_q(K;\Bbbk)$ is a free $\Bbbk$-module with a basis in bijection with the set of $q$-simplices.

 The boundary operator $d_q\colon C_q(K;\Bbbk)\to C_{q-1}(K;\Bbbk)$ for $q\geq 1$ is given by
\[d_q([v_0,\ldots, v_q]) = \sum_{i=0}^q (-1)^i[v_0,\ldots, \wh{v_i},\ldots, v_q]\] where $\wh{v_i}$ means omit $v_i$.  The augmentation $\varepsilon\colon C_0(K;\Bbbk)\to\Bbbk$ is given by $\varepsilon(v)=1$ for any $0$-simplex $v$ (where we recall that an oriented $0$-simplex is the same thing as a $0$-simplex).  The (reduced) simplicial homology $\til H_{\bullet}(K;\Bbbk)$ of $K$ with coefficients in $\Bbbk$ is then the homology of the (augmented) simplicial  chain complex.  The (reduced) cohomology of $K$ with coefficients in $\Bbbk$  is defined dually and denoted $\til H^{\bullet}(K;\Bbbk)$.  The (non-reduced) homology and cohomology of $K$ are denoted by $H_\bullet(K;\Bbbk)$ and $H^\bullet(K;\Bbbk)$, respectively.

\nomenclature[R, 09]{$f_{\ast}$}{induced chain map}%
If $f\colon K\to L$ is a \emph{simplicial map}\index{simplicial map}, then the induced chain map \[f_{\ast}\colon C_{\bullet}(K;\Bbbk)\to C_{\bullet}(L;\Bbbk)\] is defined on an oriented $q$-simplex by
\[f_*([v_0,\ldots, v_q]) = \begin{cases}[f(v_0),\ldots, f(v_q)], & \text{if}\ f(v_0),\ldots, f(v_q)\ \text{are distinct}\\ 0, & \text{else.}\end{cases}\]  The simplicial chain complex construction is functorial on the category of simplicial complexes and simplicial maps.

\nomenclature[R, 08]{$C_{\bullet}(K, L; \Bbbk)$}{simplicial chain complex of $K$ relative to the subcomplex $L$ and with coefficients in $\Bbbk$}%
If $L$ is a subsimplicial complex of a simplicial complex $K$, then $C_\bullet(L;\Bbbk)$ is a subchain complex of $C_\bullet(K;\Bbbk)$.  The \emph{relative simplicial chain complex}\index{relative simplicial chain complex} $C_\bullet(K,L;\Bbbk)$ is the quotient chain complex.

\subsubsection{Simplicial homology for connected left regular bands}

The next proposition will be used throughout the text, often without comment.

\begin{Prop}\label{makeunital}
Let $B$ be a connected left regular band acting on a simplicial complex $K$ by simplicial maps such that the stabilizer of each simplex is non-empty.  Let $L$ be a subcomplex invariant under $B$.  Then the simplicial chain complex $C_\bullet(K;\Bbbk)$ and the relative simplicial chain complex $C_\bullet(K,L;\Bbbk)$ are chain complexes of unitary $\Bbbk B$-modules.  Thus  the relative homology $H_q(K,L;\Bbbk)$ is a unitary $\Bbbk B$-module for all $q\geq 0$.
\end{Prop}
\begin{proof}
First note that $C_\bullet(K;\Bbbk)$ is a chain complex of $\Bbbk B$-modules by functoriality of simplicial homology under simplicial maps. Let us verify that the  $\Bbbk B$-module  $C_q(K;\Bbbk)$ is unitary.  Indeed, let $\sigma=[v_0,\ldots, v_q]$ be an oriented $q$-simplex and suppose that $b\in B$ stabilizes the underlying $q$-simplex $\{v_0,\ldots, v_q\}$.  Then $\{bv_0,\ldots, bv_q\}=\{v_0,\ldots, v_q\}$ and hence $b$ fixes each of $v_0,\ldots, v_q$, being an idempotent.  Thus $b_*([v_0,\ldots, v_q])=[bv_0,\ldots, bv_q]=[v_0,\ldots, v_q]$ and so $b\sigma=\sigma$.  Hence if $\eta$ is the identity of $\Bbbk B$, then we have that $\eta\sigma=\eta(b\sigma)=(\eta b)\sigma=b\sigma=\sigma$.  Therefore, $C_q(K;\Bbbk)$ is a unitary $\Bbbk B$-module for each $q\geq 0$ and so $C_{\bullet}(K;\Bbbk)$ is a chain complex of unitary $\Bbbk B$-modules.  We conclude that the subchain complex $C_\bullet(L;\Bbbk)$ consists of unitary $\Bbbk B$-modules, as does the quotient complex $C_\bullet(K,L;\Bbbk)$.  It follows that the relative homology consists of unitary $\Bbbk B$-modules.
\end{proof}

With Proposition~\ref{makeunital} in hand, we can now establish that if $B$ is a connected left regular band, then each simplicial complex $\Delta(B_{\geq X})$ with $X\in \Lambda(B)$ is acyclic.

\begin{Thm}\label{t:acyclicordercomplex}
Let $B$ be a connected left regular band.  Then $\Delta(B_{\geq X})$ is acyclic for all $X\in \Lambda(B)$.
\end{Thm}
\begin{proof}
As $B_{\geq X}$ is a connected left regular band in its own right, for all $X\in \Lambda(B)$, it suffices to prove that $\Delta(B)$ is acyclic. (One can then replace $B$ by $B_{\geq X}$ to obtain the general case.)
Note that $\Delta(B)$ is a subcomplex of $\Delta(B^1)$.  Therefore, the augmented simplicial chain complex $C_\bullet$ of $\Delta(B)$ (over $\mathbb Z$) can be viewed as a subcomplex of the augmented simplicial chain complex $C'_\bullet$ of $\Delta(B^1)$ by identifying $q$-chains on $\Delta(B)$ with $q$-chains on $\Delta(B^1)$ that are supported on oriented simplices of $\Delta(B)$.  The complex $C'_\bullet$ is acyclic (even contractible) because $\Delta(B^1)$ is contractible, having cone point $1$.

The complex $C'_\bullet$ is, in fact, a complex of $\mathbb ZB^1$-modules by Proposition~\ref{makeunital}.  Thus each element $x\in \mathbb ZB^1$ induces a chain map $C'_\bullet\to C'_\bullet$ (as chain complexes of abelian groups) via left multiplication. Moreover, since $b\Delta(B^1)\subseteq \Delta(B)$ for all $b\in B$, it follows that if $x\in \mathbb ZB$ then the image of the chain map $C'_\bullet\to C'_\bullet$ induced by $x$ is contained in $C_\bullet$.  In particular, let $\eta\in \mathbb ZB$ be the identity element guaranteed by Theorem~\ref{t:unital}.  Then $\eta$ induces an idempotent chain map $C'_\bullet\to C'_\bullet$ with image contained in $C_\bullet$ and, moreover, it fixes $C_\bullet$ by Propositions~\ref{stabilizersonsimplicesfromstabs}--\ref{makeunital}. Thus the chain complex $C_\bullet$ is a retract of the acyclic chain complex $C'_\bullet$.  This implies that $C_\bullet$ is acyclic (by the functoriality of homology).  In fact, since $C'_\bullet$ is a contractible chain complex over $\mathbb Z$, so is its retract $C_\bullet$.  This completes the proof.
\end{proof}

We do not know whether $B$ connected implies $B_{\geq X}$ is contractible for all $X\in \Lambda(B)$, although this happens to be the case in all our examples.  
Notice that if we replace $\mathbb Z$ in the above proof with an arbitrary commutative ring with unit $\Bbbk$, we would get that $\til H_q(\Delta(B);\Bbbk)=0$ for all $q\geq 0$ without appealing to the universal coefficient theorem.

\subsubsection{Cellular homology}
Let us remind the reader of some basics of cellular homology.
See~\cite{Massey} or~\cite[Chapter~4.39]{Munkres} for details.

Let $\Bbbk$ be a commutative ring with unit.
The \emph{cellular chain complex} with coefficients in $\Bbbk$ for a CW poset $P$ can be described as follows.    For $q\geq 0$, put \[P^{q}=\{\sigma\in P\mid \dim \sigma\leq q\}.\]  It is the face poset of the $q$-skeleton of $\CW(P)$.  One has $C_q(\CW(P);\Bbbk)=H_q(\Delta(P^q),\Delta(P^{q-1});\Bbbk)$, where the left hand side is the cellular chain module and the right hand side is the relative simplicial homology module.  The boundary map is the composition
\begin{align*}
H_q(\Delta(P^q),\Delta(P^{q-1});\Bbbk)&\xrightarrow{\,\,\partial\,\,} H_{q-1}(\Delta(P^{q-1});\Bbbk)\\ &\to H_{q-1}(\Delta(P^{q-1}),\Delta(P^{q-2});\Bbbk)
\end{align*}
where $\partial$ is the boundary map from the long exact sequence in relative simplicial homology.  If $\tau$ is a $q$-cell of $P$, there is a homeomorphism of pairs
\[(\Delta(P_{\leq \tau}),\Delta(P_{<\tau}))\cong (E^q,S^{q-1})\] where we recall that $E^q$ denotes the unit ball in $\mathbb R^q$.  On the level of homology this yields
 \[H_q(\Delta(P_{\leq \tau}),\Delta(P_{<\tau});\Bbbk)\cong H_q(E^q,S^{q-1};\Bbbk)\cong \Bbbk.\]  An orientation for $\tau$ is a choice of a basis element $[\tau]$ for the rank one free $\Bbbk$-module \[H_q(\Delta(P_{\leq \tau}),\Delta(P_{<\tau});\Bbbk).\]  We also write $[\tau]$ for its image under the canonical embedding \[H_q(\Delta(P_{\leq \tau}),\Delta(P_{<\tau});\Bbbk)\to H_q(\Delta(P^q),\Delta(P^{q-1});\Bbbk)=C_q(\CW(P);\Bbbk).\]  The elements $[\tau]$ with $\tau$ a $q$-cell of $P$ form a basis for $C_q(\CW(P);\Bbbk)$.

\subsubsection{Cellular homology for connected left regular bands}
We now assume that $B$ is a connected left regular band with a unitary action on a CW poset $P$ by cellular maps.
There is an induced action $B\curvearrowright \CW(P)$ by regular cellular maps by functoriality of $\CW$.
 Notice that $\Bbbk$ has the structure of a trivial $\Bbbk B$-module. The augmented cellular chain complex \[C_{\bullet}(\CW(P);\Bbbk)\xrightarrow{\,\,\varepsilon\,\,} \Bbbk\] for $\CW(P)$ with coefficients in $\Bbbk$ then becomes a chain complex of unitary $\Bbbk B$-modules (finitely generated over $\Bbbk$) by functoriality of cellular homology with respect to cellular maps and Propositions~\ref{stabilizersonsimplicesfromstabs} and~\ref{makeunital}.  Alternatively, a cellular map $f\colon P\to P$ obviously sends $P^q$ into $P^q$ and so the induced simplicial map sends $\Delta(P^q)$ into $\Delta(P^q)$.
If $P$ is acyclic, then \[C_{\bullet}(\CW(P);\Bbbk)\xrightarrow{\,\,\varepsilon\,\,} \Bbbk\] is a resolution of the trivial module $\Bbbk$. As mentioned above, $C_q(\CW(P);\Bbbk)$ has a basis consisting of the oriented $q$-cells $[\sigma]$ of $P$.  Moreover, if $b\in B$, then $b[\sigma]=0$ if $\dim b\sigma<\dim \sigma$ and otherwise $b[\sigma]\in \Bbbk [b\sigma]$ because $B$ acts by regular cellular mappings.

\subsubsection{Semi-free actions}
\label{sssec:semi-free-actions}
\nomenclature[L, 10]{$B \curvearrowright P$}{an action of the left regular band on a poset $P$}%
Let us say that an action $B\curvearrowright P$ of a left regular band $B$ on a poset $P$ by cellular maps is \emph{semi-free}\index{semi-free} if $\Stab(\tau)$ has a minimum element (or equivalently a zero element) for each $\tau\in P$.  In particular, each stabilizer is assumed non-empty and hence the action is unitary.   Notice that in this case if $Q\subseteq P$ is a $B$-invariant lower set, then the action of $B$ on $Q$ is also by cellular maps and semi-free.  For example, the action of $B$ on the left of itself (cf.~Proposition~\ref{cellactonB}) is semi-free: $b$ is the minimum element of $\Stab(b)$ for $b\in B$.  We say that an action of $B$ on a simplicial complex $K$ is semi-free if the induced action on the face poset $\mathcal P(K)$ is \emph{semi-free}\index{semi-free}.

\begin{Lemma}\label{ordercomplexissemifree}
Suppose that the action $B\curvearrowright P$ of a left regular band $B$  on a poset $P$ is semi-free.  Then the action of $B$ on $\mathcal P(\Delta(P))$ is also semi-free, i.e., the action of $B$ on $\Delta(P)$ is semi-free.
\end{Lemma}
\begin{proof}
This follows from Proposition~\ref{stabilizersonsimplicesfromstabs}, which states that the stabilizer of a simplex coincides with the stabilizer of its maximum element.
\end{proof}

In particular, the action of $B$ on $\Delta(B)$ is semi-free.    Note that if $B$ has an identity, then $\Delta(B)$ is contractible because the identity $1$ is a cone point. More generally, if $B$ is connected, then $\Delta(B)$ is acyclic by Theorem~\ref{t:acyclicordercomplex}.

If $B$ acts on a poset $P$ by cellular maps and $a,b\in B$ with $Ba=Bb$, then the subposets $aP$ and $bP$ are isomorphic. Indeed, the left action of $b$ on $P$ induces a mapping $aP\to bP$ whose inverse is given by the left action of $a$ because $ab=a$ and $ba=b$.  If $A\subseteq B$, we put $AP = \bigcup_{a\in A}aP$; note that $AP$ is a lower set.  If $A$ is a left ideal, then $AP$ is $B$-invariant.

So far the only examples we have of semi-free actions of a left regular band are on itself and on its order complex.  To provide more examples, we consider the special subclass of geometric left regular bands.

\subsubsection{Geometric left regular bands}\label{ss:geom.regular}
In~\cite{oldpaper} a left regular band $B$ was defined to be \emph{geometric}\index{geometric} if the left stabilizer of each element of $B$ is commutative.  It then follows that the stabilizer of any subset of $B$ is commutative and hence, if non-empty, has a minimum element.  Note that if $B$ embeds in the left regular band $\{0,+,-\}^n$ of covectors for some $n\geq 0$, then $B$ is geometric.  Therefore, geometric left regular bands include free partially commutative left regular bands~\cite[Theorem~3.7]{oldpaper}, COMs and strong eliminations systems such as central and affine hyperplane face semigroups, (affine) oriented matroids, face semigroups of $T$-convex sets of topes and face semigroups of CAT(0) cube complexes. The class of semigroups embeddable into $\{0,+,-\}^n$ is characterized in our paper~\cite{embedpaper}.   The main examples in this text that are not geometric left regular bands are face monoids of complex hyperplane arrangements and monoids of covectors of oriented interval greedoids.

Assume that $B$ is a connected geometric left regular band and let $X\in \Lambda(B)$.  Let $\Delta_X$ be the simplicial complex with vertex set the corresponding $\mathscr L$-class $L_X$ and where a subset $\tau\subseteq L_X$ forms a simplex if the elements of $\tau$ have a common upper bound. Because $B$ is geometric $\tau$, in fact, has a least upper bound in $B$.  As $L_X$ consists of the minimal elements of $B_{\geq X}$, it follows from the dual of Rota's cross-cut theorem (Theorem~\ref{t:crosscut}) that $\Delta_X$ is homotopy equivalent to $\Delta(B_{\geq X})$ and hence is acyclic by Theorem~\ref{t:acyclicordercomplex}. Notice that if $B$ is a monoid, then $\Delta_X$ is just the full simplex with vertex set $L_X$.

There is a natural semi-free action of $B_{\geq X}\curvearrowright \Delta_X$ by simplicial maps.  Indeed, if $\tau=\{b_0,\ldots, b_q\}$ has upper bound $b$ and $a\in B_{\geq X}$, then $a\tau=\{ab_0,\ldots,ab_q\}$ has upper bound $ab$. Also, if $b$ is an upper bound for $\tau$ then $b\in \Stab(\tau)$.  Since the stabilizer of each simplex is a non-empty commutative left regular band, it has a minimum element and hence the action is semi-free.

For example, if $B$ is the face semigroup of a CAT(0) cube complex $K$ and $X=\wh 0$ is the bottom element of the support semilattice, then $\Delta_{\wh 0}$ has vertices the vertices of $K$ and a collection of vertices forms a simplex of $\Delta_{\wh 0}$ if they belong to some cube of $K$.

\subsection{Projective resolutions}
Our goal is to show that if a connected left regular band $B$ has a semi-free action on an acyclic CW poset $P$ (by cellular maps), then the augmented cellular chain complex $C_\bullet(\CW(P);\Bbbk)\xrightarrow{\,\,\varepsilon\,\,}\Bbbk$ is a projective resolution of the trivial $\Bbbk B$-module $\Bbbk$.  In particular, this applies to $P=\mathcal P(\Delta(B))$  and so the augmented simplicial chain complex $C_\bullet(\Delta(B);\Bbbk)\xrightarrow{\,\,\varepsilon\,\,}\Bbbk$ is a projective resolution of the trivial $\Bbbk B$-module $\Bbbk$.  We shall give some other examples as well.

We begin with a lemma about recognizing Sch\"utzenberger representations. If $e\in B$ and $L_e=\sigma\inv(Be)$ is the $\mathscr L$-class of $e$, then $B$ acts on $L_e$ by partial maps (by restricting the regular action of $B$ on the left of itself).  That is, if $b\in B$ and $a\in L_e$, then
\[b\cdot a = \begin{cases} ba, & \text{if}\ Bb\supseteq Be\\ \text{undefined}, & \text{else.}\end{cases}\] The Sch\"utzenberger representation defined in Section~\ref{ss:Schutz} is then the linearization of this action by partial maps.

\begin{Lemma}\label{l:findschutz}
Let $B$ be a left regular band and let $Be\in \Lambda(B)$.  Suppose that $B$ acts on a finite non-empty set $Y$ by partial transformations such that:
\begin{enumerate}
\item $by$ is defined, for $b\in B$ and $y\in Y$, if and only if $Bb\supseteq Be$;
\item $aY\cap bY\neq \emptyset$ if and only if $a=b$ for all $a,b\in L_e$;
\item $BeY=L_eY=Y$.
\end{enumerate}
Then $Y$ is $B$-equivariantly isomorphic to a disjoint union of $|eY|$ copies of the action of $B$ on $L_e$.  More precisely, $Bx$ provides an isomorphic copy of $L_e$ for each $x\in eY$ and distinct elements of $eY$ provide disjoint copies.
\end{Lemma}
\begin{proof}
Fix $y\in eY$.  First we observe that $B$ acts transitively on $By$.  Indeed, if $Bb\supseteq Be$, then $By\supseteq Bby\supseteq Bey=By$ and so $By=Bby$.  Next, we define a map $F\colon By\to L_e$ by $F(by)=be$ for $Bb\supseteq Be$.  To see that this is well defined, observe that if $ay=by$ with $Ba\supseteq Be$, then $aey=bey$ and $ae,be\in L_e$.  Thus by assumption (2), we conclude $ae=be$ and so $F$ is well defined.  It is injective because $ae=F(ay)=F(by)=be$ implies $ay=aey=bey=by$.  It is surjective because $b=be=F(by)$ for $b\in L_e$.  If $a,b\in B$ with $Bb\supseteq Be$, then $a(by)$ is defined if and only if $Ba\supseteq Be$, which is if and only if $a\cdot be$ is defined, in which case $F(a(by))=F((ab)y)=(ab)e=a(be) = a\cdot F(by)$.  This proves that $F$ is $B$-equivariant.

Next observe that if $y'\in eY$ and $By'\cap By\neq \emptyset$, then $By'=By$ by  transitivity.  But if $by=y'$, then since $Bb\supseteq Be$, we have $eb=e$ and so $y=ey=eby=ey'=y'$.  Thus, from $BeY=Y$, we deduce that $Y$ is a disjoint union of $|eY|$ isomorphic copies of the action of $B$ on $L_e$, one for each $y\in eY$, as required.
\end{proof}

A \emph{principal series}\index{principal series} for a finite semigroup $S$ is an unrefinable chain $I_0\subset I_1\subset \cdots\subset I_n=S$ of two-sided ideals.  Set $I_{-1}=\emptyset$ for convenience.  It is known~\cite[Theorem~2.40]{CP} (see also~\cite[Proposition~1.20]{repbook}) that $I_m\setminus I_{m-1}$ is a $\mathscr J$-class for $0\leq m\leq n$ and each $\mathscr J$-class occurs in this way exactly once; moreover, $I_0$ is always the minimal ideal of $S$.  For a left regular band $B$, ideals are the same thing as left ideals and so a principal series is just a maximal chain
\begin{equation}\label{principalseries}
\wh 0=L_0\subset L_1\subset\cdots \subset L_n=B
\end{equation}
of left ideals of $B$.  The $\mathscr L$-classes each appear exactly once as $L_m\setminus L_{m-1}$ (where $L_{-1}=\emptyset$).

\begin{Rmk}
If $B$ has a unitary action on a poset $P$ by cellular maps, then we can use the principal series \eqref{principalseries} to build a filtration
\[L_0P\subseteq L_1P\subseteq \cdots\subseteq L_nP=P.\]  If $P$ is the face poset of a regular CW complex $X$, then the corresponding filtration of $X$ gives rise to a filtration of the cellular chain complex whose associated quotients will be relative cellular chain complexes.  As $L_m\setminus L_{m-1}$ is an $\mathscr L$-class, this will allow the projective indecomposable Sch\"utzenberger representations to come into play to show that the relative cellular chain complexes are chain complexes of projective modules if the action is semi-free.  From this we will deduce that the original cellular chain complex is a complex of projective modules.
\end{Rmk}

The following lemma says that a module filtered by projective modules is projective.

\begin{Lemma}\label{l:projectivefiltration}
Let $R$ be a unital ring and let $M$ be an $R$-module.  Suppose that $M$ has a filtration
\[0=M_0\subseteq M_1\subseteq\cdots\subseteq M_n=M\] with $M_i/M_{i-1}$ a projective module for $i=1,\ldots, n$.  Then we have \[M\cong \bigoplus_{i=1}^n M_i/M_{i-1}\] and hence $M$ is projective.
\end{Lemma}
\begin{proof}
We proceed via induction on $n$. If $n=1$, there is nothing to prove. In general, since the sequence
\[\xymatrix{0\ar[r] & M_{n-1}\ar[r] & M_n\to M_n/M_{n-1}\ar[r] & 0}\] is exact and $M_n/M_{n-1}$ is projective, we have $M_n\cong M_{n-1}\oplus M_n/M_{n-1}$.  The result now follows by induction applied to the filtration \[0=M_0\subseteq M_1\subseteq\cdots\subseteq M_{n-1}\] of $M_{n-1}$.
\end{proof}

The reason for using filtrations by projective modules is exhibited in the following example.

\begin{Example}
Let $B$ be the free left regular band on $\{a,b\}$.  A principal series for $B$ is
\[L_0=\{ab,ba\}\subsetneq L_1=\{ab,ba,a\}\subsetneq L_2=\{ab,ba,a,b\}\subsetneq L_3=B.\]
We have a corresponding filtration
\[\Delta(L_0)\subsetneq \Delta(L_1)\subsetneq \Delta(L_2)\subsetneq \Delta(L_3)=\Delta(B)\] of $\Delta(B)$ by $B$-invariant subcomplexes.

Obviously,
\[C_0(\Delta(B);\Bbbk)\cong \Bbbk B\cong \Bbbk L_B\oplus \Bbbk L_{Bb}\oplus \Bbbk L_{Ba}\oplus \Bbbk L_{Bab}\]
is projective.

The $2$-simplices of $\Delta(B)$ are $\{1,a,ab\}$ and $\{1,b,ba\}$.  They are both fixed by $1$ and are sent to simplices of smaller dimension by $a,b,ab,ba$.  Thus \[C_2(\Delta(B);\Bbbk)\cong \Bbbk L_B\oplus \Bbbk L_B\] and hence is projective.

The situation for $1$-simplices is more complicated.  They are $\{1,a\}$, $\{1,b\}$, $\{1,ab\}$, $\{1,ba\}$, $\{a,ab\}$, $\{b,ba\}$.  Notice that $a$ sends $\{1,a\}$ to the $0$-simplex $\{a\}$ and $b$ sends it to the $1$-simplex $\{b,ba\}$. So while $1,a$ act on the basis element of $C_1(\Delta(B);\Bbbk)$ corresponding to $\{1,a\}$ as they do on $1$ in the Sch\"utzenberger representation  $\Bbbk L_{B}$, the element $b$ does not.  Let us instead look at the relative chain module \[C_1(\Delta(B),\Delta(L_2); \Bbbk)=C_1(\Delta(B);\Bbbk)/C_1(\Delta(L_2);\Bbbk).\]  This has basis the cosets of $\{1,a\}$, $\{1,b\}$, $\{1,ab\}$, $\{1,ba\}$ and now $1$ fixes these elements and $a,b,ab,ba$ annihilate them.  So this module is isomorphic to a direct sum of $4$ copies of $\Bbbk L_{B}$ and hence is projective. Similarly, $C_1(\Delta(L_2),\Delta(L_1);\Bbbk)$ has basis the coset of $\{b,ba\}$ and the module structure is that of $\Bbbk L_{Bb}$ and $C_1(\Delta(L_1),\Delta(L_0);\Bbbk)$ has basis the coset of $\{a,ab\}$ and the module structure is that of $\Bbbk L_{Ba}$.  Thus the filtration
\[C_1(\Delta(L_0);\Bbbk)\leq C_1(\Delta(L_1);\Bbbk)\leq C_1(\Delta(L_2);\Bbbk)\leq C_1(\Delta(B);\Bbbk)\] has projective subquotients and hence $C_1(\Delta(B);\Bbbk)$ is projective by Lemma~\ref{l:projectivefiltration}.  In fact, we conclude that
\[C_1(\Delta(B);\Bbbk)\cong \Bbbk L_{B}\oplus \Bbbk L_B\oplus \Bbbk L_B\oplus \Bbbk L_B\oplus \Bbbk L_{Bb}\oplus \Bbbk L_{Ba}.\]
\end{Example}

The following theorem, whose proof generalizes that in the example above, is one of the principal results of this text.

\begin{Thm}\label{t:projectiveresgeneral}
Let $B$ be a connected left regular band and $\Bbbk$ a commutative ring with unit. Suppose that $B\curvearrowright P$ is a semi-free action of $B$ on an acyclic CW poset $P$.  Then the augmented cellular chain complex $C_\bullet(\CW(P);\Bbbk)\xrightarrow{\,\,\varepsilon\,\,}\Bbbk$ is a projective resolution of the trivial $\Bbbk B$-module.

More precisely, if we fix $e_X$ with $Be_X=X$ for all $X\in \Lambda(B)$, then \[C_q(\CW(P);\Bbbk)\cong \bigoplus_{X\in \Lambda(B)} n_X\cdot \Bbbk L_X\] where $\Bbbk L_X$ is the Sch\"utzenberger representation associated to $X$ and $n_X$ is the number of $q$-cells $\sigma$ with $e_X$ as the minimum element of $\Stab(\sigma)$.
\end{Thm}
\begin{proof}
The augmented cellular chain complex is a resolution of $\Bbbk$ because $P$ is acyclic.
Fix a principal series for $B$ as in \eqref{principalseries}. Let $P_m$ be the subposet $L_mP$ and note that $L_nP=BP=P$ since the action is unitary.  We take $P_{-1}=\emptyset$. Note that $\sigma\in P_m$ if and only if $\Stab(\sigma)\cap L_m\neq \emptyset$ and so $\sigma\in P_m\setminus P_{m-1}$ if and only if the minimum element of $\Stab(\sigma)$ is in $L_m\setminus L_{m-1}$.

 Recall that the relative chain complex $C_\bullet(\CW(P_m),\CW(P_{m-1});\Bbbk)$ is defined to be  $C_\bullet(\CW(P_m);\Bbbk)/C_\bullet(\CW(P_{m-1});\Bbbk)$.  The filtration \[\CW(P_0)\subseteq \CW(P_1)\subseteq\cdots \subseteq \CW(P_n)=\CW(P)\] of $\CW(P)$ by $B$-invariant subcomplexes induces
a filtration of chain complexes of unitary $\Bbbk B$-modules
\[0= C_\bullet(\CW(P_{-1});\Bbbk)\subseteq C_\bullet(\CW(P_0);\Bbbk)\subseteq\cdots\subseteq C_\bullet(\CW(P_n);\Bbbk) = C_\bullet(\CW(P);\Bbbk)\] with subquotients the relative chain complexes $C_\bullet(\CW(P_m),\CW(P_{m-1});\Bbbk)$.  So, to show that $C_q(\CW(P);\Bbbk)$ is projective, it  suffices to prove that the $\Bbbk B$-modules $C_q(\CW(P_m),\CW(P_{m-1});\Bbbk)$ are projective for all $m,q\geq 0$. It will then follow that
\begin{equation}\label{e:chaincomplexdecomp}
C_q(\CW(P);\Bbbk)\cong \bigoplus_{m=0}^n  C_q(\CW(P_m),\CW(P_{m-1});\Bbbk)
\end{equation}
is projective by Lemma~\ref{l:projectivefiltration}.

If $P_m$ has the same set of $q$-cells as $P_{m-1}$, then there is nothing to prove.  So assume that $P_m\setminus P_{m-1}$ contains some $q$-cells. Let $Y$ be the set of $q$-cells of $P_m$ that do not belong to $P_{m-1}$.  As $P_m$ and $P_{m-1}$ are $B$-invariant, we have that $B$ acts by partial maps on $Y$ by restricting the action of $B$ on $P$ to $Y$. That is, if $b\in B$ and $\tau\in Y$, then
\[b\cdot \tau = \begin{cases}b\tau, & \text{if}\ b\tau \in Y\\ \text{undefined}, & \text{else.}\end{cases}\]
 Suppose that $L_m\setminus L_{m-1}=\sigma\inv(X)$ with $X\in \Lambda(B)$, that is, $L_m\setminus L_{m-1}$ is the $\mathscr L$-class of $e_X$. We claim that the conditions of Lemma~\ref{l:findschutz} hold with $e=e_X$.

If $\tau\in Y$, then $\tau=a\tau$ for some $a\in L_m\setminus L_{m-1}=L_e$.  Thus the third condition is satisfied.  We continue to assume that $a\tau=\tau$ with $a\in L_e$. If $Bb\supseteq Be=Ba$, then $cb=a$ for some $c\in B$. But then $\tau=a\tau=(cb)\tau$ and so $b\tau$ is defined, i.e., $b\tau\in P_m\setminus P_{m-1}$ and is a $q$-cell.  On the other hand, if $Bb\nsupseteq Be$, then  $ba\in L_{m-1}$ and so $b\tau=ba\tau\in P_{m-1}$.  Thus $b$ is undefined at $\tau$ for the action on $Y$. This establishes the first condition of Lemma~\ref{l:findschutz}.

For the second condition, suppose that $a,b\in L_e=L_m\setminus L_{m-1}$ and $\tau\in aY\cap bY$.  Then $a,b\in \Stab(\tau)$.  Let $c$ be the minimum element of $\Stab(\tau)$ (guaranteed by the definition of a semi-free action).  Then $c\leq a,b$.  We recall that in a left regular band, $x<y$ implies $\sigma(x)<\sigma(y)$ by Proposition~\ref{p:supportiscellular}.  Therefore, if $a\neq b$, then $c<a,b$ and so $c\in L_m\setminus L_e=L_{m-1}$.  But then $\tau=c\tau\in P_{m-1}$, contradicting $\tau\in Y$.  Thus $a=b$. Note that $\sigma\in eY$ if and only if $\sigma$ is a $q$-cell and $e$ is the minimum element of $\Stab(\sigma)$.  We conclude that the module $\Bbbk Y$ is isomorphic to $n_X\cdot \Bbbk L_X$ by Lemma~\ref{l:findschutz}

Now we prove that we can find orientations for the $q$-cells of $Y$ such that the action of $B$ (restricted to $Y$) is orientation preserving.   For each $\tau\in eY$ fix an orientation $[\tau]$ of $\tau$.  Recall that an orientation for $\tau$ is a choice of a basis element $[\tau]$ for $H_q(\Delta(P_{\leq \tau}),\Delta(P_{<\tau});\Bbbk)\cong \Bbbk$.
Now, for $a\in L_m\setminus L_{m-1}$, we have that the action of $a$ on $P$ gives rise to a poset isomorphism of $P_{\leq \tau}$ and $P_{\leq a\tau}$ and hence a homeomorphism of pairs \[(\Delta(P_{\leq \tau}),\Delta(P_{<\tau}))\to (\Delta(P_{\leq a\tau}),\Delta(P_{<a\tau}))\] (with inverse given by the action of $e$). Thus we can set $[a\tau]= a[\tau]$ as an orientation for $a\tau$.  By Lemma~\ref{l:findschutz}, for each $\tau'\in Y$, there exist unique $a\in L_e$ and $\tau\in eY$ such that $\tau'=a\tau$.  Thus our scheme  provides a well-defined orientation for each element of $Y$. Then if $b\in B$ with $Bb\supseteq Be$, $a\in L_e$ and $\tau\in eY$, we have $b[a\tau] = ba[\tau]=[ba\tau]$ where the last equality follows because $ba\in L_m\setminus L_{m-1}$. Thus $b$ sends the chosen orientation of $a\tau$ to the chosen orientation of $ba\tau$. If $Bb\nsupseteq Be$, then $bY\subseteq P_{m-1}$ and so $b$ annihilates $C_q(\CW(P_m),\CW(P_{m-1});\Bbbk)$.  It now follows that $C_q(\CW(P_m),\CW(P_{m-1});\Bbbk)\cong \Bbbk Y\cong n_X\cdot \Bbbk L_X$ where the first isomorphism sends the coset $[\tau]+C_q(\CW(P_{m-1};\Bbbk))$ to $\tau$. In light of \eqref{e:chaincomplexdecomp}, this completes the proof.
\end{proof}

The following special case is of central importance. We recall that if $B$ is a connected left regular band, then $\Delta(B)$ is acyclic by Theorem~\ref{t:acyclicordercomplex}. 

\begin{Thm}\label{t:projectiveres}
Let $B$ be a connected left regular band and $\Bbbk$ a commutative ring with unit.   Then the augmented
simplicial chain complex \[C_\bullet(\Delta(B);\Bbbk)\xrightarrow{\,\,\varepsilon\,\,}\Bbbk\] is a projective resolution of the trivial $\Bbbk B$-module.

More precisely, if we fix $e_X$ with $Be_X=X$ for all $X\in \Lambda(B)$, then \[C_q(\Delta(B);\Bbbk)\cong \bigoplus_{X\in \Lambda(B)} n_X\cdot \Bbbk L_X\] where $\Bbbk L_X$ is the Sch\"utzenberger representation associated to $X$ and $n_X$ is the number of $q$-simplices in $\Delta(e_XB)$ with $e_X$ as a vertex.
\end{Thm}

Notice that $C_0(\Delta(B);\Bbbk)=\Bbbk B$ and hence the resolution in Theorem~\ref{t:projectiveres} is never minimal if $B$ is non-trivial.

Theorem~\ref{t:projectiveresgeneral}, in conjunction with Corollary~\ref{computeExt} and Theorem~\ref{t:acyclicordercomplex}, yields the following corollary.

\begin{Cor}\label{c:othersimples}
Let $B$ be a connected left regular band and $\Bbbk$ a commutative ring with unit. Suppose that $X\in \Lambda(B)$.  If $B_{\geq X}\curvearrowright P$ is a semi-free action on an acyclic CW poset complex $P$, then the augmented cellular chain complex $C_\bullet(\CW(P);\Bbbk)\xrightarrow{\,\,\varepsilon\,\,}\Bbbk_X$ is a finite projective resolution of $\Bbbk_X$ by finitely generated $\Bbbk B$-modules.  In particular,  $C_\bullet(\Delta(B_{\geq X});\Bbbk)\xrightarrow{\,\,\varepsilon\,\,}\Bbbk_X$ is a finite projective resolution of $\Bbbk_X$ by finitely generated $\Bbbk B$-modules.

Fix, for each $Y\in \Lambda(B_{\geq X})$, an element $e_Y$ with $Y=Be_Y$.   The number of occurrences of the Sch\"utzenberger representation $\Bbbk L_Y$ as a direct summand in $C_q(\CW(P);\Bbbk)$ is the number of $q$-cells $\tau$ of $P$ such that $e_Y$ is the minimum element of $\Stab(\tau)$.  In particular, the number of occurrences of $\Bbbk L_Y$ as a direct summand in $C_q(\Delta(B_{\geq X});\Bbbk)$ is the number of $q$-simplices of $\Delta(e_YB_{\geq X})$ with $e_Y$ as a vertex.
\end{Cor}

An important special case is when $B_{\geq X}$ is a CW poset.

\begin{Cor}\label{c:othersimplesCWcontraction}
Let $B$ be a connected left regular band and $\Bbbk$ a commutative ring with unit. Suppose that $X\in \Lambda(B)$ and $B_{\geq X}$ is a CW poset.   Then the augmented cellular chain complex $C_\bullet(\CW(B_{\geq X});\Bbbk)\xrightarrow{\,\,\varepsilon\,\,}\Bbbk_X$ is a finite projective resolution of $\Bbbk_X$ by finitely generated $\Bbbk B$-modules.

Moreover, we have the direct sum decomposition
\begin{equation}\label{e:schutzdecompositioncwcase}
C_q(\CW(B_{\geq X});\Bbbk)\cong \bigoplus_{\rk[X,Y]=q}\Bbbk L_Y
\end{equation}
where we note that $\Lambda(B)_{\geq X}$ is a graded poset by Proposition~\ref{p:supportiscellular} and Corollary~\ref{c:preservesgraded}, and where $\Bbbk L_Y$ denotes the Sch\"utzenberger representation associated to $Y\in \Lambda(B)$.
\end{Cor}
\begin{proof}
The CW poset $B_{\geq X}$ is acyclic by Theorem~\ref{t:acyclicordercomplex}. Corollary~\ref{c:othersimples} establishes everything except the decomposition \eqref{e:schutzdecompositioncwcase}.  Fix $e_Y\in B$ with $Y=Be_Y$.
If $Y\geq X$, then $e_Y$ is the unique cell of $B_{\geq X}$ with $e_Y$ as the minimum element of its stabilizer and the dimension of this cell in $B_{\geq X}$ is $\rk[X,Y]$ by Proposition~\ref{p:supportiscellular} and Corollary~\ref{c:preservesgraded}.  Applying again Corollary~\ref{c:othersimples} completes the proof of \eqref{e:schutzdecompositioncwcase}.
\end{proof}

\begin{Rmk}\label{r:cell.case}
The minimal element of the stabilizer of $b\in B_{\geq X}$ is $b$ and so it easily follows from how we orient the cells in the proof of Theorem~\ref{t:projectiveresgeneral} that the isomorphism in \eqref{e:schutzdecompositioncwcase} takes the oriented $q$-cell $[b]$ to $b$, for $b\in L_Y$ with $\rk[X,Y]=q$.  Indeed, if $Ba\supseteq Bb$, then $a[b]=[ab]$ by the proof of Theorem~\ref{t:projectiveresgeneral}.  If $a\in B_{\geq X}$ and $Ba\nsupseteq Bb$, then $Bab\subsetneq Bb$ and so $\rk[X,Bab]<q$.  Thus the dimension of the cell $ab$ is less than $q$ in $B_{\geq X}$ and so $a[b]=0$ in $C_q(\CW (B_{\geq X});\Bbbk)$.
\end{Rmk}

\subsection{Ext and global dimension}
We now compute $\Ext^n_{\Bbbk B}(\Bbbk_X,\Bbbk_Y)$ for $X,Y\in \Lambda(B)$.  Restricting to the case that $B$ is a monoid, we obtain an alternate proof of the main  results of~\cite{oldpaper}.
First we need a lemma.

\begin{Lemma}\label{l:homtosimples}
Let $Y\in \Lambda(B)$ and let $V$ be a $\Bbbk B$-module. Fix $e_Y$ with $Be_Y=Y$.  Then $\Hom_{\Bbbk B}(V,\Bbbk_{Y})$ can be identified as a $\Bbbk$-module with the set of $\Bbbk$-linear mappings $f\colon e_YV\to \Bbbk$ that vanish on $\Bbbk[\bd e_YB]V$ (where the latter is interpreted as $0$ if $\bd e_YB = \emptyset$).
\end{Lemma}
\begin{proof}
Let $f\in\Hom_{\Bbbk B}(V,\Bbbk_{Y})$ and $v\in V$.  Then $f(e_Yv) = e_Yf(v) = f(v)$ and so $f$ is uniquely determined by its restriction to $e_YV$. Moreover, if $b\in \bd e_YB$, then $f(bv) = bf(v)=0$.  Therefore, it remains to show that if $g\colon e_YV\to \Bbbk$ vanishes on $\Bbbk[\bd e_YB]V$ and is $\Bbbk$-linear, then defining $f(v) = g(e_Yv)$ results in a $\Bbbk B$-module homomorphism $f\colon V\to \Bbbk_Y$.  Let $b\in B$.  If $Bb\nsupseteq Y$, then $e_Yb\in \bd e_YB$ and so $f(bv) = g(e_Ybv)=0=bf(v)$.  On the other hand, if $Bb\supseteq Y$, then $e_Yb=e_Y$ and so $f(bv) = g(e_Ybv)=g(e_Yv)=bf(v)$.  This completes the proof of the lemma.
\end{proof}

The next theorem is a fundamental result of this paper.  Throughout this text we follow the standard convention that
\[\til H^{-1}(X;\Bbbk)= \begin{cases} \Bbbk, & \text{if}\ X=\emptyset\\  0, & \text{else.}\end{cases}\]
We shall use throughout that if $f\colon X\to X$ is a regular cellular map, then the image of a subcomplex of $X$ under $f$ is a subcomplex.

\begin{Thm}\label{t:extthmnew}
Let $B$ be a connected left regular band and $\Bbbk$ a commutative ring with unit.  Let $X,Y\in \Lambda(B)$.  Fix $e_Y$ with $Be_Y=Y$.   Suppose that $B_{\geq X}\curvearrowright P$ is a semi-free action of $B_{\geq X}$ on an acyclic CW poset $P$ (by cellular maps).  Then
\begin{align*}
\Ext^n_{\Bbbk B}(\Bbbk_X,\Bbbk_Y) &= \begin{cases}H^n(\CW(e_YP),\CW((\bd e_YB_{\geq X})P);\Bbbk), & \text{if}\ Y\geq X\\ 0, & \text{else.}\end{cases} \\ &= \begin{cases} \Bbbk, &\text{if}\ X=Y,\ n=0 \\ \til H^{n-1}(\CW((\bd e_YB_{\geq X})P);\Bbbk), & \text{if}\ Y> X,\ n\geq 1\\ 0, & \text{else.}\end{cases}
\end{align*}
\end{Thm}
\begin{proof}
We compute $\Ext^n_{\Bbbk B}(\Bbbk_X,\Bbbk_Y)$ using the projective resolution \[C_\bullet(\CW(P);\Bbbk)\xrightarrow{\,\,\varepsilon\,\,}\Bbbk_X\] from Corollary~\ref{c:othersimples}.  Thus $\Ext^n_{\Bbbk B}(\Bbbk_X,\Bbbk_Y)=H^n(\Hom_{\Bbbk B}(C_{\bullet}(\CW(P);\Bbbk),\Bbbk_Y))$.

First suppose that $Y\ngeq X$.  Then $e_YC_q(\CW(P);\Bbbk)=0$ and so we have $\Hom_{\Bbbk B}(C_q(\CW(P);\Bbbk),\Bbbk_Y)=0$ for all $q\geq 0$ by Lemma~\ref{l:homtosimples}.  It follows that $\Ext^n_{\Bbbk B}(\Bbbk_X,\Bbbk_Y)=0$ for all $n\geq 0$ in this case.

Next suppose that $Y\geq X$.  Then observe that
\begin{align*}
e_YC_q(\CW(P);\Bbbk) &= C_q(\CW(e_YP);\Bbbk)\\ \Bbbk[\bd e_YB]C_q(\CW(P);\Bbbk) &= C_q(\CW((\bd e_YB_{\geq X})P);\Bbbk).
\end{align*}
Thus Lemma~\ref{l:homtosimples} identifies $\Hom_{\Bbbk B}(C_{\bullet}(\CW(P);\Bbbk),\Bbbk_Y))$ with the relative co\-chain complex $C^\bullet(\CW(e_YP),\CW(\bd e_YB_{\geq X}P);\Bbbk)$ and so \[\Ext^n_{\Bbbk B}(\Bbbk_X,\Bbbk_Y) = H^n(\CW(e_YP),\bd \CW(e_YB_{\geq X}P);\Bbbk)\] for all $n\geq 0$ in this case.  The first equality is now proved.

For the second equality, we recall that $e_YP$ is acyclic by Proposition~\ref{retract}.
In the case $X=Y$, we have $\bd e_YB_{\geq X}=\emptyset$ and so \[H^n(\CW(e_YP),\CW(\bd e_YB_{\geq X}P);\Bbbk) = H^n(\CW(e_YP);\Bbbk)\] which is $\Bbbk$ in dimension zero and zero in all other dimensions.

If $Y>X$, then $\bd e_YB_{\geq X}\neq \emptyset$.  The long exact sequence for reduced relative cohomology, together with the acyclicity of $e_YP$, then yields \[H^n(\CW(e_YP),\CW((\bd e_YB_{\geq X})P);\Bbbk)\cong \til H^{n-1}(\CW((\bd e_YB_{\geq X})P);\Bbbk)\] for all $n\geq 1$. We have  $H^0(\CW(e_YP),\CW((\bd e_YB_{\geq X})P);\Bbbk)=0$ because $\CW(e_YP)$ is connected and $\CW((\bd e_YB_{\geq X})P)$ is non-empty.
\end{proof}

Specializing to  the case that $P$ is the face poset of $\Delta(B_{\geq X})$ and $B$ a monoid recovers the main result of~\cite[Theorem~4.1]{oldpaper}, which was proved using classifying spaces of small categories, Quillen's Theorem~A~\cite{Quillen} and homological trickery.

\begin{Thm}\label{t:extthm}
Let $B$ be a connected left regular band and $\Bbbk$ a commutative ring with unit.  Let $X,Y\in \Lambda(B)$.  Fix $e_Y$ with $Be_Y=Y$.   Then
\begin{align*}
\Ext^n_{\Bbbk B}(\Bbbk_X,\Bbbk_Y) &= \begin{cases}H^n(\Delta(e_YB_{\geq X}),\Delta(\bd e_YB_{\geq X});\Bbbk), & \text{if}\ Y\geq X\\ 0, & \text{else.}\end{cases} \\ &= \begin{cases} \Bbbk, &\text{if}\ X=Y,\ n=0 \\ \til H^{n-1}(\Delta(\bd e_YB_{\geq X});\Bbbk), & \text{if}\ Y> X,\ n\geq 1\\ 0, & \text{else.}\end{cases}
\end{align*}
\end{Thm}
\begin{proof}
Take $P=\Delta(B_{\geq X})$ in Theorem~\ref{t:extthmnew} and use that $\Delta(B_{\geq X})$ is acyclic by Theorem~\ref{t:acyclicordercomplex}.
\end{proof}

As a corollary, we recover a description of the quiver of a left regular band  monoid algebra~\cite{Saliola,oldpaper}, now under the weaker assumption of connectedness of the left regular band.

\begin{Thm}\label{t:quiver}
Let $B$ be a connected left regular band and $\Bbbk$ a field.  Then the quiver of $\Bbbk B$ has vertex set $\Lambda(B)$.  There are no arrows $X\to Y$ unless $X<Y$, in which case the number of arrows $X\to Y$ is one fewer than the number of connected components of $\Delta(\bd e_YB_{\geq X})$ where $Y=Be_Y$.
\end{Thm}
\begin{proof}
This follows from Theorem~\ref{t:extthm} as $\dim \til H^0(K)$ is one fewer than the number of connected components of $K$ for a simplicial complex $K$.
\end{proof}

\begin{Cor}\label{c:suspension}
$B$ be a connected left regular band and $\Bbbk$ a commutative ring with unit.  Let $X,Y\in \Lambda(B)$.  Fix $e_Y$ with $Be_Y=Y$.   Let $-\infty$ be the minimum element of $\Lambda(\mathsf S(B))=\Lambda(B)\cup \{-\infty\}$.  Then
\begin{enumerate}
\item $\Ext^n_{\Bbbk \mathsf S(B)}(\Bbbk_X,\Bbbk_Y) =  \begin{cases} \Bbbk, &\text{if}\ X=Y,\ n=0 \\ \til H^{n-1}(\Delta(\bd e_YB_{\geq X});\Bbbk), & \text{if}\ Y> X,\ n\geq 1\\ 0, & \text{else;}\end{cases}$
\item $\Ext^n_{\Bbbk \mathsf S(B)}(\Bbbk_{-\infty},\Bbbk_Y) = \begin{cases}\til H^{n-2}(\Delta(\bd e_YB);\Bbbk), & \text{if}\ n \geq 1\\  0, & \text{else;}\end{cases}$
\item $\Ext^n_{\Bbbk \mathsf S(B)}(\Bbbk_{-\infty},\Bbbk_{-\infty}) = \begin{cases} \Bbbk, &\text{if}\ n=0 \\ 0, &\text{else;}\end{cases}$;
\item $\Ext^n_{\Bbbk \mathsf S(B)}(\Bbbk_X,\Bbbk_{-\infty})=0$ for all $X\in \Lambda(B)$.
\end{enumerate}
\end{Cor}
\begin{proof}
This follows from Theorem~\ref{t:extthm} and the top\-o\-lo\-gi\-cal fact (which can be easily proved via the Meyer-Vietoris sequence) that $\til H^{q+1}(\mathsf S(X);\Bbbk) \cong \til H^q(X;\Bbbk)$.
\end{proof}

Our next corollary concerns joins.
\begin{Cor}
Let $B'$ be a left regular band, $B$ be a connected left regular band  and $\Bbbk$ a field.
\begin{enumerate}
\item  If $X,Y\in \Lambda(B)$, then \[\Ext^n_{\Bbbk (B'\ast B)}(\Bbbk_X,\Bbbk_Y) =  \begin{cases} \Bbbk, &\text{if}\ X=Y,\ n=0 \\ \til H^{n-1}(\Delta(\bd e_YB_{\geq X});\Bbbk), & \text{if}\ Y> X,\ n\geq 1\\ 0, & \text{else.}\end{cases}\]
\item If $X,Y\in \Lambda(B')$, then \[\Ext^n_{\Bbbk (B'\ast B)}(\Bbbk_X,\Bbbk_Y) =  \begin{cases} \Bbbk, &\text{if}\ X=Y,\ n=0 \\ \til H^{n-1}(\Delta(\bd e_YB'_{\geq X});\Bbbk), & \text{if}\ Y> X,\ n\geq 1\\ 0, & \text{else.}\end{cases}\]
\item  If $X\in \Lambda(B')$ and $Y\in \Lambda(B)$, then \[\Ext^n_{\Bbbk (B'\ast B)}(\Bbbk_X,\Bbbk_Y) = \bigoplus_{p+q=n-2}\til H^p(\Delta(B'_{\geq X});\Bbbk)\otimes_\Bbbk \til H^q(\Delta(\bd e_YB);\Bbbk)\] if $n\geq 1$ and otherwise is zero.
\item  If $X\in \Lambda(B)$ and $Y\in \Lambda(B')$, then $\Ext^n_{\Bbbk (B'\ast B)}(\Bbbk_X,\Bbbk_Y) = 0$.
\end{enumerate}
\end{Cor}
\begin{proof}
This follows from Theorem~\ref{t:extthm}, the observation that \[\bd e_Y(B'\ast B)_{\geq X}=B'_{\geq X}\ast \bd e_YB'\] for $X\in \Lambda(B')$ and $Y\in \Lambda(B)$ and the fact that
\[\til H^{n+1}(X\ast Y;\Bbbk)\cong \bigoplus_{p+q=n}\til H^p(X;\Bbbk)\otimes_\Bbbk \til H^q(Y;\Bbbk)\] (cf.~\cite[Equation~(9.12)]{bjornersurvey}). Here again we use the usual conventions for cohomology in negative degrees.
\end{proof}

Theorem~\ref{t:extthm} allows us to compute the global dimension of $\Bbbk B$ when $\Bbbk$ is a field. This result was first obtained in~\cite[Theorem~4.4]{oldpaper} for left regular band monoids.

\begin{Cor}\label{c:globaldimension}
If $\Bbbk$ is a field and $B$ is a connected left regular band, then the global dimension of $\Bbbk B$ is the largest $n$ such that the reduced cohomology $\til H^{n-1}(\Delta(\bd e_YB_{\geq X});\Bbbk)\neq 0$ for some $X,Y\in \Lambda(B)$ with $Y\geq X$.
\end{Cor}

Corollary~\ref{c:globaldimension} was used in~\cite{oldpaper} to compute the global dimension of almost all the examples of left regular band monoids that we have been considering in this paper.   We recall briefly some of the results.  First we give an upper bound in terms of Leray numbers, which improves upon the upper bound that~\cite{Nico1,Nico2} would provide in the case of a left regular band.

\begin{Cor}\label{c:leraybound}
Let $B$ be a finite connected left regular band and $\Bbbk$ a field. Then
$\gldim \Bbbk B$ is bounded above by the $\Bbbk$-Leray number $L_{\Bbbk}(\Delta(B))$ of the order
complex of $B$.
\end{Cor}
\begin{proof}
This is immediate from Corollary~\ref{c:globaldimension} as $\Delta(\bd e_YB_{\geq X})$ is an induced subcomplex of $\Delta(B)$ for $X<Y$ in $\Lambda(B)$ and $Y=Be_Y$.
\end{proof}

A graph $\Gamma$ is \emph{chordal}\index{chordal} if it contains no induced cycle on four or more vertices.  It is well known (see~\cite[Proposition~4.8]{oldpaper} for a proof, but we claim no originality) that a simplicial complex has $\Bbbk$-Leray number zero if and only if it is a simplex and has $\Bbbk$-Leray number at most one if and only if it is the clique complex of a chordal graph.

A left regular band monoid is called \emph{right hereditary}\index{right hereditary} if its Hasse diagram is a rooted tree.  For instance, the free left regular band monoid and  the left regular bands associated by Brown to matroids~\cite{Brown1} and by Bj\"orner~\cite{bjorner2} to interval greedoids are right hereditary.  Further examples, coming from Karnofsky-Rhodes expansions of semilattices, are given in~\cite{oldpaper}.  In~\cite[Theorem~4.9]{oldpaper}, we verified that if $B$ is a right hereditary left regular band, then $\Delta(B)$ is the clique complex of a chordal graph and hence has Leray number at most one.  From this and Corollary~\ref{c:leraybound}, we deduce the following result, which is~\cite[Theorem~4.9]{oldpaper}.

\begin{Cor}
Let $B$ be a right hereditary left regular band monoid.  Then $\Bbbk B$ is hereditary for any field $\Bbbk$.
\end{Cor}

Theorem~\ref{t:extthm} and Proposition~\ref{p:topology} imply a tight connection between the homology of induced subcomplexes of clique complexes and the representation theory of free partially commutative left regular bands~\cite[Theorem~4.16]{oldpaper}.

\begin{Cor}\label{c:freepchom}
Let $\Gamma=(V,E)$ be a  finite simple graph, $\Bbbk$ a field and $B(\Gamma)$ the free partially commutative left regular band associated to $\Gamma$. Identifying $\Lambda(B(\Gamma))$ with the power set  of $V$, we have  \[\Ext^n_{\Bbbk B(\Gamma)}(\Bbbk_X,\Bbbk_Y)\cong \til{H}^{n-1}(\Cliq (\Gamma[X\setminus Y];\Bbbk))\]  for $X\supsetneq Y$ where $\Gamma[X\setminus Y]$ is the induced subgraph on the vertex set $X\setminus Y$.
\end{Cor}

Inspection of Corollary~\ref{c:freepchom} shows that the $\Bbbk$-Leray number the clique complex $\Cliq(\Gamma)$ is exactly the global dimension of $\Bbbk B(\Gamma)$.  This result first appeared as~\cite[Corollary~4.17]{oldpaper} and gives an interpretation of the Leray number in terms of non-commutative algebra.  The Leray number is usually given a commutative algebra interpretation as the Castelnuovo-Mumford regularity of the Stanley-Reisner ring of $\Cliq(\Gamma)$.

\begin{Cor}\label{Lerayofclique}
Let $\Gamma$ be a finite simple graph and $\Bbbk$ a field.  Denoting by $B(\Gamma)$ the free partially commutative left regular band associated to $\Gamma$, we have that the global dimension of $\Bbbk B(\Gamma)$ is precisely the $\Bbbk$-Leray number of the clique complex of $\Gamma$.
\end{Cor}

Next we turn to the global dimension of connected CW left regular bands.  This is a new result, which encompasses several results from~\cite{oldpaper} (where only monoids were considered).

\begin{Cor}\label{c:cwlrb}
Let $B$ be a connected left regular band and $X\in \Lambda(B)$ such that $B_{\geq X}$ is a CW poset. (Note that $\Lambda(B)_{\geq X}$ is then a graded poset by Proposition~\ref{p:supportiscellular} and Corollary~\ref{c:preservesgraded}.)
Then, for $X,Y\in \Lambda(B)$, we have that
\begin{equation}\label{extcomputationCW}
\dim_\Bbbk \Ext^q_{\Bbbk B}(\Bbbk_X,\Bbbk_Y) = \begin{cases} 1, & \text{if}\ X\leq Y\ \text{and}\ \rk[X,Y]=q\\ 0, & \text{else.}\end{cases}
\end{equation}

In particular, if $B$ is a connected CW left regular band, then \eqref{extcomputationCW} holds for all $X,Y\in \Lambda(B)$ and consequently\[\gldim \Bbbk B=\dim \CW (B)\] holds.
\end{Cor}
\begin{proof}
Suppose that $X\leq Y$. As $\|\Delta(\bd e_YB_{\geq X})\|$ is a sphere of dimension $\dim \Delta(e_YB_{\geq X})-1=\rk[X,Y]-1$ (where $S^{-1}=\emptyset$), we deduce~\eqref{extcomputationCW} from Theorem~\ref{t:extthm}.  Since the longest chain in $\Lambda(B)$ is of length $\dim \CW(B)$, the statement on global dimension for connected CW left regular bands is immediate from \eqref{extcomputationCW}.
\end{proof}

Corollary~\ref{c:cwlrb} covers the cases of face monoids of central hyperplane arrangements~\cite{Saliolahyperplane,oldpaper} and face monoids of complex hyperplane arrangements~\cite{oldpaper}.  The principal new examples in this paper are the semigroups of covectors of COMs, including the face semigroup associated to a $T$-convex set of topes, the face semigroup of an affine hyperplane arrangement, the semigroup associated to an affine oriented matroid and the face semigroup of a CAT(0) cube complex.  Let us record a few special cases of Corollary~\ref{c:cwlrb} in the following corollary.

\begin{Cor}\label{c:special.cwlrbs}
Let $\Bbbk$ be a field.
\begin{enumerate}
\item If $(E,\mathcal L)$ is a COM, then $\gldim \Bbbk \mathcal L=d$ where $d$ is the dimension of the regular CW complex with face poset $\mathcal L$.
\item If $\mathcal A$ is an essential central hyperplane arrangement in $\mathbb R^d$ with face monoid $\FFF(\mathcal A)$, then $\gldim \Bbbk\FFF(\mathcal A)=d$.
\item  Suppose that $\mathcal A$ is an essential affine hyperplane arrangement in $\mathbb R^d$. Then $\gldim \Bbbk \FFF(\mathcal A) = d$.
\item If $\mathcal A$ is an essential complex hyperplane arrangement in $\mathbb C^d$, then $\gldim \Bbbk \FFF(\mathcal A)=2d$.
\item If $(E,\mathcal L)$ is an oriented matroid, then $\gldim \Bbbk \mathcal L$ is the rank of $(E,\mathcal L)$.
\item If $(E,\mathcal L,g)$ is an affine oriented matroid with the rank of $(E,\mathcal L)$ equal to $d+1$, then $\gldim \Bbbk\mathcal L^+(g)=d$.
\item if $K$ is a $d$-dimensional finite CAT(0) cube complex with face semigroup $\FFF(K)$, then $\gldim \Bbbk\FFF(K)=d$.
\end{enumerate}
\end{Cor}

The reader should refer back to Section~\ref{ss:geom.regular} for the terminology in the next theorem.

\begin{Thm}
Let $B$ be a connected geometric left regular band and $X\in \Lambda(B)$. Let $\Delta_X$ be the simplicial complex constructed in Section~\ref{ss:geom.regular}. Then the augmented simplicial chain complex $C_\bullet(\Delta_X;\Bbbk)\to \Bbbk_X$ is a finite projective resolution by finitely generated $\Bbbk B$-modules. Moreover,  $C_0(\Delta_X;\Bbbk)\to \Bbbk_X$ is the projective cover of $\Bbbk_X$. In particular, the global dimension of $B$ is bounded above by one less than the largest size of an $\mathscr L$-class of $B$.
\end{Thm}
\begin{proof}
Since $B_{\geq X}\curvearrowright \Delta_X$ is semi-free and $\Delta_X$ is acyclic, the first statement follows from Theorem~\ref{t:projectiveresgeneral}.  The second statement follows because $C_0(\Delta_X;\Bbbk)\cong \Bbbk L_X$.  The final statement is a consequence of the inequality $\dim \Delta_X\leq |L_X|-1$.
\end{proof}

\subsection{Minimal projective resolutions}
\nomenclature[R, 10]{$P_\bullet \to M$}{projective resultion of the module $M$}%
Assume for this section that $\Bbbk$ is a field. Let $A$ be a finite dimensional $\Bbbk$-algebra and let $M$ be a finite dimensional $A$-module.  A projective resolution $P_\bullet\to M$ is said to be \emph{minimal}\index{minimal!projective resolution}\index{resolution!minimal projective} if, for each $q>0$, one has $d(P_q)\subseteq \rad(P_{q-1})$ or, equivalently, $d\colon P_{q-1}\to d(P_{q-1})$ is a projective cover. Every finite dimensional module admits a minimal projective resolution, which is unique up to isomorphism (cf.~Proposition~\ref{minresolution} below).  The following well-known proposition is stated in the context of group algebras in~\cite[Proposition~3.2.3]{cohomologyringbook}, but the proof given there is valid for any finite dimensional algebra.  We shall sketch the proof of the equivalence of (1)--(3) for the reader's convenience, as these are the only parts we shall need.

\begin{Prop}\label{minresolution}
Let $A$ be a finite dimensional $\Bbbk$-algebra, $M$ a finite dimensional $A$-module and $P_\bullet\to M$ a projective resolution.  Then the following are equivalent.
\begin{enumerate}
\item $P_\bullet\to M$ is a minimal resolution of $M$.
\item $\Hom_A(P_q,S)\cong\Ext_A^q(M,S)$ for any $q\geq 0$ and simple $A$-module $S$.
\item The coboundary map $d^*\colon \Hom_A(P_q,S)\to \Hom_A(P_{q+1},S)$ is the zero map for all simple $A$-modules $S$ and $q\geq 0$.
\item If $Q_\bullet\to M$ is a projective resolution of $M$, then the chain map $Q_\bullet\to P_\bullet$ lifting the identity map on $M$ is surjective in each degree $q\geq 0$.
\item If $Q_\bullet\to M$ is a projective resolution of $M$, then the chain map $P_\bullet\to Q_\bullet$ lifting the identity map on $M$ is injective in each degree $q\geq 0$.
\end{enumerate}
\end{Prop}
\begin{proof}
The equivalence of (1) and (3) follows because $d(P_q)\subseteq \rad(P_{q-1})$ if and only if  $fd=0$ for all homomorphisms $f\colon P_{q-1}\to S$ with $S$ simple, i.e., if and only if $d^*(f)=0$ for all $f\in \Hom_A(P_{q-1},S)$ with $S$ simple.  Trivially, (3) implies (2) as $\Ext_A^\bullet(M,S)$ is the cohomology of the co-chain complex $\Hom_A(P_\bullet,S)$.  Conversely, (2) implies (3) because if $d^*\colon \Hom_A(P_q,S)\to \Hom_A(P_{q+1},S)$ is non-zero, then we have $\dim \Ext_A^q(M,S)\leq \dim \ker d^*<\dim \Hom_A(P_q,S)$.
\end{proof}

The last two items of the proposition explain why minimal resolutions are called ``minimal.'' The second item is a convenient way to test if a resolution is minimal.  In the case that $A$ is split over $\Bbbk$, it can be reformulated as saying that the number of occurrences of the projective cover of a simple module $S$ as a direct summand in $P_q$ is $\dim_\Bbbk \Ext_A^q(M,S)$. Indeed, $\Hom_A(P_q,S)=\Hom_{A/\rad(A)}(P_q/\rad(P_q),S)$ and so the claim follows from Schur's lemma because $S$ is absolutely irreducible.  We shall make use of this formulation of (2).

\begin{Thm}\label{minresoftrivial}
Let $B$ be a connected left regular band, $X\in \Lambda(B)$ and $\Bbbk$ be a field. Suppose that $B_{\geq X}$ is a  CW poset.  Then the augmented cellular chain complex $C_\bullet(\CW(B_{\geq X});\Bbbk)\to \Bbbk_X$ is the minimal projective resolution of $\Bbbk_{X}$.
\end{Thm}
\begin{proof}
It follows from Corollary~\ref{c:othersimplesCWcontraction} that $C_\bullet(\CW(B_{\geq X});\Bbbk)\to \Bbbk_X$ is a projective resolution and that the number of occurrences of the Sch\"utzenberger representation $\Bbbk L_Y$ associated to $Y$ in $C_q(\CW(B_{\geq X});\Bbbk)$ is $1$ if $Y\geq X$ and $\rk[X,Y]=q$, and otherwise is $0$.  On the other hand, $\dim_\Bbbk \Ext_{\Bbbk B}^q(\Bbbk_X,\Bbbk_Y)=1$ if $Y\geq X$ and $\rk[X,Y]=q$, and otherwise is $0$ by Corollary~\ref{c:cwlrb}.  We conclude that $C_\bullet(\CW(B_{\geq X});\Bbbk)\to \Bbbk_X$ is minimal by Proposition~\ref{minresolution}(2).
\end{proof}

As a corollary, we obtain minimal projective resolutions for all the simple modules in the case of a connected CW left regular band.

\begin{Cor}
Let $B$ be a connected CW left regular band and let $\Bbbk$ be a field. Then the augmented cellular chain complex $C_\bullet(\CW(B_{\geq X});\Bbbk)\to \Bbbk_X$ is the minimal projective resolution of $X$ for all $\Bbbk_X\in \Lambda(B)$.
\end{Cor}

In particular, the corollary applies to the case where $B$ is  the face monoid of a real or complex hyperplane arrangement,  the monoid of covectors of an oriented matroid or oriented interval greedoid, the face semigroup of an affine hyperplane arrangement or CAT(0) cube complex, the semigroup of covectors associated to an affine oriented matroid or, more generally, the semigroup of covectors of a COM.

\section{Quiver presentations}\label{s:quiverpres}
Let us fix for the discussion a finite connected left regular band $B$ and a field $\Bbbk$. In this section we will compute a quiver presentation for $\Bbbk B$ under some very strong hypotheses on $B$ that will be satisfied by connected CW left regular bands.

\subsection{A general result}
Recall that the quiver of $\Bbbk B$ is denoted $Q(\Bbbk B)$. It has vertex set $\Lambda(B)$.  If $X<Y$, then the number of directed edges from $X$ to $Y$ is one less than the number of connected components of $\Delta(\bd e_YB_{\geq X})$ by Theorem~\ref{t:extthm};  there are no other arrows. In particular, $Q(\Bbbk B)$ is acyclic and is independent of $\Bbbk$. Fix a complete set of primitive orthogonal idempotents $\{\eta_X\mid X\in \Lambda(B)\}$ constructed  as per Section~\ref{sssec:orthogonal-idempotents}.

Let us assume the following four properties of $B$.
\begin{itemize}
\item [(B1)] $Q(\Bbbk B)$ has no multiple edges.
\item [(B2)] $Q(\Bbbk B)$ is graded in the sense that all paths between two given vertices $X<Y$ have the same length. Set $\rk[X,Y]$ to be the common length of all paths from $X$ to $Y$.
\item [(B3)] \[\dim_\Bbbk \Ext^2_{\Bbbk B}(\Bbbk_X,\Bbbk_Y) = \begin{cases} 1, & \text{if}\ \rk[X,Y]=2\\ 0, & \text{else.}\end{cases}\]
\item [(B4)] There is a $\Bbbk B$-module homomorphism $\partial\colon \Bbbk B\to \Bbbk B$ such that:
\begin{itemize}
\item [(a)] $\partial^2=0$;
\item [(b)] $\partial(\eta_X)\eta_X=0$, for all $X\in \Lambda(B)$;
\item [(c)] $\partial(\eta_Y)\eta_X\neq 0$ if $X\to Y$ is an arrow of $Q(\Bbbk B)$.
\end{itemize}
\end{itemize}
We remark that under these assumptions, $Q(\Bbbk B)$ is the Hasse diagram of a  partial order $\preceq$ on $\Lambda(B)$ that is coarser than the original ordering. More precisely, we put $X\preceq Y$ if there is a path from $X$ to $Y$ in $Q(\Bbbk B)$ and note that $X\preceq Y$ implies $X\leq Y$ and that $(\Lambda(B),\preceq)$ is a graded poset. As usual, let us write $X\prec Y$ if $X\preceq Y$ and $X\neq Y$.

\begin{Thm}\label{t:quiverpres}
Suppose that $B$ is a connected left regular band satisfying (B1)--(B4).  Then $\Bbbk B$ has a quiver presentation $(Q(\Bbbk B),I)$ where $I$ is given by the following minimal system of relations $R$. For each $X<Y$ with $\rk[X,Y]=2$, we have a relation \[r_{X,Y}=\sum_{X\prec Z\prec Y}(X\to Z\to Y)\] in $R$.  In particular, $I$ is a homogeneous ideal  and $\Bbbk B\cong \Bbbk Q(\Bbbk B)/I$ is a quadratic algebra.
\end{Thm}
\begin{proof}
Set $Q=Q(\Bbbk B)$ and
suppose that there is an edge $X\to Y$ in $Q$.  Then, by (B2), there are no paths of length greater than one in $Q(\Bbbk B)$ from $X$ to $Y$ and so $\eta_Y\rad^2(\Bbbk B)\eta_X=0$.  Therefore, we have by (B1) that
\begin{align*}
1&=\dim_\Bbbk \Ext^1_{\Bbbk B}(\Bbbk_X,\Bbbk_Y)= \dim_\Bbbk \eta_Y[\rad(\Bbbk B)/\rad^2(\Bbbk B)]\eta_X \\ &= \dim_\Bbbk \eta_Y\Bbbk B\eta_X.
\end{align*}
Observe that, by (B4), $0\neq \partial(\eta_Y)\eta_X = \eta_Y\partial(\eta_Y)\eta_X$ and so $\{\partial(\eta_Y)\eta_X\}$ is a basis for $\eta_Y\Bbbk B\eta_X\cong \eta_Y[\rad(\Bbbk B)/\rad^2(\Bbbk B)]\eta_X$.  Thus we have a surjective homomorphism $\p\colon \Bbbk Q\to \Bbbk B$ given by $\p(\varepsilon_X)=\eta_X$ and $\p(X\to Y) = \partial(\eta_Y)\eta_X$ with $\ker \p$ admissible by the general theory of split basic finite dimensional algebras; see the discussion in Section~\ref{ss:finite.dim}.

Next we show that $r_{X,Y}\in \ker \p$ if $\rk[X,Y]=2$. Note that
\begin{equation}\label{imageofrelation}
\p(r_{X,Y}) = \sum_{X\prec Z\prec Y} \partial(\eta_Y)\eta_Z\partial(\eta_Z)\eta_X =\sum_{X\prec Z\prec Y} \eta_Y\partial(\eta_Y)\eta_Z\partial(\eta_Z)\eta_X.
\end{equation}
Notice that if $Z\in \Lambda(B)$ and $X\prec Z\prec Y$ is false, then \[\eta_Y\partial(\eta_Y)\eta_Z\partial(\eta_Z)\eta_X=0.\]  Indeed, if $Z\notin \{X,Y\}$, then there is no path in $Q$ from $X$ to $Y$ through $Z$ by definition of $\preceq$ and so $\eta_Y\Bbbk B\eta_Z\Bbbk B\eta_X=0$.  If $Z\in \{X,Y\}$ we use that $\partial(\eta_X)\eta_X=0=\partial(\eta_Y)\eta_Y$ by (B4). Thus the right hand side of \eqref{imageofrelation} is equal to
\begin{align*}
\sum_{Z\in \Lambda(B)} \eta_Y\partial(\eta_Y)\eta_Z\partial(\eta_Z)\eta_X &= \sum_{Z\in \Lambda(B)}\partial[\partial(\eta_Y)\eta_Z]\eta_X\\ &= \partial\left[\partial(\eta_Y)\cdot\sum_{Z\in \Lambda(B)}\eta_Z\right]\eta_X\\ &=\partial^2(\eta_Y)\eta_X \\ &=0
\end{align*}
using that $\partial^2=0$ and that $\sum_{Z\in \Lambda(B)}\eta_Z=1$.
This proves that $\p(r_{X,Y})=0$.

To complete the proof, let $R'$ be a minimal system of relations for $I=\ker \p$. We claim that \[R'=\{c_{X,Y}\cdot r_{X,Y}\mid \rk[X,Y]=2\}\] where each $c_{X,Y}\in \Bbbk\setminus \{0\}$.  It will then follow that the $r_{X,Y}$ also form a minimal system of relations for $I$.

By Theorem~\ref{t:bongartz} and (B3) we have that $R' = \{s_{X,Y}\mid \rk[X,Y]=2\}$ where $s_{X,Y}$ is a non-zero linear combination of paths from $X$ to $Y$ in $Q$, necessarily of length $2$ by (B2). It follows that if $\rk[X,Y]=2$, then
\begin{equation*}
r_{X,Y} = \sum_{\rk[U,V]=2}p_{U,V}s_{U,V}q_{U,V}
\end{equation*}
where $p_{U,V}\in \varepsilon_Y\Bbbk Q\varepsilon_V$ and $q_{U,V}\in \varepsilon_U\Bbbk Q\varepsilon_X$.  But since each path in $r_{X,Y}$ has length $2$ and each path in $s_{U,V}$ has length $2$, we conclude that $p_{U,V}s_{U,V}q_{U,V}=0$ if $(U,V)\neq (X,Y)$ and that $r_{X,Y}=k_{X,Y}\cdot s_{X,Y}$ for some $k_{X,Y}\in \Bbbk\setminus 0$.  This establishes the claim and completes the proof.
\end{proof}

\subsection{A quiver presentation for CW left regular bands}
One of the main results of this text is the next theorem, which computes a quiver presentation of the algebra of a connected CW left regular band, generalizing a result of the second author for face monoids of central hyperplane arrangements~\cite{Saliolahyperplane}.

\begin{Thm}\label{t:crucial}
Suppose that $B$ is a connected CW left regular band.
\begin{itemize}
\item [(a)] The quiver $Q=Q(\Bbbk B)$ of $\Bbbk B$ is the Hasse diagram of $\Lambda(B)$.
\item [(b)] $\Lambda(B)$ is graded.
\item [(c)] $\Bbbk B$ has a quiver presentation $(Q,I)$ where $I$ has minimal system of relations \[r_{X,Y} = \sum_{X<Z<Y} (X\to Z\to Y)\] ranging over all rank $2$ intervals $[X,Y]$ in $\Lambda(B)$.
\end{itemize}
\end{Thm}
\begin{proof}
The fact that $\Lambda(B)$ is graded is immediate from  Proposition~\ref{p:supportiscellular} and Corollary~\ref{c:preservesgraded} because $B$ is graded, being a CW poset.

By taking $q=1$ in \eqref{extcomputationCW} from Corollary~\ref{c:cwlrb}, we see that there is an edge $X\to Y$ in $Q$ if and only if $X<Y$ is a cover relation in $\Lambda(B)$ and that there are no multiple edges. This establishes (a).

To prove the remainder of the theorem, it suffices to show that Theorem~\ref{t:quiverpres} applies to $B$. It is immediate from the previous discussion that (B1)--(B2) are satisfied. Property (B3) is a consequence of \eqref{extcomputationCW} from Corollary~\ref{c:cwlrb}.  So it remains to verify that (B4) holds.

Choose for each $X\in \Lambda(B)$ an idempotent $e_X$ with $Be_X=X$ and let $\{\eta_X\mid X\in \Lambda(B)\}$ be the complete set of orthogonal primitive idempotents for $\Bbbk B$ from Section~\ref{sssec:orthogonal-idempotents}.
Let, as usual, $\Bbbk L_X$ denote the Sch\"utzenberger representation associated to the $\mathscr L$-class of $e_X$. Then as a left $\Bbbk B$-module we have
\begin{equation}\label{eq:dopresentation}
\Bbbk B\cong \bigoplus_{X\in \Lambda(B)} \Bbbk L_X\cong \bigoplus_{q\geq 0} C_q(\CW(B);\Bbbk)
\end{equation}
where the last isomorphism follows from \eqref{e:schutzdecompositioncwcase} of Corollary~\ref{c:othersimplesCWcontraction} with $X=\wh 0$.  Thus we can define a $\Bbbk$-module homomorphism $\partial\colon \Bbbk B\to \Bbbk B$ by transporting the differential \[d\colon \bigoplus_{q\geq 0} C_q(\CW(B);\Bbbk)\to \bigoplus_{q\geq 0} C_q(\CW(B);\Bbbk)\] under the isomorphism between $\Bbbk B$ and $\bigoplus_{q\geq 0} C_q(\CW(B);\Bbbk)$. Clearly, we have $\partial^2=0$.

Let us make the definition of $\partial$ more explicit. We recall from Corollary~\ref{basisofidempotents} that $\{b\eta_X\mid \sigma(b)=X\}$ is a basis for $\Bbbk B$ and that $\Bbbk L_X\cong \Bbbk B\eta_X$.  Moreover, there is a direct sum decomposition\[\Bbbk B=\bigoplus_{X\in \Lambda(B)} \Bbbk B\eta_X.\]  Identifying $q$-cells of $B$ with basis elements of $C_q(\CW(B);\Bbbk)$, we have \[d(b) = \sum [b:a]a\] where the sum runs over all maximal proper faces $a$ of $b$ and the $[b:a]=\pm 1$ are the incidence numbers associated to the boundary maps of $C_\bullet(\CW(B);\Bbbk)$ with respect to the orientation chosen in the proof of Theorem~\ref{t:projectiveresgeneral} (see also Remark~\ref{r:cell.case}). See~\cite[Chapter~IX, \S 7]{Massey} for details on incidence numbers for regular cell complexes.  Since the isomorphism in \eqref{eq:dopresentation} sends the cell corresponding to $b\in B$ to $b\eta_{\sigma(b)}\in \Bbbk B$, we then have, for $b\in B$, that
\[\partial(b\eta_{\sigma(b)}) = \sum [b:a]a\eta_{\sigma(a)}\]
where the sum runs over all maximal proper faces $a$ of $b$ (and, in particular, is $0$ if $b$ is a $0$-cell of $\CW(B)$).

To complete the proof, first note that \[\partial(\eta_X)\eta_X=\partial(e_X\eta_X)\eta_X=\sum [e_X:a]a\eta_{\sigma(a)}\eta_X=0\] because each maximal proper face $a$ of $e_X$ satisfies $\sigma(a)<X$ and so $\eta_{\sigma(a)}\eta_X=0$ by orthogonality (cf.~Theorem~\ref{primidempotentprops}).
On the other hand, if $Y$ covers $X$ in $\Lambda(B)$, then $\bd e_YB_{\geq X}\cong S^0$ because $B_{\geq X}$ is a CW poset and $e_Y$ is a $1$-cell of $B_{\geq X}$ (as $\sigma\colon B_{\geq X}\to \Lambda(B)_{\geq X}$ is rank preserving).  Therefore, $e_Y$ has exactly two maximal proper faces $a_1,a_2$ with support $X$.  Thus
\begin{align*}
\partial(\eta_Y)\eta_X &= \partial(e_Y\eta_Y)\eta_X= \sum [e_Y:a]a\eta_{\sigma(a)}\eta_X \\ &= [e_Y:a_1]a_1\eta_X+[e_Y:a_2]a_2\eta_X\neq 0.
\end{align*}
An application of Theorem~\ref{t:quiverpres} completes the proof of the theorem.
\end{proof}

\begin{Example}[Ladders]\label{ladders2}
Let $L_n$ be the ladder left regular band from Example~\ref{ladders}.  Theorem~\ref{t:crucial} shows that the quiver of $\Bbbk L_n$ is a directed path with $n+1$ vertices (that is, the Dynkin quiver $A_{n+1}$ oriented as a directed path) and that the relations declare all paths of length two to be zero. In other words, $\Bbbk L_n$ is the quadratic monomial algebra $\Bbbk A_{n+1}/\rad^2(\Bbbk A_{n+1})$.
\end{Example}

We shall later use Theorem~\ref{t:crucial} to show that the algebra of the face semigroup of a CAT(0) cube complex is isomorphic to the incidence algebra of its intersection semilattice.

An interesting question is whether $\mathbb ZB$ is isomorphic to the quotient of the integral path ring of the Hasse diagram of $\Lambda(B)$ modulo the ideal generated by the sum of all paths of length $2$ when $B$ is a connected CW left regular band.

\section{Quadratic and Koszul duals}\label{s:koszul}
In this section we consider quadratic and Koszul duals for connected CW left regular band algebras.

\subsection{Quadratic duals}
Let $P$ be a graded poset with Hasse diagram $Q$.  Let $I$ of $\Bbbk Q$ be the ideal generated by the elements
\begin{equation}\label{e:ourrelations}
r_{\sigma,\tau} = \sum_{\sigma<\gamma<\tau} (\sigma\to \gamma\to \tau), \quad \text{with}\ \rk[\sigma,\tau]=2.
\end{equation}
Motivated by Theorem~\ref{t:quiverpres}, we are led to consider the following question: what is the quadratic dual of $A=\Bbbk Q/I$?

In our examples the poset $P$ will also have an additional property termed strong connectedness in~\cite{abstractpolytope}.

A simplicial complex $K$ is called a \emph{chamber complex}\index{chamber complex} if it is pure and any two facets can be connected by a gallery.  Recall that if $C,C'$ are facets, then a \emph{gallery}\index{gallery} from $C$ to $C'$ is a sequence $C=C_0,\ldots,C_n=C'$ of facets such that $C_i,C_{i+1}$ are distinct and adjacent, where  $C_i,C_{i+1}$ are \emph{adjacent}\index{adjacent} if they have a common codimension one face. The reader is referred to~\cite[Appendix~A.1]{Brown:book2} for details on chamber complexes.  Crucial for us is the fact that every connected triangulated manifold without boundary is a chamber complex (cf.~\cite[Exercise~A.17]{Brown:book2}).

Notice that if $P$ is a graded poset with $\Delta(P)$ a pure simplicial complex, then two facets $\sigma_0<\cdots<\sigma_n$ and $\tau_0<\cdots<\tau_n$ of $\Delta(P)$ are adjacent if and only if there exists $0\leq i\leq n$ with $\sigma_j=\tau_j$ for $j\neq i$ and $\sigma_i\neq \tau_i$;  see Figure~\ref{f:galleryordercomplex}.
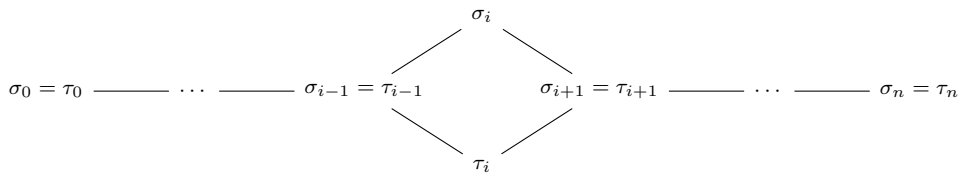
\begin{figure}[tbhp]
\begin{center}
\begin{tikzpicture}[scale=.8]
\tikzset{every node/.style={font=\scriptsize}}
\node  	(A) 								{$\sigma_0=\tau_0$};
\node  	(B) [right= 1cm of A] 		        {$\cdots$};
\node  	(C) [right= 1cm of B] 	            {$\sigma_{i-1}=\tau_{i-1}$};
\node   (F) [right= .5cm of C]               {};
\node 	(D) [above of=F]			        {$\sigma_i$};
\node 	(E) [below of=F]			        {$\tau_i$};
\node   (G) [right= .5cm of F]               {$\sigma_{i+1}=\tau_{i+1}$};
\node   (H) [right= 1cm of G]               {$\cdots$};
\node   (I) [right= 1cm of H]               {$\sigma_n=\tau_n$};
\draw (A)--(B);
\draw (B)--(C);
\draw (C)--(D);
\draw (C)--(E);
\draw (D)--(G);
\draw (E)--(G);
\draw (G)--(H);
\draw (H)--(I);
\end{tikzpicture}
\end{center}
\caption{Adjacent facets in an order complex\label{f:galleryordercomplex}}
\end{figure}

\begin{Lemma}\label{l:chamberpush}
Let $f\colon P\to Q$ be a surjective cellular map of posets such that $\sigma<\sigma'$ implies $f(\sigma)<f(\sigma')$ and suppose that $\Delta(P)$ is a chamber complex.  Then $\Delta(Q)$ is a chamber complex of the same dimension.
\end{Lemma}
\begin{proof}
Corollary~\ref{c:preservesgraded} and its proof imply that $\Delta(Q)$ is pure, $f\colon \Delta(P)\to \Delta(Q)$ preserves dimensions of simplices and $f$ maps the set of facets of $\Delta(P)$ onto the set of facets of $\Delta(Q)$. Moreover, given a pair of adjacent facets of $\Delta(P)$, either $f$ collapses them or takes them to adjacent facets.  The lemma follows easily from these observations.
\end{proof}

The following definition is from~\cite{abstractpolytope}.

\begin{Def}[Strongly connected]
Let $P$ be a graded poset. Then $P$ is said to be \emph{strongly connected}\index{strongly connected poset}\index{poset!strongly connected} if $\Delta([\sigma,\tau])$ is a chamber complex for each closed interval $[\sigma,\tau]$ of $P$.
\end{Def}

For example,  any semimodular lattice is strongly connected (see the proof of~\cite[Proposition 9.6]{Saliolahyperplane}).  Notice that $P$ is strongly connected if and only if $P^{op}$ is strongly connected.

\begin{Rmk}\label{r:coning}
Observe that if $K$ is a chamber complex, then a cone $K\ast x$ on $K$ with cone point $x$ is also a chamber complex.  Indeed, the map $F\mapsto F\cup \{x\}$ gives a bijection between facets of $K$ and facets of $K\ast x$ which raises the dimension by $1$.  Moreover, this bijection preserves adjacency.
\end{Rmk}

The reader should recall the definition of the \emph{incidence algebra}\index{incidence algebra} $I(P;\Bbbk)$ of a poset $P$ from Section~\ref{ss:jacobson}.  The following lemma provides a quiver presentation for the incidence algebra of a strongly connected graded poset.  It is closely connected to~\cite[Proposition~B.10]{THA}, which implies the lemma and has virtually the same proof; our proof predates the availability of~\cite{THA}.

\begin{Lemma}\label{l:strongconnposet}
Let $P$ be a strongly connected graded poset with Hasse diagram $Q$ and let $\Bbbk$ be a commutative ring with unit.  Then there is an isomorphism $I(P;\Bbbk)\cong \Bbbk Q/I$ where $I$ is generated by all differences $p-q$ of parallel paths $p,q$ in $Q$ of length $2$.
\end{Lemma}
\begin{proof}
First note that $I(P;\Bbbk)\cong \Bbbk Q/I'$ where $I'$ is spanned by all differences $p-q$ of parallel paths in $Q$.  Thus $I\subseteq I'$.  We claim that  $I=I'$.

Notice that paths from $\sigma$ to $\tau$ in $Q$ correspond to facets of $\Delta([\sigma,\tau])$.  If $p,q\colon \sigma\to \tau$ are two parallel paths in $Q$ corresponding to distinct adjacent facets in $\Delta([\sigma,\tau])$, then we can write $p=uav$ and $q=ubv$ where $a,b$ are parallel paths of length $2$; see Figure~\ref{f:adjacentfacetsinquiver}.
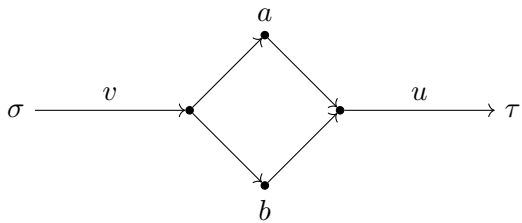
\begin{figure}[tbhp]
\begin{center}
\begin{tikzpicture}[->,vertices/.style={draw, fill=black, circle, inner sep=1pt}]
\node                                   (A) 						        {$\sigma$};
\node[vertices]  	                    (B) [right=2cm of A] 		        {};
\node                                   (C) [right of=B] 	                {};
\node[vertices,label=above:{$a$}]       (D) [above of=C]                    {};
\node[vertices,label=below:{$b$}] 	    (E) [below of=C]			        {};
\node[vertices]                         (F) [right of=C]                    {};
\node                                   (G) [right=2cm of F]                {$\tau$};
\path (A) edge node [above] {$v$}   (B)
      (B) edge                      (D)
      (B) edge                      (E)
      (D) edge                      (F)
      (E) edge                      (F)
      (F) edge node [above] {$u$}   (G);
\end{tikzpicture}
\end{center}
\caption{Parallel paths in $Q$ corresponding to adjacent facets\label{f:adjacentfacetsinquiver}}
\end{figure}
Thus $p-q=u(a-b)v\in I$.

Now suppose that $p,q\colon \sigma\to \tau$ are arbitrary parallel paths in $Q$.  Then since $\Delta([\sigma,\tau])$ is a chamber complex we can find a sequence $p=p_0,\ldots, p_n=q$ of paths from $\sigma$ to $\tau$ such that $p_i$ and $p_{i+1}$ represent distinct adjacent facets of $\Delta([\sigma,\tau])$.  Then $p-q = (p_0-p_1)+(p_1-p_2)+\cdots + (p_{n-1}-p_n)\in I$. It follows that $I=I'$, as required.  This completes the proof.
\end{proof}

The following generalizes~\cite[Proposition 9.6]{Saliolahyperplane} from the case of central hyperplane arrangements.

\begin{Thm}\label{quadraticdual}
Let $P$ be a strongly connected graded poset.  Let $Q$ be the Hasse diagram of $P$ and let $A=\Bbbk Q/I$ where $I$ is the ideal generated by the system of relations \eqref{e:ourrelations}.  Then the quadratic dual $A^!$ is isomorphic to the incidence algebra $I(P^{op};\Bbbk)$ of the opposite poset of $P$.
\end{Thm}
\begin{proof}
Clearly $I_2^{\perp}$ is spanned by all differences $p-q$ of parallel paths of length $2$ in $Q^{op}$. As $P^{op}$ is a strongly connected graded poset, the result follows from Lemma~\ref{l:strongconnposet}.
\end{proof}

In order to apply Theorem~\ref{quadraticdual} to connected CW left regular bands, we need that $\Lambda(B)$ is a strongly connected graded poset.

\begin{Lemma}\label{l:CWsupp.strong.conn}
Let $B$ be a connected CW left regular band.  Then $\Lambda(B)$ is a strongly connected graded poset.
\end{Lemma}
\begin{proof}
Theorem~\ref{t:crucial} establishes that $\Lambda(B)$ is graded.  Let $[X,Y]$ be a closed interval of $\Lambda(B)$.  Choose an idempotent $e_Y$ with $Be_Y=Y$.  Then $[X,Y]$ is the support lattice of $e_YB_{\geq X}$.  Note that $\Delta(e_YB_{\geq X})$ is the cone \[\Delta(\bd e_YB_{\geq X})\ast e_Y\] and that $\Delta(\bd e_YB_{\geq X})$ is a chamber complex, being a triangulation of a sphere.  Thus $\Delta(e_YB_{\geq X})$ is a chamber complex by Remark~\ref{r:coning}. We conclude using Proposition~\ref{p:supportiscellular} and Lemma~\ref{l:chamberpush} that $\Delta([X,Y])$ is a chamber complex and hence $\Lambda(B)$ is strongly connected.
\end{proof}

\begin{Thm}\label{t:quadraticdual}
Let $B$ be a connected CW left regular band (and so in particular $\Bbbk B$ is a quadratic algebra with the quiver presentation from Theorem~\ref{t:crucial}). Then the quadratic dual $\Bbbk B^!$ is isomorphic to $I(\Lambda(B)^{op};\Bbbk)$.
\end{Thm}
\begin{proof}
The theorem is an immediate consequence of Theorem~\ref{t:crucial}, Theorem~\ref{quadraticdual} and Lemma~\ref{l:CWsupp.strong.conn}.
\end{proof}

We next aim to give a necessary condition for the algebra of a connected CW left regular band to be an incidence algebra.  Consequently, most of our CW left regular bands are not isomorphic to incidence algebras.  We will also show that if $K$ is a CAT(0) cube complex, then the algebra of its face semigroup $\FFF(K)$ is isomorphic to the incidence algebra of its intersection semilattice $\mathcal L(K)$.  The same is true, more generally, for any lopsided system.

A graded poset is said to be \emph{thin}\index{thin poset}\index{poset!thin} (cf.~\cite[Section~4.7]{OrientedMatroids1999}) if each interval $[x,y]$ of rank $2$ has cardinality $4$; see Figure~\ref{f:thin}.
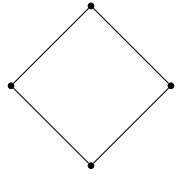
\begin{figure}[tb]
\begin{center}
\begin{tikzpicture}[vertices/.style={draw, fill=black, circle, inner sep=0.75pt}]
\node[vertices]  	(A) 								{};
\node[vertices]  	(B) [below left =of A] 		{};
\node[vertices]  	(C) [below right=of A] 	{};
\node[vertices] 	(D) [below left=of C]			{};
\draw (A)--(B);
\draw (A)--(C);
\draw (B)--(D);
\draw (C)--(D);
\end{tikzpicture}
\end{center}
\caption{A rank $2$ interval in a thin poset\label{f:thin}}
\end{figure}
For example, if $X$ is a simplicial complex and $P$ is the face poset of $X$ with an adjoined empty face, then $P$ is thin.  This is because each subposet $P_{\leq p}$ is a boolean lattice and boolean lattices are thin.

\begin{Thm}\label{t:CWLRB.thin}
Let $B$ be a connected CW left regular band and $\Bbbk$ a field.  If $\Bbbk B$ is isomorphic to the incidence algebra of a poset, then $\Lambda(B)$ is thin.
\end{Thm}
\begin{proof}
Let $Q$ be the Hasse diagram of $\Lambda(B)$.  If $P$ is a poset, then the quiver of the incidence algebra of $P$ is the Hasse diagram of $P$. Since $Q$ is the quiver of $\Bbbk B$ by Theorem~\ref{t:crucial}, we conclude that if $\Bbbk B$ is isomorphic to an incidence algebra, then it is isomorphic to the incidence algebra of $\Lambda(B)$. Let us assume that this is indeed the case.  We show that $\Lambda(B)$ is thin.

Let $[X,Y]$ be an interval of rank $2$ in $\Lambda(B)$. Let $\{\eta_X\mid X\in \Lambda(B)\}$ be a complete set of orthogonal primitive idempotents constructed as per Section~\ref{sssec:orthogonal-idempotents} and let $I$ be the admissible ideal from Theorem~\ref{t:crucial}.  Then $\dim \eta_Y\Bbbk B\eta_X=1$ because $I(\Lambda(B);\Bbbk)$ is the quotient of $\Bbbk Q$ by the ideal $I'$ generated by differences of parallel paths and so $\dim \varepsilon_Y[\Bbbk Q/I']\varepsilon_X=1$. Thus $\varepsilon_YI\varepsilon_X$ has codimension $1$ in $\varepsilon_Y\Bbbk Q\varepsilon_X$.  On the other hand, we claim $\dim \varepsilon_YI\varepsilon_X=1$.  Indeed, retaining the notation of Theorem~\ref{t:crucial}, if $r\in \varepsilon_YI\varepsilon_X$, then
\begin{equation*}
r= \sum_{\rk[U,V]=2}p_{U,V}r_{U,V}q_{U,V}
\end{equation*}
where $p_{U,V}\in \varepsilon_Y\Bbbk Q\varepsilon_V$ and $q_{U,V}\in \varepsilon_U\Bbbk Q\varepsilon_X$.  But since each path in $r$ has length $2$ and each path in $r_{U,V}$ has length $2$, it follows that $p_{U,V}r_{U,V}q_{U,V}=0$ if $(U,V)\neq (X,Y)$ and that $r=k_{X,Y}\cdot r_{X,Y}$ for some $k_{X,Y}\in \Bbbk$.  Thus $\varepsilon_Y I\varepsilon_X=\Bbbk r_{X,Y}$.  We conclude that $\dim \varepsilon_Y\Bbbk Q\varepsilon_X=1+\dim\varepsilon_Y I\varepsilon_X= 2$ and hence $[X,Y]$ consists of four elements.  Thus $\Lambda(B)$ is thin.
\end{proof}

\begin{Rmk}
The  intersection lattice of the braid arrangement in $\mathbb R^n$ for $n\geq 3$, or of any central line arrangement in $\mathbb R^2$ with at least $3$ lines, is not thin and so their algebras are not incidence algebras.
\end{Rmk}

\begin{Cor}\label{c:cat0.as.inc}
Let $K$ be a finite CAT(0) cube complex and let $\Bbbk$ be a field.  Let $\FFF(K)$ be the face semigroup of $K$ and  $\mathcal L(K)$ the intersection semilattice of $K$.  Then $\Bbbk \FFF(\mathcal A)\cong I(\mathcal L(K);\Bbbk)$.
\end{Cor}
\begin{proof}
The semigroup $\FFF(K)$ is a right ideal in $L^n$ where $n$ is the number of hyperplanes of $K$.  The support semilattice $\Lambda(\FFF(K))\cong \mathcal L(K)$ is then an ideal in $\Lambda(L^n)$, which in turn is congruent to the boolean lattice $P(\mathcal H)$ where $\mathcal H$ is the set of hyperplanes of $K$.  If we fix $e_X$ with $L^ne_X=X$ for $X\in \Lambda(L^n)$ such  that $e_X\in \FFF(K)$ whenever $X\in \Lambda(\FFF(K))$ (viewed as an ideal of $\Lambda(L^n)$), then using that $\FFF(K)$ is a right ideal we conclude that the primitive idempotent $\eta_X$, constructed as per Section~\ref{sssec:orthogonal-idempotents}, belongs to $\Bbbk \FFF(K)$ whenever $X\in \Lambda(\FFF(K))$.

We conclude that if $f=\sum_{X\in \Lambda(\FFF(K))}\eta_X$, then $\Bbbk \FFF(K)=f\Bbbk L^nf$.  If $A$ is an incidence algebra and $f$ is a sum of a subset of a complete set of orthogonal primitive idempotents, then $fAf$ is the incidence algebra of the subposet corresponding to all elements of the poset whose corresponding primitive idempotent appears in $f$.  Thus it suffices to show that $\Bbbk L^n$ is isomorphic to an incidence algebra.

But $\Bbbk [M\times N]\cong \Bbbk M\otimes \Bbbk N$ for any monoids $M,N$ and the class of incidence algebras of posets is closed under tensor product because $I(P;\Bbbk)\otimes I(P';\Bbbk)\cong I(P\times P';\Bbbk)$.  So it suffices to show that $\Bbbk L$ is the incidence algebra of a poset.  But the quiver of $L$ is $\wh 0\longrightarrow \wh 1$ and hence $\Bbbk L$ is the incidence algebra of the two element chain.  The result follows.
\end{proof}

\begin{Rmk}
The proof of Corollary~\ref{c:cat0.as.inc}  applies to any right ideal in $L^n$ that is a connected left regular band and, in particular, to lopsided systems.
\end{Rmk}

\subsection{Koszul duals}
Let $B$ be a connected CW left regular band.  Our goal in this section is to prove that $\Bbbk B$ is a Koszul algebra.  The Koszul dual is then $I(\Lambda(B)^{op};\Bbbk)$ and $\Ext(\Bbbk B)\cong I(\Lambda(B);\Bbbk)$. We follow closely~\cite[Section~9]{Saliolahyperplane}.

Fix a complete set of orthogonal primitive idempotents $\{\eta_X\mid X\in \Lambda(B)\}$ constructed as per Section~\ref{sssec:orthogonal-idempotents}. Then the grading of $\Bbbk B$ coming from the quiver presentation in Theorem~\ref{t:crucial} can be defined `intrinsically' by
\[\Bbbk B=\bigoplus_{n\geq 0}\left[\bigoplus_{\rk[X,Y]=n} \eta_Y\Bbbk B\eta_X\right]\] (or one can just check directly that this is a grading with $(\Bbbk B)_0\cong \Bbbk^{\Lambda(B)}$).

\begin{Thm}\label{t:Koszul}
Let $B$ be a connected CW left regular band.  Then $\Bbbk B$ is a Koszul algebra and the Koszul dual $\Bbbk B^!$ is isomorphic to $I(\Lambda(L)^{op};\Bbbk)$. Hence $\Ext(\Bbbk B)\cong I(\Lambda(L);\Bbbk)$.
\end{Thm}
\begin{proof}
To prove that $\Bbbk B$ is Koszul with respect to the above grading it suffices to show that each simple module $\Bbbk_X$ (viewed as a graded module concentrated in degree $0$) has a linear resolution.

The augmented cellular chain complex $C_\bullet(\CW(L_{\geq X});\Bbbk)\to \Bbbk_X$ is the minimal projective resolution by Theorem~\ref{minresoftrivial}. We need to show that it is a linear resolution. As usual let $\Bbbk L_Y$ be the Sch\"utzenberger representation associated to $Y\in \Lambda(B)$. Then
\begin{equation}\label{e:needgrade}
C_q(\CW(B_{\geq X});\Bbbk) \cong \bigoplus_{\rk[X,Y]=q}\Bbbk L_Y \cong \bigoplus_{\rk[X,Y]=q}\Bbbk B\eta_Y
\end{equation}
by Corollary~\ref{c:othersimplesCWcontraction}. As before, we identify $C_q(\CW(B_{\geq X});\Bbbk)$ with \[\bigoplus_{\rk[X,Y]=q}\Bbbk B\eta_Y\] via the isomorphism sending an oriented cell $[b]$ (as per the proof of Theorem~\ref{t:projectiveresgeneral}) to the element $b\eta_{\sigma(b)}$; see Remark~\ref{r:cell.case}.

Next we observe that if $\rk[X,Y]=q$, then we can define a grading
\begin{equation*}
    \Bbbk B\eta_Y=\bigoplus_{i\geq 0}\left[\bigoplus_{\rk [X,W]=i}\eta_W\Bbbk B\eta_Y\right]
\end{equation*}
with lowest degree term of degree $q$ using that $\eta_W\Bbbk B\eta_Y=0$ if $W\ngeq Y$ and so, in particular, $\eta_W\Bbbk B\eta_Y=0$ if $\rk[X,W]<q$.
This, of course, puts a grading on the right hand side of \eqref{e:needgrade} whose lowest degree term is $q$.
The degree $q$ component of $\Bbbk B\eta_Y$ is $\eta_Y\Bbbk B\eta_Y=\Bbbk\eta_Y$.
Thus $\Bbbk B\eta_Y$ is generated in degree $q$ and hence so is the right hand side of \eqref{e:needgrade}.
It remains to show that $d$ is a degree $0$ map.
Indeed, we have \[d(\eta_Wa\eta_Y)=\eta_Wd(a\eta_Y)\in \bigoplus_{\rk [X,Y']=q-1}\eta_W\Bbbk B\eta_{Y'}\] and so $d$ has degree $0$.
We may now conclude that $C_\bullet(\CW(B_{\geq X});\Bbbk)\to \Bbbk_X$ is a linear resolution.
\end{proof}

In particular, this result applies to the case where $B$ is the face monoid of a real or complex hyperplane arrangement, the monoid of covectors of an oriented matroid or oriented interval greedoid, the face semigroup of an affine hyperplane arrangement or CAT(0) cube complex or the semigroup of covectors of an affine oriented matroid or, more generally, a COM.

\begin{Example}[Ladders]
This is a continuation of Examples~\ref{ladders} and~\ref{ladders2}.  Theorem~\ref{t:Koszul} implies that $\Bbbk L_n$ is a Koszul algebra whose dual is the incidence algebra of a chain of $n+1$ elements.  But this incidence algebra is precisely the path algebra of the Dynkin quiver $A_{n+1}$, oriented as a directed path.  Of course, it is well known that the Koszul dual of the path algebra of the Dynkin quiver $A_{n+1}$, oriented as a directed path, is the monomial algebra obtained by factoring out the paths of length two.
\end{Example}

 Polo and Woodcock~\cite{Polo,Woodcock} independently showed that the incidence algebra $I(P;\Bbbk)$ of a graded poset $P$ is a Koszul algebra with respect to the natural grading if and only if each open interval in $P$ is $\Bbbk$-Cohen-Macaulay. We thus obtain the following theorem, which will be applied to enumerate cells of connected CW left regular bands.

\begin{Thm}\label{t:cohenmac}
Let $B$ be a connected CW left regular band.  Then each open interval of $\Lambda(B)$ is a Cohen-Macaulay poset.  In particular, if $B$ is a monoid, then the proper part $\Lambda(B)\setminus \{\wh 0,\wh 1\}$ of $\Lambda(B)$ is Cohen-Macaulay.
\end{Thm}
\begin{proof}
By the theorem of Polo and Woodcock~\cite{Polo,Woodcock} and Theorem~\ref{t:Koszul}, we have that each open interval of $\Lambda(B)$ is $\Bbbk$-Cohen-Macaulay for every field $\Bbbk$ and hence is Cohen-Macaulay by the universal coefficient theorem.
\end{proof}

\section{Injective envelopes for hyperplane arrangements, oriented matroids, CAT(0) cube complexes and COMs}\label{s:injenv}
In this section we give a geometric construction of the injective envelopes of the simple modules for $\Bbbk B$ when $B$ has the property that $bB$ is isomorphic to the face monoid of a central hyperplane arrangement or, more generally,  to the monoid of covectors of an oriented matroid for all $b\in B$.  This includes left regular bands associated to central and affine hyperplane arrangements, (possibly affine) oriented matroids, CAT(0) cube complexes and, more generally, COMs.

Recall that if $A$ is a finite dimensional $\Bbbk$-algebra and $e$ is a primitive idempotent of $A$, then $\Hom_{\Bbbk}(eA,\Bbbk)$ is the injective envelope of the simple module $Ae/\rad(A)e$.   That is, the injective envelope of a simple left module is the vector space dual of the right projective cover of the corresponding simple right module, so we proceed by constructing the right projective cover of each simple right module. We first do this for hyperplane arrangements and then we give the more general construction for oriented matroids.  Although this has some redundancy, there may be readers who are more familiar with hyperplane arrangements than oriented matroids and so we would like that case to be self-contained.

Let us fix a field $\Bbbk$ for this section.

\subsection{Generalities}
We begin with some generalities on right modules for left regular bands. Let $B$ be a connected left regular band and $X\in \Lambda(B)$.  Note first that the simple right $\Bbbk B$-module corresponding to $X$ is $\eta_X\Bbbk \Lambda(B)=\Bbbk \Lambda(B)\eta_X=\Bbbk\eta_X\cong \Bbbk$ with action
\[kb = \begin{cases}k, & \text{if}\ \sigma(b)\geq X\\ 0, & \text{else} \end{cases}\] for $k\in \Bbbk$ and $b\in B$.   We denote this simple module also by $\Bbbk_X$.

If $B$ is a left regular band and $e\in B$, then we define $\p_e\colon B\to eB$ by $\p_e(b)=eb$.  Recall that $eB=eBe$ is a left regular band monoid.

\begin{Prop}\label{p:retracthom}
The mapping $\p_e\colon B\to eB$ is a surjective homomorphism.
\end{Prop}
\begin{proof}
We have $\p_e(ab)=eab=eaeb=\p_e(a)\p_e(b)$ by the left regular band axioms.  Clearly, $\p$ is surjective.
\end{proof}

Any right $\Bbbk eB$-module can be lifted to a $\Bbbk B$-module via $\p_e$.  This applies in particular to $\Bbbk eB$ itself.

\begin{Prop}\label{p:liftup}
Let $B$ be a left regular band and $e\in B$.  There is an isomorphism of right $\Bbbk B$-modules $e\Bbbk B\to \Bbbk eB$ induced by the identity map on $eB$ where we give $\Bbbk eB$ the right $\Bbbk B$-module structure induced via $\p_e$ and we view $e\Bbbk B$ as a right ideal of $\Bbbk B$.
\end{Prop}
\begin{proof}
If $a\in eB$ and $b\in B$, then $ab=eab=eaeb=aeb=a\p_e(b)$ and the result follows.
\end{proof}

An immediate consequence of Proposition~\ref{p:liftup} is the following corollary.

\begin{Cor}\label{cor:proj.right.lift}
Let $B$ be a connected left regular band and $X\in \Lambda(B)$ with $X=Be_X$.  Then the simple right $\Bbbk B$-module $\Bbbk_X$ is obtained by lifting $\Bbbk e_XB/\Bbbk[\bd e_XB]$ from $\Bbbk e_XB$ via $\p_e$.  Moreover, if $P_X$ is a projective indecomposable right $\Bbbk e_XB$-module with $P_X/\rad(P_X)\cong \Bbbk e_XB/\Bbbk[\bd e_XB]$, then $P_X$ is a direct summand in $e_X\Bbbk B$ as a right $\Bbbk B$-module and is thus the right projective cover of $\Bbbk_X$ as a $\Bbbk B$-module.
\end{Cor}

Thus, for constructing right projective covers of simple modules, we are reduced to considering left regular band monoids $B$ and $\Bbbk_{\wh 1}$ where $\wh 1=B$ is the maximum element of the lattice $\Lambda(B)$.  The next lemma provides necessary and sufficient conditions for a quotient by a right ideal to give a right projective cover of $\Bbbk_{\wh 1}$.

\begin{Lemma}\label{l:right.proj.cover}
Let $B$ be a left regular band monoid and denote by $\wh{1}$ the maximum of $\Lambda(B)$.  Suppose that $R\subsetneq B$ is a proper right ideal.  Then the natural homomorphism $\psi\colon \Bbbk B/\Bbbk R\to \Bbbk_{\wh 1}$ induced by $1\mapsto 1$ and $b\mapsto 0$ for $b\neq 1$ is a right projective cover if and only if
\begin{enumerate}
\item $BR=\bd B$;
\item $R$ is a connected left regular band.
\end{enumerate}
\end{Lemma}
\begin{proof}
Suppose first that $\psi$ is a right projective cover.  Then the exact sequence
\[\xymatrix{0\ar[r] & \Bbbk R\ar[r] & \Bbbk B\ar[r] & \Bbbk B/\Bbbk R\ar[r] & 0}\] splits because $\Bbbk B/\Bbbk R$ is projective.  It follows that $\Bbbk R$ is  projective and hence of the form $\eta \Bbbk B$ for some idempotent $\eta\in \Bbbk B$, which must then necessarily be in $\Bbbk R$.  But then $\eta$ is a left identity for $\Bbbk R$, which is therefore a unital algebra as all left regular band algebras have a right identity by Theorem~\ref{primidempotentprops}.  Thus $R$ is connected by Theorem~\ref{t:unital}.  Suppose that $BR\neq \bd B$ and let $b\in \bd B\setminus BR$.  Notice that if $Ba\supseteq Bb$, then $ba=b$ and so  $a\in B\setminus R$.  Thus we can define a non-zero $\Bbbk B$-module homomorphism $\alpha\colon \Bbbk B/\Bbbk R\to \Bbbk_{Bb}$ by
\[\alpha(a+\Bbbk R)= \begin{cases}1, & \text{if}\ Ba\supseteq Bb\\ 0, & \text{else.}\end{cases}\]  This contradicts that $\Bbbk B/\Bbbk R$ has simple top $\Bbbk_{\wh 1}$.  Thus $BR=\bd B$.

We begin the proof of the converse by verifying that $\Bbbk B/\Bbbk R$ is indecomposable. This requires only assumption (1).
Clearly, $\psi$ is a surjective module homomorphism with kernel $\Bbbk \bd B/\Bbbk R$.  We first show that $\ker \psi=\rad(\Bbbk B/\Bbbk R)$.  The inclusion $\rad(\Bbbk B/\Bbbk R)\subseteq \ker \psi$ follows because $\Bbbk_{\wh 1}$ is simple. We show the reverse inclusion by verifying that \[\Bbbk \bd B/\Bbbk R\subseteq (\Bbbk B/\Bbbk R)\cdot\rad(\Bbbk B).\]  Let $x\in \bd B$.  By hypothesis, we can find $r\in R$ with $x\in Br$.  Then $Bx=Brx$ and so $x-rx\in \rad(\Bbbk B)$ by Proposition~\ref{rad.basis}.  But then \[(1+\Bbbk R)(x-rx) = x-rx+\Bbbk R=x+\Bbbk R\] because $rx\in R$. This completes the proof that $\rad(\Bbbk B/\Bbbk R)=\ker \psi$.  We deduce that $\Bbbk B/\Bbbk R$ is an indecomposable module (since a module over a finite dimensional algebra with simple top is indecomposable).

Now we prove that assumption (2) yields projectivity.  Indeed, since $R$ is connected we deduce that $\Bbbk R$ has an identity $\eta$ by Theorem~\ref{t:unital}. As $\Bbbk R$ is a right ideal, clearly $\eta\Bbbk B\subseteq \Bbbk R$.  As the reverse inclusion is obvious, we deduce that $\Bbbk B=\eta\Bbbk B\oplus (1-\eta)\Bbbk B=\Bbbk R\oplus (1-\eta)\Bbbk B$ and so $\Bbbk B/\Bbbk R\cong (1-\eta)\Bbbk B$ is a projective right $\Bbbk B$-module.
\end{proof}

\begin{Rmk}
Let $A$ be a set with at least two elements and $F(A)$ the free left regular band monoid on $A$.  As the elements of $A$ are in distinct maximal $\mathscr L$-classes of $F(A)$, the only proper right ideal $R$ of $F(A)$ with $F(A)R=\bd F(A)$ is $R=\bd F(A)$.  But the Hasse diagram of $\bd F(A)$ is a forest with $|A|$ connected components.  It follows from Lemma~\ref{l:right.proj.cover} that the right projective cover of $\Bbbk_{\wh 1}$ is not of the form $\Bbbk F(A)/\Bbbk R$ for any right ideal $R$ of $F(A)$.
\end{Rmk}

Our goal is to show that Lemma~\ref{l:right.proj.cover} can be applied to the algebras of a number of the CW left regular bands that we have been considering.

\subsection{Hyperplane arrangements}
Let us  begin with the case of the face monoid of a central arrangement.  The idea behind what follows was inspired by considering the line shelling associated to a zonotope~\cite[Chapter~8.2]{Ziegler}, but we will work with the dual picture. Let $\AAA=\{H_1,\ldots, H_n\}$ be an essential central hyperplane in $V=\mathbb R^d$. Let us assume that $H_i$ is defined by a form $f_i\in V^*$.  Let $\theta\colon \mathbb R^d\to \FFF(\AAA)$ be given by
\[\theta(x)=(\sgn(f_1(x)),\ldots, \sgn(f_n(x))).\]

We say that a hyperplane $H$ in $V$ is \emph{in general position}\index{in general position} or \emph{generic}\index{generic} with respect to $\AAA$ if $H$ contains no element of the intersection lattice $\mathcal L(\AAA)\cong \Lambda(\FFF(\AAA))$ except the origin.  Equivalently, if $f\in V^*$ is a form defining $H$, then $f(W)=\{0\}$ implies $W=\{0\}$
for all $W\in \mathcal L(\AAA)$.

\begin{Prop}\label{p:generic.hyp}
There exist generic hyperplanes with respect to $\AAA$.
\end{Prop}
\begin{proof}
Suppose that $\mathcal L(\AAA)\setminus \{\{0\}\}=\{W_1,\ldots, W_r\}$.  Then the annihilator $W_i^{\perp}$ is a proper subspace of $V^*$.  As a vector space over an infinite field is never a finite union of proper subspaces, we conclude that there is a form $f\in V^*\setminus \bigcup_{i=1}^r W_i^{\perp}$.  The hyperplane $H$ defined by $f=0$ is generic with respect to $\AAA$.
\end{proof}

Fix now a hyperplane $H$ generic with respect to $\AAA$ defined by a form $f\in V^*$.  We retain all the previous notation.  Let $\wh{\AAA}=\AAA\cup \{H\}$ and let $\wh{\theta}\colon \mathbb R^d\to \FFF(\wh{\AAA})$ be given by
\[\wh{\theta}(x)=(\sgn(f_1(x)),\ldots, \sgn(f_n(x)),\sgn(f(x))).\]  The natural projection $\pi\colon L^{n+1}\to L^n$ defined by \[\pi(x_1,\ldots, x_{n+1})=(x_1,\ldots, x_n)\] restricts to a surjective homomorphism $\pi\colon \FFF(\wh{\AAA})\to \FFF(\AAA)$ as $\pi(\wh{\theta}(x))=\theta(x)$ for $x\in \mathbb R^d$.   Denote by $\FFF^+(\wh{\AAA})$ the set of elements $(x_1,\ldots, x_{n+1})\in \FFF(\wh{\AAA})$ such that $x_{n+1}=+$. Note that $\FFF^+(\wh{\AAA})$ is a right ideal in $\FFF(\wh{\AAA})$ and is isomorphic to the face semigroup of the $(d-1)$-dimensional affine arrangement obtained by intersecting the hyperplanes of $\AAA$ with the affine hyperplane defined by $f=1$.  Thus $\FFF^+(\wh{\AAA})$ is a connected CW left regular band by Proposition~\ref{p:topology} (it is, in fact, an affine oriented matroid).

We define the \emph{visual hemisphere}\index{visual hemisphere} $R(H)$ of $\AAA$ with respect to $H$ to be $\pi(\FFF^+(\wh{\AAA}))$.  This is a right ideal of $\FFF(A)$.  The reason for the name is that if one performs a line shelling of the zonotope $Z(\AAA)\subseteq V^*$ polar to $\AAA$ with respect to the line spanned by $f$ in $V^*$, then the faces of $Z(\AAA)$ corresponding to the elements of $R(H)$ make up the visual hemisphere of $Z(\AAA)$ with respect to this shelling (cf.~\cite[Theorem~8.12]{Ziegler}).  Note that $R(H)$ consists of those faces $F$ containing a point $x$ with $f(x)>0$ by construction  and hence is a realizable COM.

\begin{Thm}\label{t:right.proj.central}
Let $\AAA$ be an essential central hyperplane arrangement in $\mathbb R^d$, let $\Bbbk$ be a field and let $H$ be a hyperplane in $\mathbb R^d$ which is generic with respect to $\AAA$.  Let $\wh 1$ be the maximum element of the intersection lattice $\mathcal L(\AAA)\cong \Lambda(\FFF(\AAA))$.  Then the natural homomorphism \[\psi\colon \Bbbk \FFF(\AAA)/\Bbbk R(H)\to \Bbbk_{\wh 1}\] is the right projective cover where $R(H)$ is the visual hemisphere of $\AAA$ with respect to $H$.
\end{Thm}
\begin{proof}
We verify that the right ideal $R(H)$ of $\FFF(\AAA)$ satisfies the conditions of Lemma~\ref{l:right.proj.cover}. We retain the previous notation.  First we observe that $R(H)$ is proper.  Indeed, suppose that $0\in R(H)$.  Then  $0=\pi(\wh{\theta}(x))=\theta(x)$ with $\wh{\theta}(x)\in \FFF^+(\wh{\AAA})$.   But as $\AAA$ is essential,  $0=\theta(x)$ implies $x=0$.  But then $\wh{\theta}(0)=0\notin \FFF^+(\wh{\AAA})$.   This contradiction shows that $0\notin R(H)$.  Next we observe that $\pi\colon \FFF^+(\wh{\AAA})\to R(H)$ is an isomorphism because all elements of $\FFF^+(\wh{\AAA})$ have $+$ as the last coordinate.  Therefore, $R(H)$ is a connected left regular band, being isomorphic to the face semigroup of an affine hyperplane arrangement (or affine oriented matroid).  It remains to verify that $\FFF(\AAA)R(H)=\bd \FFF(\AAA)$.

Let $F\in \FFF(\AAA)$ with $F\neq 0$.   We continue to denote by $f$ the form defining the hyperplane $H$.    As $H$ is generic with respect to $\AAA$, it follows that the form $f$ does not vanish on the span of $F$ (which belongs to $\mathcal L(\AAA)$) and hence does not vanish on $F$. Therefore, we can find $x\in F$ with $f(x)\neq 0$. Note that $F=\theta(x)$.   If $f(x)>0$, then $\wh{\theta}(x)\in \FFF^+(\wh{\AAA})$ and so $F=\theta(x)=\pi(\wh{\theta}(x))\in R(H)$.  If $f(x)<0$, then $f(-x)=-f(x)>0$ and so $\wh{\theta}(-x)\in \FFF^+(\wh{\AAA})$. We conclude that $\theta(-x)=\pi(\wh{\theta}(-x))\in R(H)$.  But $x,-x$ belong to the same elements of the intersection lattice $\mathcal L(\AAA)$ and so $F=\theta(x)=\theta(x)\theta(-x)$.  Thus $F\in \FFF(\AAA)R(H)$. An application of Lemma~\ref{l:right.proj.cover} completes the proof of the theorem.
\end{proof}

Combining Theorem~\ref{t:right.proj.central} with Corollary~\ref{cor:proj.right.lift} allows us to compute right projective covers and injective envelopes for all simple modules for a number of examples including face semigroups of central and affine hyperplane arrangements, $T$-convex sets of topes in central hyperplane arrangements, CAT(0) cube complexes and realizable COMs.

\begin{Cor}\label{c:inj.ev.hyp}
Let $B$ be a connected left regular band and $\Bbbk$ a field. Let $X\in \Lambda(B)$ and assume that $X=Be_X$.  Suppose that $\tau\colon e_XB\to \FFF(\AAA)$ is an isomorphism where $\AAA$ is an essential central hyperplane arrangement and that $H$ is a generic hyperplane with respect to $\AAA$ with associated visual hemisphere $R(H)$.
\begin{enumerate}
\item The right projective  cover of $\Bbbk_X$ is the natural homomorphism \[\psi\colon \Bbbk e_XB/\Bbbk \tau\inv(R(H))\to \Bbbk_{X}\] induced by $e_X\mapsto 1$ and $b\mapsto 0$ for $b\in \bd e_XB$.
\item   The injective envelope of $\Bbbk_X$ is the left $\Bbbk B$-module $I_X$ consisting of all mappings $f\colon e_XB\to \Bbbk$ vanishing on $\tau\inv(R(H))$ with pointwise vector space operations and the natural left action of $B$ given by $(bf)(a)=f(ab)$ for $b\in B$ and $a\in e_XB$. The simple socle of $I_X$ is the subspace of those maps vanishing on $\bd e_XB$; it is isomorphic to $\Bbbk_X$.
\end{enumerate}
\end{Cor}

Recall that if $\AAA$ is a hyperplane arrangement (central or affine) and $G\in \FFF(\AAA)$, then $G\FFF(\AAA)$ is isomorphic to the face monoid of a central hyperplane arrangement which is a deletion of $\AAA$ (one deletes all the hyperplanes except those containing $G$).  Thus Corollary~\ref{c:inj.ev.hyp} can be applied to compute the injective envelope of any simple module for a hyperplane face semigroup or, more generally, for the semigroup associated to a $T$-convex set of topes for a hyperplane arrangement or a realizable COM.

Let us consider some examples.

\begin{Example}[Boolean arrangement]
As an example, if $\AAA$ is the boolean arrangement in $\mathbb R^n$, then $F(\AAA)=L^n$.  The hyperplane defined by \[f(x_1,\ldots,x_n)=x_1+\cdots+x_n\] is generic with respect to $\AAA$ and $R(H)$ consists of those sign vectors with at least one positive coordinate.  Therefore, $\Bbbk F(\AAA)/\Bbbk R(H)$ has basis the cosets of elements of $\{0,-\}^n$ and hence has dimension $2^n$.  Consequently, the injective envelope of $\Bbbk_{\wh 1}$ has dimension $2^n$.

The zonotope associated to the boolean arrangement is the hypercube.  Thus if $K$ is a CAT(0) cube complex, $X$ is an element of the intersection semilattice of $K$ and $C_X$ is a cube crossed exactly by the hyperplanes in $X$, then the injective envelope of $\Bbbk_X$ consists of all mappings from faces of $C_X$ to $\Bbbk$ that vanish on those faces of $C_X$ that lie in the positive half-space associated to some element of $X$ where we fix an orientation of the hyperplanes of $X$.
\end{Example}

\begin{Example}[Braid arrangement]\label{ex:braid}
Recall that the braid arrangement in $V=\mathbb R^n$ is defined by the set of hyperplanes $H_{ij}=\{x\in \mathbb R^n\mid x_i=x_j\}$ with $1\leq i<j\leq n$.  This is not an essential arrangement: the intersection of the hyperplanes is the line $\ell$ given by the equation $x_1=x_2=\cdots=x_n$.  Let $\AAA$ be the essential arrangement in $V/\ell$ induced by the braid arrangement.  It is well known, and not too difficult to show (cf.~\cite{BHR}), that the elements of $\FFF(\AAA)$ are in bijection with ordered set partitions of $[n]=\{1,\ldots,n\}$.  An ordered set partition $(P_1,\ldots, P_r)$ corresponds to the face in $\mathbb R^n$ defined by the equations $x_i=x_j$ for $i,j$ belonging to the same block  and the inequalities $x_i<x_j$ if the block containing $i$ is to the left of the block containing $j$.  For example, when $n=4$, the ordered set partition $(\{1,3\},\{2,4\})$ corresponds to the face in $\mathbb R^n$ defined by $x_1=x_3<x_2=x_4$.  The product of ordered partitions is given by the simple rule:
\[(P_1,\ldots, P_r)(Q_1,\ldots, Q_s)=(P_1\cap Q_1,\ldots, P_1\cap Q_s,\ldots, P_r\cap Q_1,\ldots, P_r\cap Q_s)^{\wedge}\]
where $(B_1,\ldots, B_k)^{\wedge}$ is the result of removing all occurrences of the empty set from $(B_1,\ldots, B_k)$.
The support lattice $\Lambda(\FFF(\AAA))$ can be identified with the lattice of set partitions of $[n]$ ordered by refinement and the support map $\sigma\colon \FFF(\AAA)\to \Lambda(\FFF(\AAA))$ is given by $\sigma((P_1,\ldots, P_r))=\{P_1,\ldots, P_r\}$.

We can identify the dual space of $V/\ell$ with the space of functionals on $V$ which vanish on $\ell$.  It is not difficult to check that the functional \[f(x) = (n-1)x_n-(x_1+\cdots+x_{n-1})\] vanishes on $\ell$ and defines a generic hyperplane $H$ with respect to $\AAA$.  Clearly, $f$ is positive at some point of the face corresponding to an ordered set partition $(P_1,\ldots, P_r)$ if and only if $n\notin P_1$.  Thus $\Bbbk \FFF(\AAA)/\Bbbk R(H)$ has basis the cosets of faces corresponding to ordered partitions $(P_1,\ldots, P_r)$ with $n\in P_1$.

One can define an equivalence relation on ordered set partitions of $[n]$ by identifying cyclic conjugates, that is,
\[(P_1,P_2,\ldots, P_r)\sim (P_2,P_3,\ldots, P_r,P_1)\sim\cdots \sim (P_r,P_1,\ldots, P_{r-1}).\]  An equivalence class of ordered set partitions is called a \emph{necklace of partitions of $[n]$}\index{necklace of partitions of $[n]$}.  Clearly, each necklace contains exactly one representative $(P_1,\ldots, P_r)$ with $n\in P_1$.  Thus the dimension of the injective envelope of $\Bbbk_{\wh 1}$ is the number of necklaces of partitions of $[n]$.  On the other hand, the results of the second author~\cite[Proposition~6.4]{Saliolahyperplane} (or Theorem~\ref{t:cartan.cw} below) implies that if $X\in \Lambda(\FFF(\AAA))$, then the multiplicity of $\Bbbk _X$ as a composition factor of the injective envelope of $\Bbbk_{\wh 1}$ is given by $|\mu(X,\wh 1)|$ where $\mu$ is the M\"obius function of the partition lattice.  As a consequence we obtain the following combinatorial identity, which can also be verified directly via a straightforward computation with the M\"obius function of the partition lattice.

\begin{Prop}
The number of necklaces of partitions of $[n]$ is \[\sum_{X\in \Pi_n}|\mu(X,\wh 1)|\] where $\Pi_n$ is the lattice of set partitions of $[n]$, $\mu$ is the M\"obius function of $\Pi_n$ and $\wh 1$ is the partition into one block.
\end{Prop}
\end{Example}

\subsection{Oriented matroids}

To adapt the proof of Theorem~\ref{t:right.proj.central} to oriented matroids, we need an oriented matroid analogue of adjoining a generic hyperplane.   Let $(E,\mathcal L)$ be an oriented matroid and let $Z\colon \mathcal L\to \Lambda(\mathcal L)$ be the map taking a sign vector in $L^E$ to its zero set.  A \emph{cocircuit}\index{cocircuit} of $\mathcal L$ is a maximal non-identity element of $\mathcal L$ (with respect to the natural partial order on $\mathcal L$) or, equivalently, a non-zero element $x\in \mathcal L$ with $Z(x)$ maximal with respect to inclusion. Note that $y$ is a cocircuit if and only if $-y$ is a cocircuit. It is a general fact about oriented matroids that they are generated as a monoid by their cocircuits (cf.~the dual of~\cite[Proposition~3.7.2]{OrientedMatroids1999}).

An oriented matroid $(\til E,\til {\mathcal L})$ is a \emph{single element extension}\index{single element extension} of an oriented matroid $(E,\mathcal L)$ if $\til E=E\cup \{p\}$ for some $p\notin E$ and $\mathcal L$ is the image of $\til{\mathcal L}$ under the projection $\pi\colon L^{\til E}\to L^E$ given by $\pi(x)=x|_E$.  Let $\mathcal C^*$ denote the set of cocircuits of $\mathcal L$.  Then there is a unique mapping $\sigma\colon \mathcal C^*\to L$ satisfying $\sigma(-y)=-\sigma(y)$ for all $y\in \mathcal C^*$ and some additional axioms such that the set $\til {\mathcal C}^*$ of cocircuits of $\til{\mathcal L}$ is given by
\begin{equation}\label{eq:new.cocircuits}
\begin{split}
\til{\mathcal C}^*=&\{(y,\sigma(y))\mid y\in \mathcal C^*\}\cup {} \\ &\{(y_1y_2,0)\mid y_1,y_2\in \mathcal C^*,\ \sigma(y_1)=-\sigma(y_2)\neq 0,\\ & S(y_1,y_2)=\emptyset,\rho(y_1y_2)=2\}
\end{split}
\end{equation}
where $\rho$ is obtained from the rank function on the geometric lattice $\Lambda(\mathcal L)$ by $\rho(x)=\mathrm{rank}(E)-\mathrm{rank}(Z(x))$.  See~\cite[Proposition~7.1.4 and Theorem~7.1.8]{OrientedMatroids1999} for details.  We will write $\mathcal L_{\sigma}$ for the single element extension of $\mathcal L$ associated to a mapping $\sigma\colon \mathcal C^*\to L$ satisfying the equivalent conditions of~\cite[Theorem~7.1.8]{OrientedMatroids1999} for defining a single element extension of $\mathcal L$.

A single element extension $\mathcal L_{\sigma}$ of $\mathcal L$ is said to be \emph{in general position}\index{in general position} or \emph{generic}\index{generic} if $\sigma(y)\neq 0$ for all $y\in \mathcal C^*$.  Such extensions exist (cf.~\cite[Proposition~7.2.2]{OrientedMatroids1999}).  One method to construct them is via lexicographic extensions.   Here is the setup.  Let $E$ be the ground set of $\mathcal L$, let  $I=\{e_1,\ldots, e_k\}$ be a linearly ordered subset of $E$ and let $\alpha_1,\ldots,\alpha_k\in \{+,-\}$.  Define a mapping $\sigma\colon \mathcal C^*\to L$ by
\[\sigma(y)= \begin{cases} \alpha_iy_{e_i}, & \text{if}\ i\ \text{is minimal with}\ y_{e_i}\neq 0\\ 0, & \text{if}\ y_{e_i}=0\ \text{for all}\ 1\leq i\leq k.\end{cases}\]
Then $\sigma$ defines a single element extension $\mathcal L_{\sigma}$ called a \emph{lexicographic extension}\index{lexicographic extension} of $\mathcal L$.  See~\cite[Proposition~7.2.4]{OrientedMatroids1999}. If $I=E$, then $\mathcal L_{\sigma}$ is obviously generic.

For example, if $\mathcal L$ is the oriented matroid associated to a central hyperplane arrangement and $f_i$ is the form defining the hyperplane $e_i\in I$, for $1\leq i\leq k$, then the lexicographic extension associated to $I$ and the $\alpha_i$ is obtained by adding the hyperplane $f=0$ to the original arrangement where \[f=\varepsilon \alpha_1f_1+\varepsilon^2\alpha_2f_2+\cdots+\varepsilon^k\alpha_kf_k\] for some small $\varepsilon>0$.  See the discussion preceding~\cite[Proposition~7.2.4]{OrientedMatroids1999}.

Let $(E,\mathcal L)$ be an oriented matroid and let $(E\cup \{p\}, \mathcal L_{\sigma})$ be a generic single element extension with corresponding map $\sigma\colon \mathcal C^*\to L$ where $\mathcal C^*$ denotes the set of cocircuits of $\mathcal L$.  Then there is a corresponding affine oriented matroid $(E\cup \{p\}, \mathcal L_{\sigma},p)$ with associated left regular band $\mathcal L^+_{\sigma}(p)=\{x\in \mathcal L_{\sigma}\mid x_p=+\}$.  Let $\pi\colon \mathcal L_{\sigma}\to \mathcal L$ be the projection $\pi(x)=x|_E$.  We define $R(\sigma)=\pi(\mathcal L^+_{\sigma}(p))$ to be the \emph{visual hemisphere}\index{visual hemisphere} associated to $\sigma$.  Since $\mathcal L^+_{\sigma}(p)$ is a right ideal of $\mathcal L_{\sigma}$ and $\pi$ is surjective, we conclude that $R(\sigma)$ is a right ideal of $\mathcal L$.

\begin{Thm}\label{t:right.proj.central.oriented}
Let $(E,\mathcal L)$ be an oriented matroid, let $\Bbbk$ be a field and let $(E\cup \{p\}, \mathcal L_{\sigma})$ be a generic single element extension with corresponding map $\sigma\colon \mathcal C^*\to L$ where $\mathcal C^*$ denotes the set of cocircuits of $\mathcal L$. Let $\wh 1=\mathcal L$ be the maximum element of  $\Lambda(\mathcal L)$.  Then the natural homomorphism \[\psi\colon \Bbbk \mathcal L/\Bbbk R(\sigma)\to \Bbbk_{\wh 1}\] is the right projective cover where $R(\sigma)$ is the visual hemisphere associated to $\sigma$.
\end{Thm}
\begin{proof}
We check that the right ideal $R(\sigma)$ of $\mathcal L$ satisfies the conditions of Lemma~\ref{l:right.proj.cover}. Let us maintain the previous notation.  The set of cocircuits of $\mathcal L_{\sigma}$ will be denoted by $\til{\mathcal C}^*$.  Notice that $\til{\mathcal C}^*$ contains no element of the form $(0,z)$ with $z\in \{+,-\}$ by \eqref{eq:new.cocircuits} because $0\notin \mathcal C^*$.  Thus $\mathcal L_{\sigma}$ contains no such element because $\til {\mathcal C}^*$ generates $\mathcal L_{\sigma}$.  Therefore, $(0,+)\notin \mathcal L^+_{\sigma}(p)$ and so $R(\sigma)$ is proper.

Next observe that $\pi\colon \mathcal L^+_{\sigma}(p)\to R(\sigma)$ is an isomorphism since all elements of $\mathcal L^+_{\sigma}(p)$ have $+$ as the last coordinate.  Therefore, $R(\sigma)$ is a connected left regular band as affine oriented matroids are such.  It remains to verify that $\mathcal LR(\sigma)=\bd \mathcal L$.

Let $y\in \mathcal L\setminus \{0\}$.  Then $y=y_1\cdots y_m$ where $y_1,\ldots, y_m$ are cocircuits.  Put $\til y_i=(y_i,\sigma(y_i))\in \til{\mathcal C}^*$ (using \eqref{eq:new.cocircuits}).  Note that $\pi(\til y_i)=y_i$.   Since $\mathcal L_{\sigma}$ is generic, either $\sigma(y_1)=+$ or $\sigma(y_1)=-$.  Assume first that $\sigma(y_1)=+$.  Then $\til y_1=(y_1,+)$ and so $\til y = \til y_1\cdots\til y_m\in \mathcal L^+_{\sigma}(p)$ and hence $y=y_1\cdots y_m=\pi(\til y)\in R(\sigma)$.  Next assume that $\sigma(y_1)=-$.  Then $-\til y_1=(-y_1,+)\in \mathcal L^+_{\sigma}(p)$ and $\pi(-\til y_1)=-y_1$.  Therefore, $(-y_1)y_2\cdots y_m=\pi((-\til y_1)\til y_2\cdots \til y_m)\in \pi(\mathcal L^+_{\sigma}(p))=R(\sigma)$ and hence $y=y_1\cdots y_m = y_1(-y_1)y_2\cdots y_m\in \mathcal LR(\sigma)$ using that $x(-x)=x$ in $L^E$.

The theorem now follows via an application of Lemma~\ref{l:right.proj.cover}.
\end{proof}

One advantage of Theorem~\ref{t:right.proj.central.oriented} over Theorem~\ref{t:right.proj.central} for the case of central hyperplane arrangements is that the former is combinatorial: one can effectively construct a generic lexicographic single element extension.  Writing explicitly a generic hyperplane is often more complicated.

Theorem~\ref{t:right.proj.central.oriented}, together with Corollary~\ref{cor:proj.right.lift}, allows us to compute right projective covers and injective envelopes for all simple modules for left regular bands associated to COMs, including oriented matroids and affine oriented matroids. We summarize in the following corollary.

\begin{Cor}\label{c:inj.env.om}
Let $B$ be a connected left regular band and $\Bbbk$ a field. Let $X\in \Lambda(B)$ with $X=Be_X$.  Suppose that $\tau\colon e_XB\to \mathcal L$ is an isomorphism where $\mathcal L$ is the monoid of covectors of an oriented matroid and that $\mathcal L_{\sigma}$ is a generic single element extension of $\mathcal L$ with associated visual hemisphere $R(\sigma)$.
\begin{enumerate}
\item The right projective  cover of $\Bbbk_X$ is the natural homomorphism \[\psi\colon \Bbbk e_XB/\Bbbk \tau\inv(R(\sigma))\to \Bbbk_{X}\] induced by $e_X\mapsto 1$ and $b\mapsto 0$ for $b\in \bd e_XB$.
\item   The injective envelope of $\Bbbk_X$ is the left $\Bbbk B$-module $I_X$ consisting of all mappings $f\colon e_XB\to \Bbbk$ vanishing on $\tau\inv(R(\sigma))$ with pointwise vector space operations and the natural left action of $B$ given by $(bf)(a)=f(ab)$ for $b\in B$ and $a\in e_XB$. The simple socle of $I_X$ (isomorphic to $\Bbbk _X$) is the subspace of those maps vanishing $\bd e_XB$.\end{enumerate}
\end{Cor}

Recall that if $(E,\mathcal L)$ is an oriented matroid and $x\in \mathcal L$, then $x\mathcal L$ is isomorphic to the monoid of covectors for the deletion $(E\setminus A,\mathcal L\setminus A)$ where $A=E\setminus Z(x)$.  Thus Corollary~\ref{c:inj.env.om} can be used to compute the injective envelopes of all simple modules for the semigroup of an oriented matroid, affine oriented matroid or face semigroups of $T$-convex set of topes.  More generally, $x\mathcal L$ is isomorphic to the monoid of covectors of an oriented matroid for any COM $(E,\mathcal L)$ (cf.~\cite{COMS}) and so Corollary~\ref{c:inj.env.om} can be applied to compute injective envelopes for algebras of COMs.

It is an open problem to give an explicit construction of injective indecomposables for more general classes of left regular band algebras.

\section{Enumeration of cells for CW left regular bands}\label{s:enum}
In this section we prove analogues of Zaslavsky's theorem~\cite[Theorem~4.6.1]{OrientedMatroids1999}, and other enumerative results from hyperplane/oriented matroid theory, for connected CW left regular bands. For CAT(0) cube complexes, our results recover those of Dress \textit{et. al}~\cite{dresscube}, which were originally stated in the language of median graphs. In the case of complex hyperplane arrangements, our result specializes to that of Bj\"orner~\cite{bjorner2}. Other generalizations of Zaslavsky's theorem in the literature include~\cite{Desh1,Desh2}.
 When comparing our results to those in~\cite{OrientedMatroids1999}, one should bear in mind that people in hyperplane and oriented matroid theory work with the dual cell complex to the monoid of covectors and they do not include the identity as corresponding to a face.

\subsection{Flag vectors}
\nomenclature[C, 09]{$\chi(X)$}{Euler characteristic of the CW complex $X$}%
\nomenclature[C, 05]{$f(X)$}{$f$-vector of the CW complex $X$}%
If $X$ is a (finite) $d$-dimensional regular CW complex, then $f_k$ denotes the number of $k$-dimensional cells of $X$ for $0\leq k\leq d$.  The sequence \[f(X)=(f_0,\ldots, f_d)\] is called the \emph{$f$-vector}\index{$f$-vector} of $X$.  The
\emph{Euler characteristic}\index{Euler characteristic} of $X$ is \[\chi(X)=\sum_{k=0}^d (-1)^kf_k.\]  It is well known that \[\chi(X) = \sum_{k=0}^d(-1)^k\beta_k(X)\] where $\beta_k(X)$ is the dimension of $H_k(X;\mathbb Q)$.

\nomenclature[C, 06]{$\til f(X)$}{flag vector, or Fine $f$-vector, of the CW complex $X$}%
More generally, for $J=\{j_1,\ldots, j_k\}$ with $0\leq j_1<j_2<\cdots<j_k\leq d$, let $f_J$ be the number of chains of cells $e_1<e_2<\cdots<e_k$ in the face poset $\mathcal P(X)$ such that $\dim e_i=j_i$ for $1\leq i\leq k$.  The array $\til f(X)=(f_J)_{J\subseteq \{0,\ldots,d\}}$ is called the \emph{flag vector}\index{flag vector} (or, \emph{Fine $f$-vector}\index{Fine $f$-vector}) of $X$.  In this section, we will compute $f(\CW(B))$ and $\til f(\CW(B))$ for the regular cell complex $\CW(B)$ associated to a connected CW left regular band $B$ in terms of $\Lambda(B)$.  This generalizes well-known results for hyperplane arrangements and oriented matroids~\cite[Section 4.6]{OrientedMatroids1999}.

We begin with a lemma about cellular maps between posets satisfying properties analogous to that of the map from a connected CW left regular band to its support semilattice.

\begin{Lemma}\label{l:zaslavsky}
Let $f\colon P\to Q$ be a cellular mapping of finite posets such that:
\begin{enumerate}
\item $Q$ has a minimum $\wh 0$;
\item $f$ preserves strict inequalities;
\item $f\inv(Q_{\geq q})$ is an acyclic CW poset for all $q\in Q$.
\end{enumerate}
Then \[|f\inv(q)| = \sum_{q'\geq q}(-1)^{\dim q'-\dim q}\mu_Q(q,q')\] where $\mu_Q$ is the M\"obius function of $Q$.

If, in addition, each open interval of $Q$ is a Cohen-Macaulay poset, then
\[|f\inv(q)| = \sum_{q'\geq q}|\mu_Q(q,q')|\]  holds.
\end{Lemma}
\begin{proof}
Note that $P=f\inv(Q_{\geq \wh 0})$ is a CW poset and hence is graded.  Therefore, $Q$ is graded by Corollary~\ref{c:preservesgraded}.
As the Euler characteristic of any acyclic CW complex is $1$ and $f$ preserves dimensions (because it preserves strict inequalities), we obtain
\[1=\chi(f\inv (Q_{\geq q})) = \sum_{p\in f\inv(Q_{\geq q})}(-1)^{\dim p-\dim q}=\sum_{q'\geq q}(-1)^{\dim q'-\dim q}|f\inv (q')|.\]
Consequently, we have
\[(-1)^{\dim q} = \sum_{q'\geq q}(-1)^{\dim q'}|f\inv (q')|\] and so by M\"obius inversion we obtain
\[(-1)^{\dim q}|f\inv(q)|=\sum_{q'\geq q}(-1)^{\dim q'}\mu_Q(q,q').\]  It follows that
\[|f\inv (q)|=\sum_{q'\geq q}(-1)^{\dim q'-\dim q}\mu_Q(q,q')\] as required.

The final statement follows because $(-1)^{\dim q'-\dim q}\mu_Q(q,q')\geq 0$ holds if $\Delta((q,q'))$ is Cohen-Macaulay (cf.~\cite[Theorem~10]{StanleyCohen}).
\end{proof}

Recall that, for a poset $Q$ with minimum $\wh 0$, one puts $\mu_Q(x) =\mu_Q(\wh 0,x)$.

\begin{Cor}
With $f\colon P\to Q$ as in Lemma~\ref{l:zaslavsky}, one has \[|f\inv(\wh 0)| = \sum_{q\geq \wh 0}(-1)^{\dim q}\mu_Q(\wh 0,q) = \sum_{q\geq \wh 0}(-1)^{\dim q}\mu_Q(q).\]  In addition, if each open interval $(\wh 0,q)$ of $Q$ is Cohen-Macaulay, then \[|f\inv(\wh 0)| = \sum_{q\in Q} |\mu_Q(q)|\] holds.
\end{Cor}

Applying this result to the case of a connected CW left regular band, we obtain the following result, generalizing the Las Vergnas-Zaslavsky Theorem~\cite[Theorem 4.6.1]{OrientedMatroids1999} for hyperplane arrangements and oriented matroids.

\begin{Thm}\label{t:cwlrbenum}
Let $B$ be a connected CW left regular band with support homomorphism $\sigma\colon B\to \Lambda(B)$.  Let $X\in \Lambda(B)$.  Then \[|\sigma\inv(X)|=\sum_{Y\geq X}|\mu_{\Lambda(B)}(X,Y)|\] where $\mu_{\Lambda(B)}$ is the M\"obius function of $\Lambda(B)$.  In particular, the size of the minimal ideal of $B$ is \[f_0(\CW(B))=\sum_{X\in \Lambda(B)}|\mu_{\Lambda(B)}(X)|.\]  More generally, \[f_k(\CW(B)) = \sum_{\dim A=k,B\geq A}|\mu_{\Lambda(B)}(A,B)|\] holds.
\end{Thm}
\begin{proof}
The support homomorphism $\sigma\colon B\to \Lambda(B)$ is a cellular mapping preserving strict inequalities by Proposition~\ref{p:supportiscellular}. The semilattice $\Lambda(B)$ has a minimum element and $\sigma\inv(\Lambda(B)_{\geq X})=B_{\geq X}$ is an acyclic CW poset for $X\in \Lambda(B)$ by Theorem~\ref{t:acyclicordercomplex} and the definition of a CW left regular band. Moreover, each open interval in $\Lambda(B)$ is Cohen-Macaulay by Theorem~\ref{t:cohenmac}.  Lemma~\ref{l:zaslavsky} now yields the result.
\end{proof}

Notice in particular, that the cardinalities of $B$ and all its $\mathscr L$-classes depend only on $\Lambda(B)$.  This can also be deduced from the fact that $\Bbbk B$ depends only on $\Lambda(B)$ up to isomorphism.

We now specialize Theorem~\ref{t:cwlrbenum} to the case of a CAT(0) cube complex in order to recover a result of Dress \textit{et al.}~\cite[Remark of Section~3]{dresscube}, which  is stated there in the language of median graphs.  To translate between our results and theirs, one should use the equivalence between CAT(0) cube complexes and median graphs and the bijection between hyperplanes and what these authors call splits, cf.~\cite{roller}.

\begin{Thm}\label{t:enum.Cat(0)}
Let $K$ be a finite CAT(0) cube complex with set of hyperplanes $\mathcal H$.  Let $\Gamma$ be the graph with vertex set $\mathcal H$ and where two hyperplanes are connected if they intersect.  Then we have
\[f_k(K) = \sum_{i\geq k}\binom{i}{k} f_{i-1}(\Cliq(\Gamma))\] where $f_{-1}(\Cliq(\Gamma))=1$ by convention.  In particular, $f_0(K)$ is one more than the number of non-empty simplices in $\Cliq(\Gamma)$.
\end{Thm}
\begin{proof}
If $K$ is a CAT(0) cube complex, then $\Lambda(\FFF(K))$ is isomorphic to the face poset, together with an empty face, of the nerve of the collection $\mathcal H$ by Proposition~\ref{p:intersectcat0faceposet} and Theorem~\ref{embedinhypercube}.  Note, though, that the empty simplex has rank $0$ and hence the rank of a simplex in $\Lambda(\FFF(K))$ is its  number of vertices, rather than it dimension.   Also, by the Helly property for CAT(0) cube complexes~\cite{chepoi}, the nerve of $\mathcal H$ is a flag complex and hence isomorphic to the clique complex of  $\Gamma$. Note that the interval $[X,Y]$ is isomorphic to a boolean algebra for each $X\leq Y$ in $\Lambda(\FFF(K))$.  Therefore, $\mu_{\Lambda(\FFF(K))}(X,Y)=(-1)^{\rk[X,Y]}$ and so $|\mu_{\Lambda(\FFF(K))}(X,Y)|=1$.

If $i\geq k$, then there are $f_{i-1}(\Cliq(\Gamma))$ elements $Y\in \Lambda(\FFF(K))$ with $i$ vertices.  Identifying $Y$ with an $i$-element clique in $\Gamma$, we see that $Y$ has $\binom{i}{k}$ subcliques with $k$ elements and hence there are $\binom{i}{k}$ choices of $X\in \Lambda(\FFF(K))$ with $X$ having rank $k$ and $X\leq Y$.  Denoting the number of vertices of $X\in \Lambda(\FFF(K))$ by $|X|$, it follows from Theorem~\ref{t:cwlrbenum} that
\begin{align*}
f_k(K)&=  \sum_{|X|=k,Y\geq X}|\mu_{\Lambda(\FFF(K))}(X,Y)|\\ &=\sum_{i\geq k}\sum_{|Y|=i}\left(\sum_{|X|=k, X\leq Y}|\mu_{\Lambda(\FFF(K))}(X,Y)|\right)\\ &=\sum_{i\geq k}\binom{i}{k}f_{i-1}(\Cliq(\Gamma))
\end{align*}
using that $|\mu_{\Lambda(\FFF(K))}(X,Y)|=1$.   This completes the proof.
\end{proof}

For example, if $T$ is a tree (i.e., a one-dimensional CAT(0) cube complex), then the hyperplanes of $T$ are the midpoints of edges.  No two hyperplanes intersect.  Theorem~\ref{t:enum.Cat(0)} then recovers, in this case, the well-known fact that the number of vertices of a tree is one more than the number of edges.

The next proposition generalizes~\cite[Proposition~4.6.2]{OrientedMatroids1999}. We use here that the support map $\sigma\colon B\to \Lambda(B)$ takes chains to chains of the same dimension.

\begin{Prop}\label{p:4.6.2}
Let $B$ be a connected CW left regular band and suppose that
$X_1<X_2<\cdots<X_k$ is a chain in $\Lambda(B)$. Then
\[|\sigma\inv(X_1,X_2,\ldots,X_k)|=\prod_{i=1}^k\sum_{X_i\leq Y\leq X_{i+1}}|\mu_{\Lambda(B)}(X_i,Y)|\]
where $X_k\leq Y\leq X_{k+1}$ is interpreted as $X_k\leq Y$.
\end{Prop}
\begin{proof}
We proceed by induction, the case $k=1$ being handled by Theorem~\ref{t:cwlrbenum}.  Assume it is true for $k-1$.  By Lemma~\ref{l:chainlifting}, to choose a preimage of $(X_1,X_2,\ldots, X_k)$ is to choose a preimage $x_k$ of $X_k$ in $B$ and then a preimage of $(X_1,X_2,\ldots, X_{k-1})$ in $x_kB$ under the restriction $\sigma\colon x_kB\to \Lambda(x_kB)=\Lambda(B)_{\leq X_k}$. Note that $x_kB$ is a connected CW left regular band.  By Theorem~\ref{t:cwlrbenum} there are $\sum_{X_k\leq Y}|\mu_{\Lambda(B)}(X_k,Y)|$ ways to choose $x_k$ and by induction there are then \[\prod_{i=1}^{k-1}\sum_{X_i\leq Y\leq X_{i+1}}|\mu_{\Lambda(B)}(X_i,Y)|\] ways to choose a preimage of $(X_1,\ldots,X_{k-1})$ in $x_kB$.  Thus
\[|\sigma\inv(X_1,X_2,\ldots,X_k)|=\prod_{i=1}^k\sum_{X_i\leq Y\leq X_{i+1}}|\mu_{\Lambda(B)}(X_i,Y)|\]
completing the induction.
\end{proof}

We can now prove the following analogue for connected CW left regular bands of a result of Bayer and Sturmfels~\cite[Corollary~4.6.3]{OrientedMatroids1999}.

\begin{Cor}\label{c:bayersturmfels}
The flag vector $\til f(\CW(B))$ for a connected CW left regular band $B$ depends only on its support semilattice $\Lambda(B)$.  More precisely, if $\dim B=d$ and $J=\{j_1,\ldots, j_k\}$ with $0\leq j_1<j_2<\cdots<j_k\leq d$, then
\[f_J(\CW(B)) = \sum\prod_{i=1}^k\sum_{X_i\leq Y\leq X_{i+1}}|\mu_{\Lambda(B)}(X_i,Y)|\] where the first sum runs over all chains $(X_1,\ldots, X_k)$ in $\Lambda(B)$ with $\dim X_i=j_i$ and where $X_k\leq Y\leq X_{k+1}$ is interpreted as $X_k\leq Y$.
\end{Cor}
\begin{proof}
Since $\sigma\colon B\to \Lambda(B)$ preserves dimensions and flags, it is clear that
\[f_J(\CW(B)) = \sum |\sigma\inv(X_1,\ldots, X_k)|\] where the sum runs over all chains $(X_1,\ldots, X_k)$ in $\Lambda(B)$ with $\dim X_i=j_i$. The result now follows from Proposition~\ref{p:4.6.2}.
\end{proof}

\subsection{Cartan invariants}
We use the previous results of this section to compute the Cartan matrix for a connected CW left regular band.
If $B$ is any left regular band, recall that the Cartan matrix of $\Bbbk B$ is the mapping $C\colon \Lambda(B)\times \Lambda(B)\to \mathbb N$ with $C_{X,Y}$ the multiplicity of the simple module $\Bbbk_X$ as a composition factor in the projective cover $\Bbbk L_Y$ of $\Bbbk_Y$.  Suppose that $X=Be_X$ for $X\in \Lambda(B)$.  Then Theorem~\ref{t:Cartan} states
\[C_{X,Y} = \sum_{Y\leq Z\leq X} |e_ZB\cap L_Y|\cdot \mu_{\Lambda(B)}(Z,X).\]
Therefore, we have by M\"obius inversion that
\begin{equation}\label{e:computecartan}
|e_XB\cap L_Y|= \sum_{Y\leq Z\leq X} C_{Z,Y}
\end{equation}
for $X\geq Y$.

Assume now that $B$ is a connected CW left regular band and $Y\leq X$ in $\Lambda(B)$.  Then $e_XB_{\geq Y}$ is a connected CW left regular band monoid with support lattice $[Y,X]$ and $e_XB\cap L_Y$ is the vertex set of this complex.  Thus, by Theorem~\ref{t:cwlrbenum}, we have that
\begin{equation}\label{e:computecartan2}
|e_XB\cap L_Y|=\sum_{Y\leq Z\leq X}|\mu_{\Lambda(B)}(Y,Z)|.
\end{equation}
A comparison of \eqref{e:computecartan} and \eqref{e:computecartan2}, together with an application of M\"obius inversion, yields $C_{X,Y} = |\mu_{\Lambda(B)}(Y,X)|$ for $X\geq Y$.  We have thus proved the following theorem, generalizing a result of the second author for hyperplane face monoids~\cite[Proposition~6.4]{Saliolahyperplane}.

\begin{Thm}\label{t:cartan.cw}
Let $B$ be a connected CW left regular band and $\Bbbk$ a field.  Then the Cartan matrix for $\Bbbk B$ is given by \[C_{X,Y} = \begin{cases} |\mu_{\Lambda(B)}(Y,X)|, & \text{if}\ X\geq Y\\ 0, & \text{else.}\end{cases}\]
\end{Thm}

\begin{Example}
If $K$ is a CAT(0) cube complex, then $\Lambda(\FFF(K))$ is isomorphic to the face poset, together with an empty face, of the nerve of the collection of hyperplanes of $K$ and hence each interval $[Y,X]$ is isomorphic to a boolean algebra.  Therefore, $\mu_{\Lambda(\FFF(K))}(Y,X)=\pm 1$ and so $|\mu_{\Lambda(\FFF(K))}(Y,X)|=1$.  Thus the Cartan matrix of $\Bbbk\FFF(K)$ is the zeta function of $\Lambda(\FFF(K))$.  This also follows from the fact that $\Bbbk\FFF(K)$ is isomorphic to the incidence algebra of $\Lambda(\FFF(K))$ by Corollary~\ref{c:cat0.as.inc}.
\end{Example}

\section{Cohomological dimension}\label{s:cohom}
In this section, we compute the cohomological dimension of a finite left regular band monoid.  Monoid cohomology has a long history, in part because of its connections with
string rewriting systems~\cite{Squier,brownrewriting,KobayashiFP1,KobayashiOtto,Lafont,Pridefiniteness,GubaPride,GubaPrideBurnside,Nicocohom1,AdamsRieffel,Nicocohom2,Nunes,GrayPride,CohenFP,GraySteinberg1,GraySteinberg2}. See~\cite{Browncohomology} for background on cohomology of groups and~\cite{Webb} for the cohomology of categories with applications.

\nomenclature[R, 11]{$\cdim_{\Bbbk} M$}{cohomological dimension of the module $M$}%
Let $M$ be a monoid.  Then the cohomology of $M$ with coefficients in a $\mathbb ZM$-module $V$ is defined by \[H^n(M,V) = \Ext_{\mathbb ZM}^n(\mathbb Z,V).\]  If $V$ is a $\Bbbk M$-module for some commutative ring with unit $\Bbbk$, then it is known that $\Ext_{\Bbbk M}^n(\Bbbk, V)= H^n(M,V)$ (use the standard bar resolution).  The \emph{$\Bbbk$-cohomological dimension}\index{$\Bbbk$-cohomological dimension} of $M$ is \[\cdim_{\Bbbk} M=\sup\{n\mid H^n(M,V)\neq 0,\ \text{for some $\Bbbk M$-module}\ V\}.\] Obviously $\cdim_{\Bbbk} M\leq \cdim_{\mathbb Z} M=\cdim M$.  Using the projective resolution of the trivial module from Theorem~\ref{t:projectiveres}, we can compute the cohomological dimension of a finite left regular band monoid over any base ring.

\begin{Thm}\label{t:cohomdim}
Let $B$ be a  left regular band monoid and $\Bbbk$ a commutative ring with unit. Then $\cdim_{\Bbbk} B$ is the largest $n$ such that  $\til H^{n-1}(\Delta(\bd e_YB);\Bbbk)\neq 0$ for some $Y\in \Lambda(B)$.
\end{Thm}
\begin{proof}
Theorem~\ref{t:projectiveres} shows that the trivial $\Bbbk B$-module admits a finite projective resolution by finitely generated $\Bbbk B$-modules.  Thus $\cdim_{\Bbbk} B<\infty$.  Suppose that $n=\cdim_{\Bbbk} B$.  Then the functor $H^n(B,-)$ is right exact by consideration of the long exact sequence for cohomology.  Because every module is a quotient of a free module, it then follows that $H^n(B,F)\neq 0$ for some free $\Bbbk B$-module $F$.
Since the trivial $\Bbbk B$-module admits a projective resolution that is finitely generated in each degree, $H^n(B,-)$ preserves direct limits by a result of Brown~\cite{Brownfiniteness}. As every free module is a direct limit of finitely generated free modules and $H^n(B,-)$ commutes with finite direct sums, it follows that $H^n(B,\Bbbk B)\neq 0$.  We conclude that $n$ is the largest non-negative integer for which $H^n(B,\Bbbk B)\neq 0$.

Let $J=\ker \sigma$ where $\sigma\colon \Bbbk B\to \Bbbk \Lambda(B)$ is induced by the support homomorphism.  Recall from Theorem~\ref{t:nilkernel} that $J$ is nilpotent and let $m\geq 1$ be such that $J^m=0$.  Then we have a filtration \[0=J^m\Bbbk B\subseteq J^{m-1}\Bbbk B\subseteq\cdots \subseteq J^0\Bbbk B=\Bbbk B.\]  It follows using the long exact sequence for cohomology that $\cdim_{\Bbbk} B$ is the largest integer $n\geq 0$ such that \[H^n(B,J^j\Bbbk B/J^{j+1}\Bbbk B)\neq 0\] for some $0\leq j\leq m-1$.  But each $J^j\Bbbk B/J^{j+1}\Bbbk B$ is a $\Bbbk \Lambda(B)$-module.  Thus   $\cdim_{\Bbbk} B$ is the largest integer $n$ such that $H^n(B,V)\neq 0$ for some $\Bbbk \Lambda(B)$-module $V$.  But, arguing as above, this is the same as the largest integer $n$ such that $H^n(B,\Bbbk\Lambda(B))\neq 0$.  We conclude from \eqref{eq:decomposelambdaB} that \[\cdim_{\Bbbk} B=\max\{n\mid H^n(B,\Bbbk_X)\neq 0,\ \text{for some}\ X\in \Lambda(B)\}\]  The result now follows from Theorem~\ref{t:extthm}.
\end{proof}

Let us state some corollaries.

\begin{Cor}
Let $B$ be a left regular band monoid and $\Bbbk$ a commutative ring with unit.  Then $\cdim_\Bbbk \mathsf S(B)=\cdim_\Bbbk B+1$
\end{Cor}
\begin{proof}
This is immediate from Theorem~\ref{t:cohomdim} and Corollary~\ref{c:suspension}.
\end{proof}
Note that in the next two corollaries we do \emph{not} require $B_{\geq X}$ to be a CW poset for all $X\in \Lambda(B)$.

\begin{Cor}\label{c:Cwposetcohom}
Let $B$ be a left regular band monoid which is a CW poset.  Then $\cdim_\Bbbk B=\cdim B=\dim \CW(B)$.  This applies in particular to real and complex hyperplane face monoids, oriented matroids and oriented interval greedoids.
\end{Cor}
\begin{proof}
The $\|\Delta(\bd e_YB)\|$ are precisely the boundary spheres of the closed cells of $\CW(B)$. The result is then immediate from Theorem~\ref{t:cohomdim}.
\end{proof}

\begin{Cor}
Let $B$ be a left regular band which is an acyclic CW poset.  Then $\cdim_\Bbbk B^1=\cdim B^1=\dim \CW(B)$.  This applies in particular to covector semigroups of COMs such as affine hyperplane arrangements, CAT(0) cube complexes and affine oriented matroids.
\end{Cor}
\begin{proof}
One has $\Delta(\bd B^1) = \Delta(B)$ is acyclic and $\bd e_YB^1=\bd e_YB$ for $Y\in \Lambda(B)$.
The $\|\Delta(\bd e_YB)\|$ are precisely the boundary spheres of the closed cells of $\CW(B)$. The result is then immediate from Theorem~\ref{t:cohomdim}.
\end{proof}

We end by giving a surprising connection between the cohomological dimension of free partially commutative left regular bands and the Leray number of flag complexes.  Connections between right-angled Artin groups and flag complexes can be found in~\cite{BestvinaBrady}.

\begin{Cor}
If $\Gamma$ is a graph, then the $\Bbbk$-cohomological dimension of the free partially commutative left regular band associated to $\Gamma$ is the $\Bbbk$-Leray number of $\Gamma$. In particular, $\cdim (B(\Gamma))\leq 1$ if and only if $\Gamma$ is chordal.
\end{Cor}
\begin{proof}
This is immediate from Proposition~\ref{p:topology} and Theorem~\ref{t:cohomdim}.
\end{proof}

It is well known that in $M$ and $N$ are monoids, then
\begin{equation}\label{chdir}
\max\{\cdim M,\cdim N\}\leq \cdim (M\times N)\leq \cdim (M)+\cdim(N)
\end{equation}
 cf.~\cite{GubaPride}. In particular, if $\cdim N=0$, then $\cdim (M\times N)=\cdim M$. Also, it is known that if $M$ is a monoid containing a right zero, then $\cdim M=0$. Indeed, if $e$ is a right zero then $\mathbb Z\cong \mathbb ZMe$ and so $\mathbb Z$ is projective and hence has projective dimension $0$. In particular, if $B$ is a left regular band monoid then $\cdim B^{op}=0$.  Recall that a \emph{regular band}\index{regular band} is a subdirect product of a left regular band and a right regular band (defined dually).

\begin{Cor}
For any pair of non-negative integers $m,n$, there is a finite regular band monoid $B$ with $\cdim B=m$ and $\cdim B^{op}=n$.
\end{Cor}
\begin{proof}
Recall the definition of ladders from Example~\ref{ladders}.
We have $\cdim L_m=m$ and $\cdim L_n=n$ by Corollary~\ref{c:Cwposetcohom} and so if $B=L_m\times L_n^{op}$, then $\cdim B =\cdim L_m=m$ and $\cdim B^{op}=\cdim L_n=n$ by \eqref{chdir}.
\end{proof}

Of course, we can replace $L_m$ and $L_n$ by face monoids of essential hyperplane arrangements in $\mathbb R^m$ and $\mathbb R^n$, respectively, and still get the same result.
In~\cite{GubaPride} it is shown that, for any monoid $M$ and for any $0\leq m,n\leq \infty$, there is a monoid $N$ containing $M$ as a submonoid with $\cdim N=m$ and $\cdim N^{op}=n$.

\section*{Acknowledgments}
We would like to thank Jonathan Leech for reading an earlier version of this manuscript and for sending us corrections and helpful comments.  We also would  like to thank the referee for a very thorough reading of the manuscript and for making extensive and very useful remarks.

\newpage
\bibliographystyle{abbrv}
\bibliography{lrbcomplexes}

\newpage
\printnomenclature[1.0in]

\newpage
\printindex

\end{document}